\documentclass[12pt]{book}

\usepackage{amsmath,amssymb,amsthm,amsfonts}
\usepackage{graphicx,pstricks}
\usepackage{graphics}
\usepackage{epsfig}
\usepackage{geometry}
\usepackage{setspace}

\theoremstyle{plain}
\newtheorem{thm}{Theorem}[chapter]
\newtheorem*{thmn}{Theorem}
\newtheorem{prop}[thm]{Proposition}
\newtheorem{lemma}[thm]{Lemma}
\newtheorem{cor}[thm]{Corollary}
\newtheorem{claim}[thm]{Claim}

\theoremstyle{definition}
\newtheorem*{defnn}{Definition}
\newtheorem{remark}[thm]{Remark}

\newcommand{\supp}{\operatorname{supp}}
\newcommand{\dist}{\operatorname{dist}}
\newcommand{\distt}{\operatorname{dist_{\mathbb{T}}}}
\newcommand{\distr}{\operatorname{dist_{\mathbb{R}}}}
\newcommand{\re}{\operatorname{Re}}
\newcommand{\im}{\operatorname{Im}}
\newcommand{\llogl}{\text{$L\log{\it L}$}}
\newcommand{\llogln}{\text{$L(\log{\it L})^n$}}
\newcommand{\esssup}{\operatorname{ess \, sup}}

\title{End-point Estimates and Multi-parameter Paraproducts on Higher Dimensional Tori}
\author{John Tyler Workman}

\begin{document}

\pagestyle{plain}

\maketitle

\null\vfill
\begin{center}
{\it To Shelby and Isaac.}
\end{center}
\vfill\null
\clearpage

\begin{center}
{\fontsize{18}{18} \bf Abstract}
\end{center}

\hskip0.5in

Analogues of multi-parameter multiplier operators on $\mathbb{R}^d$ are
defined on the torus $\mathbb{T}^d$.  It is shown that these operators
satisfy the classical Coifman-Meyer theorem.  In addition, $\llogl$ and
$\llogln$ end-point estimates are proved.
\clearpage

\begin{center}
{\fontsize{18}{18} \bf Biographical Sketch}
\end{center}

\hskip0.5in

John Tyler Workman was born September 9, 1981, to parents Jim and Marilyn
Workman in Chattanooga, Tennessee.  He attended Cleveland High School in
Cleveland, Tennessee, graduating as valedictorian in 1999.

He attended the University of Tennessee from 1999 to 2004 and the
Royal Melbourne Institute of Technology in Melbourne, Australia, during
a semester abroad in 2001.  As an undergraduate, John was honored with a
Barry M. Goldwater Research Scholarship, along with National Science
Foundation and Department of Defense Fellowships for future graduate
work.  He earned a Bachelor of Science degree in Mathematics in 2004.

In the same year, he began attending Cornell University as a graduate
student in pure mathematics.  He received the degree of Master of Science
in 2007 and Doctor of Philosophy in 2008.  To date, he has three
publications~\cite{me1, me2, me3}.

John was married in August of 2004 to Shelby Grant-Workman, then Shelby
Grant, with whom he has a son, Isaac Workman.
\clearpage

\begin{center}
{\fontsize{18}{18} \bf Acknowledgements}
\end{center}

\hskip0.5in

I would like to extend my sincerest thanks to my advisor Camil Muscalu.
First, for being a co-author of several papers on and introducing me to
the principal subject of this text.  But, more notably, for spending
hour upon hour in his office answering my questions on all things mathematical.

I would also like to thank the National Science Foundation and the Department
of Defense.  Being a beneficiary of the NSF Graduate Fellowship and NDSEG Fellowship,
this work, and all work I have done in graduate school, has been funded
by these organizations.
\clearpage

\begin{singlespacing}
\tableofcontents
\end{singlespacing}

\chapter*{Preface}
\addcontentsline{toc}{chapter}{Preface}

Consider the classical Marcinkiewicz multiplier operator $\Lambda_m$ on
$\mathbb{R}^d$ defined $\Lambda_m f(x) = \int_{\mathbb{R}^d} m(t)
\widehat{f}(t) e^{2\pi itx} \, dt$, where $m$ satisfies a standard
Marcinkiewicz-Mihlin-H\"{o}rmander type condition~\cite{marcin}.
This arises, in part, as a natural extension of the Hilbert
transform and Riesz transforms. In 1991, Coifman and
Meyer~\cite{meyer} considered a multilinear extension

\begin{equation*}
\Lambda_m(f_1, \ldots, f_n)(x) = \int_{\mathbb{R}^{dn}} m(t)
\widehat{f}_1(t_1) \cdots \widehat{f}_n(t_n) e^{2\pi ix(t_1 + \ldots
+ t_n)} \, dt,
\end{equation*}

\noindent where $m$, now acting on $\mathbb{R}^{dn}$, satisfies the
same kind of condition.  This operator is known to map $L^{p_1}
\times \ldots \times L^{p_n} \rightarrow L^p$ for $1/p_1 + \ldots +
1/p_n = 1/p$ and $1 < p_j < \infty$.  The case when $p \geq 1$ was
originally shown by Coifman and Meyer.  The general case $p > 1/n$
was settled later in~\cite{grafakos, kenig}.

An important application of this result occurs in non-linear partial
differential equations.  If $\widehat{D^\alpha f}(t) = |t|^{\alpha}
\widehat{f}(t)$, $\alpha > 0$, is the homogenous derivative, then

\begin{equation*}
\|D^\alpha (fg)\|_r \lesssim \|D^\alpha f\|_p \|g\|_q + \|f\|_p
\|D^\alpha g\|_q
\end{equation*}

\noindent for Schwartz functions $f, g$,
where $1 < p, q < \infty$ and $1/r = 1/p + 1/q$.  This inequality
was originally proved by Kato and Ponce~\cite{kato}, and can also be
established via the Coifman-Meyer theorem (see~\cite{camil1}).

In a more general setting, one can consider an operator $(D^\alpha_1
D^\beta_2 f) \, \widehat{} \,\, (t_1, t_2) = |t_1|^\alpha
|t_2|^\beta \widehat{f}(t_1, t_2)$ for $\alpha, \beta > 0$.  It is
natural to ask, then, is there an analogue to the inequality of Kato
and Ponce for this operator.  Heuristically, we should have
something like $\|D^\alpha_1 D^\beta_2 (fg)\|_r \lesssim
\|D^\alpha_1 D^\beta_2f\|_p \|g\|_q + \|D^\alpha_1 f\|_p \|D^\beta_2
g\|_q + \|D^\beta_2 f\|_p \|D^\alpha_1 g\|_q + \|f\|_p \|D^\alpha_1
D^\beta_2 g\|_q$.  Attempts to prove this kind of inequality by a
Coifman-Meyer type argument lead to a wider class of multipliers
$m$, which behave like the product of two standard multipliers.

Special cases of these multiplier operators had been previously
considered by Christ and Journ\'{e}~\cite{christ, journe}.  Muscalu
et.~al.~\cite{camil1} showed in 2004 that this so-called
bi-parameter multiplier operator satisfies the same $L^{p_1} \times
\ldots \times L^{p_n} \rightarrow L^p$ estimates. The original proof
for the Coifman-Meyer operator~\cite{meyer, grafakos, kenig}
involved the $T1$ theorem, BMO theory, and Carleson measures.  Many
of these methods, most notably the Calder\'{o}n-Zygmund
decomposition, do not extend to this bi-parameter setting.
In~\cite{camil1}, an entirely new method based on a strong geometric
structure and stopping time arguments is used.  This method was
further extended in~\cite{camil2} to show that arbitrary
multi-parameter multiplier operators satisfy the same bounds.

Another important side-effect of this method is its application to
the original Coifman-Meyer operator, giving a much simpler proof. In
particular, it establishes the ``end-point" estimates of the case
when any of the $p_j$ are equal to 1.  Here, we have $L^{p_1} \times
\ldots \times L^{p_n} \rightarrow L^{p,\infty}$. However, in the
multi-parameter setting of~\cite{camil1, camil2}, no such end-point
estimates are known.

A natural candidate for such an estimate would involve $\llogl$
spaces, because of how they arise in interpolation results. Naively,
it is often believed an operator which maps $L^1 \rightarrow
L^{1,\infty}$, and also satisfies some $L^p$ result, should take
$\llogl$ into $L^1$. However, it is rarely this straightforward.
In~\cite{jessen}, Jessen, Marcinkiewicz, and Zygmund showed that if
$f$ is in $\llogl$ then $Mf$ (the standard maximal function) is in
$L^1$.  But this was only for $f, Mf$ on $[0,1]$.
Wiener~\cite{wiener} improved this by noting that if $f$, defined on
all of $\mathbb{R}^n$, is in $\llogl$, then $Mf$ is locally
integrable. Stein~\cite{steinllogl} showed the converse is true.
Indeed, $Mf$ is locally integrable if and only if $f$ is locally in
$\llogl$.

Similarly, C. Fefferman~\cite{fefferman} examined the role of
$\llogl$ as an end-point estimate for the double Hilbert transform
and maximal double Hilbert transform.  Heuristically, a $\llogl$ to
weak-$L^1$ estimate should be expected.  Indeed, this is what is
shown, but truncated on the unit square.  That is, the maximal
double Hilbert transform maps $\llogl([0,1]^2)$ to
$L^{1,\infty}([0,1]^2)$.

This problem, that $\llogl$ estimates can only be gained in the
compact setting, is rather common.  Therefore, the desired end-point
estimate for the bi-parameter multiplier operator

\begin{equation*}
\Lambda_m : \llogl \times \ldots \times \llogl \rightarrow L^{1/n, \infty}
\end{equation*}

\noindent is likely to hold only in the compactified sense. However,
this leads to an interesting idea: analogues of multiplier operators
$\Lambda_m$, from the single-parameter case of Coifman and Meyer to the
multi-parameter situation of Muscalu~et.~al., instead defined on the
torus $\mathbb{T}^d$.  In this setting, that of a probability space,
$\llogl$ estimates are often cleaner and conceptually simpler. The
establishment and study of the correct operators on tori, and
in particular the desire for appropriate end-point estimates, is the
focus of this text.

The organization is as follows.  Chapter 1 is composed of three
parts. The first is a survey of some of the standard analytical
tools on the torus.  The second part is a series of somewhat
technical results concerning special smooth functions.  These
results are used sporadically throughout the text, but their proofs
are similar, so they are presented together.  The third part is an
interpolation theorem.  Chapter 2 covers several different maximal
operators, and Chapter 3 deals with a particular square function of
Littlewood-Paley type. In Chapter 4, characterizations of $\llogl$
and $\llogln$ are developed for any probability space, and several
important results therein are proved. Chapter 5 introduces and
studies single-parameter multipliers, and in particular, analogues
of the Marcinkiewicz and Coifman-Meyer multipliers.  In Chapter 6,
bi-parameter multiplier operators are handled.  Chapter 7 is a
non-rigorous survey of the proof for multi-parameter multipliers.

\chapter{The Circle and Smooth Functions}

\section{Preliminaries}\label{sec:prelim}

Consider the space $\mathbb{T} = \mathbb{R}/\mathbb{Z}$.  It has a
natural correspondence with a circle of diameter 1 or the interval
$[0,1) \subset \mathbb{R}$, where 0 and 1 are identified. In this
way, we can consider Lesbegue measure $m$ on sets $E \subseteq
\mathbb{T}$, by considering the corresponding set in $[0,1)$.
Then, $(\mathbb{T}, m)$ is a probability space.

Addition is also naturally defined on $\mathbb{T}$ by the group
structure of $\mathbb{R}/\mathbb{Z}$. That is, $x, y \in \mathbb{T}$
can be thought of as elements in $[0,1)$, and $x + y$ in $\mathbb{T}$
is $(x+y)$ mod-1 in $\mathbb{R}$.

Let $\distr(\cdot, \cdot)$ be the Euclidean metric on $\mathbb{R}$,
and $\distt(\cdot, \cdot)$ the standard metric on $\mathbb{T}$
induced by the geometry of the circle.  In particular, if $x, y \in
\mathbb{T}$ are thought of as elements in $[0,1)$, then,
$\distt(x,y) = \min\{\distr(x,y), 1 - \distr(x,y)\}$.  For sets $A,
B \subseteq \mathbb{T}$, let $\distt(x, A) = \min\{ \dist(x,y) : y
\in A\}$ and $\distt(A, B) = \min\{\distt(x, y) : x \in A, y \in
B\}$, as usual.

Functions $f$ acting on $\mathbb{T}$ can simultaneously be thought
of as 1-periodic functions acting on $\mathbb{R}$.  In this way, we
define integration on $(\mathbb{T}, m)$ by

\begin{equation*}
\int_\mathbb{T} f \, dm = \int_0^1 f(x) \, dx,
\end{equation*}

\noindent where the function on the right is defined on $\mathbb{R}$
and integrated over $[0,1)$. Further, we inherit from $\mathbb{R}$
notions of continuity, differentiability, smoothness, etc..  We will
most often consider complex-valued functions, which we write as $f :
\mathbb{T} \rightarrow \mathbb{C}$.  This notation is somewhat
misleading, as we allow functions to take infinite values.

For complex scalars $\alpha$, we will use $|\alpha|$ to denote the
modulus or absolute value, and we will denote Lebesgue measure of a
set $A$ by $|A|$.  This double use should not
cause any confusion.  We say two sets $A, B$ are disjoint if $|A
\cap B| = 0$.

There is a natural notion of intervals in $\mathbb{T}$ as well, that
is, connected subsets.  We will always use the terminology interval
to mean a non-empty, closed interval in $\mathbb{T}$.  For
simplicity, we allow $\mathbb{T}$ to be considered an interval.  We
say an interval $I$ is dyadic if $I = [2^{-k} j, 2^{-k} (j+1)]$ for
some $j \in \mathbb{Z}$, $k \in \mathbb{N}$.  Note that $\mathbb{T}$
itself is not considered a dyadic interval.  One can easily show
that there is a kind of trichotomy: for any two dyadic intervals
either they are equal, one is strictly contained in the other, or
they are disjoint.

For any interval $I$ and $0 \leq \alpha \leq 1/|I|$, let $\alpha I$
denote the interval concentric with $I$ which satisfies $|\alpha I|
= \alpha |I|$. That is, if $I = \{x : \distt(x,x_I) \leq |I|/2\}$,
then $\alpha I = \{x : \distt(x,x_I) \leq \alpha|I|/2\}$. For
integers $n$, let $I^n = I + n|I|$, the interval gained by shifting
$n$ steps of length $|I|$.

Finally, we will use the somewhat standard notation $A \lesssim B$
to mean that there is some ``universal" constant $C$ such that $A
\leq C \cdot B$.  We will write $A \sim B$ if $A \lesssim B$ and $B
\lesssim A$.  It will be our attempt throughout to make as clear as
possible precisely what these unspoken constants depend on.

\section{Analysis on $\mathbb{T}$}\label{sec:analysisont}

Many of the fundamental analytical tools which we use on
$\mathbb{R}^n$ can be easily extended to $\mathbb{T}$.
Katznelson~\cite{katznelson} gives a comprehensive
introduction to this topic.

Considering the probability space $(\mathbb{T}, m)$, we can define
$\|f\|_p = (\int_\mathbb{T} |f|^p \, dm)^{1/p}$ for $0 < p < \infty$
and $\|f\|_\infty = \esssup_\mathbb{T} |f|$ as normal.  Then, the
spaces $L^p(\mathbb{T})$ of functions for which $\|f\|_p < \infty$
are Banach spaces for $1 \leq p \leq \infty$.  Similarly, we define
weak-$L^p(\mathbb{T})$ or $L^{p,\infty}(\mathbb{T})$ as the
functions for which

\begin{equation*}
\|f\|_{p,\infty} := \sup_{\lambda > 0} \lambda \big| \big\{ x \in
\mathbb{T} : |f(x)| > \lambda \big\} \big|^{1/p} < \infty.
\end{equation*}

\noindent Note, $\|\cdot\|_{p,\infty}$ is only a quasi-norm, in that
it does not always satisfy the triangle inequality.  However, it is
true that $|f| \leq |g|$ a.e.~implies $\|f\|_{p,\infty} \leq
\|g\|_{p,\infty}$ and $f_n \uparrow |f|$ a.e.~implies
$\|f_n\|_{p,\infty} \uparrow \|f\|_{p,\infty}$.

Denote the $L^2$ inner product by $\langle \cdot, \cdot \rangle$,
i.e., $\langle f, g \rangle = \int_\mathbb{R} f(x) \overline{g(x)}
\, dx$, where $\overline{g}$ is the complex conjugate.  It will be
our practice, when studying an operator $T$, to write $T : L^p
\rightarrow L^p$ or that $T$ maps $L^p$ to $L^p$, when it is
actually meant maps boundedly.  In particular, that there is some
constant $C$ so that $\|Tf\|_p \leq C \|f\|_p$ for all $f$.

For $f \in L^1(\mathbb{T})$ we define the Fourier coefficients of
$f$ by

\begin{equation*}
\widehat{f}(n) = \int_{\mathbb{T}} f(x) e^{-2\pi inx} \, dx
\end{equation*}

\noindent for all $n \in \mathbb{Z}$.  It is easily
shown~\cite{katznelson} that the usual properties hold.  In
particular, this operation is linear, and if we define convolution
as

\begin{equation*}
(f * g)(x) = \int_\mathbb{T} f(y)g(x-y) \, dy,
\end{equation*}

\noindent then $\widehat{(f*g)}(n) = \widehat{f}(n) \widehat{g}(n)$
for all $n$. Further, we have a version of Plancherel's theorem: for
$f, g \in L^2(\mathbb{T})$

\begin{equation*}
\langle f, g \rangle = \sum_{n \in \mathbb{Z}} \widehat{f}(n) \,
\overline{\widehat{g}(n)}
\end{equation*}

\noindent or equivalently,

\begin{equation*}
\int_\mathbb{T} f(x) g(x) \, dx = \sum_{n \in \mathbb{Z}}
\widehat{f}(n) \widehat{g}(-n).
\end{equation*}

\noindent It is also well-known that if a function $f$ is smooth

\begin{equation*}
f(x) = \sum_{n \in \mathbb{Z}} \widehat{f}(n) e^{2\pi inx}.
\end{equation*}

Recall that a function $f : \mathbb{R} \rightarrow \mathbb{C}$ is a
Schwartz function~\cite{stein3, stein2} if it is infinitely
differentiable and $\sup_\mathbb{R} |x|^k |f^{(l)}(x)| < \infty$ for
all integers $k, l \geq 0$.  For a Schwartz function $f$, define its
periodization by

\begin{equation*}
F(x) = \sum_{j \in \mathbb{Z}} f(x + j).
\end{equation*}

\noindent This function is clearly 1-periodic, so we may think of
$F$ as a function on $\mathbb{T}$.  This sum converges absolutely
for all $x$, which follows because $|f(x+j)| \leq C |x+j|^{-2}$ for
some $C$ is guaranteed by the Schwartz condition, . For $h \not= 0$,
we have by the mean value theorem that $\frac{1}{h} [F(x+h) - F(x)]
= \sum_j f'(x_{j,h})$, where $x_{j,h}$ is some number between $x+j$
and $x+j+h$.  Using the Schwartz property again, we can apply the
dominated convergence theorem to let $h \rightarrow 0$ and see $F$
is differentiable, with $F'(x) = \sum_j f'(x+j)$.  Iterating this,
we find that $F$ is smooth (infinitely differentiable) and $F^{(l)}$
is simply the periodization of $f^{(l)}$.

Furthermore,

\begin{equation*}
\begin{split}
\widehat{F}(n) &= \int_\mathbb{T} F(x) e^{-2\pi inx} \, dx =
\int_0^1 F(x) e^{-2\pi inx} \, dx \\
&= \sum_j \int_0^1 f(x + j) e^{-2\pi inx} \, dx = \int_{\mathbb{R}}
f(x) e^{-2\pi inx} \, dx = \widehat{f}(n).
\end{split}
\end{equation*}

\noindent That is, the Fourier coefficients of $F$ coincide with the
Fourier transform of $f$ on the integers.

Finally, we can define $\mathbb{T}^d = \mathbb{R}^d/\mathbb{Z}^d$,
the $d$-fold product of $\mathbb{T}$. Lebesgue measure can be
attained from $\mathbb{R}^d$ just as before, or as the appropriate
product measure, so that $(\mathbb{T}^d, m)$ is a probability space.
Functions $f : \mathbb{T}^d \rightarrow \mathbb{C}$ can be thought
of as functions on $\mathbb{R}^d$ which are 1-periodic in each
coordinate, and integration is defined as before.  The Fourier
coefficients $\widehat{f}(n_1, \ldots, n_d) = \int_{\mathbb{T}^d}
f(\vec{x}) e^{-2\pi i \vec{n} \cdot \vec{x}} \, d\vec{x}$ are
defined in a natural way, and all the normal results hold.

\section{Bump Functions}

Our first goal will be to generate a sequence of smooth functions
whose Fourier coefficients are a kind of ``partition of unity."  It
turns out the easiest way to do this is to first create the
functions on $\mathbb{R}$ and then periodize.

\begin{thm}\label{thm:rbumps} There are Schwartz functions
$\theta_k^1, \theta_k^2 : \mathbb{R} \rightarrow \mathbb{C}$, $k \in
\mathbb{Z}$, and constants $C_m > 0$, $m \in \mathbb{N}$, so that

\begin{gather*}
\sum_{k \in \mathbb{Z}} \widehat{\theta^1_k}(t)
\widehat{\theta^2_k}(-t) = \chi_{\mathbb{R} - 0}(t) , \\
\supp(\widehat{\theta^1_k}) \subseteq
[-2^{k-2}, -2^{k-4}] \cup [2^{k-4}, 2^{k-2}], \,\, \widehat{\theta_k^2}(0) = 0, \\
|\theta_k^1(x)|, |\theta_k^2(x)| \leq 2^k C_m \Big(1 + 2^k \distr(x,
[0, 2^{-k}]) \Big)^{-m} \,\, \text{ for all } x \in \mathbb{R}, \, m
\in \mathbb{N}, \\
|\theta^{1'}_k(x)|, |\theta^{2'}_k(x)| \leq 4^k C_m \Big(1 + 2^k
\distr(x, [0, 2^{-k}]) \Big)^{-m} \,\, \text{ for all } x \in
\mathbb{R}, \, m \in \mathbb{N}.
\end{gather*}
\end{thm}

\begin{proof} Choose a Schwartz function $\alpha : \mathbb{R}
\rightarrow \mathbb{C}$ so that $\widehat{\alpha} = 1$ on $[-1/8,
1/8]$ and $\supp(\widehat{\alpha}) \subseteq [-1/4, 1/4]$. Define
$\widehat{\theta^1}(t) = \widehat{\alpha}(t) -
\widehat{\alpha}(2t)$.  Let $\widehat{\theta^1_k}(t) =
\widehat{\theta^1}(2^{-k}t)$ for all $k \in \mathbb{Z}$.

Fix $t \not= 0$.  Choose any $N \in \mathbb{N}$ so that $|t| \leq
2^{N - 3}$ and $|t| > 2^{-N - 3}$.  Then,

\begin{equation*}
\begin{split}
\sum_{k=-N}^N \widehat{\theta^1_k}(t) &= \Big(
\widehat{\alpha}(2^{N}t) - \widehat{\alpha}(2^{N+1}t) \Big) + \Big(
\widehat{\alpha}(2^{N-1}t) - \widehat{\alpha}(2^{N}t)
\Big) + \ldots + \\
&\qquad \Big( \widehat{\alpha}(2^{-N}t) -
\widehat{\alpha}(2^{-N+1}t) \Big) \\
&= \widehat{\alpha}(2^{-N}t) - \widehat{\alpha}(2^{N+1}t) = 1 - 0 =
1.
\end{split}
\end{equation*}

\noindent As this holds for all $N$ big enough, and $t$ is
arbitrary, it follows that $\sum_k \widehat{\theta^1_k}(t) = 1$ for
all $t \not= 0$.  On the other hand, as $\widehat{\theta^1_k}(0) =
\widehat{\theta^1}(0) = 0$ for all $k$, it is clear the sum is 0 at
$t = 0$.

Fix $k \in \mathbb{Z}$.  Let $|t| \leq 2^{k - 4}$.  Then $|2^{-k}t|,
|2^{-k + 1}t| \leq 1/8$ implying $\widehat{\theta^1_k}(t) =
\widehat{\theta^1}(2^{-k}t) = \widehat{\alpha}(2^{-k}t) -
\widehat{\alpha}(2^{-k + 1}t) = 1 - 1 = 0$. Similarly, $|t| >
2^{k-2}$ implies $|2^{-k + 1}t|, |2^{-k}t| > 1/4$, and
$\widehat{\theta^1_k}(t) = \widehat{\alpha}(2^{-k}t) -
\widehat{\alpha}(2^{-k + 1}t) = 0$.  That is,
$\supp(\widehat{\theta^1_k}) \subseteq [-2^{k-2}, -2^{k-4}] \cup
[2^{k-4}, 2^{k-2}]$.

Choose a Schwartz function $\theta^2$ so that $\widehat{\theta^2} =
1$ on $[-1/8, -1/16] \cup [1/16, 1/8]$ and is supported away from 0.
Define $\widehat{\theta^2_k}(t) = \widehat{\theta^2}(2^{-k}t)$.
Then, $\widehat{\theta^2_k}(0) = 0$ and $\widehat{\theta^2_k} = 1$
on $[-2^{k-2}, -2^{k-4}] \cup [2^{k-4}, 2^{k-2}] \supseteq
\supp(\widehat{\theta^1_k})$, so that

\begin{equation*}
\sum_{k \in \mathbb{Z}} \widehat{\theta^1_k}(t)
\widehat{\theta^2_k}(-t) = \sum_{k \in \mathbb{Z}}
\widehat{\theta^1_k}(t) = \chi_{\mathbb{R} - 0}(t).
\end{equation*}

Finally, note that $\theta^i_k(x) = 2^k \theta^i(2^k x)$ for $i =
1,2$. As $\theta^i$ and $\theta^{i'}$ are Schwartz functions and $(1
+ \distr(x, [0, 1]))^m$ has polynomial growth, we can choose $C_m$
so that $|\theta^i(x)|, |\theta^{i'}(x)| \leq C_m(1 + \distr(x,
[0,1]))^{-m}$ for all $x$ and $m$ and $i = 1,2$.  Then,

\begin{equation*}
\begin{split}
|\theta_k^i(x)| &= 2^k |\theta^i(2^kx)| \leq 2^k C_m (1 + \distr(2^k
x, [0,1]))^{-m} \\
&= 2^k C_m (1 + 2^k \distr(x, [0, 2^{-k}]))^{-m}.
\end{split}
\end{equation*}

\noindent By the same argument, $|\theta^{i'}_k(x)| = 4^k |\theta'(2^k)| \leq 4^k
C_m (1 + 2^k \distr(x, [0, 2^{-k}]))^{-m}$.
\end{proof}

\begin{claim}\label{claim:stupidf} Fix $j, k \in \mathbb{N}$ and
define $f(t) = j(1 + 2^k \min(t, 1-t)) - 2^k(t+j-1)$.  For any $t
\in [0,1]$, $f(t) \leq 1$.
\end{claim}

\begin{proof} For $t \in [0,1/2]$, $f(t) = j(1 + 2^kt) -
2^k(t+j-1)$, which is an increasing linear function in $t$.  Indeed,
$f'(t) = j2^k - 2^k \geq 0$.  For $t \in [1/2,1]$, $f(t) = j(1 +
2^k(1-t)) - 2^k(t+j-1)$, which is a decreasing linear function in
$t$, as $f'(t) = -j2^k - 2^k < 0$.  Thus, $\max_{x \in [0,1]} f(t) =
f(1/2) = j(1 + 2^{k-1}) - 2^k(j - 1/2) = j + 2^{k-1} - j2^{k-1} \leq
1$.  This last inequality follows as $a + b - ab \leq 1$ for any
positive integers $a, b$, which is easily shown through induction.
\end{proof}

\begin{lemma}\label{lemma:adapted} Let $\theta_k : \mathbb{R}
\rightarrow \mathbb{C}$ be a Schwartz function and $\psi_k :
\mathbb{T} \rightarrow \mathbb{C}$ its periodization.  If

\begin{equation*}
\begin{split}
|\theta_k(x)| &\leq C_m 2^k \big(1 + 2^k \distr(x,
[0,2^{-k}])\big)^{-m} \quad \text{ and } \\
|\theta'_k(x)| &\leq C_m 4^k \big(1 + 2^k \distr(x, [0,
2^{-k}])\big)^{-m},
\end{split}
\end{equation*}

\noindent then there exist constants $C_m'$ so that

\begin{equation*}
\begin{split}
|\psi_k(x)| &\leq C_m' 2^k \big(1 + 2^k \distt(x,
[0,2^{-k}])\big)^{-m} \quad \text{ and } \\
|\psi'_k(x)| &\leq C_m' 4^k \big(1 + 2^k \distt(x, [0,
2^{-k}])\big)^{-m}.
\end{split}
\end{equation*}
\end{lemma}

\begin{proof} Fix $x \in [0,1)$.  Clearly, $\distr(x, [0,2^{-k}])
\geq \distt(x, [0,2^{-k}])$.  Hence, $|\theta_k(x)| \leq C_m 2^k (1 +
\distr(x, [0,2^{-k}]))^{-m} \leq C_m 2^k (1 + \distt(x,
[0,2^{-k}]))^{-m}$.  For any $j \in \mathbb{N}$, note $\distr(x+j,
[0, 2^{-k}]) \geq \distr(x, [0, 2^{-k}]) + j - 1$. Set $t =
\distr(x, [0,2^{-k}])$, and observe that $t \in [0, 1]$ and
$\distt(x, [0,2^{-k}]) \leq \min(t, 1-t)$. Thus, by
Claim~\ref{claim:stupidf},

\begin{equation*}
\begin{split}
j\big(1 + 2^k \distt(x, [0,2^{-k}])\big) &\leq j\big(1 + 2^k \min(t,
1-t)\big) = f(t) + 2^k(t + j - 1) \\
&\leq 1 + 2^k(t + j - 1) \leq 1 + 2^k \distr(x+j, [0, 2^{-k}]).
\end{split}
\end{equation*}

\noindent Therefore, we see that for any integer $m > 1$,

\begin{equation*}
\begin{split}
\sum_{j=1}^\infty |\theta_k(x+j)| &\leq \sum_{j=1}^\infty C_m 2^k
\Big(1 + \distr(x+j, [0,2^{-k}])\Big)^{-m} \\
&\leq C_m 2^k \sum_{j=1}^\infty j^{-m} \Big(1 + 2^k \distt(x, [0,
2^{-k}]) \Big)^{-m} \\
&\leq 2 C_m 2^k \Big(1 + 2^k \distt(x, [0, 2^{-k}]) \Big)^{-m}
\end{split}
\end{equation*}

Similarly, $\distr(x-1, [0, 2^{-k}]) = \distr(x-1,0) = \distr(x,1)
\geq \distt(x,1) \geq \distt(x, [0,2^{-k}])$. Therefore,
$|\theta_k(x-1)| \leq C_m 2^k (1 + \distr(x-1,[0,2^{-k}]))^{-m} \leq
C_m 2^k (1 + 2^k \distt(x, [0,2^{-k}]))^{-m}$, and for $j \in
\mathbb{N}$, we have $\distr(x-j,[0,2^{-k}]) =
\distr(x-1,[0,2^{-k}])+j-1$. Set $t = \distr(x-1, [0,2^{-k}])$, and
again observe that $t \in [0,1]$ and $\distt(x, [0,2^{-k}]) \leq
\min(t, 1-t)$.  Using the claim as before, it follows that $j(1 +
2^k \distt(x, [0,2^{-k}])) \leq 1 + 2^k \distr(x-j,[0,2^{-k}])$.
Thus, for $m > 1$,

\begin{equation*}
\begin{split}
\sum_{j=2}^\infty |\theta_k(x-j)| &\leq \sum_{j=2}^\infty C_m 2^k
\Big(1 + \distr(x-j, [0,2^{-k}])\Big)^{-m} \\
&\leq C_m 2^k \sum_{j=2}^\infty j^{-m} \Big(1 + 2^k \distt(x, [0,
2^{-k}]) \Big)^{-m} \\
&\leq  2 C_m 2^k \Big(1 + 2^k \distt(x, [0, 2^{-k}]) \Big)^{-m}
\end{split}
\end{equation*}

\noindent Hence,

\begin{equation*}
\begin{split}
|\psi_k(x)| &\leq \sum_{j \in \mathbb{Z}} |\theta_k(x+j)| \\
&= |\theta_k(x)| + |\theta_k(x-1)| + \sum_{j=1}^\infty
|\theta_k(x+j)| + \sum_{j=2}^\infty |\theta_k(x-j)| \\
&\leq (C_m + C_m + 2C_m + 2C_m) 2^k \Big(1 + 2^k \distt(x,
[0,2^{-k}])\Big)^{-m}.
\end{split}
\end{equation*}

Now, this holds for all $m > 1$.  But, of course, the $m = 1$ case
follows as $|\psi_k(x)| \leq 6C_2 2^k (1 + 2^k \distt(x, [0,
2^{-k}]))^{-2} \leq 6C_2 2^k (1 + 2^k \distt(x, [0, 2^{-k}]))^{-1}$.
The condition on $\psi'_k$ is proven in exactly the same manner.
Thus, the statement holds with $C_m' = 6C_m$ for $m > 1$ and $C_1' =
6C_2$.
\end{proof}

\begin{thm}\label{thm:tbumps} There are smooth functions $\psi_k^1, \psi_k^2 :
\mathbb{T} \rightarrow \mathbb{C}$, $k \in \mathbb{N}$, and
constants $C_m > 0$, $m \in \mathbb{N}$, so that

\begin{gather*}
\sum_{k = 1}^\infty \widehat{\psi^1_k}(n) \widehat{\psi^2_k}(-n) =
\chi_{\mathbb{Z} - 0}(n) , \\
\supp(\widehat{\psi^1_k}) \subseteq
[-2^{k-2}, -2^{k-4}] \cup [2^{k-4}, 2^{k-2}], \,\, \widehat{\psi^2_k}(0) = 0, \\
|\psi^1_k(x)|, |\psi^2_k(x)| \leq 2^k C_m \Big(1 + 2^k \distt(x, [0,
2^{-k}]) \Big)^{-m} \,\, \text{ for all } x \in \mathbb{T}, \, m \in \mathbb{N}, \\
|\psi^{1'}_k(x)|, |\psi^{2'}_k(x)| \leq 4^k C_m \Big(1 + 2^k
\distt(x, [0, 2^{-k}]) \Big)^{-m} \,\, \text{ for all } x \in
\mathbb{T}, \, m \in \mathbb{N}.
\end{gather*}
\end{thm}

\begin{proof} Let $\theta^1_k, \theta_k^2 : \mathbb{R} \rightarrow
\mathbb{C}$, $k \in \mathbb{Z}$, be the functions guaranteed by
Theorem~\ref{thm:rbumps}, and $\psi^1_k$, $\psi^2_k$ their
respective periodizations.  As $\widehat{\theta_k^i}(n) =
\widehat{\psi^i_k}(n)$, it follows $\widehat{\psi^2_k}(0) = 0$ and
$\supp(\widehat{\psi^1_k}) \subseteq [-2^{k-2}, -2^{k-4}] \cup
[2^{k-4}, 2^{k-2}]$. From this, we see for any integer $n$ that
$\widehat{\psi_k^1}(n) \widehat{\psi^2_k}(-n) = 0$ for $k \leq 0$.
Thus,

\begin{equation*}
\sum_{k=1}^\infty \widehat{\psi^1_k}(n) \widehat{\psi^2_k}(-n) =
\sum_{k \in \mathbb{Z}} \widehat{\theta^1_k}(n)
\widehat{\theta^2_k}(-n) = \chi_{\mathbb{Z}-0}(n).
\end{equation*}

\noindent Finally, the inequalities on $\psi_k^i$ and $\psi_k^{i'}$
follow from Theorem~\ref{thm:rbumps} and Lemma~\ref{lemma:adapted}.
\end{proof}

\begin{thm}\label{thm:doublerbumps} There are Schwartz functions
$\theta^{a,i}_k : \mathbb{R} \rightarrow \mathbb{C}$, $1 \leq a, i
\leq 3$, $k \in \mathbb{Z}$, and constants $C_m > 0$, $m \in
\mathbb{N}$, so that

\begin{gather*}
\sum_{a=1}^3 \sum_{k \in \mathbb{Z}} \widehat{\theta^{a,1}_k}(t_1)
\widehat{\theta^{a,2}_k}(t_2) \widehat{\theta^{a,3}_k}(-t_1-t_2) =
\chi_{\mathbb{R}^2 - (0,0)}(t_1,t_2) \\
\supp(\widehat{\theta^{a,i}_k}) \subseteq [-2^{k-2}, -2^{k-10}] \cup
[2^{k-10}, 2^{k-2}] \quad \text{ for } a \not= i, \\
\supp(\widehat{\theta^{a,i}_k}) \subseteq [-2^{k-2}, 2^{k-2}] \quad \text{ for } a = i, \\
|\theta^{a,i}_k(x)| \leq 2^k C_m \Big(1 + 2^k \distr(x, [0, 2^{-k}])
\Big)^{-m} \,\, \text{ for all } x \in \mathbb{R}, \, m \in \mathbb{N}, \\
|\theta^{a,i'}_k(x)| \leq 4^k C_m \Big(1 + 2^k \distr(x, [0,
2^{-k}]) \Big)^{-m} \,\, \text{ for all } x \in \mathbb{R}, \, m \in
\mathbb{N}.
\end{gather*}
\end{thm}

\begin{proof} Similar to the proof of Theorem~\ref{thm:rbumps}, start
with a Schwartz bump $\widehat{\alpha}$ which is identically 1 on
$[-1/64, 1/64]$ and supported in $[-1/32,1/32]$.  Set
$\widehat{\beta}(t) = \widehat{\alpha}(t) - \widehat{\alpha}(2t)$.
Define $\widehat{\beta^1_k}(t) = \widehat{\beta}(2^{-k}t)$ and
$\widehat{\beta^2_k}(t) = \widehat{\alpha}(2^{-k+3}t)$.  Set
$\widehat{\beta_k^3}(t) = \sum_{j = k-2}^{k+2}
\widehat{\beta^1_j}(t)$.

By construction of $\alpha$, we can see $\supp(\widehat{\beta_k^2})
\subseteq [-2^{k-8}, 2^{k-8}]$.  By an argument similar to that in
Theorem~\ref{thm:rbumps}, $\supp(\widehat{\beta_k^1}) \subseteq
[-2^{k-5}, -2^{k-7}] \cup [2^{k-7}, 2^{k-5}]$.  Thus,
$\supp(\widehat{\beta_k^3}) \subseteq [-2^{k-3}, -2^{k-9}] \cup
[2^{k-9}, 2^{k-3}]$.

Fix $t \in \mathbb{R}$, $t \not= 0$, and choose $N \in \mathbb{N}$
so that $|t| > 2^{-N-6}$.  Then, $|2^{N+1}t| > 1/32$ and by the same
telescoping argument as before

\begin{equation*}
\begin{split}
\sum_{j=-N}^{k-3} \widehat{\beta^1_j}(t) &= \Big(
\widehat{\alpha}(2^{N}t) - \widehat{\alpha}(2^{N+1}t) \Big) + \ldots
+ \Big( \widehat{\alpha}(2^{-k+3}t) - \widehat{\alpha}(2^{-k+3}t)
\Big) \\
&= \widehat{\alpha}(2^{-k+3}t) - \widehat{\alpha}(2^{N+1}t) =
\widehat{\alpha}(2^{-k+3}t) = \widehat{\beta^2_k}(t).
\end{split}
\end{equation*}

\noindent As $N$ and $t$ are arbitrary, we have that $\sum_{j < k -
2} \widehat{\beta^1_j}(t) = \widehat{\beta^2_k}(t)$ for $t \not= 0$.
By the same argument used in Theorem~\ref{thm:rbumps}, $\sum
\widehat{\beta^1_k}(t) = 1$ for all $t \not= 0$.

Fix $t_1, t_2 \in \mathbb{R}$, both non-zero.  Then,

\begin{equation*}
\begin{split}
1 &= \bigg( \sum_{k_1 \in \mathbb{Z}} \widehat{\beta^1_{k_1}}(t_1)
\bigg)
\bigg( \sum_{k_2 \in \mathbb{Z}} \widehat{\beta^1_{k_2}}(t_2) \bigg) \\
&= \sum_{k_1 \in \mathbb{Z}} \sum_{k_2 > k_1 + 2}
\widehat{\beta^1_{k_1}}(t_1) \widehat{\beta^1_{k_2}}(t_2) +
\sum_{k_1 \in \mathbb{Z}} \sum_{k_2 < k_1 - 2}
\widehat{\beta^1_{k_1}}(t_1) \widehat{\beta^1_{k_2}}(t_2) \\
&\qquad + \sum_{k_1 \in \mathbb{Z}} \sum_{k_2 = k_1 - 2}^{k_1 + 2}
\widehat{\beta^1_{k_1}}(t_1) \widehat{\beta^1_{k_2}}(t_2) \\
&= \sum_{k \in \mathbb{Z}} \widehat{\beta^2_{k}}(t_1)
\widehat{\beta_{k}^1}(t_2) + \sum_{k \in \mathbb{Z}}
\widehat{\beta^1_{k}}(t_1) \widehat{\beta_{k}^2}(t_2) + \sum_{k \in
\mathbb{Z}} \widehat{\beta^1_{k}}(t_1) \widehat{\beta_{k}^3}(t_2).
\end{split}
\end{equation*}

\noindent On the other hand, $\widehat{\beta^2_k}(0) =
\widehat{\alpha}(0) = 1$.  Hence, in the $t_1 = 0$ case, we see that
for any $t_2 \not= 0$

\begin{gather*}
\sum_{k \in \mathbb{Z}} \widehat{\beta^2_{k}}(0)
\widehat{\beta_{k}^1}(t_2) + \sum_{k \in \mathbb{Z}}
\widehat{\beta^1_{k}}(0) \widehat{\beta_{k}^2}(t_2) + \sum_{k \in
\mathbb{Z}} \widehat{\beta^1_{k}}(0) \widehat{\beta_{k}^3}(t_2) =
\sum_{k \in \mathbb{Z}} \widehat{\beta_{k}^1}(t_2) = 1.
\end{gather*}

\noindent The $t_2 = 0$ case is symmetrical.  We note that when $t_1
= t_2 = 0$, the triple sum is equal to 0.  Hence,

\begin{equation*}
\sum_{k \in \mathbb{Z}} \widehat{\beta^2_{k}}(t_1)
\widehat{\beta_{k}^1}(t_2) + \sum_{k \in \mathbb{Z}}
\widehat{\beta^1_{k}}(t_1) \widehat{\beta_{k}^2}(t_2) + \sum_{k \in
\mathbb{Z}} \widehat{\beta^1_{k}}(t_1) \widehat{\beta_{k}^3}(t_2) =
\chi_{\mathbb{R}^2 - (0,0)}(t_1, t_2).
\end{equation*}

\noindent Define $\beta^1_k = \theta^{1,2}_k = \theta^{2,1}_k =
\theta^{3,1}_k$, $\beta^2_k = \theta^{1,1}_k = \theta^{2,2}_k$, and
$\beta^3_k = \theta^{3,2}_k$ and observe

\begin{equation*}
\sum_{k \in \mathbb{Z}} \widehat{\theta^{1,1}_{k}}(t_1)
\widehat{\theta^{1,2}_{k}}(t_2) + \sum_{k \in \mathbb{Z}}
\widehat{\theta^{2,1}_{k}}(t_1) \widehat{\theta^{2,2}_{k}}(t_2) +
\sum_{k \in \mathbb{Z}} \widehat{\theta^{3,1}_{k}}(t_1)
\widehat{\theta_{k}^{3,2}}(t_2) = \chi_{\mathbb{R}^2 - (0,0)}(t_1,
t_2).
\end{equation*}

Choose a Schwartz function $\gamma^1$ supported in $[-2^{-3},
-2^{-9}] \cup [2^{-9}, 2^{-3}]$ and identically 1 on $[-2^{-4},
-2^{-8}] \cup [2^{-8}, 2^{-4}]$.  Let $\widehat{\gamma^1_k}(t) =
\widehat{\gamma^1}(2^{-k}t)$. Then, $\supp(\widehat{\gamma^1_k})
\subseteq [-2^{k-3}, -2^{k-9}] \cup [2^{k-9}, 2^{k-3}]$ and
$\widehat{\gamma^1_k} = 1$ on $[-2^{k-4}, -2^{k-8}] \cup [2^{k-8},
2^{k-4}]$.  Now, if $2^{k-7} \leq |t_1| \leq 2^{k-5}$ and $|t_2|
\leq 2^{k-8}$, then $2^{k-8} \leq |t_1 + t_2| \leq 2^{k-4}$. Hence,
$\widehat{\gamma_k^1}(-t_1 - t_2) = 1$ for such $t_1, t_2$. In
particular,

\begin{equation*}
\begin{split}
\widehat{\theta_k^{2,1}}(t_1) \widehat{\theta_k^{2,2}}(t_2)
\widehat{\gamma^1_k}(-t_1 - t_2) &= \widehat{\beta_k^1}(t_1)
\widehat{\beta_k^2}(t_2) \widehat{\gamma_k^1}(-t_1-t_2) \\
&= \widehat{\beta_k^1}(t_1) \widehat{\beta_k^2}(t_2) =
\widehat{\theta_k^{2,1}}(t_1) \widehat{\theta_k^{2,2}}(t_2).
\end{split}
\end{equation*}

\noindent By symmetry,

\begin{equation*}
\widehat{\theta_k^{1,1}}(t_1) \widehat{\theta_k^{1,2}}(t_2)
\widehat{\gamma^1_k}(-t_1-t_2) = \widehat{\theta_k^{1,1}}(t_1)
\widehat{\theta_k^{1,2}}(t_2).
\end{equation*}

\noindent Set $\theta_k^{1,3} = \theta_k^{2,3} = \gamma^1$.

Similarly, if we choose a Schwartz function $\gamma^2$ so that
$\widehat{\gamma^2}$ is supported in $[-1/4, 1/4]$ and identically 1
on $[-1/8-1/32, 1/8+1/32]$, and let $\widehat{\gamma^2_k}(t) =
\widehat{\gamma^2}(2^{-k}t)$, then $\widehat{\gamma_k^2}$ is
supported in $[-2^{k-2}, 2^{k-2}]$ and identically 1 on
$[-2^{k-3}-2^{k-5}, 2^{k-5}+2^{k-3}]$.  Thus,

\begin{equation*}
\widehat{\theta^{3,1}_k}(t_1) \widehat{\theta^{3,2}_k}(t_2)
\widehat{\gamma_k^2}(-t_1-t_2) = \widehat{\theta^{3,1}_k}(t_1)
\widehat{\theta^{3,2}_k}(t_2).
\end{equation*}

\noindent Set $\theta^{3,3}_k = \gamma^2_k$. It is now clear that
the appropriate sum condition holds.

As $\beta_0^1, \beta_0^2, \beta_0^3, \gamma_0^1, \gamma_0^2$ are all
Schwartz bumps, we can choose constants $C_m$ so that $|\beta_0^i|,
|\beta_0^{i'}|, |\gamma_0^j|, |\gamma_0^{j'}| \leq C_m(1 + \distr(x,
[0,1]))^{-m}$ for $i = 1, 2, 3$ and $j = 1, 2$.  Then, as in the
proof of Theorem~\ref{thm:rbumps}, $|\theta^{a,i}_k(x)| \leq C_m 2^k
(1 + 2^k \distr(x, [0,2^{-k}]))^{-m}$ and $|\theta^{a,i'}_k(x)| \leq
4^k C_m (1 + 2^k \distr(x, [0, 2^{-k}]))^{-m}$.
\end{proof}

\begin{thm}\label{thm:doubletbumps} There are smooth functions
$\psi^{a,i}_k : \mathbb{T} \rightarrow \mathbb{C}$, $1 \leq a, i
\leq 3$, $k \in \mathbb{N}$, and constants $C_m > 0$, $m \in
\mathbb{N}$, so that

\begin{gather*}
\sum_{a=1}^3 \sum_{k = 1}^\infty \widehat{\psi^{a,1}_k}(n_1)
\widehat{\psi^{a,2}_k}(n_2) \widehat{\psi^{a,3}_k}(-n_1-n_2) =
\chi_{\mathbb{Z}^2 - (0,0)}(n_1,n_2) \\
\supp(\widehat{\psi^{a,i}_k}) \subseteq [-2^{k-2}, -2^{k-10}] \cup
[2^{k-10}, 2^{k-2}] \quad \text{ for } a \not= i, \\
\supp(\widehat{\psi^{a,i}_k}) \subseteq [-2^{k-2}, 2^{k-2}] \quad \text{ for } a = i, \\
|\psi^{a,i}_k(x)| \leq 2^k C_m \Big(1 + 2^k \distt(x, [0, 2^{-k}])
\Big)^{-m} \,\, \text{ for all } x \in \mathbb{T}, \, m \in \mathbb{N}, \\
|\psi^{a,i'}_k(x)| \leq 4^k C_m \Big(1 + 2^k \distt(x, [0, 2^{-k}])
\Big)^{-m} \,\, \text{ for all } x \in \mathbb{T}, \, m \in
\mathbb{N}.
\end{gather*}
\end{thm}

\begin{proof} Let $\theta^{a,i}_k$ be the functions
guaranteed by Theorem~\ref{thm:doublerbumps}, and let $\psi^{a,i}_k$
be their respective periodizations. Noting that
$\widehat{\psi^{a,i}_k}(n) = 0$ for all integers $n \not= 0$ when $k
\leq 0$, everything follows immediately from
Theorem~\ref{thm:doublerbumps}.
\end{proof}

\section{Adapted Families}

\begin{defnn} We say a smooth function $\varphi : \mathbb{T}
\rightarrow \mathbb{C}$ is adapted to an interval $I$ with constants
$C_m > 0$, $m \in \mathbb{N}$, if

\begin{gather*}
|\varphi(x)| \leq C_m \bigg(1 + \frac{\distt(x,I)}{|I|} \bigg)^{-m}
\text { for all } x \in \mathbb{T}, m \in \mathbb{N}, \\
|\varphi'(x)| \leq C_m \frac{1}{|I|} \bigg(1 +
\frac{\distt(x,I)}{|I|} \bigg)^{-m} \text { for all } x \in
\mathbb{T}, m \in \mathbb{N}. \\
\end{gather*}

A family of smooth functions $\varphi_I : \mathbb{T} \rightarrow
\mathbb{C}$, indexed by the dyadic intervals, is called an adapted
family if each $\varphi_I$ is adapted to $I$ with the same universal
constants.  We say $\{\varphi_I\}_I$ is a 0-mean adapted family if
it is an adapted family, with the additional property that
$\int_\mathbb{T} \varphi_I \, dm = 0$ for all $I$. \end{defnn}

The first question we should address is whether such a family
exists. Take either $\psi^1_k$ or $\psi^2_k$ from
Theorem~\ref{thm:tbumps}. We will write $\psi_k$ for simplicity. For
each dyadic interval $I = [2^{-k} j, 2^{-k}(j+1)]$, define
$\varphi_I(x) = 2^{-k} \psi_k(x - 2^{-k}j)$.  Then,

\begin{equation*}
\begin{split}
|\varphi_I(x)| &= |2^{-k} \psi_k(x - 2^{-k}j)| \\
&\leq C_m \Big( 1 + 2^k \distt(x - 2^{-k}j, [0,2^{-k}]) \Big)^{-m}
\\
&= C_m \bigg( 1 + \frac{\distt(x, I)}{|I|} \bigg)^{-m}.
\end{split}
\end{equation*}

\noindent Similarly, $|\varphi_I'(x)| = |2^{-k} \psi_k'(x -
2^{-k}j)| \leq C_m \frac{1}{|I|} (1 + \frac{\distt(x,
I)}{|I|})^{-m}$.  Therefore, we have established a way to generate
adapted families.  In fact, this is a 0-mean adapted family, as
$\widehat{\psi_k^1}(0) = \widehat{\psi_k^2}(0) = 0$. However, there
are adapted families with even more specific properties.

\begin{thm}\label{thm:specialadapted} There exists a 0-mean adapted
family $\{\varphi_I\}_I$ and a constant $a > 0$ so that $|\varphi_I|
\geq a \chi_I$ for all $I$.
\end{thm}

\begin{proof} Choose a Schwartz function $\alpha : \mathbb{R}
\rightarrow \mathbb{C}$ so that $\widehat{\alpha} = 1$ on
$[-1/2, 1/2]$, $\supp(\widehat{\alpha}) \subseteq
[-1, 1]$, and $s = |\alpha(0)| > 0$.  By continuity, choose an
integer $k_0 \geq 0$ so that $|x| \leq 2^{-k_0}$ implies $|\alpha(x)
- \alpha(0)| < s/4$.  Then, for $x \in [0, 2^{-k_0}]$, we have
$|\alpha(x)| - s \leq |\alpha(x) - \alpha(0)| < s/4$ or $|\alpha(x)|
< \frac{5s}{4}$. Similarly, $s - |\alpha(x)| \leq |\alpha(x) -
\alpha(0)| < s/4$ or $|\alpha(x)| > \frac{3s}{4}$. Set $\beta(x) =
\alpha(2^{-k_0}x)$, giving $\frac{3s}{4} < |\beta(x)| <
\frac{5s}{4}$ for all $x \in [0, 1]$.  Define $\widehat{\theta}(x) =
\widehat{\beta}(x) - \widehat{\beta}(2x)$ and $\widehat{\theta}_k(x)
= \widehat{\theta}(2^{-k}x)$ for all $k \in \mathbb{N}$.

Now, $\theta_k(x) = 2^k \theta(2^k x)$ and $\theta(x) = \beta(x) -
\frac{1}{2} \beta(\frac{1}{2} x)$.  For any $x \in [0, 1]$, we see
$|\theta(x)| \geq |\beta(x)| - \frac{1}{2} |\beta(\frac{1}{2}x)|
\geq \frac{3s}{4} - \frac{5s}{8} = \frac{s}{8} =: c$.  Thus, for any
$x \in [0, 2^{-k}]$, we have $|\theta_k(x)| = 2^k|\theta(2^k x)|
\geq c2^k$.  Namely, $|\theta_k| \geq c 2^k \chi_{[0, 2^{-k}]}$.  It
is easily seen that $\widehat{\theta_k}(0) = \widehat{\theta}(0) =
\widehat{\beta}(0) - \widehat{\beta}(0) = 0$.

Note $\theta$ and $\theta'$ are Schwartz functions and $(1 +
\distr(x, [0, 1]))^m$ has polynomial growth.  Choose $C_m$ so that
$|\theta(x)|, |\theta'(x)| \leq C_m(1 + \distr(x, [0,1]))^{-m}$ for
all $x$ and $m$.  By the same manipulations as before, this implies
$|\theta_k(x)| = 2^k |\theta(2^kx)| \leq 2^k C_m (1 + \distr(2^k x,
[0,1]))^{-m} = 2^k C_m (1 + 2^k \distr(x, [0, 2^{-k}]))^{-m}$, and
$|\theta'_k(x)| = 4^k |\theta'(2^k)| \leq 4^k C_m (1 + 2^k \distr(x,
[0, 2^{-k}]))^{-m}$.

Let $\psi_k$ be the periodization of $\theta_k$.  As
$\widehat{\psi}_k(0) = \widehat{\theta}_k(0) = 0$, each $\psi_k$ has
integral 0.  Let $k \in \mathbb{N}$ and $x \in [0, 2^{-k}]$. Note,
for $j \geq 1$, we have $\distr(x+j, [0,2^{-k}]) = \distr(x+j,
2^{-k}) = x+j-2^{-k} \geq j - 2^{-k}$. For $j \leq -1$, we see
$\distr(x+j, [0,2^{-k}]) = \distr(x+j, 0) = |j| - x \geq |j| -
2^{-k}$.  So,

\begin{equation*}
\begin{split}
|\psi_k(x)| &\geq |\theta_k(x)| - \bigg|\sum_{j\not= 0}
\theta_k(x+j)\bigg| \, \geq \, c2^k - \sum_{j\not= 0} |\theta_k(x+j)| \\
&\geq c2^k - \sum_{j\not= 0} C_2 2^k \Big(1 + 2^k \distr(x+j, [0,
2^{-k}])\Big)^{-2} \\
&\geq c 2^k - C_2 2^k \sum_{j\not= 0} \Big(1 + 2^k(|j| -
2^{-k})\Big)^{-2} \\
&= c2^k - C_2 2^k \sum_{j\not= 0} (2^k |j|)^{-2} \, \geq \, 2^k
\big[ c - C_2 4^{1-k} \big]
\end{split}
\end{equation*}

\noindent In particular, $|\psi_k| \geq \frac{c}{2} 2^k
\chi_{[0,2^{-k}]}$ for all $k \geq K$, where $K$ is the smallest
integer with $K \geq \log(2C_2/c) (\log 4)^{-1} + 1$.

For each dyadic interval $I = [2^{-k}j, 2^{-k}(j+1)]$ with $k \geq
K$, set $\varphi_I(x) = 2^{-k} \psi_k(x - 2^{-k}j)$. Each
$\varphi_I$ has 0 mean and is adapted to $I$ with constants $C_m'$ by
Lemma~\ref{lemma:adapted}.  Further, $|\varphi_I(x)| = 2^{-k}
|\psi_k(x - 2^{-k}j)| \geq
\frac{c}{2} \chi_{[0, 2^{-k}]}(x - 2^{-k}j) = a\chi_I(x)$, if $a =
c/2$.

Let $I$ be a dyadic interval with $|I| > 2^{-K}$, of which there are
only finitely many.  Choose a smooth function $f_I$ so that $|f_I|
\geq a \chi_I$.  Let $g_I$ be a smooth function, supported away from
$I$, with $\int_\mathbb{T} g_I = 1$, and set $\varphi_I = f_I -
(\int_{\mathbb{T}} f_I) g_I$. Then, $|\varphi_I| \geq a \chi_I$ and
$\varphi_I$ has mean 0. Do this for each remaining $I$, and choose $C''$
so that $\|\varphi_I\|_\infty, \|\varphi_I'\|_\infty \leq C''$ for
all such $I$.  Again, this is possible as there only finitely many.
Set $C_m'' = (1 + 2^K)^m C''$. Then, for any $x \in \mathbb{T}$,

\begin{equation*}
|\varphi_I(x)| \leq C_m'' (1 + 2^K)^{-m} \leq C_m'' \Big(1 +
\frac{1}{2|I|}\Big)^{-m} \leq C_m'' \Big(1 + \frac{\distt(x,
I)}{|I|}\Big)^{-m},
\end{equation*}

\noindent and

\begin{equation*}
|\varphi_I'(x)| \leq C_m'' (1 + 2^K)^{-m} \leq C_m'' \frac{1}{|I|}
\Big(1 + \frac{1}{2|I|}\Big)^{-m} \leq C_m'' \frac{1}{|I|} \Big(1 +
\frac{\distt(x, I)}{|I|}\Big)^{-m}.
\end{equation*}

Hence, $\{\varphi_I\}_I$ is a 0-mean adapted family, with constants
$\max(C_m', C_m'')$, and $|\varphi_I| \geq a \chi_I$.
\end{proof}

The following is an important consequence of the definition, and the
proof is the first of many which make use of a ``geometric" argument
and the adapted property.

\begin{prop}\label{prop:1} For any adapted family $\varphi_I$,
we have $\|\varphi_I\|_1 \lesssim |I|$, where the underlying
constant does not depend on $I$.
\end{prop}

\begin{proof} Fix $I$.  If $|I| = 2^{-k}$, let $N = 2^{k-1}$ so that
$\mathbb{T} = \bigcup\{ I^m : -N+1 \leq m \leq N\}$ and this union
is disjoint.  Then,

\begin{equation*}
\begin{split}
\|\varphi_I\|_1 &= \int_\mathbb{T} |\varphi_I(x)| \, dx = \sum_{m =
-N+1}^N \int_{I^m} |\varphi_I(x)| \, dx \\
&\leq C_2 \sum_{m = -N+1}^N \int_{I^m} \Big(1 + \frac{\distt(x,
I)}{|I|} \Big)^{-2} \, dx \\
&\leq C_2 \sum_{m = -N+1}^N \int_{I^m} \Big(1 + \frac{\distt(I^m,
I)}{|I|} \Big)^{-2} \, dx.
\end{split}
\end{equation*}

\noindent Observe that $\distt(I^m, I) = |I|(|m| - 1)$ for $-N+1
\leq m \leq N$, $m \not= 0$.  Thus,

\begin{equation*}
\begin{split}
\|\varphi_I\|_1 &\lesssim \sum_{m = -N+1}^N \int_{I^m} \Big(1 +
\frac{\distt(I^m, I)}{|I|} \Big)^{-2} \, dx \\
&= |I| + \sum_{-N+1 \leq m \leq N, m \not= 0} |I^m| |m|^{-2} \\
&\leq |I| \Big[ 1 + 2\sum_{m=1}^N \frac{1}{m^2} \Big] \leq |I| \Big[
1 + 2\sum_{m=1}^\infty \frac{1}{m^2} \Big]\lesssim |I|.
\end{split}
\end{equation*}

\end{proof}

Conceptually, we often think of functions which are adapted to an
interval $I$ as being ``almost supported" in $I$.  The following
theorems give some rigid meaning to this.

\begin{thm}\label{thm:decomposition1} Let $\varphi_I : \mathbb{T}
\rightarrow \mathbb{C}$ be adapted to an interval $I$, with $|I| =
2^{-N}$. Then, we can write

\begin{equation*}
\varphi_I = \sum_{k=1}^\infty 2^{-10k} \varphi_I^k,
\end{equation*}

\noindent where $\varphi_I^k$ are adapted to $I$ uniformly in $k$.
In addition, $\supp(\varphi_I^k) \subseteq 2^kI$ for $1 \leq k \leq
N$, and $\varphi_I^k = 0$ otherwise.
\end{thm}

\begin{proof} Assume $\varphi_I$ is adapted to $I$ with constants
$C_m$.  Let $\psi : \mathbb{R} \rightarrow \mathbb{C}$ be smooth,
supported in $[-1/2, 1/2]$, identically 1 on $[-1/4, 1/4]$, with $0
\leq \psi \leq 1$ and $|\psi'| \leq 4$.  For any interval $J$ with
center $x_J$, define $\psi_J(x) = \psi(\frac{x - x_J}{|J|})$.  For
each $0 \leq k < N$, periodize the appropriate functions to create
smooth functions $\psi_{2^kI}$ on $\mathbb{T}$ such that $0 \leq
\psi_{2^kI} \leq 1$, $|\psi_{2^kI}'| \leq 4/|I|$,
$\supp(\psi_{2^kI}) \subseteq 2^kI$, and $\psi_{2^kI} = 1$ on
$2^{k-1}I$.

We start by noting that

\begin{equation*}
1 = \psi_I + (\psi_{2^2I} - \psi_I) + \ldots + (\psi_{2^{N-1}I} -
\psi_{2^{N-2}I}) + (1 - \psi_{2^{N-1}I}).
\end{equation*}

\noindent Therefore, if we define $\varphi_I^1 = 2^{10} \varphi_I
\psi_I$, $\varphi_I^k = 2^{10k} \varphi_I (\psi_{2^kI} -
\psi_{2^{k-1}I})$ for $1 < k < N$, $\varphi_I^N = 2^{10N} \varphi_I
(1 - \psi_{2^{N-1}I})$, and $\varphi_I^k = 0$ for $k > N$, then

\begin{equation*}
\varphi_I = \sum_{k=1}^\infty 2^{-10k} \varphi_I^k.
\end{equation*}

\noindent Further, $\supp(\varphi_I^k) \subseteq 2^kI$ by
construction (for $k = N$, this is an empty statement).

Clearly, $|\varphi_I^1(x)| \leq 2^{10} |\varphi_I(x)| |\psi_I(x)|
\leq 2^{10} |\varphi_I(x)| \leq 2^{10} C_m (1 +
\frac{\distt(x,I)}{|I|})^{-m}$. Also, $|\varphi_I^{1'}(x)| \leq
2^{10} |\varphi_I'(x)| |\psi_I(x)| + 2^{10} |\varphi_I(x)|
|\psi_I'(x)| \leq 2^{10} |\varphi_I'(x)| + 2^{10} \frac{4}{|I|}
|\varphi_I(x)| \leq 2^{10} \cdot 5 C_m \frac{1}{|I|} (1 +
\frac{\distt(x,I)}{|I|})^{-m}$.

Now, for each $1 < k < N$, $\psi_{2^kI} - \psi_{2^{k-1}I}$ is
supported in $2^kI - 2^{k-2}I$.  Also, $1 - \psi_{2^{N-1}I}$ is
supported in $\mathbb{T} - 2^{N-2}I = 2^NI - 2^{N-2}I$.  So, fix $1
< k \leq N$ and let $x \in 2^kI - 2^{k-2}I$.  Then, $2^{k-3} <
\frac{\distt(x, x_I)}{|I|} \leq 2^{k-1}$, where $x_I$ is the center
of $I$. However, $\distt(x, x_I) = \distt(x, I) + |I|/2$, which
gives $2^{k-3} - 1/2 < \frac{\distt(x, I)}{|I|} \leq 2^{k-1} - 1/2$.
Hence,

\begin{equation*}
\begin{split}
|\varphi_I(x)| &\leq C_{m+10} \Big(1 + \frac{\distt(x,I)}{|I|}
\Big)^{-m-10} \leq C_{m+10} (2^{k-3} + 1/2)^{-m-10} \\
&= C_{m+10} (2^{k-3} + 1/2)^{-10} (2^{k-3} + 1/2)^{-m} \\
&\leq C_{m+10} \big(2^{30} \cdot 2^{-10k} \big) \big(4^m (2^{k-1} + 1/2)^{-m} \big) \\
& \leq 4^{m+15} C_{m+10} 2^{-10k} \Big(1 + \frac{\distt(x,I)}{|I|}
\Big)^{-m}.
\end{split}
\end{equation*}

\noindent By precisely the same argument, $|\varphi_I'(x)| \leq
4^{m+15} C_{m+10} 2^{-10k} \frac{1}{|I|} (1 +
\frac{\distt(x,I)}{|I|} )^{-m}$.  Thus, for all $x \in \mathbb{T}$,

\begin{gather*}
|\varphi_I^k(x)| \leq 2^{10k} |\varphi_I(x)| \leq 4^{m+15} C_{m+10}
\Big(1 + \frac{\distt(x,I)}{|I|} \Big)^{-m}, \\
|\varphi_I^{k'}(x)| \leq 2^{10k} \Big[ |\varphi_I'(x)| +
|\varphi_I(x)| \frac{8}{|I|} \Big] \leq 9 \cdot 4^{m+15} C_{m+10}
\frac{1}{|I|} \Big(1 + \frac{\distt(x,I)}{|I|} \Big)^{-m}.
\end{gather*}

\noindent In particular, $\varphi_I^k$ is adapted to $I$ with
constants $9 \cdot 4^{m+15} C_{m+10}$ for all $k$.
\end{proof}

\begin{thm}\label{thm:decomposition2} Let $\varphi_I : \mathbb{T}
\rightarrow \mathbb{C}$ be adapted to an interval $I$, $|I| =
2^{-N}$, with $\int_{\mathbb{T}} \varphi_I \, dm = 0$. Then, we can
write

\begin{equation*}
\varphi_I = \sum_{k=1}^\infty 2^{-10k} \varphi_I^k,
\end{equation*}

\noindent where $\varphi_I^k$ are adapted to $I$ uniformly in $k$
and $\int \varphi_I^k \, dm = 0$. In addition, $\supp(\varphi_I^k)
\subseteq 2^kI$ for $1 \leq k \leq N$, and $\varphi_I^k = 0$
otherwise.  \end{thm}

\begin{proof} Using Theorem~\ref{thm:decomposition1}, write
$\varphi_I = \sum 2^{-10k} \varphi_I^k$, where $\supp(\varphi_I^k)
\subseteq 2^kI$ for $1 \leq k \leq N$ and $\varphi_I^k = 0$
otherwise. Further, $\varphi_I^k$ are adapted to $I$ with uniform
constants.

Choose a smooth function $\psi : \mathbb{T} \rightarrow \mathbb{C}$
so that $0 \leq \psi \leq 2/|I|$, $|\psi'| \leq 8/|I|^2$, $\int \psi
\, dm = 1$, and $\supp(\psi) \subseteq I$. Set $\varphi_{0,I}^k =
\varphi_I^k - (\int \varphi_I^k \, dm) \psi$. Then, each
$\varphi_{0,I}^k$ has integral 0, and is still supported in $2^kI$.
Further,

\begin{equation*}
\sum_{k=1}^\infty 2^{-10k} \varphi_{0,I}^k = \sum_{k=1}^\infty
2^{-10k} \varphi_I^k - \psi \Big( \int_\mathbb{T} \sum_{k=1}^\infty
2^{-10k} \varphi_I^k \, dm \Big) = \varphi_I - \psi \Big(
\int_\mathbb{T} \varphi_I \, dm \Big) = \varphi_I.
\end{equation*}

As $\varphi_I^k$ are uniformly adapted to $I$, we see by
Proposition~\ref{prop:1} that $\|\varphi_I^k\|_1 \lesssim |I|$.  So,
for $x \in I$,

\begin{gather*}
\Big| \Big( \int \varphi_I^k \, dm \Big) \psi(x) \Big| \lesssim 1
= \Big(1 + \frac{\dist(x, I)}{|I|} \Big)^{-m}, \\
\Big| \Big( \int \varphi_I^k \, dm \Big) \psi'(x) \Big| \lesssim
\frac{1}{|I|} = \frac{1}{|I|} \Big(1 + \frac{\dist(x, I)}{|I|}
\Big)^{-m}.
\end{gather*}

\noindent Of course, for $x \notin I$, these quantities are 0.
It follows that $\varphi_{0,I}^k$ are uniformly adapted to
$I$.
\end{proof}

\section{Interpolation Theorems}

Let $(X,\rho)$ be a measure space and $(B, \|\cdot\|_B)$ a (complex)
Banach space and its associated norm.  Consider functions $f :
(X,\rho) \rightarrow B$ which take values in this Banach space. Let
$\mathcal{M}(X,B)$ be the set of such functions such that the map $x
\mapsto \|f(x)\|_B$ is measurable.

For $0 < p < \infty$ and $f \in \mathcal{M}(X,B)$, define

\begin{equation*}
\|f\|_{p,B} = \Big( \int_X \|f(x)\|_B^p \, \rho(dx) \Big)^{1/p},
\end{equation*}

\noindent and $\|f\|_{\infty,B} = \esssup_X \|f(x)\|_B$.  Let
$L^p_B(X)$ be the set of functions for which these quantities are
finite.  It is easily established that $L^p_B(X)$ are Banach spaces,
as usual, for $1 \leq p \leq \infty$.  Let

\begin{equation*}
\|f\|_{p,\infty,B} = \sup_{\lambda > 0} \lambda \, \rho\{x \in X :
\|f(x)\|_B > \lambda\}^{1/p},
\end{equation*}

\noindent and define $L^{p,\infty}_B(X)$ accordingly.

The principal reason for considering such spaces is to attain
interpolation results for operators $T$ which take
$\mathcal{M}(X,B)$ to $\mathcal{M}(X,B)$.  We say an operator is
sublinear if $\|T(f+g)(x)\|_B \leq \|Tf(x)\|_B + \|Tg(x)\|_B$
and $\|T(\alpha f)(x)\|_B = |\alpha| \|Tf(x)\|_B$ for all scalars
$\alpha \in \mathbb{C}$ and almost every $x \in X$.  Consider the
following~\cite{spanish}.

\begin{thm}\label{thm:biginterpolation} Let $T$ be a sublinear
operator on $\mathcal{M}(X,B)$.  Suppose that for some $0 < p_0 <
p_1 \leq \infty$, $T : L^{p_j}_B(X) \rightarrow L^{p_j,\infty}_B(X)$
for $j = 0, 1$ (where $L^{\infty,\infty}_B = L^\infty_B$). Then, for
every $p_0 < p < p_1$, $T : L^p_B(X) \rightarrow L^p_B(X)$.
\end{thm}

\begin{proof}  Fix $p$ and $f$.  First, suppose $p_1 < \infty$.
For each $t > 0$, let $f^t = f$ when $\|f\|_B > t$ and 0 otherwise.
Similarly, let $f_t = f$ when $\|f\|_B \leq t$ and 0 otherwise, so
that $f = f_t + f^t$.

Note, $\|Tf(x)\|_B \leq \|Tf_t(x)\|_B + \|Tf^t(x)\|_B$, and by
hypothesis,

\begin{equation*}
\begin{split}
\rho\{x : \|Tf(x)\|_B > t\} &\leq \rho\{\|Tf^t\|_B > t/2\} +
\rho\{\|Tf_t\|_B > t/2\} \\
&\lesssim (t/2)^{-p_0} \|f^t\|_{p_0,B}^{p_0} + (t/2)^{-p_1}
\|f_t\|_{p_1,B}^{p_1},
\end{split}
\end{equation*}

\noindent where the underlying constants are the operator norms of
$T$.  So,

\begin{equation*}
\begin{split}
\|Tf\|_{p,B}^p &= \int_X \|Tf\|_B^p \, d\rho = p \int_0^\infty
t^{p-1} \rho \{\|Tf\|_B > t \} \, dt \\
&\lesssim \int_0^\infty t^{p-p_0-1} \|f^t\|_{p_0,B}^{p_0} +
t^{p-p_1-1} \|f_t\|_{p_1,B}^{p_1} \, dt \\
&= \int_0^\infty t^{p-p_0-1} \int_{\{\|f\|_B > t\}} \|f\|_B^{p_0} \,
d\rho \, dt + \int_0^\infty t^{p-p_1-1} \int_{\{\|f\|_B \leq t\}}
\|f\|_B^{p_1} \, d\rho \, dt \\
&= \int_X \|f\|_B^{p_0} \int_0^{\|f\|_B} t^{p-p_0-1} \, dt \, d\rho
+ \int_X \|f\|_B^{p_1} \int_{\|f\|_B}^\infty t^{p-p_1-1} \, dt \,
d\rho \\
&= \frac{1}{p-p_0} \int_X \|f\|_B^p \, d\rho + \frac{1}{p_1-p}
\int_X \|f\|_B^p \, d\rho \lesssim \|f\|_{p,B}^p.
\end{split}
\end{equation*}

Now, suppose $p_1 = \infty$.  Let $C$ be the operator norm of $T :
L^\infty_B \rightarrow L^\infty_B$.  For each $t > 0$, set $f_t = f$
for $\|f\|_B \leq t/(2C)$ and 0 otherwise.  Define $f^t$ accordingly
so that $f = f_t + f^t$.  Note, $\|Tf_t(x)\|_B \leq
\|Tf_t\|_{\infty, B} \leq C \|f_t\|_{\infty, B} \leq t/2$ for almost
every $x \in X$.  Thus, $\rho\{x : \|Tf_t(x)\|_B > t/2\} = 0$.
Hence,

\begin{equation*}
\begin{split}
\|Tf\|_{p,B}^p &= \int_X \|Tf\|_B^p \, d\rho = p \int_0^\infty
t^{p-1} \rho \{\|Tf\|_B > t \} \, dt \\
&\lesssim \int_0^\infty t^{p-p_0-1} \|f^t\|_{p_0,B}^{p_0} \, dt \\
&= \int_0^\infty t^{p-p_0-1} \int_{\{\|f\|_B > t/(2C) \}}
\|f\|_B^{p_0} \, d\rho \, dt \\
&= \int_X \|f\|_B^{p_0} \int_0^{2C\|f\|_B} t^{p-p_0-1} \, dt \,
d\rho \\
&= \frac{(2C)^{p-p_0}}{p-p_0} \int_X \|f\|_B^p \, d\rho \lesssim
\|f\|_{p,B}^p.
\end{split}
\end{equation*}
\end{proof}

The preceding theorem is a generalization of the classical
Marcinkiewicz interpolation theorem~\cite{marcin2, steininter}.
Indeed, the proof is nearly identical.  To recover the classical
version, we need only take the Banach space $B$ to be $\mathbb{C}$
with norm $| \cdot |$.  Like the Marcinkiewicz interpolation
theorem, we can actually prove a version where $T : L^{p_j}_B
\rightarrow L^{q_j, \infty}_B$, for $j = 0,1$, implies $T : L^p_B
\rightarrow L^q_B$, with the standard relationships between $p, p_0,
p_1$ and $q, q_0, q_1$.  However, the proof presented here is
slightly neater, and is all we will need.

\chapter{Maximal Operators}\label{chap:max}

Given an adapted family $\varphi_I$ and a function $f : \mathbb{T}
\rightarrow \mathbb{C}$, we will be interested in ``averages" of $f$
with respect to the family.  In particular, given the sequence

\begin{equation*}
\Big\{ \frac{1}{|I|} |\langle \varphi_I, f \rangle| \chi_I(x)
\Big\}_I,
\end{equation*}

\noindent the associated $\ell^2$ and $\ell^\infty$-norms will be
useful quantities.  The $\ell^\infty$-norm is examined in this
chapter. The $\ell^2$-norm is the principal subject of
Chapter~\ref{chap:square}.  Let

\begin{equation*}
M'f(x) = \sup_I \frac{1}{|I|} |\langle \varphi_I, f \rangle|
\chi_I(x).
\end{equation*}

\noindent Instead of studying this operator directly, it will be
more useful to study a different, but related operator; one which is
independent of any adapted family.

\section{Hardy-Littlewood Maximal Function}

\begin{defnn} For $f : \mathbb{T} \rightarrow \mathbb{C}$, define
the Hardy-Littlewood maximal function~\cite{hardylittlewood1} by

\begin{equation*}
Mf(x) = \sup_{x \in I} \frac{1}{|I|} \int_I |f(y)| \, dy,
\end{equation*}

\noindent where the supremum is taken over all intervals in
$\mathbb{T}$ containing $x$.  Similarly, define the dyadic maximal
function $M_D f(x)$ where the supremum is instead taken over all
dyadic intervals containing $x$.
\end{defnn}

We will not be interested in proving results for $M_D$, per se, but
it will prove a useful tool in this and other chapters.

\begin{prop} For any complex-valued functions $f_k, f, g$ on
$\mathbb{T}$ and any scalars $\alpha$,

\begin{enumerate}
\item $M(f+g) \leq Mf + Mg$ and $M(\alpha f) = |\alpha| Mf$,

\item $|f| \leq |g|$ a.e.~implies $Mf \leq M g$ pointwise,

\item $|f_k| \uparrow |f|$ a.e.~implies $M f_k \uparrow M f$
pointwise.
\end{enumerate}

\noindent The same is true for $M_D$.
\end{prop}

\begin{proof} The proofs for $M$ and $M_D$ are essentially identical,
and we handle only the $M$ case.

(1) Fix $x \in \mathbb{T}$ and an interval $I$ containing $x$. Then,
it is clear that $|I|^{-1} \int_I |f(y) + g(y)| \, dy \leq |I|^{-1}
\int_I |f(y)| \, dy + |I|^{-1} \int_I |g(y)| \, dy \leq Mf(x) +
Mg(x)$.  As $I$ is arbitrary, $M(f+g)(x) \leq Mf(x) + Mg(x)$.  Also,
$M(\alpha f)(x) = \sup |I|^{-1} \int_I |\alpha f(y)| \, dy =
|\alpha| \sup |I|^{-1} \int_I |f(y)| \, dy = |\alpha| Mf(x)$.

(2) Fix $x \in \mathbb{T}$ and an interval $I$ containing $x$.  Then
we have immediately that $|I|^{-1} \int_I |f(y)| \, dy \leq |I|^{-1} \int_I |g(y)| \, dy \leq Mg(x)$,
which implies $Mf(x) \leq Mg(x)$.

(3) From statement 2 above, $Mf_1 \leq Mf_2 \leq \ldots \leq Mf$.
Fix $x \in \mathbb{T}$ and $\epsilon > 0$.  There exists an interval $I$ containing
$x$ so that $Mf(x) \leq |I|^{-1} \int_I |f(y)| \, dy + \epsilon/2$.
By the monotone convergence theorem, $\int_I |f_k(y)| \, dy \uparrow
\int_I |f(y)| \, dy$.  So, choose $N$ such that $k \geq N$ implies
$|I|^{-1} \int_I |f_k(y)| \, dy > |I|^{-1} \int_I |f(y)| \, dy -
\epsilon/2$.  Then, for all $k \geq N$, it follows $Mf(x) \leq
|I|^{-1} \int_I |f(y)| \, dy + \epsilon/2 \leq |I|^{-1} \int_I
|f_k(y)| \, dy + \epsilon \leq Mf_k(x) + \epsilon$. As $\epsilon >
0$ is arbitrary, $Mf_k(x) \uparrow Mf(x)$.
\end{proof}

\begin{prop}\label{prop:1'} For any $f : \mathbb{T} \rightarrow
\mathbb{C}$, we have $M'f \lesssim Mf$ pointwise, where the
underlying constant is independent of $f$. \end{prop}

\begin{proof} Let $\varphi_I$ be an adapted family and $f :
\mathbb{T} \rightarrow \mathbb{C}$.  By
Theorem~\ref{thm:decomposition1}, write

\begin{equation*}
\varphi_I = \sum_{k=1}^\infty 2^{-10k} \varphi_I^k,
\end{equation*}

\noindent for each $I$, where $\varphi_I^k$ are uniformly adapted to
$I$.  In particular, $\|\varphi_I^k\|_\infty \lesssim 1$
uniformly in $I$ and $k$.  Further, $\supp(\varphi_I^k) \subseteq
2^kI$ when $k$ is small enough and identically 0 otherwise.

Fix $I$, and suppose $|I| = 2^{-n}$.  Let $x \in I$.  Then,

\begin{equation*}
\begin{split}
\frac{1}{|I|} |\langle \varphi_I, f \rangle| \chi_I(x) &\leq
\frac{1}{|I|} \sum_{k=1}^n 2^{-10k} \int_\mathbb{T} |f(x)|
|\varphi_I^k(x)| \, dx = \frac{1}{|I|} \sum_{k=1}^n \int_{2^k I}
|f(x)| |\varphi_I^k(x)|
\, dx \\
&\lesssim \frac{1}{|I|} \sum_{k=1}^n 2^{-10k} \int_{2^k I} |f(x)| \,
dx = \sum_{k=1}^n 2^{-9k} \frac{1}{|2^k I|} \int_{2^k I} |f(x)| \,
dx \\
&\leq \sum_{k=1}^n 2^{-9k} Mf(x) \leq \sum_{k=1}^\infty 2^{-9k}
Mf(x) \lesssim Mf(x).
\end{split}
\end{equation*}

\noindent Of course, if $x \notin I$, this holds trivially.  As $I$
is arbitrary, take the supremum to see $M'f(x) \lesssim Mf(x)$.
\end{proof}

In light of this, any boundedness property of $M$ will hold
automatically for $M'$.  Therefore, we restrict our attention to $M$
for the remainder of the chapter.

For any interval $I \subseteq \mathbb{T}$, denote by $I^* = 3I$, if
$|I| \leq 1/3$, and $I^* = \mathbb{T}$ if $|I| > 1/3$. Thus, $I
\subseteq I^*$ and $|I^*| \leq 3|I|$.

\begin{claim}\label{claim:1} Let $A, B$ be intervals in $\mathbb{T}$.
If $A \cap B$ is non-empty and $|B| \leq |A|$, then $B \subseteq
A^*$.
\end{claim}

\begin{proof} Suppose $A, B$ have centers $x_A, x_B$. Pick
$z \in A \cap B$.  Then, for $x \in B$,

\begin{equation*}
\begin{split}
\dist(x, x_A) &\leq \dist(x,x_B) + \dist(x_B,z) + \dist(z,x_A)
\\
&\leq |B|/2 + |B|/2 + |A|/2 \leq 3|A|/2.
\end{split}
\end{equation*}

\noindent Namely, $x \in A^*$ and $B \subseteq A^*$.
\end{proof}

Of course, $\dist(\cdot, \cdot)$ refers to $\distt(\cdot, \cdot)$.
As we will be exclusively on $\mathbb{T}$ from now on, we no longer
make this distinction.

\begin{claim}\label{claim:2} Let $A, B$ be intervals in $\mathbb{T}$.
If $A \cap B$ and $A - B^*$ are both nonempty, then $B \subseteq
A^*$.
\end{claim}

\begin{proof} Suppose $A, B$ have centers $x_A, x_B$.  Let $u \in A
\cap B$ and $v \in A - B^*$. That is, $\dist(u,x_A) \leq |A|/2$ and
$\dist(u,x_B) \leq |B|/2$. Also, $\dist(v, x_A) \leq |A|/2$, but
$\dist(v, x_B) > 3|B|/2$. Then,

\begin{equation*}
\begin{split}
3|B|/2 &< \dist(v,x_B) \leq \dist(v, x_A) + \dist(x_A, u) + \dist(u,
x_B) \\
&\leq |A|/2 + |A|/2 + |B|/2,
\end{split}
\end{equation*}

\noindent which implies $|B| < |A|$.  It now follows by
Claim~\ref{claim:1} that $B \subseteq A^*$.
\end{proof}

The following is a decomposition lemma similar to that of
Calder\'{o}n and Zygmund~\cite{cz}, of which we will ultimately
prove several different versions.

\begin{lemma}\label{lemma:cute} Let $f : \mathbb{T} \rightarrow
\mathbb{C}$ and $\alpha > 0$ so that $\{Mf > \alpha\}$ is non-empty.
Then, there exists a sequence of disjoint intervals $I_j$ such that
$\{Mf > \alpha\} \subseteq \bigcup_j I_j^*$ and

\begin{gather*}
\frac{\alpha}{4} \leq \frac{1}{|I_j|} \int_{I_j} |f(x)| \, dx
\text{ for all } I_j. \\
\end{gather*}
\end{lemma}

\begin{proof} Let $\Omega = \{M_D f > \alpha/4\}$.  Assume $\Omega$
is non-empty.  This will be justified shortly.  Let $\mathcal{D}$ be
the countable collection of all dyadic intervals $I$ such that
$\frac{1}{|I|} \int_I |f(y)| \, dy > \alpha/4$.  By construction,
$\Omega = \bigcup_\mathcal{D} I$. We say a dyadic interval $I \in
\mathcal{D}$ is maximal if for every $I' \in \mathcal{D}$, we have
either $I' \subseteq I$ or $I, I'$ are disjoint. Clearly, every $I
\in \mathcal{D}$ is contained in a maximal interval.  Let $I_1, I_2,
\ldots$ be the maximal intervals of $\mathcal{D}$, which are
necessarily disjoint.  Further, it is clear that

\begin{equation*}
\Omega = \bigcup_\mathcal{D} I = \bigcup_\mathbb{N} I_j.
\end{equation*}

Let $x \in \{Mf > \alpha\}$.  By definition, there is an interval
$J$ containing $x$ so that $\frac{1}{|J|} \int_J |f| \, dm
> \alpha$. Write $J = [a, a + |J|]$.  Choose $k \in \mathbb{N}$ so
that $2^{-k-1} \leq |J| < 2^{-k}$, and pick an integer $j$ so that
$(j-1)2^{-k} \leq a < j2^{-k}$.  Then, $a + |J| < (j+1)2^{-k}$.  Let
$I = [2^{-k}(j-1), 2^{-k}j]$ and $I' = [2^{-k}j, 2^{-k}(j+1)]$,
which are both dyadic and $J \subseteq I \cup I'$.  It follows that
either

\begin{equation*}
\int_{I \cap J} |f(x)| \,dx > \alpha |J|/2 \quad \text{or} \quad
\int_{I' \cap J} |f(x)| \, dx > \alpha |J|/2.
\end{equation*}

\noindent Without loss of generality, assume it is the first.  But,
$|J| \geq 2^{-k-1} = |I|/2$.  Thus, $\int_I |f(x)| \, dx >
\alpha|I|/4$, or $I \in \mathcal{D}$ (it is now clear that $\Omega$
is non-empty). So, $I \subseteq I_j$ for some $j$. As $I \cap J$ is
non-empty and $|J| \leq |I|$, we have by Claim~\ref{claim:1} that $x
\in J \subseteq I^* \subseteq I_j^*$. As $x$ is arbitrary, $\{Mf >
\alpha\} \subseteq \bigcup_j I_j^*$.
\end{proof}

\begin{thm}\label{thm:weakL1} $M : L^1 \rightarrow L^{1,\infty}$.
\end{thm}

\begin{proof} Fix $\alpha > 0$ and set $E = \{Mf > \alpha\}$.
If $E$ is empty, the $|E| \leq \|f\|_1/\alpha$ trivially.  Assume it
is not empty, and apply Lemma~\ref{lemma:cute} to find disjoint
intervals $I_j$.  Then,

\begin{equation*}
|E| \leq \sum_j |I_j^*| \leq 3 \sum_j |I_j| \lesssim
\frac{1}{\alpha} \sum_j \int_{I_j} |f(x)| \, dx = \frac{1}{\alpha}
\int_{\bigcup_j I_j} |f(x)| \, dx \leq \frac{1}{\alpha} \|f\|_1.
\end{equation*}

\noindent As $\alpha > 0$ is arbitrary, this completes the proof.
\end{proof}

\begin{cor}\label{cor:Lp} $M : L^p \rightarrow L^p$ for all $1 < p
\leq \infty$.
\end{cor}

\begin{proof} As $M$ is sublinear, it suffices by the Marcinkiewicz
interpolation theorem to show $M : L^\infty \rightarrow L^\infty$.
But, for any $x \in \mathbb{T}$ and any interval $I$ containing $x$,

\begin{equation*}
\frac{1}{|I|} \int_{I} |f(y)| \, dy \leq \|f\|_\infty,
\end{equation*}

\noindent which implies $\|M f\|_\infty \leq \|f\|_\infty$.
\end{proof}

\begin{cor}\label{cor:pointwise} For $f \in L^1(\mathbb{T})$,
$|f| \leq M_D f \leq Mf$ a.e.. \end{cor}

\begin{proof} The fact that $M_D f \leq Mf$ pointwise is clear,
as the supremum in $M$ is taken over a larger class of sets.

For each $x \in \mathbb{T}$, let $I_k(x)$ be the dyadic interval
containing $x$ with $|I| = 2^{-k}$. Define

\begin{equation*}
V_f(x) = \limsup_{k \rightarrow \infty} \frac{1}{|I_k(x)|}
\int_{I_k(x)} |f(y) - f(x)| \, dy.
\end{equation*}

\noindent Let $\epsilon > 0$.  As the continuous functions are dense
in $L^1(\mathbb{T})$, choose $g$ continuous so that $\|h\|_1 <
\epsilon$ where $f = g + h$. Define $V_g$ and $V_h$ accordingly, and
note $V_f \leq V_g + V_h$.

Fix $x \in \mathbb{T}$ and let $\delta > 0$.  As $g$ is continuous
at $x$, there is some $r$ so that $|x - y| < r$ implies $|g(x) -
g(y)| < \delta$.  Then, for all $k > - \log_2 r$, we see

\begin{equation*}
\frac{1}{|I_k(x)|} \int_{I_k(x)} |g(y) - g(x)| \, dy < \delta.
\end{equation*}

\noindent That is, $V_g(x) \leq \delta$.  As $\delta$ and $x$ are
arbitrary, $V_g = 0$.  So, $V_f \leq V_h$.

On the other hand, we clearly have

\begin{equation*}
\begin{split}
V_h(x) &= \limsup_k \frac{1}{|I_k(x)|} \int_{I_k(x)} |h(y) - h(x)| \, dy
\\
&\leq \limsup_k \frac{1}{|I_k(x)|} \int_{I_k(x)} |h(y)| \, dy + |h(x)|
\leq M_D h(x) + |h(x)|.
\end{split}
\end{equation*}

\noindent Thus, for all $t > 0$,

\begin{equation*}
\begin{split}
|\{x \in \mathbb{T} : V_f(x) > t\} &\leq |\{V_h > t\}| \\
&\leq |\{M_D h > t/2\}| + |\{|h| > t/2\}| \\
&\leq \frac{2}{t} \|M\|_{L^1 \rightarrow L^{1,\infty}} \|h\|_1 +
\frac{2}{t}\|h\|_1 \\
&\leq \frac{2\epsilon}{t} (1 + \|M\|_{L^1 \rightarrow
L^{1,\infty}}).
\end{split}
\end{equation*}

\noindent  As $\epsilon > 0$ is arbitrary, $|\{V_f > t\}| = 0$.  As
$t$ is arbitrary, $V_f = 0$ a.e.. Namely, $f(x) = \lim_k
|I_k(x)|^{-1} \int_{I_k(x)} |f(y)| \, dy \leq M_D f(x)$ a.e..
\end{proof}

\section{Fefferman-Stein Inequalities}

Our goal in this section will be to prove the classical
Fefferman-Stein inequalities~\cite{feffermanstein} below.

\begin{thmn} For any sequence $f_1, f_2, \ldots$ of complex-valued
functions on $\mathbb{T}$ and any $1 < p, r < \infty$

\begin{gather*}
\bigg\| \Big( \sum_{k=1}^\infty |M f_k|^r \Big)^{1/r} \bigg\|_p
\lesssim \bigg\| \Big( \sum_{k=1}^\infty |f_k|^r \Big)^{1/r}
\bigg\|_p, \\
\bigg\| \Big( \sum_{k=1}^\infty |M f_k|^r \Big)^{1/r}
\bigg\|_{1,\infty} \lesssim \bigg\| \Big( \sum_{k=1}^\infty |f_k|^r
\Big)^{1/r} \bigg\|_1,
\end{gather*}

\noindent where the underlying constants depend only on $p$ and $r$.
\end{thmn}

Note the similarities between these results and what we know about
$M$.  In both cases, we have $L^1 \rightarrow L^{1,\infty}$ and $L^p
\rightarrow L^p$ ($1 < p < \infty$) results.  But, here, there can
be no $L^\infty$ estimate.  Indeed, fix $1 < r < \infty$ and set
$f_k = \chi_{[2^{-k-1}, 2^{-k})}$.  Then, $(\sum |f_k|^r)^{1/r} =
\chi_{(0,1/2)}$, which has  $L^\infty$-norm 1.  But, if $x \in [0,
2^{-N}]$, then $x \in [0, 2^{-k}]$ for any $k \leq N$.  So, $Mf_k(x)
\geq |[0, 2^{-k}]|^{-1} \int_{[0, 2^{-k}]} f_k(y) \, dy = 1/2$.
Namely, $(\sum |Mf_k(x)|^r)^{1/r} \geq N^{1/r}/2$ for every $x \in
[0,2^{-N}]$.  As $N$ is arbitrary, $\|(\sum |Mf_k|^r)^{1/r}\|_\infty
= \infty$.

One can also see that no $r = 1$ estimate could exist.  Fix a
positive integer $N$ and set $f_k = \chi_{[(k-1)/N, k/N)}$ for $k
\leq N$ and $f_k = 0$ for $k > N$.  Then, $\sum |f_k| = 1$, which
has $L^p$-norm 1 for every $1 \leq p \leq \infty$.  On the other
hand, fix $x \in \mathbb{T}$.  Choose $1 \leq j \leq N$ so that $x
\in [\frac{j-1}{N}, \frac{j}{N})$.  Fix $1 \leq k \leq N$ and denote
$r = \frac{|k-j|}{N} + \frac{1}{N}$.  Then, there is an interval
$I$, with $|I| = r$, containing $[\frac{k-1}{N}, \frac{k}{N}]$ and
$[\frac{j-1}{N}, \frac{j}{N}]$, thus $x$. So,

\begin{equation*}
Mf_k(x) \geq \frac{1}{|I|} \int_{I} f_k(y) \, dy = \frac{1}{rN} =
\frac{1}{|k-j|+1}.
\end{equation*}

\noindent Hence, $\sum |Mf_k(x)| \geq \sum_{k=1}^N \frac{1}{|k-j|+1}
\geq \sum_{k=1}^N \frac{1}{k+1} \geq \log N - 1$. This holds for all
$x$, so $\|\sum |Mf_k|\|_p \geq \log N - 1$. As $N$ is arbitrary, no
$r=1$ estimate could exist. These two counterexamples are taken from
Stein~\cite{stein}.

Finally, we note that the $r = \infty$ case also holds (even for $p
= \infty$), almost trivially.  One only needs to note that $\sup_k
Mf_k \leq M(\sup_k |f_k|)$ pointwise, and apply the $L^p$ theory of
$M$.

\begin{lemma}\label{lemma:cz} Let $f \in L^1(\mathbb{T})$ and
$\alpha > \|f\|_1$ a constant.  Then, there exists a sequence of
disjoint dyadic intervals $I_j$ such that, if $\Omega = \bigcup_j
I_j$, then $|f| \leq \alpha$ a.e.~on $\Omega^c$ and

\begin{gather*}
|\Omega| = \sum_{j=1}^\infty |I_j| \leq \frac{1}{\alpha} \|f\|_1, \\
\frac{1}{|I_j|} \int_{I_j} |f(x)| \, dx \leq 2 \alpha \text{ for all } I_j. \\
\end{gather*}
\end{lemma}

\begin{proof} Define $\Omega = \{M_D f > \alpha\}$.  As $|f| \leq
M_D f$ a.e., we see immediately that $|f| \leq M_D f \leq \alpha$
a.e.~on $\Omega^c$.  If $\Omega$ is empty, then $|f| \leq \alpha$
everywhere.  Thus, $|I|^{-1} \int_I |f(y)| \, dy \leq \alpha$ for
any interval $I$.  Simply choose a dyadic interval $I_1$ so that
$|I_1| \leq \|f\|_1/\alpha$, and let $I_j$ be empty for $j > 1$.
Then, $|\Omega| \leq \|f\|_1/\alpha$, and all conditions are
satisfied.

Now, assume $\Omega$ is not empty.  Let $\mathcal{D}$ be the
countable collection of all dyadic intervals $I$ such that
$\frac{1}{|I|} \int_I |f(y)| \, dy > \alpha$.  By construction,
$\Omega = \bigcup_\mathcal{D} I$.  We say a dyadic interval $I \in
\mathcal{D}$ is maximal if for every $I' \in \mathcal{D}$, we have
either $I' \subseteq I$ or $I, I'$ are disjoint. Clearly, every $I
\in \mathcal{D}$ is contained in a maximal interval.  Let $I_1, I_2,
\ldots$ be the maximal intervals of $\mathcal{D}$, which are
necessarily disjoint.  Further, it is clear that

\begin{equation*}
\Omega = \bigcup_\mathcal{D} I = \bigcup_\mathbb{N} I_k.
\end{equation*}

As each $I_k \in \mathcal{D}$, we have $\alpha|I_k| < \int_{I_k}
|f(y)| \, dy$.  As the $I_k$ are disjoint, simply sum over $k$ to
see $\alpha |\Omega| \leq \int_\Omega |f(y)| \, dy \leq \|f\|_1$. On
the other hand, if $|I_k| < 1/2$, then there is some dyadic interval
$I_k'$ which contains $I_k$ and satisfies $|I_k'| = 2 |I_k|$.  But,
$I_k' \notin \mathcal{D}$, because otherwise $I_k$ could not be
maximal. Thus, $\alpha|I_k'| \geq \int_{I_k'} |f(y)| \, dy$, which
implies $\int_{I_k} |f(y)| \, dy \leq \int_{I_k'} |f(y)| \, dy \leq
\alpha |I_k'| = 2 \alpha |I_k|$.  Similarly, if $|I_k| = 1/2$, then
$\int_{I_k} |f(y)| \, dy \leq \|f\|_1 < \alpha = 2 \alpha |I_k|$.
\end{proof}

\begin{lemma}\label{lemma:fs1} For any sequence $f_1, f_2, \ldots$
on $\mathbb{T}$ and $1 < r <\infty$

\begin{equation*}
\bigg\| \Big( \sum_{k=1}^\infty |M f_k|^r \Big)^{1/r} \bigg\|_r
\lesssim \bigg\| \Big( \sum_{k=1}^\infty |f_k|^r \Big)^{1/r}
\bigg\|_r,
\end{equation*}

\noindent where the underlying constants depend only on $r$.
\end{lemma}

\begin{proof} Simply note that

\begin{equation*}
\begin{split}
\bigg\| \Big( \sum_{k=1}^\infty |M f_k|^r \Big)^{1/r} \bigg\|_r^r &=
\int_{\mathbb{T}} \Big( \sum_{k=1}^\infty |M f_k(x)|^r \Big) \, dx
= \sum_{k=1}^\infty \int_{\mathbb{T}} |M f_k(x)|^r \, dx \\
&\leq \|M\|_{L^r \rightarrow L^r}^r \sum_{k=1}^\infty
\int_{\mathbb{T}} |f_k(x)|^r \, dx \\
&= \|M\|_{L^r \rightarrow L^r}^r \bigg\| \Big( \sum_{k=1}^\infty
|f_k|^r \Big)^{1/r} \bigg\|_r^r.
\end{split}
\end{equation*}
\end{proof}

\begin{thm}\label{thm:fs2} For any sequence $f_1, f_2, \ldots$
on $\mathbb{T}$ and $1 < r < \infty$

\begin{equation*}
\bigg\| \Big( \sum_{k=1}^\infty |M f_k|^r \Big)^{1/r}
\bigg\|_{1,\infty} \lesssim \bigg\| \Big( \sum_{k=1}^\infty |f_k|^r
\Big)^{1/r} \bigg\|_1,
\end{equation*}

\noindent where the underlying constants depend only on $r$.
\end{thm}

\begin{proof} Denote $F(x) = (\sum_{k=1}^\infty
|f_k(x)|^r)^{1/r} \geq 0$.  If $F$ is not in $L^1$, then there is
nothing to prove.  So, assume $F \in L^1(\mathbb{T})$.  Let $\alpha
> \|F\|_1$.  Then, by applying Lemma~\ref{lemma:cz} to $F$ and $\alpha$, find
disjoint intervals $I_j$, $\Omega = \bigcup I_j$, satisfying

\begin{equation*}
\begin{split}
(a) \quad &|\Omega| = \sum_{j=1}^\infty |I_j| \leq \frac{1}{\alpha} \|F\|_1, \\
(b) \quad &F \leq \alpha \text{ on } \Omega^c, \\
(c) \quad &\frac{1}{|I_j|} \int_{I_j} F(y) \, dy \leq 2 \alpha
\text{ for each } I_j.
\end{split}
\end{equation*}

\noindent Decompose each $f_k$ into $f_k = f_k' + f_k''$ where $f_k'
= f_k \chi_{\Omega^c}$ and $f_k''= f_k\chi_\Omega$.  Denote $F' =
(\sum |f_k'|^r)^{1/r}$.

As $f_k' \leq f_k$ pointwise, it is clear that $F' \leq F$
pointwise. On the other hand, $F'$ is 0 on $\Omega$.  So, by (b)
above, we see $F' \leq \alpha$.  Thus,

\begin{equation*}
\|F'\|_r^r = \int_{\mathbb{T}} |F'(x)|^r \, dx \leq \alpha^{r-1}
\|F'\|_1 \leq \alpha^{r-1} \|F\|_1.
\end{equation*}

\noindent Applying Lemma~\ref{lemma:fs1}, we have

\begin{equation*}
\bigg\| \Big( \sum_{k=1}^\infty |M f_k'|^r \Big)^{1/r} \bigg\|_r^r
\lesssim \bigg\| \Big( \sum_{k=1}^\infty |f_k'|^r \Big)^{1/r}
\bigg\|_r^r = \|F'\|_r^r \leq \alpha^{r-1} \|F\|_1.
\end{equation*}

\noindent An application of Chebyshev's inequality yields

\begin{equation*}
\bigg| \bigg\{ \Big( \sum_{k=1}^\infty |M f_k'|^r \Big)^{1/r} >
\alpha/2 \bigg\} \bigg| \lesssim \frac{1}{\alpha^r} \bigg\| \Big(
\sum_{k=1}^\infty |M f_k'|^r \Big)^{1/r} \bigg\|_r^r \lesssim
\frac{1}{\alpha} \|F\|_1.
\end{equation*}

On the other hand, define functions $g_k$ by

\begin{equation*}
g_k(x) = \begin{cases} \frac{1}{|I_j|} \int_{I_j} |f_k(y)| \, dy,&
\quad \text{ if } x \in I_j, \\ 0,& \quad \text{ if } x \notin
\Omega.
\end{cases}
\end{equation*}

\noindent As the $I_j$ are disjoint, this is well-defined a.e.. Let
$G(x) = (\sum |g_k|^r)^{1/r}$, which is supported on $\Omega$.

Fix $x \in \Omega$.  Then, $x$ is in some $I_j$.  By the generalized
Minkowski inequality (see Lieb and Loss~\cite{lieb} or
Rudin~\cite{rudin}) and (c) above, we have

\begin{equation*}
\begin{split}
G(x) &= \bigg( \sum_{k=1}^\infty \Big[ \frac{1}{|I_j|} \int_{I_j}
|f_k(y)| \, dy \Big]^r \bigg)^{1/r} \leq \frac{1}{|I_j|} \int_{I_j}
\Big( \sum_{k=1}^\infty |f_k(y)|^r \Big)^{1/r} \, dy \\
&= \frac{1}{|I_j|} \int_{I_j} F(y) \, dy \leq 2 \alpha.
\end{split}
\end{equation*}

\noindent Hence, as $G$ is supported in $\Omega$ and bounded by $2
\alpha$, we see $\|G\|_r^r \lesssim \alpha^r |\Omega| \leq
\alpha^{r-1} \|F\|_1$.  Precisely as was done above, apply
Lemma~\ref{lemma:fs1} and Chebyshev to see

\begin{equation*}
\bigg| \bigg\{ \Big( \sum_{k=1}^\infty |M g_k|^r \Big)^{1/r}
> \alpha/6 \bigg\} \bigg| \lesssim \frac{1}{\alpha^r}
\bigg\| \Big( \sum_{k=1}^\infty |M g_k|^r \Big)^{1/r} \bigg\|_r^r
\lesssim \frac{1}{\alpha} \|F\|_1.
\end{equation*}

Now, we would now like to establish some relationship between $Mg_k$
and $M f_k''$.  First, note that for any $I_j$,

\begin{equation*}
\int_{I_j} |g_k(x)| \, dx = \int_{I_j} \bigg( \frac{1}{|I_j|}
\int_{I_j} |f_k(y)| \, dy \bigg) \, dx = \int_{I_j} |f_k(y)| \, dy =
\int_{I_j} |f_k''(y)| \, dy.
\end{equation*}

\noindent Set $\Omega^* = \bigcup I_j^*$. By (a),

\begin{equation*}
|\Omega^*| \leq \sum_j |I_j^*| \leq 3 \sum_j |I_j| \lesssim
\frac{1}{\alpha} \|F\|_1.
\end{equation*}

\noindent Fix $x \notin \Omega^*$ and $I$ an interval containing
$x$. As each $f_k''$ is supported on $\Omega = \bigcup I_j$, we see

\begin{equation*}
\frac{1}{|I|} \int_{I} |f_k''(y)| \, dy = \frac{1}{|I|} \sum_{j \in
\mathbb{N}} \int_{I \cap I_j} |f_k''(y)| \, dy = \frac{1}{|I|}
\sum_{j \in J} \int_{I \cap I_j} |f_k''(y)| \, dy,
\end{equation*}

\noindent where $J = \{j : I_j \cap I \not= \emptyset\}$.  But, for
$j \in J$, we have $I_j \cap I \not= \emptyset$ and $x \in I -
\Omega^* \subseteq I - I_j^*$.  By Claim~\ref{claim:2}, this implies
$I_j \subseteq I^*$.  So,

\begin{equation*}
\begin{split}
\frac{1}{|I|} \int_{I} |f_k''(y)| \, dy &= \frac{1}{|I|} \sum_J
\int_{I \cap I_j} |f_k''(y)| \, dy  \leq \frac{1}{|I|}
\sum_J \int_{I_j} |f_k''(y)| \, dy \\
&= \frac{1}{|I|} \sum_J \int_{I_j} |g_k(y)| \, dy \leq \frac{1}{|I|}
\int_{I^*} |g_k(y)| \, dy \\
&\leq \frac{3}{|I^*|} \int_{I^*} |g_k(y)| \, dy \leq 3 M g_k(x).
\end{split}
\end{equation*}

\noindent As $I$ is arbitrary, $M f_k''(x) \leq 3 M g_k(x)$. As $x
\notin \Omega^*$ is arbitrary, this holds on $\mathbb{T} -
\Omega^*$.  Hence, $(\sum |M f_k''|^r)^{1/r} \leq 3 (\sum
|Mg_k|^r)^{1/r}$ on $\mathbb{T} - \Omega^*$, and

\begin{equation*}
\begin{split}
\bigg| \bigg\{x \in \mathbb{T} &- \Omega^* : \Big(\sum_{k=1}^\infty
|M f_k''(x)|^r \Big)^{1/r} > \alpha/2 \bigg\} \bigg| \\
&\leq \bigg| \bigg\{x \in \mathbb{T} - \Omega^* :
\Big(\sum_{k=1}^\infty |M g_k(x)|^r \Big)^{1/r} > \alpha/6 \bigg\} \bigg| \\
&\leq \bigg| \bigg\{\Big(\sum_{k=1}^\infty |M g_k|^r \Big)^{1/r}
> \alpha/6 \bigg\} \bigg| \lesssim \frac{1}{\alpha} \|F\|_1.
\end{split}
\end{equation*}

\noindent Therefore,

\begin{equation*}
\begin{split}
\bigg| \bigg\{ \Big( \sum_{k=1}^\infty |M f_k''|^r \Big)^{1/r} >
\alpha/2 \bigg\} \bigg| &= \bigg| \bigg\{x \in \mathbb{T} - \Omega^*
: \Big( \sum_{k=1}^\infty |M f_k''(x)|^r \Big)^{1/r}
> \alpha/2 \bigg\} \bigg| \\
&\qquad + \bigg| \bigg\{x \in \Omega^* : \Big( \sum_{k=1}^\infty |M
f_k''(x)|^r \Big)^{1/r} > \alpha/2 \bigg\}
\bigg| \\
&\lesssim \frac{1}{\alpha} \|F\|_1 + |\Omega^*| \lesssim
\frac{1}{\alpha} \|F\|_1.
\end{split}
\end{equation*}

Recall $f_k = f_k' + f_k''$, so that $M f_k \leq M f_k' + M f_k''$.
By Minkowski,

\begin{equation*}
\Big( \sum_{k=1}^\infty |M f_k(x)|^r \Big)^{1/r} \leq \Big(
\sum_{k=1}^\infty |M f_k'(x)|^r \Big)^{1/r} + \Big(
\sum_{k=1}^\infty |M f_k''(x)|^r \Big)^{1/r}.
\end{equation*}

\noindent Finally, we see

\begin{equation*}
\begin{split}
\bigg| \bigg\{ \Big(&\sum_{k=1}^\infty |M f_k|^r \Big)^{1/r} >
\alpha \bigg\} \bigg| \\
&\leq \bigg| \bigg\{ \Big(\sum_{k=1}^\infty |M f_k'|^r \Big)^{1/r}
> \alpha/2 \bigg\} \bigg| + \bigg| \bigg\{ \Big(\sum_{k=1}^\infty |M f_k''|^r \Big)^{1/r} >
\alpha/2 \bigg\} \bigg| \\
&\lesssim \frac{1}{\alpha} \|F\|_1.
\end{split}
\end{equation*}

\noindent This holds for all $\alpha > \|F\|_1$.  But, if $\alpha
\leq \|F\|_1$, then $|\{ (\sum |M f_k|^r)^{1/r} > \alpha \}| \leq 1
\leq \|F\|_1/\alpha$ trivially. This completes the proof.
\end{proof}

\begin{thm}\label{thm:fs3} For any sequence $f_1, f_2, \ldots$
on $\mathbb{T}$ and $1 < p \leq r < \infty$

\begin{equation*}
\bigg\| \Big( \sum_{k=1}^\infty |M f_k|^r \Big)^{1/r} \bigg\|_p
\lesssim \bigg\| \Big( \sum_{k=1}^\infty |f_k|^r \Big)^{1/r}
\bigg\|_p,
\end{equation*}

\noindent where the underlying constants depend only on $p$ and $r$.
\end{thm}

\begin{proof} The case $p = r$ has already been shown in
Lemma~\ref{lemma:fs1}.  Let $B = \ell^r$, a Banach space. Then,
$\mathcal{M}(\mathbb{T}, B)$ is the set of sequences of functions $f
= (f_1, f_2, \ldots)$ where each $f_k : \mathbb{T} \rightarrow
\mathbb{C}$ is measurable.  Further, $\|f(x)\|_B = (\sum_k
|f_k(x)|^r)^{1/r}$.

Define $\overline{M}$ on $\mathcal{M}(\mathbb{T}, B)$ by
$\overline{M}(f_1, f_2, \ldots) = (Mf_1, Mf_2, \ldots)$. Then,
$\overline{M}$ is sublinear by Minkowski. Theorem~\ref{thm:fs2} says
$\overline{M} : L^1_B \rightarrow L^{1,\infty}_B$, and
Lemma~\ref{lemma:fs1} says $\overline{M} : L^r_B \rightarrow L^r_B$.
It follows then from Theorem~\ref{thm:biginterpolation} that
$\overline{M} : L^p_B \rightarrow L^p_B$ for all $1 < p < r$, which
is exactly what we wanted to prove.
\end{proof}

\begin{lemma}\label{lemma:weight} For any $1 < r < \infty$ and
any $f , \phi : \mathbb{T} \rightarrow
\mathbb{C}$, we have

\begin{equation*}
\int_{\mathbb{T}} |Mf|^r |\phi| \,\, dx \lesssim \int_{\mathbb{T}}
|f|^r M \phi \, dx,
\end{equation*}

\noindent where the underlying constants depend only on $r$.
\end{lemma}

\begin{proof} Fix $\phi : \mathbb{T} \rightarrow \mathbb{C}$.  If
$\phi$ is identically 0, there is nothing to prove. So, assume
otherwise. We consider the operator $M$ from $(\mathbb{T}, M\phi
\,\, dx)$ to $(\mathbb{T}, |\phi| \,\, dx)$.

As $\phi$ is not identically 0, $M \phi > 0$ everywhere.  Hence,
$\|\cdot\|_\infty = \|\cdot\|_{L^\infty(M \phi \, dx)}$. On the
other hand, it is clear that $\|\cdot\|_{L^\infty(|\phi| \, dx)}
\leq \|\cdot\|_\infty$.  Thus, $\|M f\|_{L^\infty(|\phi| \, dx)}
\leq \|M f\|_\infty \leq \|f\|_\infty = \|f\|_{L^\infty(M \phi \,
dx)}$.  Namely, $M : L^\infty(\mathbb{T}, M \phi \,\, dx)
\rightarrow L^\infty(\mathbb{T}, |\phi| \,\, dx)$.

Fix $\alpha > 0$ and $f : \mathbb{T} \rightarrow \mathbb{C}$.
Consider $\{Mf > \alpha\}$.  Assume for the moment that this set is
non-empty. By Lemma~\ref{lemma:cute}, choose disjoint intervals
$I_j$ so that $|I_j|^{-1} \int_{I_j} |f| \, dm \geq \alpha/4$ and
$\{Mf > \alpha\} \subseteq \bigcup_j I_j^*$.  Then,

\begin{equation*}
\begin{split}
\int_{I_j} f(x) M \phi(x) \, dx &\geq \int_{I_j} f(x)
\bigg( \frac{1}{|I_j^*|} \int_{I_j^*} |\phi(y)| \, dy \bigg) \, dx \\
&\geq \frac{1}{3} \bigg( \int_{I_j^*} |\phi(y)| \, dy \bigg) \cdot
\bigg( \frac{1}{|I_j|} \int_{I_j} f(x) \, dx \bigg) \\
&\geq \frac{\alpha}{12} \int_{I_j^*} |\phi(y)| \, dy.
\end{split}
\end{equation*}

\noindent Summing over $j$, we have

\begin{equation*}
\begin{split}
\alpha \int_{\{M f > \alpha\}} |\phi(x)| \, dx \leq 12 \sum_j
\int_{I_j} f(x) M\phi(x) \, dx \lesssim \int_{\mathbb{T}} f(x) M
\phi(x) \, dx.
\end{split}
\end{equation*}

\noindent This holds so long as $\{Mf > \alpha\}$ is non-empty.
However, if this set is empty, the above holds trivially.  This says
$M : L^1(\mathbb{T}, M \phi \,\, dx) \rightarrow
L^{1,\infty}(\mathbb{T}, |\phi| \,\, dx)$.

Therefore, we see $M : L^r(\mathbb{T}, M \phi \,\, dx) \rightarrow
L^r(\mathbb{T}, |\phi| \,\, dx)$ for all $1 < r < \infty$ by the
Marcinkiewicz interpolation theorem.  This is precisely the
statement we wanted to prove.
\end{proof}

\begin{thm}\label{thm:fs4} For any sequence $f_1, f_2, \ldots$
on $\mathbb{T}$ and $1 < p, r < \infty$

\begin{equation*}
\bigg\| \Big( \sum_{k=1}^\infty |M f_k|^r \Big)^{1/r} \bigg\|_p
\lesssim \bigg\| \Big( \sum_{k=1}^\infty |f_k|^r \Big)^{1/r}
\bigg\|_p,
\end{equation*}

\noindent where the underlying constants depend only on $p$ and $r$.
\end{thm}

\begin{proof}  The case $1 < p \leq r < \infty$ has already been
shown in Theorem~\ref{thm:fs3}.

Fix $1 < r < p < \infty$.  Let $q = p/r > 1$ and $\|\phi\|_{q'} \leq
1$ (where $1/q + 1/q' = 1)$. Then, by Lemma~\ref{lemma:weight}

\begin{equation*}
\begin{split}
\int_{\mathbb{T}} \sum_{k=1}^\infty |M f_k|^r |\phi| \, dx &\lesssim
\int_{\mathbb{T}} \sum_{k=1}^\infty  |f_k|^r M \phi \, dx \leq
\Big\| \sum_{k=1}^\infty |f_k|^r \Big\|_q \|M
\phi\|_{q'} \\
&\lesssim \|\phi\|_{q'} \Big\| \sum_{k=1}^\infty |f_k|^r \Big\|_q
\leq \Big\| \sum_{k=1}^\infty |f_k|^r \Big\|_q.
\end{split}
\end{equation*}

\noindent As $\phi$ in the unit ball of $L^{q'}$ is arbitrary, we
have

\begin{equation*}
\begin{split}
\bigg\| \Big( \sum_{k=1}^\infty |M f_k|^r \Big)^{1/r} \bigg\|_p^r &=
\Big\| \sum_{k=1}^\infty |M f_k|^r \Big\|_q = \sup \bigg\{
\int_{\mathbb{T}} \sum_{k=1}^\infty |M f_k|^r |\phi| \, dx :
\|\phi\|_{q'} \leq 1 \bigg\} \\
&\lesssim \Big\| \sum_{k=1}^\infty |f_k|^r \Big\|_q = \bigg\| \Big(
\sum_{k=1}^\infty |f_k|^r \Big)^{1/r} \bigg\|_p^r
\end{split}
\end{equation*}
\end{proof}

\section{Strong Maximal Operator}\label{sec:strong}

There are multiple ways to define maximal operators for functions $f
: \mathbb{T}^d \rightarrow \mathbb{C}$.  If the maximal function is
defined to be the supremum over one-parameter ``cubes" in
$\mathbb{T}^d$, then it would satisfy all the preceding results by
essentially the same arguments.  However, we will be most interested
in a multi-parameter maximal function.  This will require the
following definition.

\begin{defnn} We say a set $R \subseteq \mathbb{T}^d$ is a rectangle
if $R = I_1 \times I_2 \times \ldots \times I_d$, where each $I_j$
is an interval.  \end{defnn}

\begin{defnn} For $f : \mathbb{T}^d \rightarrow \mathbb{C}$,
define the strong maximal function by

\begin{equation*}
M_S f(\vec{x}) = \sup_{\vec{x} \in R} \frac{1}{|R|} \int_R
|f(\vec{y})| \, d\vec{y},
\end{equation*}

\noindent where the supremum is taken over all rectangles in
$\mathbb{T}^d$ containing $\vec{x}$. \end{defnn}

It is immediately clear that $\|M_S f\|_\infty \leq \|f\|_\infty$,
as before.  In addition, $M_S$ satisfies the same $L^p \rightarrow
L^p$ estimates.  To prove this, we take a slight detour.

Denote $\mathcal{M}(\mathbb{T}^d, \mathbb{C})$ the set of measurable
functions $f : \mathbb{T}^d \rightarrow \mathbb{C}$.  For an
operator $L : \mathcal{M}(\mathbb{T}, \mathbb{C}) \rightarrow
\mathcal{M}(\mathbb{T}, \mathbb{C})$, and $1 \leq j \leq d$, define
$L_j : \mathcal{M}(\mathbb{T}^d, \mathbb{C}) \rightarrow
\mathcal{M}(\mathbb{T}^d, \mathbb{C})$ as the operator which applies
$L$ to functions with all but the $j^{th}$ variable fixed.
Explicitly,

\begin{equation*}
L_jf(x_1, \ldots, x_d) = L\big(f(x_1, \ldots, x_{j-1}, \cdot,
x_{j+1}, \ldots, x_d) \big)(x_j).
\end{equation*}

\begin{thm}\label{thm:Lj} If $L : L^p(\mathbb{T}) \rightarrow
L^p(\mathbb{T})$ for some $0 < p \leq \infty$, then it follows $L_j
: L^p(\mathbb{T}^d) \rightarrow L^p(\mathbb{T}^d)$ for all $1 \leq j
\leq d$. Similarly, if $L : L^p(\mathbb{T}) \rightarrow
L^{p,\infty}(\mathbb{T})$ for some $0 < p < \infty$, then $L_j :
L^p(\mathbb{T}^d) \rightarrow L^{p,\infty}(\mathbb{T}^d)$. Finally,
if $L$ satisfies any Fefferman-Stein inequalities on $\mathbb{T}$
for any $r$ and/or $p$, then $L_j$ satisfies the same inequalities on
$\mathbb{T}^d$. \end{thm}

\begin{proof} For simplicity, we assume $d = 2$ and $j = 1$.  Suppose
$L : L^p(\mathbb{T}) \rightarrow L^p(\mathbb{T})$ with $p$
finite. Let $f : \mathbb{T}^2 \rightarrow \mathbb{C}$, and fix $x_2
\in \mathbb{T}$. Write $f_{x_2}(x_1) = f(x_1, x_2)$. Then,

\begin{equation*}
\begin{split}
\int_\mathbb{T} |L_1 f(x_1, x_2)|^p \, dx_1 &=
\int_\mathbb{T} |L(f_{x_2})(x_1)|^p \, dx_1 \\
&\lesssim \int_\mathbb{T} |f_{x_2}(x_1)|^p \, dx_1 = \int_\mathbb{T}
|f(x_1, x_2)|^p \, dx_1.
\end{split}
\end{equation*}

\noindent Integrating in the $x_2$-variable, we see

\begin{equation*}
\|L_1 f\|_{L^p(\mathbb{T}^2)}^p = \int_{\mathbb{T}^2} |L_1 f(x_1,
x_2)|^p \, dx_1 \, dx_2 \lesssim \int_{\mathbb{T}^2} |f(x_1, x_2)|^p
\, dx_1 \, dx_2 = \|f\|_{L^p(\mathbb{T}^2)}^p.
\end{equation*}

\noindent On the other hand, if $p = \infty$, then $|L_1 f(x_1,
x_2)| \lesssim \|f(\cdot, x_2)\|_{L^\infty(\mathbb{T})}$ for
a.e.~$x_1$. But, $\|f(\cdot, x_2)\|_{L^\infty(\mathbb{T})} \leq
\|f\|_{L^\infty(\mathbb{T}^2)}$ for a.e.~$x_2$.  Thus,
$\|L_1f\|_{L^\infty(\mathbb{T}^2)} \lesssim
\|f\|_{L^\infty(\mathbb{T}^2)}$.

Now suppose $L : L^p(\mathbb{T}) \rightarrow
L^{p,\infty}(\mathbb{T})$.  Then, for any $\lambda > 0$ and any $x_2
\in \mathbb{T}$, we have

\begin{equation*}
\lambda^p \big| \big\{x_1 \in \mathbb{T} : |L_1 f(x_1, x_2)| >
\lambda \big\} \big| \lesssim \int_\mathbb{T} |f(x_1, x_2)|^p \,
dx_1.
\end{equation*}

\noindent Integrating

\begin{equation*}
\begin{split}
\lambda^p \big| \big\{(x_1, x_2) \in \mathbb{T}^2 : |L_1 f(x_1,
x_2)| > \lambda \big\} \big| &= \lambda^p \int_\mathbb{T} \big|
\big\{x_1 \in \mathbb{T} : |L_1 f(x_1, x_2)| > \lambda \big\} \big| \, dx_2 \\
&\lesssim \int_{\mathbb{T}^2} |f(x_1, x_2)|^p \, dx_1 \, dx_2.
\end{split}
\end{equation*}

\noindent As $\lambda$ is arbitrary, we have $\|L_1
f\|_{L^{p,\infty}(\mathbb{T}^2)}^p \lesssim
\|f\|_{L^p(\mathbb{T}^2)}^p$.  Any Fefferman-Stein type inequalities
are extended in the same way.
\end{proof}

Applying the definition above to $M$, consider $M_j$.  Explicitly,

\begin{equation*}
M_j f(\vec{x}) = \sup_{x_j \in I} \frac{1}{|I|} \int_I |f(x_1,
\ldots, x_{j-1}, y_j, x_{j+1}, \ldots, x_d)| \, dy_j.
\end{equation*}

\noindent By the theorem, $M_j : L^p(\mathbb{T}^d) \rightarrow
L^p(\mathbb{T}^d)$ for all $1 < p \leq \infty$.

On the other hand, fix $\vec{x} \in \mathbb{T}^d$.  Let $\epsilon >
0$ and choose a rectangle $\vec{x} \in R$ so that

\begin{equation*}
M_S f(\vec{x}) \leq \frac{1}{|R|} \int_R |f(\vec{y})| \, d\vec{y} \,
+ \epsilon.
\end{equation*}

\noindent Write $R = I_1 \times \ldots \times I_d$, so that $x_j \in
I_j$ for each $j$. Then,

\begin{equation*}
\begin{split}
M_S f(\vec{x}) - \epsilon &\leq \frac{1}{|I_1| \cdots |I_d|}
\int_{I_1 \times \ldots \times I_d} |f(\vec{y})| \, d\vec{y} \\
&= \frac{1}{|I_1|} \int_{I_1} \cdots \frac{1}{|I_d|} \int_{I_d}
|f(y_1, \ldots, y_d)| \, dy_d \cdots dy_1 \\
&\leq \frac{1}{|I_1|} \int_{I_1} \cdots \frac{1}{|I_{d-1}|}
\int_{I_{d-1}} M_d f(y_1, \ldots, y_{d-1}, x_d) \, dy_{d-1} \cdots dy_1 \\
&\leq M_1 \circ M_2 \circ \cdots \circ M_d f(\vec{x}).
\end{split}
\end{equation*}

\noindent As $\epsilon$ is arbitrary, $M_S f \leq M_1 \circ \cdots
\circ M_d f$.  From this, it is easily observed that

\begin{equation*}
\|M_S f\|_p \leq \|M\|_{L^p \rightarrow L^p}^d \|f\|_p
\end{equation*}

\noindent for all $1 < p \leq \infty$.  However, $M_S$ does not
satisfy an $L^1 \rightarrow L^{1,\infty}$ estimate.  Precisely which
set of functions is mapped to weak-$L^1$ by $M_S$ is the subject of
later chapters.  For now, we postpone this topic.

\chapter{Littlewood-Paley Square Function}\label{chap:square}

In this chapter, we focus on a particular square function of
Littlewood-Paley theory~\cite{littlewoodpaley1, littlewoodpaley2,
littlewoodpaley3, littlewoodpaley4}.

For an adapted family $\varphi_I$, define $\phi_I = |I|^{-1/2}
\varphi_I$, and note $\|\phi_I\|_2 \lesssim 1$ for all $I$.  Often,
$\phi_I$ is called an $L^2$-normalized family.  Unless otherwise
noted, $\varphi_I$ will always represent an adapted family, and
$\phi_I$ will always represent the $L^2$-normalization.

For the rest of this chapter, we focus on 0-mean adapted families.
For a 0-mean adapted family $\varphi_I$ and its normalization
$\phi_I$, define the Littlewood-Paley (discrete) square function by

\begin{equation*}
Sf(x) = \bigg( \sum_I \frac{|\langle \phi_I, f \rangle|^2}{|I|}
\chi_I(x) \bigg)^{1/2},
\end{equation*}

\noindent where the sum is over all dyadic intervals.  Note that $S$
is sublinear. We are interested in proving $L^p \rightarrow L^p$
estimates for this operator. All the underlying norm constants will
depend on the original choice of $\varphi_I$, and, in particular,
the constants $C_m$. However, for the sake of neatness, we suppress
that dependence.

\section{The $L^2$ Estimate}

Recall the notation $I^n = I + n|I|$.  The ``canonical"
representation is $I^n$ where $|n| \leq 1/2|I|$. That is, the
smallest $|n|$ giving this set.

\begin{lemma}\label{lemma:L2tech1} For any 0-mean adapted family and
any integer $|n| \leq 1/2|I|$,

\begin{equation*}
|\langle \phi_I, \phi_{I^n} \rangle|
\lesssim \frac{1}{(|n|+1)^2}.
\end{equation*}
\end{lemma}

\begin{proof} First, if $|n| \leq 1$, then $|\langle \phi_I,
\phi_{I^n} \rangle| \leq \|\phi_I\|_2 \|\phi_{I^n}\|_2 \lesssim 1
\leq 4 \frac{1}{(|n| + 1)^2}$. So, assume $|n| > 1$.

Suppose, for simplicity, that $n > 0$.  The other case follows in
the same manner.  If $|I| = 2^{-k}$, set $N = 2^{k-1}$, so that
$\mathbb{T} = \bigcup\{I^m : -N+1 \leq m \leq N\}$, and this union
is disjoint.  Set $\alpha(n) = \frac{n-1}{2}$ if $n$ is odd and
$\frac{n}{2}$ if $n$ is even, so that $\alpha(n)$ is a positive
integer, which is strictly less that $n$. Observe,

\begin{equation*}
\begin{split}
|\langle \phi_I, \phi_{I^n} \rangle| &= \frac{1}{|I|} \bigg|
\int_{\mathbb{T}} \varphi_I(x) \overline{\varphi_{I^n}(x}) \, dx
\bigg| = \frac{1}{|I|} \bigg| \sum_{m = -N+1}^N \int_{I^m}
\varphi_I(x) \overline{\varphi_{I^n}(x)} \, dx \bigg| \\
&\leq \frac{1}{|I|} \sum_{m = -N+1}^N \int_{I^m} |\varphi_I(x)|
|\varphi_{I^n}(x)| \, dx \\
&\lesssim \frac{1}{|I|} \sum_{m = -N+1}^N \int_{I^m} \bigg(1 +
\frac{\dist(x, I)}{|I|} \bigg)^{-3} \bigg(1 + \frac{\dist(x,
I^n)}{|I|} \bigg)^{-3} \, dx \\
&\leq \sum_{m = -N+1}^N \bigg(1 + \frac{\dist(I^m, I)}{|I|}
\bigg)^{-3} \bigg(1 + \frac{\dist(I^m, I^n)}{|I|} \bigg)^{-3}.
\end{split}
\end{equation*}

It is clear that

\begin{equation*}
\frac{\dist(I, I^m)}{|I|} = |m| - 1 \,\, (m \not= 0), \quad
\frac{\dist(I^n, I^m)}{|I|} = \min \big\{|n-m|, |n+m| \big\} - 1
\,\, (m \not= n).
\end{equation*}

\noindent Therefore,

\begin{equation*}
\begin{split}
|\langle \phi_I, \phi_{I^n} \rangle| &\lesssim \sum_{m = -N+1}^N
\bigg(1 + \frac{\dist(I^m, I)}{|I|} \bigg)^{-3} \bigg(1 +
\frac{\dist(I^m, I^n)}{|I|} \bigg)^{-3} \\
&\leq \sum_{|m| \leq \alpha(n)} \bigg(1 + \frac{\dist(I^m,
I^n)}{|I|} \bigg)^{-3} + \sum_{\alpha(n) < |m| \leq N} \bigg(1
+ \frac{\dist(I^m, I)}{|I|} \bigg)^{-3} \\
&= \sum_{|m| \leq \alpha(n)} \frac{1}{\min(|n+m|, |n-m|)^3} +
\sum_{\alpha(n) < |m| \leq N} \frac{1}{|m|^3} \\
&\leq 2 \sum_{m=0}^{\alpha(n)} \frac{1}{|n - m|^3} + 2 \sum_{m =
\alpha(n)}^N \frac{1}{m^3} \leq 2 \sum_{m = \alpha(n)}^n
\frac{1}{m^3} + 2 \sum_{m = \alpha(n)}^N \frac{1}{m^3} \\
&\leq 4 \sum_{m = \alpha(n)}^\infty \frac{1}{m^3} \lesssim
\frac{1}{\alpha(n)^2} \lesssim \frac{1}{(n + 1)^2}.
\end{split}
\end{equation*}
\end{proof}

Let $I$ be a dyadic interval with $|I| = 2^{-k}$.  Then, for $1 \leq
j \leq k-1$, let $J$ be the unique dyadic interval containing $I$
with $|J| = 2^j |I|$.  For $|n| \leq 1/(2|J|)$, denote $I(j,n) =
J^n$.  That is, for an interval $I$, $I(j,n)$ is the interval
obtained by enlarging to the dyadic interval of length $2^j|I|$ and
shifting $n$ units of the new length.

\begin{lemma}\label{lemma:L2tech2} For any 0-mean adapted family
with $j$ and $n$ as above,

\begin{equation*}
|\langle \phi_I, \phi_{I(j,n)} \rangle|
\lesssim 2^{-j} \frac{1}{(|n|+1)^2}.
\end{equation*}
\end{lemma}

\begin{proof} Suppose $|I| = 2^{-k}$.  Let $J$ be the dyadic interval
containing $I$ with $|J| = 2^j |I|$.  Then, $J^n = I(j,n)$.  Set $N
= 2^{k-j-1}$ so that $\mathbb{T} = \bigcup\{J^m : -N+1 \leq m \leq
N\}$ and $\mathbb{T} = \bigcup\{I^m : -2^jN+1 \leq m \leq 2^j N\}$,
and these unions are disjoint.

For a moment, let us think of $\varphi_I$ as a periodic function on
the real line.  Let $I'$ be an interval in $\mathbb{T}$, which can
be thought of as an interval on the real line contained in $[0,1]$.
Then, for $x, z \in I'$, we have by the mean value theorem that
$|\varphi_I(x) - \varphi_I(z)| = |\varphi_I'(z_x)| |x - z| \leq
|\varphi_I'(z_x)| |I'|$, for some $z_x$ in $I'$.  Thus, if we fix a
$z^m$ in each $I^m$, as $\varphi_I$ has integral 0,

\begin{equation*}
\begin{split}
|\langle \phi_I, \phi_{I(j,n)} \rangle| &\leq \frac{1}{|I|^{1/2}}
\frac{1}{|J|^{1/2}} \sum_{m = -2^jN +1}^{2^jN} \bigg| \int_{I^m}
\varphi_I(x) \overline{\varphi_{J^n}(x)} \, dx \bigg| \\
&= 2^{-j/2} \frac{1}{|I|} \sum_{m = -2^jN +1}^{2^jN} \bigg|
\int_{I^m} \varphi_I(x) \Big[\overline{\varphi_{J^n}(x)} -
\overline{\varphi_{J^n}(z^m)} \Big] \, dx \bigg| \\
&\leq 2^{-j/2} \sum_{m = -2^jN +1}^{2^jN} \int_{I^m}
|\varphi_I(x)| |\varphi'_{J^n}(z_x^m)| \, dx \\
&\lesssim 2^{-j/2} \sum_{m = -2^jN +1}^{2^jN} \frac{|I^m|}{|J^n|}
\bigg(1 + \frac{\dist(I^m, I)}{|I|} \bigg)^{-4}
\bigg(1 + \frac{\dist(I^m, J^n)}{|J|} \bigg)^{-10} \\
&= 2^{-3j/2} \sum_{m = -2^jN +1}^{2^jN} \bigg(1 + \frac{\dist(I^m,
I)}{|I|} \bigg)^{-4} \bigg(1 +
\frac{\dist(I^m, J^n)}{|J|} \bigg)^{-10} \\
\end{split}
\end{equation*}

Hence, if $|n| \leq 1$, then

\begin{equation*}
\begin{split}
|\langle \phi_I, \phi_{I(j,n)} \rangle| &\lesssim 2^{-3j/2} \sum_{m
= -2^jN +1}^{2^jN} \bigg(1 + \frac{\dist(I^m, I)}{|I|} \bigg)^{-4}
\\
&\leq 2^{-3j/2} \bigg[ 1 + 2 \sum_{m = 1}^{2^jN} \frac{1}{m^4}
\bigg] \leq 2^{-3j/2} \bigg[ 1 + 2 \sum_{m = 1}^\infty \frac{1}{m^4}
\bigg] \\
&\lesssim 2^{-3j/2} \leq 2^{-j} \leq 4 \cdot 2^{-j}
\frac{1}{(|n|+1)^2}.
\end{split}
\end{equation*}

\noindent Therefore, assume $|n| > 1$.  As before, consider only the
$n > 0$ case, as the other is done in the same way.  Let $\alpha(n)$
be as previously defined.  First, we see

\begin{gather*}
\sum_{2^j\alpha(n) < |m| \leq 2^j N} \bigg(1 + \frac{\dist(I^m,
I)}{|I|} \bigg)^{-4} \bigg(1 + \frac{\dist(I^m,
J^n)}{|J|} \bigg)^{-10} \leq \\
\sum_{2^j\alpha(n) < |m| \leq 2^j N} \bigg(1 + \frac{\dist(I^m,
I)}{|I|} \bigg)^{-4} \leq 2 \sum_{m = 2^j
\alpha(n)}^{2^j N} \frac{1}{m^4} \leq \\
2 \sum_{m = \alpha(n)}^\infty \frac{1}{m^4} \lesssim
\frac{1}{\alpha(n)^2} \lesssim \frac{1}{(|n|+1)^2}.
\end{gather*}

On the other hand, by H\"{o}lder, we have

\begin{gather*}
\sum_{|m| \leq 2^j\alpha(n)} \bigg(1 + \frac{\dist(I^m, I)}{|I|}
\bigg)^{-4} \bigg(1 + \frac{\dist(I^m,
J^n)}{|J|} \bigg)^{-10} \leq \\
\Bigg( \sum_{|m| \leq 2^j\alpha(n)} \Bigg(1 + \frac{\dist(I^m,
I)}{|I|} \bigg)^{-2} \Bigg)^{1/2} \cdot \bigg( \sum_{|m| \leq
2^j\alpha(n)} \Bigg(1 + \frac{\dist(I^m, J^n)}{|J|} \bigg)^{-5}
\Bigg)^{1/2} \leq \\
\bigg(1 + 2\sum_{m=1}^\infty \frac{1}{m^2} \bigg)^{1/2} \cdot \Bigg(
\sum_{|m| \leq 2^j\alpha(n)} \bigg(1 + \frac{\dist(I^m, J^n)}{|J|}
\bigg)^{-5} \Bigg)^{1/2} \lesssim \\
\Bigg( \sum_{|m| \leq 2^j\alpha(n)} \bigg(1 + \frac{\dist(I^m,
J^n)}{|J|} \bigg)^{-5} \Bigg)^{1/2}.
\end{gather*}

\noindent For each $|m| \leq 2^j \alpha(n)$, there is an $m'$ so
that $I^m \subset J^{m'}$ and $|m'| \leq \alpha(n)$.  Further, there
are exactly $2^j$ of these $I^m$ contained in each $J^{m'}$.  Thus,

\begin{gather*}
\Bigg( \sum_{|m| \leq 2^j\alpha(n)} \bigg(1 + \frac{\dist(I^m,
J^n)}{|J|} \bigg)^{-5} \Bigg)^{1/2} \leq  \Bigg(2^j \sum_{|m| \leq
\alpha(n)} \bigg(1 + \frac{\dist(J^m, J^n)}{|J|} \bigg)^{-5}
\Bigg)^{1/2} = \\
\bigg( 2^j \sum_{|m| \leq \alpha(n)} \frac{1}{\min(|n+m|, |n-m|)^5}
\bigg)^{1/2} \lesssim \bigg(2^j \sum_{m=0}^{\alpha(n)}
\frac{1}{|n-m|^5} \bigg)^{1/2} \leq \\
\bigg(2^j \sum_{m = \alpha(n)}^\infty \frac{1}{m^5} \bigg)^{1/2}
\lesssim 2^{j/2} \frac{1}{\alpha(n)^2} \lesssim 2^{j/2}
\frac{1}{(|n|+1)^2}.
\end{gather*}

Finally, combining all of this, we have

\begin{equation*}
\begin{split}
|\langle \phi_I, \phi_{I(j,n)} \rangle| &\lesssim  2^{-3j/2} \sum_{m
= -2^jN +1}^{2^jN} \bigg(1 + \frac{\dist(I^m, I)}{|I|} \bigg)^{-4}
\bigg(1 + \frac{\dist(I^m, J^n)}{|J|} \bigg)^{-10} \\
&\lesssim 2^{-3j/2} \Big[ \frac{1}{(|n|+1)^2} + 2^{j/2}
\frac{1}{(|n|+1)^2} \Big] \lesssim 2^{-j} \frac{1}{(|n|^2+1)}.
\end{split}
\end{equation*}
\end{proof}

For any $N \in \mathbb{N}$, define the linear operator $L_N$ by

\begin{equation*}
L_N f(x) = \sum_{|I| \geq 2^{-N}} \langle \phi_I, f \rangle
\overline{\phi_I(x)}.
\end{equation*}

\noindent The following is the crucial estimate in our desired $L^2$
result.

\begin{lemma}\label{lemma:LN} For any 0-mean adapted family and any
function $f : \mathbb{T} \rightarrow \mathbb{C}$,

\begin{equation*}
\|L_N f\|_2^2 \lesssim \sum_{|I| \geq 2^{-N}}
|\langle \phi_I, f \rangle|^2,
\end{equation*}

\noindent where the underlying constant is
independent of $N$ and $f$.
\end{lemma}

\begin{proof} We note that

\begin{equation*}
\begin{split}
\|L_N f\|_2^2 &= \int_{\mathbb{T}} L_N f(x) \overline{L_N f(x)} \,
dx \\
&= \int_\mathbb{T} \bigg[ \sum_{|I| \geq 2^{-N}} \langle \phi_I, f
\rangle \overline{\phi_I(x)} \bigg] \bigg[ \sum_{|J| \geq 2^{-N}}
\overline{\langle \phi_J, f \rangle} \phi_J(x) \bigg] \, dx \\
&= \sum_{|I|, |J| \geq 2^{-N}} \langle \phi_I, f \rangle \overline{
\langle \phi_J, f \rangle} \langle \phi_J, \phi_I \rangle \\
&\leq \sum_{|I|, |J| \geq 2^{-N}} |\langle \phi_I, f \rangle|
|\langle \phi_J, f \rangle| |\langle \phi_I, \phi_J \rangle|.
\end{split}
\end{equation*}

We break this sum into three pieces: the terms where $|I| = |J|$,
where $|I| < |J|$, and where $|J| < |I|$.  The last two pieces are
symmetric, and we only prove one of them.  For the first piece,

\begin{equation*}
\begin{split}
\sum_{|I| = |J| \geq 2^{-N}} |\langle \phi_I, f \rangle| &|\langle
\phi_J, f \rangle| |\langle \phi_J, \phi_I \rangle| \\
&= \sum_{|I| \geq 2^{-N}} \sum_{n = -1/(2|I|) + 1}^{1/(2|I|)}
|\langle \phi_I, f \rangle| |\langle \phi_{I^n}, f \rangle| |\langle
\phi_I, \phi_{I^n} \rangle|.
\end{split}
\end{equation*}

\noindent For the purposes of this proof only, we adopt a notational
convention.  For an interval $I$ and integer $|n| \leq 1/(2|I|)$,
let $I^n$ be as normal, and $\phi_{I^n}$ the adapted family member
for this interval.  But, for $n$ not satisfying this property, let
$\phi_{I^n}$ be identically 0.  Then, by Lemma~\ref{lemma:L2tech1},
$|\langle \phi_I, \phi_{I^n} \rangle| \lesssim (|n|+1)^{-2}$ for all
$n$.  Further, we can write

\begin{equation*}
\begin{split}
\sum_{|I| = |J| \geq 2^{-N}} |\langle \phi_I, f \rangle| &|\langle
\phi_J, f \rangle| |\langle \phi_I, \phi_J \rangle| = \sum_{n \in
\mathbb{Z}} \sum_{|I| \geq 2^{-N}} |\langle \phi_I, f \rangle|
|\langle \phi_{I^n}, f \rangle| |\langle \phi_I, \phi_{I^n}
\rangle| \\
&\lesssim \sum_{n \in \mathbb{Z}} \frac{1}{(|n|+1)^2} \sum_{|I| \geq
2^{-N}} |\langle \phi_I, f \rangle| |\langle \phi_{I^n}, f \rangle|
\\
&\leq \sum_{n \in \mathbb{Z}} \frac{1}{(|n|+1)^2} \bigg( \sum_{|I|
\geq 2^{-N}} |\langle \phi_I, f \rangle|^2 \bigg)^{1/2} \bigg(
\sum_{|I| \geq 2^{-N}} |\langle \phi_{I^n}, f \rangle|^2
\bigg)^{1/2} \\
&= \bigg( \sum_{|I| \geq 2^{-N}} |\langle \phi_I, f \rangle|^2
\bigg) \sum_{n \in \mathbb{Z}} \frac{1}{(|n|+1)^2} \\
&\lesssim \bigg( \sum_{|I| \geq 2^{-N}} |\langle \phi_I, f
\rangle|^2 \bigg).
\end{split}
\end{equation*}

\noindent The transition from the fourth to fifth line follows
because for a fixed $n$, summing over all $I^n$ is equivalent to
summing over all $I$.  The shift is irrelevant in this regard.

Now let us focus on the case $|I| < |J|$.  Again, we adopt here some
unusual notation.  For appropriate $j$ and $n$, let $I(j,n)$ be as
defined before and $\phi_{I(j,n)}$ as normal.  If either $j$ or $n$
is not small enough with respect to $I$, then set $\phi_{I(j,n)}$ to
be 0.  Then, by Lemma~\ref{lemma:L2tech2}, $|\langle \phi_I,
\phi_{I(j,n)} \rangle| \lesssim 2^{-j} (|n|+1)^{-2}$ for all $j$ and
$n$.  Further,

\begin{equation*}
\begin{split}
\sum_{|J| > |I| \geq 2^{-N}} &|\langle \phi_I, f \rangle| |\langle
\phi_J, f \rangle| |\langle \phi_I, \phi_J \rangle| \\
&= \sum_{k=1}^N \sum_{|I| = 2^{-k}} \sum_{j=1}^{k-1} \sum_{n =
-2^{k-j-1} + 1}^{2^{k-j-1}} |\langle \phi_I, f \rangle| |\langle
\phi_{I(j,n)}, f \rangle| |\langle \phi_I, \phi_{I(j,n)} \rangle| \\
&= \sum_{j \in \mathbb{N}} \sum_{n \in \mathbb{N}} \sum_{k=1}^N
 \sum_{|I| = 2^{-k}} |\langle \phi_I, f \rangle| |\langle
\phi_{I(j,n)}, f \rangle| |\langle \phi_I, \phi_{I(j,n)} \rangle| \\
&\lesssim \sum_{j \in \mathbb{N}} 2^{-j} \sum_{n \in \mathbb{Z}}
\frac{1}{(|n| + 1)^2} \sum_{|I| \geq 2^{-N}} |\langle \phi_I, f
\rangle| |\langle \phi_{I(j,n)}, f \rangle| \\
&\leq \sum_{j \in \mathbb{N}} 2^{-j} \sum_{n \in \mathbb{Z}}
\frac{1}{(|n| + 1)^2} \bigg( \sum_{|I| \geq 2^{-N}} |\langle \phi_I,
f \rangle|^2 \bigg)^{1/2} \bigg( \sum_{|I| \geq 2^{-N}} |\langle
\phi_{I(j,n)}, f \rangle|^2 \bigg)^{1/2}.
\end{split}
\end{equation*}

\noindent Fix $j$ and $n$, and consider a dyadic interval $|J| \geq
2^{-N}$. One of two things is true.  Either there are no $|I| \geq
2^{-N}$ such that $J = I(j,n)$, due to the incompatibility of $j$,
$n$, and/or $N$.  Or, there are exactly $2^j$ such $I$.  Indeed, if
there is an $I$ such that $I \subset J_0$, where $|J_0| = |J|$ and
$J_0^n = J$, then $J = I(j,n)$ for all $I$ contained in this $J_0$.
Hence,

\begin{equation*}
\begin{split}
\sum_{|J| > |I| \geq 2^{-N}} &|\langle \phi_I, f \rangle| |\langle
\phi_J, f \rangle| |\langle \phi_I, \phi_J \rangle| \\
&\lesssim \sum_{j \in \mathbb{N}} 2^{-j} \sum_{n \in \mathbb{Z}}
\frac{1}{(|n| + 1)^2} \bigg( \sum_{|I| \geq 2^{-N}} |\langle \phi_I,
f \rangle|^2 \bigg)^{1/2} \bigg( \sum_{|I| \geq 2^{-N}} |\langle
\phi_{I(j,n)}, f \rangle|^2 \bigg)^{1/2} \\
&\leq \sum_{j \in \mathbb{N}} 2^{-j} \sum_{n \in \mathbb{Z}}
\frac{1}{(|n| + 1)^2} \bigg( \sum_{|I| \geq 2^{-N}} |\langle \phi_I,
f \rangle|^2 \bigg)^{1/2} \bigg(2^j \sum_{|J| \geq 2^{-N}} |\langle
\phi_J, f \rangle|^2 \bigg)^{1/2} \\
&= \bigg( \sum_{|I| \geq 2^{-N}} |\langle \phi_I, f \rangle|^2
\bigg) \sum_{j \in \mathbb{N}} 2^{-j/2} \sum_{n \in \mathbb{Z}}
\frac{1}{(|n| + 1)^2} \\
&\lesssim \bigg( \sum_{|I| \geq 2^{-N}} |\langle \phi_I, f
\rangle|^2 \bigg).
\end{split}
\end{equation*}
\end{proof}

\begin{thm}\label{thm:SL2} For any 0-mean adapted family, $S : L^2
\rightarrow L^2$.  \end{thm}

\begin{proof} Let $f \in L^2$ and fix $N \in \mathbb{N}$.  First, we
note

\begin{equation*}
\begin{split}
\sum_{|I| \geq 2^{-N}} |\langle \phi_I, f \rangle|^2 &\leq \sum_{|I|
\geq 2^{-N}} \|f\|_2^2 \|\phi_I\|_2^2 \lesssim \|f\|_2^2 \sum_{|I|
\geq 2^{-N}} 1 \\
&= \|f\|_2^2 \Big(2 + 2^2 + \ldots + 2^N \Big) \leq 2^{N+1}
\|f\|_2^2 < \infty.
\end{split}
\end{equation*}

\noindent Thus,

\begin{equation*}
\begin{split}
\sum_{|I| \geq 2^{-N}} |\langle \phi_I, f \rangle|^2 &= \sum_{|I|
\geq 2^{-N}} \langle \phi_I, f \rangle \overline{\langle \phi_I, f
\rangle} = \Big\langle \sum_{|I| \geq 2^{-N}} \langle \phi_I, f
\rangle \overline{\phi_I}, \overline{f} \Big\rangle \\
&= \langle L_N f, \overline{f} \rangle \leq \|L_N f\|_2 \|f\|_2
\lesssim \|f\|_2 \bigg( \sum_{|I| \geq 2^{-N}} |\langle \phi_I, f
\rangle|^2 \bigg)^{1/2}
\end{split}
\end{equation*}

\noindent implies

\begin{equation*}
\bigg( \sum_{|I| \geq 2^{-N}} |\langle \phi_I, f \rangle|^2
\bigg)^{1/2} \lesssim \|f\|_2.
\end{equation*}

\noindent As $N$ is arbitrary, and the bounds do not depend on $N$,
let $N$ tend to infinity.  Then,

\begin{equation*}
\|Sf\|_2^2 = \int_\mathbb{T} \sum_I \frac{|\langle \phi_I, f
\rangle|^2}{|I|} \chi_I(x) \, dx = \sum_I |\langle \phi_I, f
\rangle|^2 \lesssim \|f\|_2^2.
\end{equation*}
\end{proof}

\section{The Weak-$L^1$ Estimate}

\begin{lemma}\label{lemma:cz2} Let $f \in
L^1(\mathbb{T})$ and $\alpha > \|f\|_1$ a constant.  Then, there
exists a sequence of disjoint dyadic intervals $I_1, I_2, \ldots$,
with $\Omega = \bigcup_k I_k$, and a decomposition $f = g + b$, $b =
\sum_k b_k$, such that

\begin{gather*}
\|g\|_2^2 \lesssim \alpha \|f\|_1, \\
\supp(b_k) \subseteq I_k, \,\,\, \|b_k\|_1 \lesssim \alpha |I_k|,
\,\,\, \int_{\mathbb{T}} b_k(x) \, dx = 0, \\
|\Omega| = \sum_{k=1}^\infty |I_k| \leq \frac{\|f\|_1}{\alpha}.
\end{gather*}
\end{lemma}

\begin{proof} Define $\Omega = \{M_D f > \alpha\}$.  As $|f| \leq
M_D f$ a.e., we see immediately that $|f| \leq M_D f \leq \alpha$
a.e.~on $\Omega^c$.  If $\Omega$ is empty, then $|f| \leq \alpha$
a.e. on $\mathbb{T}$. Simply set $g = f$, $b_k = 0$, and $I_k$ empty
for each $k$. Then, the conditions are trivially satisfied.

Now, assume $\Omega$ is not empty.  Let $\mathcal{D}$ be the
countable collection of all dyadic intervals $I$ such that
$\frac{1}{|I|} \int_I |f(y)| \, dy > \alpha$.  By construction,
$\Omega = \bigcup_\mathcal{D} I$.  We say a dyadic interval $I \in
\mathcal{D}$ is maximal if for every $I' \in \mathcal{D}$, we have
either $I' \subseteq I$ or $I, I'$ are disjoint. Clearly, every $I
\in \mathcal{D}$ is contained in a maximal interval.  Let $I_1, I_2,
\ldots$ be the maximal intervals of $\mathcal{D}$, which are
necessarily disjoint.  Further, it is clear that

\begin{equation*}
\Omega = \bigcup_\mathcal{D} I = \bigcup_\mathbb{N} I_k.
\end{equation*}

As each $I_k \in \mathcal{D}$, we have $\alpha|I_k| < \int_{I_k}
|f(y)| \, dy$.  As the $I_k$ are disjoint, simply sum over $k$ to
see $\alpha |\Omega| \leq \int_\Omega |f(y)| \, dy \leq \|f\|_1$. On
the other hand, if $|I_k| < 1/2$, then there is some dyadic interval
$I_k'$ which contains $I_k$ and satisfies $|I_k'| = 2 |I_k|$.  But,
$I_k' \notin \mathcal{D}$, because otherwise $I_k$ could not be
maximal. Thus, $\alpha|I_k'| \geq \int_{I_k'} |f(y)| \, dy$, which
implies $\int_{I_k} |f(y)| \, dy \leq \int_{I_k'} |f(y)| \, dy \leq
\alpha |I_k'| = 2 \alpha |I_k|$.  Similarly, if $|I_k| = 1/2$, then
$\int_{I_k} |f(y)| \, dy \leq \|f\|_1 < \alpha = 2 \alpha |I_k|$.

Define the function $g$ by

\begin{equation*}
g(x) = f(x) \chi_{\Omega^c}(x) + \sum_k \Big( \frac{1}{|I_k|}
\int_{I_k} f(y) \, dy \Big) \chi_{I_k}(x).
\end{equation*}

\noindent It is easily seen that $g(x)^2 = f(x)^2 \chi_{\Omega^c}(x)
+ \sum_k (\frac{1}{|I_k|} \int_{I_k} f(y) \, dy)^2 \chi_{I_k}(x)$.
Thus,

\begin{equation*}
\begin{split}
\|g\|_2^2 &= \int_{\Omega^c} |f(x)|^2 \, dx + \sum_k \frac{1}{|I_k|}
\bigg( \int_{I_k} f(y) \, dy \bigg)^2 \\
&\leq \int_{\Omega^c} \alpha |f(x)| \, dx + \sum_k 4 \alpha^2 |I_k| \\
&\leq \alpha \|f\|_1 + 4 \alpha^2 |\Omega| \leq 5 \alpha \|f\|_1.
\end{split}
\end{equation*}

Set $b = f - g$ and $b_k = (f - \frac{1}{|I_k|} \int_{I_k} f(y) \,
dy) \chi_{I_k}$.  Then, we immediately have $\int b_k(x) \, dx = 0$.
Further, each $b_k$ is supported on $I_k$ and $b = \sum_k b_k$.
Finally,

\begin{equation*}
\|b_k\|_1 = \int_{I_k} \bigg|f(x) - \frac{1}{|I_k|} \int_{I_k} f(y)
\, dy \bigg| \, dx \leq 2 \int_{I_k} |f(x)| \, dx \leq 4 \alpha
|I_k|.
\end{equation*}
\end{proof}

\begin{lemma}\label{lemma:atom} If $a : \mathbb{T}
\rightarrow \mathbb{C}$ is in $L^1$, supported in some dyadic
interval $I$, and satisfies $\int_\mathbb{R} a(x) \, dx = 0$, then
$\|Sa\|_{L^1(\mathbb{T} - 2I)} \lesssim \|a\|_1$.  \end{lemma}

\begin{proof} If $|I| = 1/2$, then $2I = \mathbb{T}$, and the
result is trivially satisfied.  So, assume $|I| < 1/2$.  Pick a
dyadic interval $J$ such that $|J| < |I|$.  Note, either $J \subset
2I$ or $J$ and $2I$ are disjoint.  Assume it is the later, i.e.~$J
\nsubseteq 2I$. Then,

\begin{equation*}
\begin{split}
\frac{|\langle \phi_J, a \rangle|}{|J|^{1/2}} &\leq
\frac{1}{|J|^{1/2}} \int_{I} |a(x)| |\phi_J(x)| \, dx =
\frac{1}{|J|} \int_{I} |a(x)| |\varphi_J(x)| \, dx \\
&\lesssim \frac{1}{|J|} \int_{I} |a(x)| \bigg(1 + \frac{\dist(x,
J)}{|J|}\bigg)^{-2} \, dx \\
&\leq \frac{1}{|J|} \|a\|_1 \bigg(1 + \frac{\dist(I,
J)}{|J|}\bigg)^{-2}.
\end{split}
\end{equation*}

\noindent Therefore,

\begin{equation*}
\begin{split}
\bigg\| \sum_{|J| < |I|} \frac{| \langle \phi_J, a
\rangle|}{|J|^{1/2}} \chi_J \bigg\|_{L^1(\mathbb{T} - 2I)} &\lesssim
\|a\|_1 \sum_{|J| < |I|, J \nsubseteq 2I}
\bigg(1 + \frac{\dist(J,I)}{|J|}\bigg)^{-2} \\
&= \|a\|_1 \sum_{j=1}^\infty \bigg[ \sum_{J \nsubseteq 2I, |J| =
2^{-j} |I|} \bigg(1 + \frac{\dist(J,I)}{|J|}\bigg)^{-2} \bigg].
\end{split}
\end{equation*}

Consider the dyadic intervals $J$ so that $|J| = 2^{-j}
|I|$ and $J \nsubseteq 2I$.  The minimum value of $\dist(J, I)$ for
all such $J$ is $|I|/2$, due to the definition of $2I$. But, $|I|/2
= 2^{j-1} |J|$. There are two such $J$, on either side of $2I$.
Taking one step further from $2I$, there are two $J$ with $\dist(J,
I) = (2^{j-1}+1) |J|$. Taking another step, there are two $J$ with
$\dist(J, I) = (2^{j-1}+2) |J|$, and so on, until we have exhausted
$\mathbb{T}$.  Thus,

\begin{equation*}
\sum_{J \nsubseteq 2I, |J| = 2^{-j} |I|} \bigg(1 +
\frac{\dist(J,I)}{|J|}\bigg)^{-2} \leq 2 \sum_{i=2^{j-1}}^\infty
(1+i)^{-2} \leq 2^{2-j},
\end{equation*}

\noindent and

\begin{equation*}
\begin{split}
\bigg\| \sum_{|J| < |I|} \frac{| \langle \phi_J, a
\rangle|}{|J|^{1/2}} \chi_J \bigg\|_{L^1(\mathbb{T} - 2I)} &\lesssim
\|a\|_1 \sum_{j=1}^\infty 2^{-j} = \|a\|_1.
\end{split}
\end{equation*}

Now, let $J$ be a dyadic interval with $|J| \geq |I|$.  Fix $z \in
I$. As in the proof of Lemma~\ref{lemma:L2tech2}, by the mean value
theorem, for all $x \in I$ there exists a $z_x \in I$ such that
$|\varphi_J(x) - \varphi_J(z)| \leq |\varphi_J'(z_x)| |I|$.
Recalling that the integral of $a$ is 0, we have

\begin{equation*}
\begin{split}
\frac{|\langle \phi_J, a \rangle|}{|J|^{1/2}} &= \frac{1}{|J|} \Big|
\int_{I} \varphi_J(x) \overline{a(x)} \, dx \Big| = \frac{1}{|J|}
\Big| \int_{I} \overline{a(x)} \big(\varphi_J(x) - \varphi_J(z)\big)
\, dx
\Big| \\
&\leq \frac{|I|}{|J|} \int_{I} |a(x)| |\varphi_J'(z_x)| \, dx
\lesssim \|a\|_1 \frac{|I|}{|J|^2} \bigg(1 + \frac{\dist(J, I)}{|J|}
\bigg)^{-2}.
\end{split}
\end{equation*}

\noindent So,

\begin{equation*}
\begin{split}
\bigg\| \sum_{|J| \geq |I|} \frac{| \langle \phi_J, a
\rangle|}{|J|^{1/2}} \chi_J \bigg\|_{L^1(\mathbb{T} - 2I)} &\leq
\bigg\| \sum_{|J| \geq |I|} \frac{| \langle \phi_J, a
\rangle|}{|J|^{1/2}} \chi_J \bigg\|_1 \\
&\lesssim \|a\|_1 \sum_{|J| \geq |I|}
\frac{|I|}{|J|} \bigg(1 + \frac{\dist(J,I)}{|J|}\bigg)^{-2} \\
&= \|a\|_1 \sum_{j=0}^{k-1} \bigg[ \sum_{|J| = 2^{j} |I|}
\frac{1}{2^j} \bigg(1 + \frac{\dist(J,I)}{|J|}\bigg)^{-2} \bigg],
\end{split}
\end{equation*}

\noindent if $|I| = 2^{-k}$.  Consider the $J$ with $|J| = 2^j |I|$.
There is one such interval $J'$ with $I \subseteq J'$. Now, for
every other such $J$, we have $\dist(J, I) \geq \dist(J, J')$. There
are three such $J$ (including $J'$) with $\dist(J,J') = 0$. Moving
farther to the left and right, there are two with $\dist(J, J') =
|J|$, two with $\dist(J,J') = 2|J|$, and so on, until we exhaust all
such $J$.  Thus,

\begin{equation*}
\sum_{|J| = 2^{j} |I|} \bigg(1 + \frac{\dist(J,I)}{|J|}\bigg)^{-2}
\leq \sum_{|J| = 2^{j} |I|} \bigg(1 +
\frac{\dist(J,J')}{|J|}\bigg)^{-2} \leq 3 + 2 \sum_{i = 1}^\infty
(1+i)^{-2} \leq 5,
\end{equation*}

\noindent and

\begin{equation*}
\begin{split}
\bigg\| \sum_{|J| \geq |I|} \frac{| \langle \phi_J, a
\rangle|}{|J|^{1/2}} \chi_J \bigg\|_{L^1(\mathbb{T} - 2I)} &\lesssim
\|a\|_1 \sum_{j=0}^{k-1} 2^{-j} \leq \|a\|_1 \sum_{j=0}^\infty
2^{-j} = 2 \|a\|_1.
\end{split}
\end{equation*}

Recalling that the $\ell^2$-norm is always less than or equal to the
$\ell^1$-norm, it follows

\begin{equation*}
\|Sa\|_{L^1(\mathbb{T} - 2I)} \leq \bigg\| \sum_J \frac{|\langle a,
\phi_J \rangle|}{|J|^{1/2}} \chi_J \bigg\|_{L^1(\mathbb{T} - 2I)}
\lesssim \|a\|_1.
\end{equation*}
\end{proof}

\begin{thm}\label{thm:Sweak} For any 0-mean adapted family,
$S : L^1 \rightarrow L^{1,\infty}$.
\end{thm}

\begin{proof} Let $f \in L^1(\mathbb{T})$ and $\alpha \leq
\|f\|_1$.  Then, $|\{Sf > \alpha\}| \leq 1 \leq \|f\|_1/\alpha$. Now
take $\alpha > \|f\|_1$.  Apply Lemma~\ref{lemma:cz2} to find
disjoint dyadic intervals $I_k$ and write $f = g + b$. Then, by
Chebyshev

\begin{equation*}
|\{Sg > \alpha/2\}| \lesssim \frac{1}{\alpha^2} \|Sg\|_2^2 \lesssim
\frac{1}{\alpha^2} \|g\|_2^2 \lesssim \frac{1}{\alpha} \|f\|_1.
\end{equation*}

Applying Lemma~\ref{lemma:atom} to each $b_k$, we see
$\|Sb_k\|_{L^1(\mathbb{T} - 2I_k)} \lesssim \|b_k\|_1 \lesssim
\alpha |I_k|$.  Define $\Omega^* = \bigcup_k 2I_k$, and note
$|\Omega^*| \leq \sum_k |2I_k| = 2 \sum_k |I_k| \lesssim
\|f\|_1/\alpha$.  As $S$ is sublinear,

\begin{equation*}
\begin{split}
|\{Sb > \alpha/2\}| &\leq |\{x \in \Omega^* : Sb(x) > \alpha/2\} +
|\{x \in \mathbb{T} - \Omega^* : Sb(x) > \alpha/2\}| \\
&\leq |\Omega^*| + \frac{2}{\alpha} \|Sb\|_{L^1(\mathbb{T} -
\Omega^*)} \lesssim \frac{1}{\alpha} \|f\|_1 + \frac{2}{\alpha}
\sum_k
\|Sb_k\|_{L^1(\mathbb{T} - 2I_k)} \\
&\leq \frac{1}{\alpha} \|f\|_1 + 2 \sum_k |I_k| \lesssim
\frac{\|f\|_1}{\alpha}.
\end{split}
\end{equation*}

Hence,

\begin{equation*}
|\{Sf > \alpha\}| \leq |\{Sg > \alpha/2\}| + |\{Sb > \alpha/2\}|
\lesssim \frac{\|f\|_1}{\alpha}.
\end{equation*}

\noindent As $\alpha$ is arbitrary, this completes the proof.
\end{proof}

\section{The Linearization $T_\epsilon$}\label{sec:Te}

In order to complete the $L^p$ estimates of $S$, it is necessary to
consider a kind of linearization.  Let $\varphi_I^1$, $\varphi_I^2$
be two 0-mean adapted families.  Let $\epsilon_I$ be a sequence of
scalars, indexed by the dyadic intervals, which is uniformly
bounded.  Define the linear operator $T_\epsilon$ by

\begin{equation*}
T_\epsilon f(x) = \sum_I \epsilon_I \langle \phi^1_I, f \rangle
\phi^2_I(x),
\end{equation*}

\noindent where $\phi_I^1$, $\phi_I^2$ are, of course, the
corresponding normalized families.  By dividing out a constant, we
can assume $|\epsilon_I| \leq 1$.  Our first goal will be to prove
$T_\epsilon$ maps $L^2$ to $L^2$. This follows easily using what we
know about $S$.

\begin{thm}\label{thm:TL2} For any 0-mean adapted families, $T_\epsilon
: L^2 \rightarrow L^2$, where the underlying constant is independent
of the sequence $\epsilon$.
\end{thm}

\begin{proof} Fix a sequence $(\epsilon_I)$ where $|\epsilon_I| \leq
1$ for all $I$.  Let $\varphi_I^1$, $\varphi_I^2$ be two 0-mean
adapted families, and $S^1$, $S^2$ the associated square functions.
Fix $f \in L^2(\mathbb{T})$. Let $\|g\|_2 \leq 1$. Then, by two
applications of H\"{o}lder,

\begin{equation*}
\begin{split}
|\langle T_\epsilon f, g \rangle| &= \Big| \sum_I \epsilon_I \langle
\phi_I^1, f \rangle \langle \phi_I^2, g \rangle \Big| \leq \sum_I
|\langle \phi_I^1, f \rangle| |\langle \phi_I^2, g \rangle| \\
&= \int_\mathbb{T} \sum_I \frac{|\langle \phi_I^1, f
\rangle|}{|I|^{1/2}} \frac{|\langle \phi_I^2, g \rangle|}{|I|^{1/2}}
\chi_I(x) \, dx \\
&\leq \int_\mathbb{T} \Big( \sum_I \frac{|\langle \phi_I^1, f
\rangle|^2}{|I|} \chi_I(x) \Big)^{1/2} \Big( \sum_I \frac{|\langle
\phi_I^2, g \rangle|^2}{|I|} \chi_I(x) \Big)^{1/2} \, dx \\
&= \int_\mathbb{T} S^1f(x) S^2 g(x) \, dx \leq \|S^1f\|_2 \|S^2
g\|_2 \lesssim \|f\|_2 \|g\|_2 \leq \|f\|_2.
\end{split}
\end{equation*}

\noindent As $g$ in the unit ball of $L^2$ is arbitrary, we see

\begin{equation*}
\|T_\epsilon f\|_2 = \sup \Big\{ |\langle T_\epsilon f, g \rangle| :
\|g\|_2 \leq 1 \Big\} \lesssim \|f\|_2.
\end{equation*}
\end{proof}

Next, we will show $T_\epsilon$ maps $L^1$ into weak-$L^1$.  First,
we prove a useful ``dualizaton" of weak-$L^p$.

\begin{lemma}\label{lemma:dual} Fix $0 < p < \infty$ and $f :
\mathbb{T} \rightarrow \mathbb{C}$.  Suppose that for every set $|E|
> 0$ in $\mathbb{T}$, we can choose a subset $E'
\subseteq E$ with $|E'| > |E|/2$ and $|\langle f, \chi_{E'} \rangle|
\leq A|E|^{1 - 1/p}$.  Then, $\|f\|_{p,\infty} \lesssim A$.
Conversely, if $\|f\|_{p, \infty} \leq A$, then for any set $|E| >
0$ there exists $E' \subseteq E$ with $|E'| > |E|/2$ and $|\langle
f, \chi_{E'} \rangle| \lesssim A |E|^{1-1/p}$.
\end{lemma}

\begin{proof} Start with the first statement.  Fix $\lambda > 0$.
Let $E_1 = \{\re f > \lambda\}$. If $|E_1| = 0$, then clearly
$\lambda^p |E_1| \leq (2 A)^p$. Otherwise, choose $E' \subset E_1$
as per the hypothesis.  Now, $|\langle \re f, \chi_{E'} \rangle| =
|\int_{E'} \re f(x) \, dx| = \int_{E'} \re f(x) \, dx \geq \lambda
|E'|$. So, $\lambda |E_1| < 2 \lambda |E'| \leq 2|\langle \re f,
\chi_{E'} \rangle| = 2 |\re \langle f, \chi_{E'} \rangle| \leq 2
|\langle f, \chi_{E'} \rangle| \leq 2 A |E_1|^{1 - 1/p}$. It follows
$\lambda^p |E_1| \leq (2 A)^p$. Do the same for $E_2 = \{\re f <
-\lambda\}$, $E_3 = \{\im f > \lambda\}$, and $E_4 = \{\im f <
-\lambda\}$ to get $\lambda^p |E_j| \leq (2 A)^p$ for $j = 1, 2, 3,
4$. But, $F = \{|f| > \lambda\sqrt{2} \} \subseteq \bigcup_j E_j$. So,
$(\lambda\sqrt{2})^p |F| \leq 4(\sqrt{2})^p(2 A)^p$.  As $\lambda$
is arbitrary, we have $\|f\|_{p,\infty} \leq 2^{3/2+2/p} A$.

Now suppose $\|f\|_{p,\infty} \leq A$.  Let $|E| > 0$.  Note,
$|\{|f| > 3^{1/p} A |E|^{-1/p}\} \big| \leq \frac{|E|}{3 A^p}
\|f\|_{p,\infty}^p < |E|/2$. Thus, if $E' = E - \{|f| > 3^{1/p} A
|E|^{-1/p}\}$, then $E' \subseteq E$ and $|E'| > |E|/2$.  Further,
$|\langle f, \chi_{E'} \rangle| \leq \int_{E'} |f| \, dm \leq |E'|
3^{1/p} A |E|^{-1/p} \lesssim A |E|^{1-1/p}$.
\end{proof}

\begin{thm}\label{thm:Tweak} For any 0-mean adapted families, $T_\epsilon
: L^1 \rightarrow L^{1,\infty}$, where the underlying constant is
independent of the sequence $\epsilon$.
\end{thm}

\begin{proof} As $T_\epsilon$ is linear, it suffices to prove the result
for $\|f\|_1 = 1$.  Let $|E| > 0$.  By Lemma~\ref{lemma:dual}, we
will be done if we can find $E' \subseteq E$, $|E'| > |E|/2$ so that
$|\langle T_\epsilon f, \chi_{E'} \rangle| \lesssim 1$. By
Theorem~\ref{thm:decomposition2}, decompose each $\phi_I^2$ into

\begin{equation*}
\phi_I^2 = \sum_{k=1}^\infty 2^{-10k} \phi_I^{2,k}
\end{equation*}

\noindent where $\phi_I^{2,k}$ is the normalization of a 0-mean
adapted family $\varphi_I^{2,k}$, which are universally adapted to
$I$. Further, $\supp(\phi_I^{2,k}) \subseteq 2^kI$ for $k$ small
enough, while $\phi_I^{2,k}$ is identically 0 otherwise.  Now write

\begin{equation*}
\langle T_\epsilon f, \chi_{E'} \rangle = \sum_{k=1}^\infty 2^{-10k}
\sum_{I} \epsilon_I \langle \phi_I^1, f \rangle \langle
\phi_I^{2,k}, \chi_{E'} \rangle.
\end{equation*}

\noindent Hence, it suffices to show $|\sum \epsilon_I \langle
\phi_I^1, f \rangle \langle \phi_I^{2,k}, \chi_{E'} \rangle|
\lesssim 2^{3k}$, so long as the underlying constants are
independent of $k$.

Denote by $S^1$, $S^{2,k}$ the square functions associated to the
appropriate 0-mean adapted families. For each $k \in \mathbb{N}$,
define

\begin{gather*}
\Omega_{-3k} = \{S^1f > C2^{3k} \}, \\
\widetilde{\Omega}_k = \{M(\chi_{\Omega_{-3k}}) > 1/100\}, \\
\widetilde{\widetilde{\Omega}}_k = \{M(\chi_{\widetilde{\Omega}_k})
> 2^{-k-1} \}.
\end{gather*}

\noindent and

\begin{gather*}
\Omega = \bigcup_{k \in \mathbb{N}}
\widetilde{\widetilde{\Omega}}_k.
\end{gather*}

\noindent Observe,

\begin{equation*}
|\Omega| \leq \sum_{k=1}^\infty 2^{-3k} 2^{k+1}  \frac{100}{C}
\|M\|^2_{L^1 \rightarrow L^{1,\infty}} \|S^1\|_{L^1 \rightarrow
L^{1,\infty}}.
\end{equation*}

\noindent Therefore, we can choose $C$ independent of $f$ so that
$|\Omega| < |E|/2$. Set $E' = E - \Omega = E \cap \Omega^c$. Then,
$E' \subseteq E$ and $|E'| > |E|/2$.

Fix $k \in \mathbb{N}$.  Set $Z_k = \{S^1f = 0\} \cup \{S^{2,k}
(\chi_{E'}) = 0\}$. Let $\mathcal{D}$ be any finite collection of
dyadic intervals. We divide this collection into three
subcollections.  First, define $\mathcal{D}_1 = \{I \in \mathcal{D}
: I \cap Z_k \not= \emptyset \}$. For the remaining intervals, let
$\mathcal{D}_2 = \{I \in \mathcal{D} - \mathcal{D}_1 : I \subseteq
\widetilde{\Omega}_k\}$ and $\mathcal{D}_3 = \{I \in \mathcal{D} -
\mathcal{D}_1 : I \cap \widetilde{\Omega}_k^c \not= \emptyset\}$.

If $I \in \mathcal{D}_1$, then there is some $x \in I \cap Z_k$.
Namely, either $S^1f(x) = 0$ or $S^{2,k}(\chi_{E'})(x) = 0$.  If it
is the first, then from the definition of $S^1f$, it must be that
$\langle \phi^1_J, f \rangle = 0$ for all dyadic $J$ containing $x$.
In particular, $\langle \phi^1_I, f \rangle = 0$.  If instead
$S^{2,k}(\chi_{E'})(x) = 0$, then $\langle \phi^{2,k}_I, \chi_{E'}
\rangle = 0$. As this holds for all $I \in \mathcal{D}_1$, we have

\begin{equation*}
\sum_{I \in \mathcal{D}_1} |\langle \phi_I^1, f \rangle| |\langle
\phi_I^{2,k}, \chi_{E'} \rangle| = 0.
\end{equation*}

Now suppose $I \in \mathcal{D}_2$, namely $I \subseteq
\widetilde{\Omega}_k$.  If $k$ is big enough so that $2^k > 1/|I|$,
then $\phi_I^{2,k}$ is identically 0 and $\langle \phi_I^{2,k},
\chi_{E'} \rangle = 0$. If $2^k \leq 1/|I|$, then $\phi_I^{2,k}$ is
supported in $2^k I$. Let $x \in 2^k I$, and observe

\begin{equation*}
M(\chi_{\widetilde{\Omega}_k})(x) \geq \frac{1}{|2^kI|} \int_{2^kI}
\chi_{\widetilde{\Omega}_k} \,\, dm \geq \frac{1}{2^k} \frac{1}{|I|}
\int_I \chi_{\widetilde{\Omega}_k} \,\, dm = 2^{-k} > 2^{-k-1}.
\end{equation*}

\noindent That is, $2^kI \subseteq \widetilde{\widetilde{\Omega}}_k
\subseteq \Omega$, a set disjoint from $E'$.  Thus, $\langle
\phi_I^{2,k}, \chi_{E'} \rangle = 0$.  As this holds for all $I
\in \mathcal{D}_2$, we have

\begin{equation*}
\sum_{I \in \mathcal{D}_2} |\langle \phi_I^1, f \rangle| |\langle
\phi_I^{2,k}, \chi_{E'} \rangle| = 0.
\end{equation*}

Finally, we concentrate on $\mathcal{D}_3$.  Define $\Omega_{-3k+1}$
and $\Pi_{-3k+1}$ by

\begin{gather*}
\Omega_{-3k+1} = \{S^1f > C2^{3k-1}\}, \\
\Pi_{-3k+1} = \{I \in \mathcal{D}_3 : |I \cap \Omega_{-3k+1}| >
|I|/100\}.
\end{gather*}

\noindent Inductively, define for all $n > -3k+1$,

\begin{gather*}
\Omega_{n} = \{S^1f > C2^{-n}\}, \\
\Pi_{n} = \{I \in \mathcal{D}_3 - \bigcup_{j=-3k+1}^{n-1} \Pi_{j} :
|I \cap \Omega_{n}|
> |I|/100\}.
\end{gather*}

\noindent As every $I \in \mathcal{D}_3$ is not in $\mathcal{D}_1$,
that is $S^1f > 0$ on $I$, it is clear that each $I \in
\mathcal{D}_3$ will be in one of these collections.

We can choose an integer $N$ big enough so that $\Omega_{-N}' =
\{S^{2,k} (\chi_{E'}) > 2^N\}$ has very small measure.  In
particular, we take $N$ big enough so that $|I \cap \Omega_{-N}'| <
|I|/100$ for all $I \in \mathcal{D}_3$, which is possible since
$\mathcal{D}_3$ is a finite collection. Define

\begin{gather*}
\Omega_{-N+1}' = \{S^{2,k}(\chi_{E'}) > 2^{N-1}\}, \\
\Pi_{-N+1}' = \{I \in \mathcal{D}_3 : |I \cap \Omega_{-N+1}'| >
|I|/100\},
\end{gather*}

\noindent and

\begin{gather*}
\Omega_{n}' = \{S^{2,k}(\chi_{E'}) > 2^{-n}\}, \\
\Pi_{n}' = \{I \in \mathcal{D}_3 - \bigcup_{j=-N+1}^{n-1} \Pi_j' :
|I \cap \Omega_{n}'|
> |I|/100\},
\end{gather*}

\noindent Again, all $I \in \mathcal{D}_3$ must be in one of these
collections.

Consider $I \in \mathcal{D}_3$, so that $I \cap
\widetilde{\Omega}_k^c \not= \emptyset$. Then, there is some $x \in
I \cap \widetilde{\Omega}_k^c$ which implies $|I \cap
\Omega_{-3k}|/|I| \leq M(\chi_{\Omega_{-3k}})(x) \leq 1/100$. Write
$\Pi_{n_1, n_2} = \Pi_{n_1} \cap \Pi'_{n_2}$. So,

\begin{equation*}
\begin{split}
\sum_{I \in \mathcal{D}_3} |\langle \phi_I^1, f \rangle| |\langle
\phi_I^{2,k}, \chi_{E'} \rangle| &= \sum_{n_1
> -3k, \, n_2 > -N} \bigg[ \sum_{I \in \Pi_{n_1, n_2}} |\langle \phi_I^1,
f \rangle| |\langle \phi_I^{2,k}, \chi_{E'}\rangle| \bigg]\\
&= \sum_{n_1 > -3k, \, n_2 > -N} \bigg[ \sum_{I \in \Pi_{n_1, n_2}}
\frac{|\langle \phi_I^1, f \rangle|}{|I|^{1/2}} \frac{|\langle
\phi_I^{2,k}, \chi_{E'} \rangle|}{|I|^{1/2}} |I| \bigg].
\end{split}
\end{equation*}

\noindent Suppose $I \in \Pi_{n_1, n_2}$.  If $n_1 > -3k + 1$,
then $I \in \Pi_{n_1}$, which in particular says $I \notin \Pi_{n_1
- 1}$. So, $|I \cap \Omega_{n_1 - 1}| \leq |I|/100$.  If $n_1 = -3k
+ 1$, then we still have $|I \cap \Omega_{-3k}| \leq |I|/100$, as $I
\in \mathcal{D}_3$. Similarly, if $n_2 > -N+1$, then $I \notin
\Pi'_{n_2 - 1}$ and $|I \cap \Omega'_{n_2 - 1}| \leq |I|/100$.  If
$n_2 = -N+1$, then $|I \cap \Omega'_{-N}| \leq |I|/100$ by the
choice of $N$.  So, $|I \cap \Omega_{n_1 - 1}^c \cap \Omega_{n_2 -
1}'^c| \geq \frac{98}{100} |I|$.  Let $\Omega_{n_1, n_2} = \bigcup\{
I : I \in \Pi_{n_1, n_2}\}$.  Then, $|I \cap \Omega_{n_1 - 1}^c \cap
\Omega'^c_{n_2-1} \cap \Omega_{n_1, n_2}| \geq \frac{98}{100} |I|$
for all $I \in \Pi_{n_1, n_2}$, and

\begin{equation*}
\begin{split}
\sum_{I \in \Pi_{n_1, n_2}} \frac{|\langle \phi_I^1, f
\rangle|}{|I|^{1/2}} &\frac{|\langle \phi_I^{2,k}, \chi_{E'}
\rangle|}{|I|^{1/2}} |I| \\
&\lesssim \sum_{I \in \Pi_{n_1, n_2}}
\frac{|\langle \phi_I^1, f \rangle|}{|I|^{1/2}} \frac{|\langle
\phi_I^{2,k}, \chi_{E'} \rangle|}{|I|^{1/2}} |I \cap \Omega_{n_1 -
1}^c
\cap \Omega_{n_2 - 1}'^c \cap \Omega_{n_1, n_2}| \\
&= \int_{\Omega_{n_1 - 1}^c \cap \Omega_{n_2 - 1}'^c \cap
\Omega_{n_1, n_2}} \,\, \sum_{I \in \Pi_{n_1, n_2}} \frac{|\langle
\phi_I^1, f \rangle|}{|I|^{1/2}} \frac{|\langle \phi_I^{2,k},
\chi_{E'} \rangle|}{|I|^{1/2}} \chi_I(x) \,
dx \\
&\leq \int_{\Omega_{n_1 - 1}^c \cap \Omega_{n_2 - 1}'^c \cap
\Omega_{n_1, n_2}} S^1f(x)
S^{2,k} (\chi_{E'})(x) \, dx \\
&\lesssim C 2^{-n_1} 2^{-n_2} |\Omega_{n_1, n_2}|.
\end{split}
\end{equation*}

Observe that $|\Omega_{n_1, n_2}| \leq |\bigcup \{I : I \in
\Pi_{n_1}\}| \leq |\{M(\chi_{\Omega_{n_1}}) > 1/100\}| \lesssim
|\Omega_{n_1}| = |\{S^1f > C2^{-n_1}\}| \lesssim 2^{n_1}/C$.  By the
same argument, $|\Omega_{n_1, n_2}| \lesssim |\Omega_{n_2}'| =
|\{S^{2,k}(\chi_{E'}) > 2^{-n_2}\}| \lesssim 2^{\alpha n_2}$ for
$\alpha = 1, 2$, as $S : L^p \rightarrow L^{p,\infty}$ for $p = 1,
2$.  Thus, $|\Omega_{n_1, n_2}| \lesssim C^{-1} 2^{\theta_1 n_1}
2^{\theta_2 \alpha n_2}$ for any $\theta_1 + \theta_2 = 1$, $0 \leq
\theta_1, \theta_2 \leq 1$.  Hence,

\begin{equation*}
\begin{split}
\sum_{I \in \mathcal{D}_3} |\langle \phi_I^1, f \rangle| |\langle
\phi_I^{2,k}, \chi_{E'} \rangle| &\lesssim \sum_{n_1 > -3k, \,\, n_2
> 0} 2^{n_1(\theta_1 - 1)} 2^{n_2(\theta_2 \alpha -1)} \quad + \\
&\,\,\, \sum_{n_1 > -3k, \,\, -N < n_2 \leq 0}
2^{n_1(\theta_1 - 1)} 2^{n_2(\theta_2 \alpha -1)} \\
&= A + B.
\end{split}
\end{equation*}

\noindent For the first term, take $\alpha = 1$, $\theta_1 =
\theta_2 = 1/2$, and for the second term, take $\alpha = 2$,
$\theta_1 = 1/4$, and $\theta_2 = 3/4$ to see

\begin{gather*}
A = \sum_{n_1 > -3k, \,\, n_2 > 0} 2^{-n_1/2} 2^{-n_2/2}
\lesssim 2^{3k/2} \leq 2^{3k}, \\
B = \sum_{n_1 > -3k, \,\, -N < n_2 \leq 0} 2^{-3n_1/4} 2^{n_2/2}
\leq \sum_{n_1 = -3k}^\infty \sum_{n_2 \leq 0} 2^{-3n_1/4} 2^{n_2/2}
\lesssim 2^{9k/4} \leq 2^{3k}.
\end{gather*}

\noindent The important thing to notice is that there is no
dependence on the number $N$, which depends on $\mathcal{D}$, or
$C$, which depends on $E$.

Combining the estimates for $\mathcal{D}_1$, $\mathcal{D}_2$, and
$\mathcal{D}_3$, we see

\begin{equation*}
\sum_{I \in \mathcal{D}} |\langle \phi_I^1, f \rangle| |\langle
\phi_I^{2,k}, \chi_{E'} \rangle| \lesssim 2^{3k},
\end{equation*}

\noindent where the constant has no dependence on the collection
$\mathcal{D}$.  Hence, as $\mathcal{D}$ is arbitrary, we have

\begin{equation*}
\Big| \sum_I \epsilon_I \langle \phi_I^1, f \rangle \langle
\phi_I^{2,k}, \chi_{E'} \rangle \Big| \leq \sum_I |\langle \phi_I^1,
f \rangle| |\langle \phi_I^{2,k}, \chi_{E'} \rangle| \lesssim
2^{3k},
\end{equation*}

\noindent which completes the proof.
\end{proof}

\begin{thm}\label{thm:TLp} For any 0-mean adapted families, $T_\epsilon
: L^p \rightarrow L^p$ for $1 < p < \infty$, where the underlying
constants are independent of the sequence $\epsilon$.
\end{thm}

\begin{proof} Fix a sequence $\epsilon_I$, and let $\varphi_I^1$,
$\varphi_I^2$ be any two 0-mean adapted families.  By
Theorems~\ref{thm:TL2} and~\ref{thm:Tweak}, $T_\epsilon : L^2
\rightarrow L^2$ and $T_\epsilon : L^1 \rightarrow L^{1,\infty}$. By
the Marcinkiewicz interpolation theorem, $T_\epsilon : L^p
\rightarrow L^p$ for all $1 < p \leq 2$. By symmetry, the operator
$T_\epsilon^* f = \sum \langle \phi_I^2, f \rangle \phi_I^1$
satisfies the same properties.

Fix $f \in L^p$ with $2 < p < \infty$.  Let $\|g\|_{p'} \leq 1$,
where $1/p + 1/p' = 1$ and $1 < p' < 2$.  Then,
\begin{equation*}
\begin{split}
|\langle T_\epsilon f, g \rangle| &= |\langle T_\epsilon^* g, f
\rangle| \leq \|T_\epsilon^* g \|_{p'} \|f\|_p \lesssim \|g\|_{p'}
\|f\|_p \leq \|f\|_p.
\end{split}
\end{equation*}

\noindent As $g$ in the unit ball of $L^{p'}$ is arbitrary, we see
$\|T_\epsilon f\|_p \lesssim \|f\|_p$.
\end{proof}

\section{The $L^p$ Estimates}

The main tool is this section is a randomization argument using
Khinchtine's inequality. Given a probability space $(\Omega, P)$, we
say a random variable $r : \Omega \rightarrow \mathbb{C}$ is a
Rademacher function if $P(r = 1) = P(r = -1) = 1/2$.  For more
background information on probability spaces and independence,
see~\cite{billingsley}.

\begin{lemma}\label{lemma:rad} Let $r_1, \ldots, r_N$ be an independent
sequence of Rademacher functions on $(\Omega, P)$.  For any $t
> 0$ and any $(a_1, \ldots, a_N) \in \mathbb{C}$ such that
$\sum_{j=1}^N |a_j|^2 \leq 1$,

\begin{equation*}
P\bigg\{ \Big| \sum_{j=1}^N a_j r_j \Big| > t \bigg\} \leq
4e^{-t^2/4}.
\end{equation*}
\end{lemma}

\begin{proof} First, suppose the $a_j$ are real.  Write $S_N(\omega)
= \sum_{j=1}^N a_j r_j(\omega)$.  We will use the notation
$E(\cdot)$ for expectation; that is, $E(f) = \int_\Omega f \, dP$.
Recall, if $f$ and $g$ are independent, $E(fg) = E(f) E(g)$.  So,

\begin{equation*}
E(e^{t S_N}) = E \Big( \prod_{j=1}^N e^{t a_j r_j} \Big) =
\prod_{j=1}^N E(e^{t a_j r_j}) = \prod_{j=1}^N \frac{e^{t a_j} +
e^{-t a_j}}{2} = \prod_{j=1}^N \cosh(t a_j).
\end{equation*}

\noindent Observe $\cosh(x) \leq e^{x^2/2}$ for all real $x$. So,
$E(e^{tS_n}) \leq \prod e^{t^2 a_j^2/2} \leq e^{t^2/2}$.  On the other
hand,

\begin{equation*}
E(e^{tS_N}) \geq \int_{\{S_N > t\}} e^{t S_N(\omega)} \, P(d\omega)
\geq e^{t^2} P\{S_N > t\},
\end{equation*}

\noindent which implies $P\{S_N > t\} \leq e^{-t^2} E(e^{t S_N})
\leq e^{-t^2/2}$.

Alternatively, $\{S_N < -t\} = \{-S_N > t\}$, where $-S_N = \sum a_j
(-r_j)$.  As $-r_j$ is also an independent Rademacher sequence, the
same applies to $-S_N$.  In particular, $P\{-S_N > t\} \leq
e^{-t^2/2}$, which gives $P\{|S_N| > t\} \leq P\{S_N > t\} + P\{S_N
< -t\} \leq 2e^{-t^2/2}$.

Now allow $a_j$ to be complex with $\sum |a_j|^2 \leq 1$, so that
$\sum |\re a_j|^2, \sum |\im a_j|^2 \leq 1$.  Let
$S_N$ be as before, with $S_N' = \sum \re(a_j) r_j$ and $S_N'' =
\sum \im(a_j) r_j$.  The above argument works with $S_N'$ and $S_N''$,
and therefore $P\{|S_N| > t\} \leq P\{|S_N'| > t/\sqrt{2}\} + P\{|S_N''| >
 t/\sqrt{2}\} \leq 4e^{-t^2/4}$.
\end{proof}

\begin{thm}\label{thm:Khinchine} For each $0 < p < \infty$, any
sequence of complex numbers  $\{a_j\}_{j \in \mathbb{N}}$ in
$\ell^2$, and any independent sequence of Rademacher functions
$\{r_j\}$ on $\Omega$, we have

\begin{equation*}
\bigg(\sum_{j=1}^\infty |a_j|^2 \bigg)^{1/2} \sim \bigg\|
\sum_{j=1}^\infty a_j r_j \bigg\|_{L^p(\Omega)},
\end{equation*}

\noindent where the underlying constants depend only on $p$.
\end{thm}

\begin{proof} Fix $N \in \mathbb{N}$.  Write $\sigma^2 = \sum_{j=1}^N
|a_j|^2$ and define $b_j = a_j/\sigma$, so that $\sum_{j=1}^N
|b_j|^2 = 1$.  Let $S_N = \sum_{j=1}^N b_j r_j$.  Then, using the
previous lemma,

\begin{equation*}
\int_\Omega |S_N(\omega)|^p \, P(d\omega) = \int_0^\infty p t^{p-1}
P\{|S_N| > t\} \, dt \leq 4p \int_0^\infty t^{p-1} e^{-t^2/4} \, dt
=: K_p^p,
\end{equation*}

\noindent where $K_p < \infty$ for all $0 < p < \infty$.

Suppose $1 < p < \infty$.  Note, by independence, $E(r_j r_k) =
E(r_j) E(r_k) = 0$ for $j \not= k$ and $E(r_j r_j) = 1$.  So,

\begin{equation*}
\int_\Omega |S_N|^2 \, dP = \int_\Omega S_N \overline{S_N} \, dP =
\sum_{1 \leq j, k \leq N} b_j \overline{b_k} \int_\Omega r_j r_k \,
dP = \sum_{j=1}^N |b_j|^2 = 1.
\end{equation*}

\noindent But, by above, $\|S_N\|_{p'} \leq K_{p'}$. By
H\"{o}lder, $1 \leq \|S_N\|_p \|S_N\|_{p'} \leq \|S_N\|_p K_{p'}$,
or $K_{p'}^{-1} \leq \|S_N\|_p$.  Now suppose $0 < p \leq 1$.  Then,
$1 = \int_\Omega |S_N|^2 \, dP \leq \| |S_N|^{p/2}\|_2 \||S_N|^{2 -
p/2}\|_2 = \|S_N\|_p^{2/p} \|S_N\|_{4-p}^{2/(4-p)}$. Note,
$\|S_N\|_{4-p}^{2/(4-p)} \leq K_{4-p}^{2/(4-p)}$.  Therefore,
$K^{p/(p-4)}_{4-p} \leq \|S_N\|_p$.  Let $K_p' = K_{p'}^{-1}$ for $p
> 1$ and $K_p' = K_{4-p}^{p/(4-p)}$ for $p \leq 1$.  Then, we have
shown

\begin{equation*}
K_p' \bigg(\sum_{j=1}^N |a_j|^2 \bigg)^{1/2} \leq \bigg\|
\sum_{j=1}^N a_j r_j \bigg\|_{L^p(\Omega)} \leq K_p
\bigg(\sum_{j=1}^N |a_j|^2 \bigg)^{1/2}.
\end{equation*}

\noindent for all $0 < p < \infty$.  To finish, we note that by
Fatou's Lemma

\begin{equation*}
\int_\Omega \Big| \sum_{j=1}^\infty a_j r_j \, dP \Big|^p \leq
\liminf_{N \rightarrow \infty} \int_\Omega \Big| \sum_{j=1}^N a_j
r_j \, dP \Big|^p \leq K_p \bigg(\sum_{j=1}^\infty |a_j|^2
\bigg)^{p/2}.
\end{equation*}

Fix $1 \leq p < \infty$.  Then, by Minkowski,

\begin{equation*}
\bigg\| \sum_{j=1}^N a_j r_j \bigg\|_{L^p(\Omega)} - \bigg\|
\sum_{j=1}^\infty a_j r_j \bigg\|_{L^p(\Omega)} \leq \bigg\| \sum_{j
= N+1}^\infty a_j r_j \bigg\|_{L^p(\Omega)} \leq K_p \bigg(
\sum_{j=N+1}^\infty |a_j|^2 \bigg)^{1/2},
\end{equation*}

\noindent the last term tending to 0 as $N \rightarrow \infty$,
because $(a_j)$ is in $\ell^2$. Thus,

\begin{equation*}
\bigg\| \sum_{j=1}^\infty a_j r_j \bigg\|_{L^p(\Omega)} \geq
\limsup_{N \rightarrow \infty} \bigg\| \sum_{j=1}^N a_j r_j
\bigg\|_{L^p(\Omega)} \geq K_p' \bigg( \sum_{j=1}^\infty |a_j|^2
\bigg)^{1/2}.
\end{equation*}

Finally, let $0 < p < 1$.  Set $t = (2-2p)/(2-p)$ so that $0 < t <
1$ and $1 = (1-t)/p + t/2$.  Let $F = \sum_{j=1}^\infty a_j r_j$.
Then,

\begin{equation*}
\begin{split}
\|F\|_{L^1(\Omega)} &= \||F|^{1-t} |F|^t\|_{L^1(\Omega)} \leq
\||F|^{1-t}\|_{L^{p/(1-t)}(\Omega)} \||F|^t\|_{L^{2/t}(\Omega)}
= \|F\|_{L^p(\Omega)}^{1-t} \|F\|_{L^2(\Omega)}^t \\
&\leq \|F\|_{L^p(\Omega)}^{1-t} (K_2 \|a\|_{\ell^2})^t \leq
\|F\|_{L^p(\Omega)}^{1-t} (K_2 K_1'^{-1} \|F\|_{L^1(\Omega)})^t
\end{split}
\end{equation*}

\noindent which implies $K_1' \|a\|_{\ell^2} \leq
\|F\|_{L^1(\Omega)} \leq (K_2/K_1')^{t/(1-t)} \|F\|_{L^p(\Omega)}$,
completing the proof.
\end{proof}

\begin{thm}\label{thm:SLp} For any 0-mean adapted family, $S :
L^p \rightarrow L^p$ for $1 < p < \infty$.  \end{thm}

\begin{proof} Let $\varphi_I$ be a 0-mean adapted family and $S$ the
associated square operator.  By Theorem~\ref{thm:specialadapted},
let $\varphi_I^2$ be a second adapted family, with the additional
property that $\chi_I \lesssim |\varphi_I^2|$ for all $I$.  That is,
$\chi_I/|I|^{1/2} \lesssim |\phi_I^2|$.

Let $\{r_I\}$ be an independent sequence of Rademacher functions on
a probability space $(\Omega, P)$ indexed by the dyadic intervals.
For each $\omega \in \Omega$, denote the sequence $\{r_I(\omega)\}$
by $\epsilon(\omega)_I$, and note $|\epsilon(\omega)_I| \leq 1$ for
all $I$.  Let $T_{\epsilon(\omega)}$ be the linearization associated
to $\varphi_I$, $\varphi_I^2$, and the sequence $\epsilon(\omega)$.

Fix $1 < p < \infty$ and $f \in L^p$.  By Khinctine,

\begin{equation*}
\begin{split}
|Sf(x)|^p &= \bigg( \sum_I \frac{|\langle \phi_I, f \rangle|^2}{|I|}
\chi_I(x) \bigg)^{p/2} \lesssim \bigg( \sum_I |\langle \phi_I, f
\rangle|^2 |\phi_I^2(x)|^2 \bigg)^{p/2} \\
&\lesssim \int_\Omega \Big| \sum_I r_I(\omega) \langle \phi_I, f
\rangle \phi_I^2(x) \Big|^p \, P(d\omega) = \int_\Omega
|T_{\epsilon(\omega)} f(x)|^p \, P(d\omega).
\end{split}
\end{equation*}

\noindent So,

\begin{equation*}
\begin{split}
\|Sf\|_p^p &= \int_\mathbb{T} |Sf(x)|^p \, dx \lesssim \int_\mathbb{T}
\int_\Omega |T_{\epsilon(\omega)} f(x)|^p \, dx \, P(d\omega) \\
&= \int_\Omega \big\|T_{\epsilon(\omega)} f \big\|_p^p \, P(d\omega)
\lesssim \int_\Omega \|f\|_p^p \, P(d\omega) = \|f\|_p^p.
\end{split}
\end{equation*}
\end{proof}

\section{Fefferman-Stein Inequalities}

We are also able to prove a special case of
Fefferman-Stein inequalities ($r = 2$) for the square function.
First, we need the following characterization of weak-$L^p$,
sometimes called the Kolmogorov condition.

\begin{lemma}\label{lemma:kol} Let $0 < r < p < \infty$, and choose $s$
so that $1/s = 1/r - 1/p$.  Denote

\begin{equation*}
M_{p,r}(f) = \sup \bigg\{ \frac{\|f \chi_E\|_r}{\|\chi_E\|_s} : |E|
> 0 \bigg\}.
\end{equation*}

\noindent Then, $\|f\|_{p,\infty} \sim M_{p,r}(f)$ for all $f$,
where the underlying constant depends only on $p$ and $r$.
\end{lemma}

\begin{proof} Let $\lambda > 0$ and $E = \{|f| > \lambda\}$.  If
$|E| = 0$,  then $\lambda |E|^{1/p} \leq M_{p,r}(f)$ trivially. So,
assume $|E| > 0$.  Then,

\begin{equation*}
|E|^{1/r} = \Big( \int_E \, dx \Big)^{1/r} \leq \lambda^{-1}
\Big(\int_E |f(x)|^r \, dx \Big)^{1/r} = \lambda^{-1} \|f \chi_E\|_r
\leq \lambda^{-1} \|\chi_E\|_s M_{p,r}(f).
\end{equation*}

\noindent Hence, $M_{p,r}(f) \geq \lambda |E|^{1/r - 1/s} = \lambda
|E|^{1/p}$.  As $\lambda$ is arbitrary, $\|f\|_{p,\infty} \leq
M_{p,r}(f)$.

If $\|f\|_{p,\infty} = \infty$, the reverse inequality is trivially
satisfied. So, assume it is finite. If $\|f\|_{p,\infty} = 0$, then
$f = 0$ a.e., and again the reverse inequality holds.  Assume
$\|f\|_{p,\infty} > 0$.  Set $g = f/\|f\|_{p,\infty}$ which gives
$\|g\|_{p,\infty} = 1$.  Let $|E| > 0$. Then, $|\{|g \chi_E| >
\lambda\}| \leq \min(|E|, \lambda^{-p})$. Thus, for any $h
> 0$

\begin{equation*}
\begin{split} \|g\chi_E\|_r^r &= \int_0^\infty r\lambda^{r-1} \big| \{|g
\chi_E| > \lambda\} \big| \, d\lambda \\
&\leq r |E| \int_0^h \lambda^{r-1} \, d\lambda + r
\int_h^\infty \lambda^{r-p-1} \, d\lambda \\
&= h^r |E| + \frac{r}{p-r} h^{r-p}.
\end{split}
\end{equation*}

\noindent Setting $h = |E|^{-1/p}$ gives $\|g \chi_E\|_r^r \leq
|E|^{r/s} + \frac{r}{p-r} |E|^{r/s}$ and $\|g \chi_E\|_r \leq
(\frac{p}{p-r})^{1/r} |E|^{1/s} = (\frac{p}{p-r})^{1/r}
\|\chi_E\|_s$.  As $E$ is arbitrary, $M_{p,r}(g) \leq
(\frac{p}{p-r})^{1/r}$.  Noting that $M_{p,r}$ is quasi-linear, we
have $M_{p,r}(f) \leq (\frac{p}{p-r})^{1/r} \|f\|_{p,\infty}$.
\end{proof}

\begin{thm} For $1 < p < \infty$ and any sequence $f_1, f_2, \ldots$
of complex-valued functions on $\mathbb{T}$

\begin{gather*}
\bigg\| \Big( \sum_{k=1}^\infty |S f_k|^2 \Big)^{1/2} \bigg\|_p
\lesssim \bigg\| \Big( \sum_{k=1}^\infty |f_k|^2 \Big)^{1/2}
\bigg\|_p, \\
\bigg\| \Big( \sum_{k=1}^\infty |S f_k|^2 \Big)^{1/2}
\bigg\|_{1,\infty} \lesssim \bigg\| \Big( \sum_{k=1}^\infty |f_k|^2
\Big)^{1/2} \bigg\|_1.
\end{gather*}
\end{thm}

\begin{proof} Let $r_I$ be a sequence of independent Rademacher functions,
indexed by the dyadic intervals, on a probability space $(\Omega,
P)$.  Let $r_k'$ be another sequence of independent Rademacher
functions, indexed by $\mathbb{N}$, on a probability space
$(\Omega', P')$. Note, $r_{I,k}(\omega, \omega') = r_I(\omega)
r_k'(\omega')$ is an independent Rademacher sequence on $\Omega
\times \Omega'$.

Let $1 < p < \infty$.  Fix $N \in \mathbb{N}$.  Then, by Khinchtine,

\begin{equation*}
\begin{split}
\bigg\| \Big( \sum_{k=1}^N |S f_k|^2 &\Big)^{1/2} \bigg\|_p^p \\
&= \int_\mathbb{T} \Big( \sum_{k=1}^N \sum_I \frac{|\langle
\phi_I, f_k \rangle|^2}{|I|} \chi_I(x) \Big)^{p/2} \, dx \\
&\lesssim \int_\mathbb{T} \int_{\Omega \times \Omega'} \Big|
\sum_{k=1}^N \sum_I r_I(\omega) r_k'(\omega') \frac{1}{|I|^{1/2}}
\langle \phi_I, f_k \rangle \chi_I(x) \Big|^p \, P(d\omega) \,
P(d\omega') \, dx \\
&= \int_\mathbb{T} \int_{\Omega \times \Omega'} \Big| r_I(\omega)
\frac{1}{|I|^{1/2}} \Big\langle \phi_I, \sum_{k=1}^N r_{k}'(\omega')
f_k \Big\rangle \chi_I(x) \Big|^p \, P(d\omega) \, P'(d\omega') \,
dx.
\end{split}
\end{equation*}

\noindent Now use the reverse inequality of Khinctine, first in
$\Omega$, then $\Omega'$, to see

\begin{equation*}
\begin{split}
\bigg\| \Big( \sum_{k=1}^N |S f_k|^2 \Big)^{1/2} \bigg\|_p^p
&\lesssim \int_{\Omega'} \int_\mathbb{T} \Big( \sum_I \frac{1}{|I|}
\Big|\Big\langle \phi_I, \sum_{k=1}^N r_{k}'(\omega') f_k
\Big\rangle\Big|^2 \chi_I(x) \Big)^{p/2} \, dx \, P'(d\omega') \\
&= \int_{\Omega'} \int_\mathbb{T} \Big| S\Big(\sum_{k=1}^N
r_k'(\omega') f_k \Big)(x) \Big|^p \, dx \, P'(d\omega') \\
&\lesssim \int_{\Omega'} \int_\mathbb{T} \Big| \sum_{k=1}^N
r_k'(\omega') f_k(x) \Big|^p \, dx \, P'(d\omega') \\
&\lesssim \int_\mathbb{T} \Big( \sum_{k=1}^N |f_k(x)|^2 \Big)^{p/2}
\, dx = \bigg\| \Big( \sum_{k=1}^N |f_k|^2 \Big)^{1/2} \bigg\|_p^p.
\end{split}
\end{equation*}

\noindent Simply apply the monotone convergence theorem to let $N
\rightarrow \infty$ and gain the desired result.

Now let $|E| > 0$.  Fix $0 < r < 1$ and $1/s = 1/r - 1$. As
$\|Sf\|_{1,\infty} \lesssim \|f\|_1$, it follows from
Lemma~\ref{lemma:kol} that $\|S(f) \chi_E\|_r \lesssim \|\chi_E\|_s
\|f\|_1$.  Again, fix $N \in \mathbb{N}$.  So,

\begin{equation*}
\begin{split}
\bigg\| \Big( &\sum_{k=1}^N |S f_k|^2 \Big)^{1/2} \chi_E \bigg\|_r^r
\\
&= \int_\mathbb{T} \Big( \sum_{k=1}^N \sum_I \frac{|\langle
\phi_I, f_k \rangle|^2}{|I|} \chi_I(x) \chi_E(x) \Big)^{r/2} \, dx \\
&\lesssim \int_\mathbb{T} \int_{\Omega \times \Omega'} \Big|
\sum_{k=1}^N \sum_I r_I(\omega) r_k'(\omega') \frac{1}{|I|^{1/2}}
\langle \phi_I, f_k \rangle \chi_I(x) \chi_E(x) \Big|^r \,
P(d\omega) \, P(d\omega') \, dx \\
&= \int_\mathbb{T} \int_{\Omega \times \Omega'} \Big| r_I(\omega)
\frac{1}{|I|^{1/2}} \Big\langle \phi_I, \sum_{k=1}^N r_{k}'(\omega')
f_k \Big\rangle \chi_I(x) \chi_E(x) \Big|^r \, P(d\omega)
\, P'(d\omega') \, dx \\
&\lesssim \int_{\Omega'} \int_\mathbb{T} \Big| S\Big(\sum_{k=1}^N
r_k'(\omega') f_k\Big)(x) \chi_E(x) \Big|^r \, dx \, P'(d\omega') \\
&\lesssim \|\chi_E\|_s^r \int_{\Omega'} \bigg[ \int_\mathbb{T} \Big|
\sum_{k=1}^N r_k'(\omega') f_k(x) \Big| \, dx \bigg]^r \, P'(d\omega') \\
\end{split}
\end{equation*}

\noindent As $\Omega'$ is a probability space and $r < 1$, we can apply
Jensen's inequality to see

\begin{equation*}
\begin{split}
\bigg\| \Big( \sum_{k=1}^N |S f_k|^2 \Big)^{1/2} \chi_E \bigg\|_r
&\lesssim \|\chi_E\|_s \bigg( \int_{\Omega'} \Big[ \int_\mathbb{T} \Big|
\sum_{k=1}^N r_k'(\omega') f_k(x) \Big| \, dx \Big]^r \, P(d\omega')
\bigg)^{1/r} \\
&\leq \|\chi_E\|_s \int_{\Omega'} \int_\mathbb{T} \Big|
\sum_{k=1}^N r_k'(\omega') f_k(x) \Big| \, dx \, P(d\omega') \\
&\lesssim \|\chi_E\|_s^r \int_\mathbb{T} \Big( \sum_{k=1}^N
|f_k(x)|^2 \Big)^{r/2} \, dx \\
&= \|\chi_E\|_s^r \bigg\| \Big(
\sum_{k=1}^N |f_k|^2 \Big)^{1/2} \bigg\|_r^r.
\end{split}
\end{equation*}

\noindent Taking the supremum over all such $E$, and applying
Lemma~\ref{lemma:kol},

\begin{equation*}
\bigg\| \Big( \sum_{k=1}^N |S f_k|^2 \Big)^{1/2} \bigg\|_{1,\infty}
\lesssim M_{1,r}\bigg( \Big( \sum_{k=1}^N |S f_k|^2 \Big)^{1/2}
\bigg) \lesssim \bigg\| \Big( \sum_{k=1}^N |f_k|^2 \Big)^{1/2}
\bigg\|_1.
\end{equation*}

\noindent Letting $N \rightarrow \infty$ completes the proof.
\end{proof}

\chapter{Zygmund Spaces and $\llogl$}

In this chapter, we begin by focusing on a general measure space
$(X, \rho)$. Our goal is to introduce new spaces of functions and
interpolation results that will ultimately give us the ``end-point"
estimates of certain operators. Many of the preliminary proofs of
this chapter are taken from~\cite{inter}.

\section{Decreasing Rearrangements}

\begin{defnn} For $f : (X, \rho) \rightarrow \mathbb{C}$, the
distribution function of $f$ is defined

\begin{equation*}
\mu_f(\lambda) = \rho\{x \in X : |f(x)| > \lambda\}, \quad \lambda
\geq 0.
\end{equation*}

\noindent Two function $f, g$ (even if they act on different measure
spaces) are said to be equimeasurable if $\mu_f(\lambda) =
\mu_g(\lambda)$ for all $\lambda \geq 0$.
\end{defnn}

\begin{defnn} For $f : (X, \rho) \rightarrow \mathbb{C}$, the
deceasing rearrangement of $f$ is defined

\begin{equation*}
f^*(t) = \inf \{\lambda \geq 0 : \mu_f(\lambda) \leq t\}, \quad t
\geq 0,
\end{equation*}

\noindent where we use the convention that $\inf \{\emptyset\} =
\infty$.
\end{defnn}

Note, if $(X, \rho)$ is a finite measure space, then $\mu_f(\lambda)
\leq \rho(X)$ for all $\lambda \geq 0$. Hence, $f^*(t) = 0$ for all
$t > \rho(X)$. That is, $f^*$ is supported in $[0, \rho(X)]$.

\begin{prop}\label{prop:equi} For any $f, f_n, g : (X, \rho) \rightarrow
\mathbb{C}$ and $\alpha \in \mathbb{C}$,

\begin{enumerate}
\item $f^*$ is nonnegative, decreasing, and identically 0 if and
only if $f = 0$ a.e.[$\rho$],

\item $|f| \leq |g|$ a.e.[$\rho$] implies $f^* \leq g^*$ pointwise,

\item $f^*(\mu_f(\lambda)) \leq \lambda$ for $\mu_f(\lambda) < \infty$,
and $\mu_f(f^*(t)) \leq t$ for $f^*(t) < \infty$,

\item $(f+g)^*(t_1 + t_2) \leq f^*(t_1) + g^*(t_2)$,

\item $(\alpha f)^* = |\alpha| f^*$,

\item $|f_n| \uparrow |f|$ a.e.[$\rho$] implies $f_n^* \uparrow f^*$
pointwise,

\item $f$ and $f^*$ are equimeasurable.
\end{enumerate}
\end{prop}

\begin{proof} (1) The fact that $f^* \geq 0$ follows from the
definition.  Let $0 \leq t_1 < t_2$ and $\epsilon > 0$.  Choose
$\lambda \geq 0$ so that $\mu_f(\lambda) \leq t_1$ and $f^*(t_1) +
\epsilon \geq \lambda$. Then, $\mu_f(\lambda) \leq t_1 < t_2$ which
implies $f^*(t_2) \leq \lambda \leq f^*(t_1) + \epsilon$. As
$\epsilon$ is arbitrary, $f^*(t_1) \geq f^*(t_2)$.  But, since $f^*$
is decreasing, $f^*$ is identically 0 if and only if $f^*(0) = 0$.
This is true if and only if $\mu_f(0) = 0$, which means $f = 0$
a.e..

(2) Fix $t$ and $\epsilon > 0$.  As $|f| \leq |g|$ a.e., it is
immediately clear that $\mu_f \leq \mu_g$.  Choose $\lambda \geq 0$
so that $\mu_g(\lambda) \leq t$ and $g^*(t) + \epsilon \geq
\lambda$.  Then, $\mu_f(\lambda) \leq \mu_g(\lambda) \leq t$, which
implies that $f^*(t) \leq \lambda \leq g^*(t) + \epsilon$.  As
$\epsilon$ is arbitrary, $f^*(t) \leq g^*(t)$.

(3) Fix $\lambda \geq 0$ and set $t = \mu_f(\lambda)$. Then,
$\lambda \in \{\lambda' \geq 0 : \mu_f(\lambda') \leq t\}$ giving
$f^*(\mu_f(\lambda)) = f^*(t) = \inf\{ \lambda' : \mu_f(\lambda')
\leq t \} \leq \lambda$. Now fix $t \geq 0$ and assume $\lambda =
f^*(t) < \infty$. Let $\lambda_n$ be a sequence of positive numbers
so that $\lambda_n \downarrow \lambda$.  Then, $\mu_f(\lambda_n)
\leq t$ for each $n$. Therefore, as $\{|f| >
\lambda_n\} \subseteq \{|f| > \lambda\}$ for all $n$ and

\begin{equation*}
\bigcup_n \big\{|f| > \lambda_n \big\} = \big\{ |f| > \lambda
\big\},
\end{equation*}

\noindent it follows from simple properties of measures that
$\mu_f(\lambda_n) \uparrow \mu_f(\lambda)$.  That is,
$\mu_f(\lambda) = \lim_n \mu_f(\lambda_n) \leq t$.

(4) Let $t_1, t_2 \geq 0$.  Let $\lambda = f^*(t_1) + f^*(t_2)$ and
$t = \mu_{f+g}(\lambda)$.  Then,

\begin{equation*}
\begin{split}
t &= |\{|f+g| > \lambda\}| \leq |\{|f| > f^*(t_1)\}| + |\{|g| >
g^*(t_2)\}| \\
&= \mu_f(f^*(t_1)) + \mu_g(g^*(t_2)) \leq t_1 + t_2.
\end{split}
\end{equation*}

\noindent So, $(f+g)^*(t_1 + t_2) \leq (f+g)^*(t) =
(f+g)^*(\mu_{f+g}(\lambda)) \leq \lambda = f^*(t_1) + f^*(t_2)$.

(5) For $\alpha \in \mathbb{C}$, we have $\mu_{\alpha f}(\lambda) =
\rho\{|\alpha f| > \lambda\} = \rho\{|f| > \lambda/|\alpha|\} =
\mu_{f}(\lambda/|\alpha|)$.  Thus, $(\alpha f)^*(t) = \inf\{\lambda
\geq 0 : \mu_{\alpha f}(\lambda) \leq t\} = \inf\{|\alpha| \lambda
\geq 0 : \mu_f(\lambda) \leq t\} = |\alpha| f^*(t)$.

(6) It is clear from (2) that $f_1^* \leq f_2^* \leq \ldots \leq
f^*$ pointwise.  Fix $\lambda$.  By the same argument used in (3),
we see $\{|f_n| > \lambda\} \subseteq \{|f| > \lambda\}$ and
$\bigcup \{|f_n| > \lambda\} = \{|f| > \lambda\}$.  Thus,
$\mu_{f_n}(\lambda) \uparrow \mu_f(\lambda)$.  By the same token, it
is now clear that $\{\lambda : \mu_f(\lambda) \leq t\} \subseteq
\{\lambda : \mu_{f_n}(\lambda) \leq t\}$ and $\bigcap_n \{\lambda :
\mu_{f_n}(\lambda) \leq t\} = \{\lambda : \mu_f(\lambda) \leq t\}$.
Therefore, taking infimums, we see $f_n^*(t) \uparrow f^*(t)$.

(7) Simply from the definition, $f^*(t)> \lambda$ if and only if $t
< \mu_f(\lambda)$. Thus, $\mu_{f^*}(\lambda) = |\{t \geq 0 : f^*(t)
> \lambda\}| = |[0, \mu_f(\lambda))| = \mu_f(\lambda)$. \end{proof}

\begin{lemma}\label{lemma:f*map} Let $\Psi : [0,\infty) \rightarrow
[0, \infty)$ be continuous and increasing with $\Psi(0) = 0$.  Then,
$\int_X \Psi(|f|) \, d\rho = \int_0^\infty \Psi(f^*) \, dt$.
\end{lemma}

\begin{proof} First consider the case where $f$ is positive and
simple.  That is, there are constants $a_1 > a_2 > \ldots > a_n > 0$
and disjoint sets $E_1, \ldots, E_n$ so that $f = \sum a_j
\chi_{E_j}$. It is easy to calculate that $f^*(t) = \sum a_j
\chi_{[m_{j-1}, m_j)}$, where $m_0 = 0$ and $m_j = \rho(E_1) +
\ldots + \rho(E_j)$.  Thus,

\begin{equation*}
\begin{split}
\int_X \Psi(f) \, d\rho &= \sum_{j=1}^n \int_{E_j} \Psi(a_j) \, d\rho
= \sum_{j=1}^n \Psi(a_j) \rho(E_j) = \sum_{j=1}^n \Psi(a_j) \big[ m_j
- m_{j-1} \big] \\
&= \sum_{j=1}^n \int_0^\infty \Psi(a_j) \chi_{[m_{j-1}, m_j)}(t) \,
dt = \int_0^\infty \Psi(f^*) \, dt.
\end{split}
\end{equation*}

\noindent Note that $\Psi(0) = 0$ was used here to say $\Psi(a_j
\chi_{E_j}) = \Psi(a_j) \chi_{E_j}$.

Now consider a general $f : (X, \rho) \rightarrow \mathbb{C}$.
Choose positive simple functions $f_n$ so that $f_n \uparrow |f|$.
As $\Psi$ is continuous and increasing, it follows that $\Psi(f_n)
\uparrow \Psi(|f|)$.  Also, as $f_n^* \uparrow f^*$, we have
$\Psi(f_n^*) \uparrow \Psi(f^*)$.  So, by the monotone convergence
theorem,

\begin{equation*}
\int_X \Psi(|f|) \, d\rho = \lim_{n \rightarrow \infty} \int_X
\Psi(f_n) \, d\rho = \lim_{n \rightarrow \infty} \int_0^\infty
\Psi(f_n^*) \, dt = \int_0^\infty \Psi(f^*) \, dt.
\end{equation*}
\end{proof}

\begin{cor}\label{cor:f*p} For $f : (X, \rho) \rightarrow
\mathbb{C}$ and $0 < p < \infty$, we have $\|f\|_p = \|f^*\|_p$.
Furthermore, $\|f\|_\infty = \|f^*\|_\infty = f^*(0)$.
\end{cor}

\begin{proof} In the case $0 < p < \infty$, simply let $\Psi(t) =
t^p$ and apply the previous lemma.  Secondly, note that
$\|f\|_\infty = f^*(0)$ by definition.  As $f^*$ is decreasing,
$\|f^*\|_\infty = f^*(0)$. \end{proof}

\section{Lorentz Spaces}

\begin{defnn} Let $0 < p < \infty$ and $0 < q \leq \infty$.
For $f : (X, \rho) \rightarrow \mathbb{C}$, define $\|f\|_{p,q}$ by

\begin{equation*}
\|f\|_{p,q} = \begin{cases} \displaystyle{\bigg( \int_0^\infty \Big(
t^{1/p}
f^*(t) \Big)^q \, \frac{dt}{t} \bigg)^{1/q}},& \qquad q < \infty \\
\displaystyle{\sup_{t > 0} t^{1/p} f^*(t)},& \qquad q = \infty.
\end{cases}
\end{equation*}

\noindent Denote by $L^{p,q}(X)$ be the set of functions $f$ for
which $\|f\|_{p,q} < \infty$.  \end{defnn}

It is clear from Corollary~\ref{cor:f*p} that $\|f\|_{p,p} =
\|f\|_p$. Further, one can check that $L^{p,\infty}$ here coincides
with the definition of weak-$L^p$ given in
Section~\ref{sec:analysisont}.

\begin{lemma} Let $0 < p < \infty$ and $0 < q < r \leq \infty$.
Then, $\|f\|_{p,r} \lesssim \|f\|_{p,q}$, where the underlying
constants depend only on $p, q, r$.  \end{lemma}

\begin{proof} As $f^*$ is decreasing,

\begin{equation*}
\begin{split}
t^{1/p} f^*(t) &= \bigg( \frac{p}{q} \int_0^t \Big( s^{1/p} f^*(t)
\Big)^q \, \frac{ds}{s} \bigg)^{1/q} \\
&\leq \bigg( \frac{p}{q} \int_0^t \Big( s^{1/p} f^*(s) \Big)^q \,
\frac{ds}{s} \bigg)^{1/q} = \left( \frac{p}{q} \right)^{1/q}
\|f\|_{p,q}.
\end{split}
\end{equation*}

\noindent Taking the supremum over all $t$, we see $\|f\|_{p,\infty}
\leq (\frac{p}{q})^{1/q} \|f\|_{p,q}$.  This gives the $r = \infty$
case.  Now, suppose $r < \infty$.  Then,

\begin{equation*}
\begin{split}
\|f\|_{p,r} &= \bigg( \int_0^\infty \Big( t^{1/p} f^*(t)
\Big)^{r-q+q} \frac{dt}{t} \bigg)^{1/r} \\
&\leq \|f\|_{p,\infty}^{1 - q/r} \|f\|_{p,q}^{q/r} \leq \bigg(
\Big(\frac{p}{q}\Big)^{1/q} \|f\|_{p,q} \bigg)^{1 - q/r}
\|f\|_{p,q}^{q/r} = \Big(\frac{p}{q} \Big)^{1/q - 1/r} \|f\|_{p,q}.
\end{split}
\end{equation*}
\end{proof}

\begin{lemma}\label{lemma:weak} Let $T$ be a sublinear operator
which which maps $L^{p_0}(X) \rightarrow L^{q_0, \infty}(X)$ and
$L^{p_1}(X) \rightarrow L^{q_1, \infty}(X)$, where $1 \leq p_0 < p_1
< \infty$, $1 \leq q_0, q_1 < \infty$, and $q_0 \not= q_1$. Then,

\begin{equation*}
(Tf)^*(t) \lesssim \bigg[ t^{-1/q_0} \int_0^{t^m} s^{1/p_0} f^*(s)
\, \frac{ds}{s} + t^{-1/q_1} \int_{t^m}^\infty s^{1/p_1} f^*(s) \,
\frac{ds}{s} \bigg], \quad t > 0,
\end{equation*}

\noindent where $m = (\frac{1}{q_0} - \frac{1}{q_1}) (\frac{1}{p_0}
- \frac{1}{p_1})^{-1}$.
\end{lemma}

\begin{proof} Let $\alpha(x)$ be a complex-valued
function with $|\alpha(x)| = 1$ so that $|f(x)| \alpha(x) = f(x)$.
Fix $t > 0$. Define $f_0$ and $f_1$ by

\begin{gather*}
f_0(x) = \max \big\{|f(x)| - f^*(t^m), 0 \big\} \cdot \alpha(x), \\
f_1(x) = \min \big\{|f(x)|, f^*(t^m) \big\} \cdot \alpha(x).
\end{gather*}

\noindent Then, $f = f_0 + f_1$, and it is easily shown that
$f_0^*(s) = \max\{f^*(s) - f^*(t^m), 0\}$ and $f_1^*(s) =
\min\{f^*(s), f^*(t^m)\}$. Further,

\begin{gather*}
\|f_0\|_{p_0,1} = \int_0^{t^m} s^{1/p_0} f^*(s) \,
\frac{ds}{s} - p_0 t^{m/p_0} f^*(t^m), \\
\|f_1\|_{p_1,1} = p_1 t^{m/p_1} f^*(t^m) + \int_{t^m}^\infty
s^{1/p_1} f^*(s) \, \frac{ds}{s}.
\end{gather*}

\noindent As $T$ is sublinear, $(Tf)^*(t) \leq (Tf_0
+ Tf_1)^*(t) \leq (Tf_0)^*(t/2) + (Tf_1)^*(t/2)$.  By the hypotheses
on $T$,

\begin{equation*}
\Big(\frac{t}{2}\Big)^{1/q_0} (T f_0)^*(t/2) \leq \|Tf_0\|_{q_0,
\infty} \lesssim \|f_0\|_{p_0} \lesssim \|f_0\|_{p_0,1},
\end{equation*}

\noindent or

\begin{equation*}
(Tf_0)^*(t/2) \lesssim t^{-1/q_0} \|f_0\|_{p_0,1}.
\end{equation*}

\noindent Similarly,

\begin{equation*}
(Tf_1)^*(t/2) \lesssim t^{-1/q_1} \|f_1\|_{p_1,1}.
\end{equation*}

\noindent Hence,

\begin{equation*}
\begin{split}
(Tf)^*(t) &\leq (Tf_0)^*(t/2) + (Tf_1)^*(t/2) \\
&\lesssim \bigg[ \frac{1}{p_0} t^{-1/q_0} \|f_0\|_{p_0,1} +
\frac{1}{p_1} t^{-1/q_1} \|f_1\|_{p_1,1} \bigg] \\
&= \bigg[ \frac{1}{p_0} t^{-1/q_0} \int_0^{t^m} s^{1/p_0} f^*(s) \,
\frac{ds}{s} + \frac{1}{p_1} t^{-1/q_1} \int_{t^m}^\infty
s^{1/p_1} f^*(s) \, \frac{ds}{s} \\
&\qquad \qquad + t^{m/p_1 - 1/q_1}f^*(t^m) - t^{m/p_0 - 1/q_0}
f^*(t^m) \bigg]
\end{split}
\end{equation*}

\noindent Noting that $\frac{m}{p_0} - \frac{1}{q_0} = \frac{m}{p_1}
- \frac{1}{q_1}$, the $f^*(t^m)$ terms cancel.  Thus,

\begin{equation*}
(Tf)^*(t) \lesssim \bigg[ t^{-1/q_0} \int_0^{t^m} s^{1/p_0} f^*(s)
\, \frac{ds}{s} + t^{-1/q_1} \int_{t^m}^\infty s^{1/p_1} f^*(s) \,
\frac{ds}{s} \bigg].
\end{equation*}
\end{proof}

\section{The 2-Star Operator}

The next step is to define a kind of maximal operator of $f^*$,
which we call the 2-star operator.

\begin{defnn} For $f : (X, \rho) \rightarrow \mathbb{C}$, define

\begin{equation*}
f^{**}(t) = \frac{1}{t} \int_0^t f^*(s) \, ds, \quad t > 0.
\end{equation*}
\end{defnn}

\begin{prop} For any $f, f_n, g : (X, \rho) \rightarrow \mathbb{C}$ and
$\alpha \in \mathbb{C}$,

\begin{enumerate}
\item $f^{**}$ is nonnegative, decreasing, and identically 0 if and
only if $f = 0$ a.e.[$\rho$],

\item $f^* \leq f^{**}$,

\item $|f| \leq |g|$ a.e.[$\rho$] implies $f^{**} \leq g^{**}$ pointwise,

\item $(\alpha f)^{**} = |\alpha| f^{**}$,

\item $|f_n| \uparrow |f|$ a.e.[$\rho$] implies $f_n^{**} \uparrow
f^{**}$ pointwise.
\end{enumerate}
\end{prop}

\begin{proof} The fact that $f^{**}$ is nonnegative and equal to
0 if and only if $f = 0$ a.e. follows as $f^*$ satisfies the same
properties.  Let $0 \leq t_1 < t_2$.  As $f^*$ is decreasing $f^*(s)
\leq f^*(s t_1/t_2)$ for any $s \geq 0$.  Thus,

\begin{equation*}
f^{**}(t_2) = \frac{1}{t_2} \int_0^{t_2} f^*(s) \, ds \leq
\frac{1}{t_2} \int_0^{t_2} f^*(st_1/t_2) \, ds = \frac{1}{t_1}
\int_0^{t_1} f^*(u) \, du = f^{**}(t_1).
\end{equation*}

\noindent This establishes (1).  Again, as $f^*$ is decreasing,

\begin{equation*}
f^{**}(t) = \frac{1}{t} \int_0^t f^*(s) \, ds \geq f^*(t)
\frac{1}{t} \int_0^t \, dt = f^*(t).
\end{equation*}

\noindent This establishes (2).  Properties (3), (4), and (5) follow
immediately from the fact that $f^*$ satisfies the same properties,
in addition to the monotone convergence theorem for (5).
\end{proof}

We will also want to show that the 2-star operator is sublinear.
This is more difficult than the preceding results, and needs the
following intermediary step.

\begin{lemma}\label{lemma:f**sum} For all $t > 0$, $\displaystyle{
\inf_{f = g + h} \big\{ \|g\|_1 + t \|h\|_\infty \big\} = t
f^{**}(t)}$.
\end{lemma}

\begin{proof} Fix $t > 0$ and $f : (X, \rho) \rightarrow \mathbb{C}$.
Let $\alpha_t$ be the value of the infimum on the left-hand side of
the equality.  We first show $t f^{**}(t) \leq \alpha_t$.

We can assume that $f$ can be decomposed into $g + h$ as implied, as
otherwise $\alpha_t = \infty$ and there is nothing to prove.  So,
write $f = g + h$ where $g \in L^1(X)$ and $h \in L^\infty(X)$.  Let
$n \in \mathbb{N}$.  Then,

\begin{equation*}
\begin{split}
t f^{**}(t) &= \int_0^t f^*(s) \, ds \leq \int_0^t g^*
\Big(\frac{n-1}{n} \, s \Big) \, ds + \int_0^t h^* \Big( \frac{1}{n}
\, s \Big) \, ds \\
&= \frac{n}{n-1} \int_0^{t(n-1)/n} g^*(u) \, du + n \int_0^{t/n}
h^*(u) \, du \\
&\leq \frac{n}{n-1} \int_0^\infty g^*(u) \, du + n h^*(0)
\int_0^{t/n} \, du \\
&= \frac{n}{n-1} \|g\|_1 + t \|h\|_\infty.
\end{split}
\end{equation*}

\noindent As $n$ is arbitrary, let $n \rightarrow \infty$ to see $t
f^{**}(t) \leq \|g\|_1 + t \|h\|_\infty$.  As this decomposition is
arbitrary, $tf^{**}(t) \leq \alpha_t$.

For the reverse inequality, we can assume $f^{**}(t)$ is finite, or
there is nothing to prove; so $f^*(t) \leq f^{**}(t) < \infty$.  Let
$E = \{x \in X : |f(x)| > f^*(t)\}$ and $t_0 = \rho(E)$. By
Proposition~\ref{prop:equi}, $t_0 = \mu_f(f^*(t)) \leq t$.  As $f$
and $f^*$ are equimeasurable, and $f^*$ is decreasing, it follows
that $f^*(s) = f^*(t)$ for $t_0 < s \leq t$.

As $|f \chi_E| \leq |f|$, we see $(f \chi_E)^* \leq f^*$.  But, $f
\chi_E$ is supported on a set of measure $t_0$.  So, $(f
\chi_E)^*(s) = 0$ for $s > t_0$.  Thus,

\begin{equation*}
\int_E |f| \, d\rho = \int_0^\infty (f \chi_E)^*(s) \, ds =
\int_0^{t_0} (f \chi_E)^*(s) \, ds
\leq \int_0^{t_0} f^*(s) \, ds.
\end{equation*}

Define $g$ and $h$ by

\begin{gather*}
g(x) = \max \big\{|f(x)| - f^*(t), 0 \big\} \cdot \alpha(x), \\
h(x) = \min \big\{|f(x)|, f^*(t) \big\} \cdot \alpha(x),
\end{gather*}

\noindent where $\alpha(x) |f(x)| = f(x)$, so that $f = g + h$.
Observe,

\begin{equation*}
\|g\|_1 = \int_E |f| \, d\rho - \rho(E) f^*(t) \leq \int_0^{t_0}
f^*(s) \, ds - t_0 f^*(t).
\end{equation*}

\noindent On the other hand, $\|h\|_\infty \leq f^*(t)$ is clear
from construction.  Therefore,

\begin{equation*}
\begin{split}
\alpha_t &\leq \|g\|_1 + t
\|h\|_\infty \leq \int_0^{t_0} f^*(s) \, ds + (t - t_0) f^*(t)
= \int_0^t f^*(s) \, ds = t f^{**}(t).
\end{split}
\end{equation*}
\end{proof}

\begin{thm}\label{thm:2starlinear} The 2-star operator is sublinear,
i.e., for any $f_1, f_2 : (X, \rho) \rightarrow \mathbb{C}$ and $t
> 0$, $(f_1 + f_2)^{**}(t) \leq f_1^{**}(t) + f_2^{**}(t)$.
\end{thm}

\begin{proof} Fix $t > 0$ and $\epsilon > 0$.  By the preceding
lemma, choose $g_1, g_2 \in L^1(X)$ and $h_1, h_2 \in L^\infty(X)$
so that $f_j = g_j + h_j$ and $\|g_j\|_1 + t \|h_j\|_\infty \leq t
f_j^{**}(t) + \epsilon$ for $j = 1, 2$. Then,

\begin{equation*}
\begin{split}
t(f_1 + f_2)^{**}(t) &\leq \|g_1 + g_2\|_1 + t\|h_1 + h_2\|_\infty
\\
&\leq \Big(\|g_1\|_1 + t \|h_1\|_\infty \Big) + \Big( \|g_2\|_1 + t
\|h_2\|_\infty \Big) \\
&\leq t f_1^{**}(t) + t f_2^{**}(t) + 2 \epsilon.
\end{split}
\end{equation*}

\noindent As $\epsilon$ is arbitrary, this completes the proof.
\end{proof}

\section{A Characterization of $\llogl$}

The space Zygmund space $\llogl$ arises naturally in a number of
ways, particularly interpolation results.  However, the exact
definition of the space differs with the given application, and most
definitions are somewhat unwieldy.  The definition we present here,
and use for the remainder of the text, is less conceptually natural,
but once certain properties are established, is much easier to use.

For this section, we restrict $(X, \rho)$ to be a probability space.
For functions $f$ on $X$, $f^*(t) = 0$ for $t > 1$.  So, for
simplicity, we can think of $f^*$ and $f^{**}$ as functions defined
only on $[0, 1]$.

\begin{defnn} For functions $f : (X, \rho) \rightarrow
\mathbb{C}$ define $\|f\|_{\llogl}$ by

\begin{equation*}
\|f\|_{\llogl} = \int_0^1 f^{**}(t) \, dt.
\end{equation*}

\noindent Define the Zygumnd space $\llogl(X)$ as the set of all
functions $f : (X,\rho) \rightarrow \mathbb{C}$ with $\|f\|_{\llogl}
< \infty$.
\end{defnn}

It is clear from what we know about the 2-star operator that
$\|\cdot\|_{L\log L}$ is a norm and $\llogl(X)$ is a Banach space.
Further, we know that if $|g| \leq |f|$ a.e.[$\rho$] then
$\|g\|_{\llogl} \leq \|f\|_{\llogl}$ and $|f_n| \uparrow |f|$
a.e.[$\rho$] implies $\|f_n\|_{\llogl} \uparrow \|f\|_{\llogl}$.
What is not clear is the reason for choosing this definition. This is
explained by the following.

\begin{thm} $f \in \llogl(X)$ if and only if

\begin{equation*}
\int_{X} |f(x)| \log^+ |f(x)| \, \rho(dx) < \infty,
\end{equation*}

\noindent where $\log^+(x) = \max(\log x, 0)$.
\end{thm}

\begin{proof} As the map $x \mapsto x \log^+ x$ is continuous,
increasing, and has value $0$ at $x = 0$, we have by
Lemma~\ref{lemma:f*map} that $\int_X |f| \log^+|f| \, d\rho$ is
finite if and only if $\int_0^1 f^*(t) \log^+ f^*(t) \, dt$ is
finite.  On the other hand, changing the order of integration shows

\begin{equation*}
\int_0^1 f^{**}(t) \, dt = \int_0^1 f^*(s) \int_s^1 \frac{1}{t} \,
dt \, ds = \int_0^1 f^*(s) \log(1/s) \, ds.
\end{equation*}

Assume $\int_0^1 f^*(t) \log^+ f^*(t) \, dt$ is finite. Let $E
= \{t \in (0,1) : f^*(t)
> t^{-1/2}\}$ and $F = (0,1) - E$.  Then,

\begin{equation*}
\begin{split}
\int_0^1 f^*(t) \log(1/t) \, dt & \leq \int_E f^*(t) \log( f^*(t)^2)
\, dt + \int_F t^{-1/2} \log(1/t) \, dt \\
&\leq 2 \int_0^1 f^*(t) \log^+ f^*(t) + \int_0^1 t^{-1/2} \log(1/t)
\, dt \\
&= 2 \int_0^1 f^*(t) \log^+ f^*(t) + 4 < \infty.
\end{split}
\end{equation*}

Now suppose $\int_0^1 f^{**}(t) \, dt$ is finite.  Then, $\|f\|_1 =
\int_0^1 f^*(t) \, dt \leq \int_0^1 f^{**}(t) \, dt < \infty$.  If
$\|f\|_1 = 0$ there is nothing to prove, so assume otherwise.  Let
$g = f/\|f\|_1$ so that $\|g\|_1 = 1$. Then, $g^*(t) \leq g^{**}(t)
\leq \|g\|_1/t = 1/t$. Also,

\begin{equation*}
\begin{split}
\int_0^1 g^*(t) \log^+ g^*(t) \, dt &\leq \int_0^1 g^*(t) \log^+
(1/t) \, dt = \int_0^1 g^*(t) \log (1/t) \, dt \\
&= \frac{1}{\|f\|_1} \int_0^1 f^*(t)
\log(1/t) \, dt < \infty \\
\end{split}
\end{equation*}

\noindent But,

\begin{equation*}
\begin{split}
\int_0^1 f^*(t) \log^+ f^*(t) \, dt &= \|f\|_1 \int_0^1 g^*(t)
\log^+( \|f\|_1 g^*(t)) \, dt \\
&\leq \|f\|_1 \Big[ \int_0^1 g^*(t) \log^+ \|f\|_1 \, dt + \int_0^1
g^*(t) \log^+ g^*(t) \, dt \Big] \\
&= \|f\|_1 \Big[ \log^+ \|f\|_1 + \int_0^1 g^*(t) \log^+ g^*(t) \,
dt \Big] < \infty.
\end{split}
\end{equation*}
\end{proof}

The quantity $\int_X |f| \log^+ |f| \, d\rho$ is often taken as the
definition of $\|\cdot\|_{\llogl}$.  Indeed, this quantity naturally
arises in many arguments.  However, it is clearly not a norm, and
makes any deep analysis difficult.

Our next goal is to show how $\llogl$ is related to $L^p$.  First,
we prove a special case of Hardy's
inequality~\cite{hardyinequality}.

\begin{lemma}\label{lemma:hardy} Let $1 < p < \infty$ and $\psi$ be
a nonnegative, measurable function on $(0,1)$.  Then,

\begin{equation*}
\bigg[ \int_0^1 \Big( \frac{1}{t} \int_0^t \psi(s) \, ds \Big)^p dt
\bigg]^{1/p} \lesssim \Big( \int_0^1 \psi(s)^p \, ds \Big)^{1/p},
\end{equation*}

\noindent where the underlying constants depend only on $p$.
\end{lemma}

\begin{proof} Fix $p$.  Let $p'$ be the conjugate exponent of $p$;
that is, $\frac{1}{p} + \frac{1}{p'} = 1$.  Write $\psi(s) =
[s^{-1/pp'}] [ s^{1/pp'} \psi(s)]$ and apply H\"{o}lder to see

\begin{equation*}
\begin{split}
\frac{1}{t} \int_0^t \psi(s) \, ds &\leq \Big( \frac{1}{t} \int_0^t
s^{-1/p} \, ds \Big)^{1/p'} \Big( \frac{1}{t} \int_0^t s^{1/p'}
\psi(s)^p \, ds \Big)^{1/p} \\
&= p'^{1/p'} t^{-1/p - 1/pp'} \Big( \int_0^t s^{1/p'} \psi(s)^p \,
ds \Big)^{1/p}.
\end{split}
\end{equation*}

\noindent Thus,

\begin{equation*}
\begin{split}
\int_0^1 \Big( \frac{1}{t} \int_0^t \psi(s) \, ds \Big)^p \, dt
&\leq p'^{p/p'} \int_0^1 t^{-1 - 1/p'} \int_0^t s^{1/p'} \psi(s)^p
\, ds \, dt \\
&= p'^{p/p'} \int_0^1 s^{1/p'} \psi(s)^p \int_s^1 t^{-1 - 1/p'} \,
dt \, ds  \\
&= p'^{p/p'} \int_0^1 s^{1/p'} \psi(s)^p \big[ p' (s^{-1/p'} - 1)
\big] \, ds \\
&\leq p'^{p/p'} \int_0^1 s^{1/p'} \psi(s)^p \big[ p' s^{-1/p'} \big]
\, ds \\
&= p'^p \int_0^1 \psi(s)^p \, ds.
\end{split}
\end{equation*}
\end{proof}

\begin{thm}\label{thm:lloglp} For any $1 < p \leq \infty$, $L^p(X)
\subseteq \llogl(X) \subseteq L^1(X)$, with $\|f\|_1 \leq
\|f\|_{\llogl} \lesssim \|f\|_p$.  \end{thm}

\begin{proof} Fix $f : \mathbb{T} \rightarrow \mathbb{C}$.  We
have trivially that $\|f\|_1 = \int_0^1 f^*(t) \, dt \leq \int_0^1
f^{**}(t) \, dt = \|f\|_{\llogl}$.

Now let $1 < p < \infty$.  First, as $(0,1)$ is a probability space,
we have by H\"{o}lder that $\|f\|_{\llogl} \leq (\int_0^1
f^{**}(t)^p \, dt)^{1/p}$.  Now apply Hardy's inequality with
$\psi(t) = f^*(t)$ to see $(\int_0^1 f^{**}(t)^p \, dt)^{1/p}
\lesssim (\int_0^1 f^*(t)^p \, dt)^{1/p} = \|f\|_p$.
\end{proof}

The principal reason for defining $\llogl$ as we have is the ease in
which we gain interpolation results.

\begin{thm} Let $T$ be a sublinear operator which maps $L^1(X)
\rightarrow L^{1,\infty}(X)$ and $L^p(X) \rightarrow
L^{q,\infty}(X)$, for some $1 < p, q < \infty$.  Then, $T :
\llogl(X) \rightarrow L^1(X)$.
\end{thm}

\begin{proof} Set $m = (\frac{1}{q} - 1) (\frac{1}{p} - 1)^{-1}$,
which is positive and finite. By Lemma~\ref{lemma:weak},

\begin{equation*}
(Tf)^*(t) \lesssim \bigg[ \frac{1}{t} \int_0^{t^m} f^*(s) \, ds +
t^{-1/q} \int_{t^m}^1 s^{1/p} f^*(s) \, \frac{ds}{s} \bigg],
\end{equation*}

\noindent for all $0 < t < 1$.  Note, the second integral's upper
limit is now 1, instead of $\infty$, as $f^*$ is supported on
$[0,1]$.  A simple change of variables gives

\begin{equation*}
\begin{split}
\int_0^1 \frac{1}{t} \int_0^{t^m} f^*(s) \, ds \, dt &= \frac{1}{m}
\int_0^1 \frac{1}{u} \int_0^u f^*(s) \, ds \, du
\\
&= \frac{1}{m} \int_0^1 f^{**}(u) \, du = \frac{1}{m} \|f\|_{L \log
L}.
\end{split}
\end{equation*}

\noindent On the other hand, using Fubini,

\begin{equation*}
\begin{split}
\int_0^1 t^{-1/q} \int_{t^m}^1 s^{1/p-1} f^*(s) \, ds \,
dt &= \int_0^1 s^{1/p-1} f^*(s) \int_0^{s^{1/m}} t^{-1/q} \, dt \, ds \\
&= \frac{1}{1 - 1/q} \int_0^1 s^{1/p - 1} s^{1/m - 1/mq} f^*(s)
\, ds \\
&= \frac{1}{1 - 1/q} \int_0^1 f^*(s) \, ds \\
&\leq \frac{1}{1 - 1/q} \int_0^1 f^{**}(s) \, ds = \frac{1}{1 - 1/q}
\|f\|_{\llogl}.
\end{split}
\end{equation*}

\noindent Hence,

\begin{equation*}
\|Tf\|_1 = \int_0^1 (Tf)^*(t) \, dt \lesssim \left( \frac{1}{m} +
\frac{1}{1 - 1/q} \right) \|f\|_{\llogl}.
\end{equation*}
\end{proof}

\begin{cor}\label{cor:f-sllogl} Let $T$ be a sublinear operator.  If
for some $1 < p, r < \infty$

\begin{gather*}
\bigg\| \Big( \sum_{k=1}^\infty |Tf_k|^r \Big)^{1/r}
\bigg\|_{1,\infty} \lesssim \bigg\| \Big( \sum_{k=1}^\infty |f_k|^r
\Big)^{1/r} \bigg\|_1 \quad \text{and} \\
\bigg\| \Big( \sum_{k=1}^\infty |Tf_k|^r \Big)^{1/r} \bigg\|_{p}
\lesssim \bigg\| \Big( \sum_{k=1}^\infty |f_k|^r \Big)^{1/r}
\bigg\|_{p},
\end{gather*}

\noindent then

\begin{gather*}
\bigg\| \Big( \sum_{k=1}^\infty |Tf_k|^r \Big)^{1/r} \bigg\|_1
\lesssim \bigg\| \Big( \sum_{k=1}^\infty |f_k|^r \Big)^{1/r}
\bigg\|_{\llogl}.
\end{gather*}
\end{cor}

\begin{proof} Recall Banach-valued functions $f \in \mathcal{M}(X,
B)$ as in Theorem~\ref{thm:biginterpolation}.  Although we did not
do so for stylistic purposes, this chapter could have been presented
in this more general setting.  For instance, for $f \in
\mathcal{M}(X,B)$, define $\mu_f(\lambda) = \rho\{x \in X :
\|f(x)\|_B > \lambda\}$ and $f^*(t) = \inf \{\lambda \geq 0 :
\mu_f(\lambda) \leq t\}$.  In this manner, we could redo this entire
chapter replacing $\mathbb{C}$ and $|\cdot|$ with $B$ and
$\|\cdot\|_B$, and everything would follow as before.

Specifically, the previous theorem holds; if $T$ is sublinear
operator mapping $L^{1}_B(X)$ to $L^{1,\infty}_B(X)$ and $L^p_B(X)$
to $L^{q,\infty}_B(X)$, then $T : \llogl_B(X) \rightarrow L^1_B(X)$.
But, simply by definition, $f^*(t) = (\|f\|_B)^*(t)$, where
$(\|f\|_B)^*$ is understood as the decreasing rearrangement of the
map $x \mapsto \|f(x)\|_B$. Thus, $\|f\|_{\llogl_B} = \big\| \|f\|_B
\big\|_{\llogl}$.

Let $B = \ell^r$.  For $f \in \mathcal{M}(X, B)$, let
$\overline{T}(f) = (Tf_1, Tf_2, \ldots)$, which is sublinear,
because $T$ is. By hypothesis, $\overline{T} : L^1_B(X) \rightarrow
L^{1,\infty}_B(X)$ and $L^p_B(X) \rightarrow L^p_B(X)$.  Thus,
$\overline{T} : \llogl_B(X) \rightarrow L^1_B(X)$, which is what we
wanted to prove.
\end{proof}

\section{The n-Star Operator and $\llogln$}

To extend the definition of $\llogl$, we first must extend the
definition of the 2-star operator.  We remain with the convention
that $(X,\rho)$ is a probability space.

\begin{defnn} For $f : (X,\rho) \rightarrow \mathbb{C}$, let
$f^{(*,1)}(t) = f^*(t)$ and for integers $n \geq 2$, set
$f^{(*,n)}(t) = \frac{1}{t} \int_0^t f^{(*,n-1)}(s) \, ds$.
\end{defnn}

\begin{prop} For any $f, f_k, g : (X, \rho) \rightarrow \mathbb{C}$ and
$\alpha \in \mathbb{C}$,

\begin{enumerate}
\item $f^{(*,n)}$ is nonnegative, decreasing, and identically 0 if and
only if $f = 0$ a.e.[$\rho$],

\item $f^{(*,n)} \leq f^{(*,n+1)}$,

\item $|f| \leq |g|$ a.e.[$\rho$] implies $f^{(*,n)} \leq g^{(*,n)}$ pointwise,

\item $(\alpha f)^{(*,n)} = |\alpha| f^{(*,n)}$,

\item $|f_k| \uparrow |f|$ a.e.[$\rho$] implies $f_k^{(*,n)} \uparrow
f^{(*,n)}$ pointwise,

\item $(f+g)^{(*,n)} \leq f^{(*,n)} + g^{(*,n)}$ ($n \geq 2$ only).
\end{enumerate}
\end{prop}

\begin{proof} It is known that $f^{(*,1)} = f^*$ is decreasing.
Assume $f^{(*,n-1)}$ is decreasing.  Let $0 < t_1 < t_2$.  Then,
$f^{(*,n-1)}(s) \leq f^{(*,n-1)}(s t_1/t_2)$ for any $s > 0$.  Thus,

\begin{equation*}
\begin{split}
f^{(*,n)}(t_2) &= \frac{1}{t_2} \int_0^{t_2} f^{(*,n-1)}(s) \, ds
\leq \frac{1}{t_2} \int_0^{t_2} f^{(*,n-1)}(st_1/t_2) \, ds \\
&= \frac{1}{t_1} \int_0^{t_1} f^{(*,n-1)}(u) \, du = f^{(*,n)}(t_1).
\end{split}
\end{equation*}

\noindent By induction, $f^{(*,n)}$ is decreasing.  This gives

\begin{equation*}
f^{(*,n+1)}(t) = \frac{1}{t} \int_0^t f^{(*,n)}(s) \, ds \geq
f^{(*,n)}(t) \frac{1}{t} \int_0^t \, ds = f^{(*,n)}(t).
\end{equation*}

\noindent All other properties are easily established by induction
and that each is known to hold for $n = 1$ (or $n = 2$ in the case
of (6)).  \end{proof}

\begin{defnn} For functions $f : (X, \rho) \rightarrow
\mathbb{C}$ and integers $n \geq 0$, define $\|f\|_{\llogln}$ by

\begin{equation*}
\|f\|_{\llogln} = \int_0^1 f^{(*,n+1)}(t) \, dt.
\end{equation*}

\noindent Define the Zygumnd space $\llogln(X)$ as the set of all
functions $f$ with $\|f\|_{\llogln} < \infty$.  \end{defnn}

We note that $L(\log L)^0(X) = L^1(X)$, which is a useful notational
shortcut.  As before, it is clear that $\llogln(X)$ is a Banach
space, and $\|\cdot\|_{\llogln}$ is a norm with the additional
properties that $|f| \leq |g|$ a.e.[$\rho$] implies $\|f\|_{\llogln}
\leq \|g\|_{\llogln}$ and $|f_k| \uparrow |f|$ a.e.[$\rho$] implies
$\|f_k\|_{\llogln} \uparrow \|f\|_{\llogln}$. Further, this
definition is related to the intuitive value, as before.

\begin{thm} $f \in \llogln(X)$ if and only if

\begin{equation*}
\int_X |f(x)| (\log^+ |f(x)|)^n \, \rho(dx) < \infty.
\end{equation*}
\end{thm}

\begin{proof} The $n = 0$ case is trivial, and the $n = 1$ is already
known.  So, fix $n \geq 2$.  As the map $x \mapsto x (\log^+ x)^n$
is continuous, increasing, and has value $0$ at $x = 0$, we have by
Lemma~\ref{lemma:f*map} that $\int_X |f| (\log^+|f|)^n \, d\rho$ is
finite if and only if $\int_0^1 f^*(t) (\log^+ f^*(t))^n \, dt$ is
finite.  On the other hand, changing the order of integration several
times shows

\begin{equation*}
\int_0^1 f^{(*,n+1)}(t) \, dt = \frac{1}{n!} \int_0^1 f^*(t)
\log(1/t)^{n} \, dt
\end{equation*}

Suppose $\int_0^1 f^*(t) (\log^+ f^*(t))^n \, dt$ is finite. Let $E = \{t
\in (0,1) : f^*(t) > t^{-1/2}\}$ and $F = (0,1) - E$.  Then,

\begin{equation*}
\begin{split}
\int_0^1 f^*(t) \log(1/t)^n \, dt & \leq \int_E f^*(t) \log(
f^*(t)^2)^n
\, dt + \int_F t^{-1/2} \log(1/t)^n \, dt \\
&\leq 2^n \int_0^1 f^*(t) \big(\log^+ f^*(t)\big)^n + \int_0^1
t^{-1/2} \log(1/t)^n \, dt \\
&= 2^n \int_0^1 f^*(t) \big(\log^+ f^*(t)\big)^n + 2^{n+1} n! <
\infty.
\end{split}
\end{equation*}

Now suppose $\int_0^1 f^{(*,n+1)}(t) \, dt$ is finite.  Then, we have
$\|f\|_1 = \int_0^1 f^*(t) \, dt \leq \int_0^1 f^{(*,n+1)}(t) \, dt
< \infty$. If $\|f\|_1 = 0$ there is nothing to prove, so assume
otherwise. Let $g = f/\|f\|_1$ so that $\|g\|_1 = 1$. Then, $g^*(t)
\leq g^{**}(t) \leq \|g\|_1/t = 1/t$. Also,

\begin{equation*}
\begin{split}
\int_0^1 g^*(t) \big( \log^+ g^*(t) \big)^n \, dt &\leq \int_0^1
g^*(t) \log^+ (1/t)^n \, dt = \int_0^1 g^*(t) \log (1/t)^n \, dt \\
&= \frac{1}{\|f\|_1} \int_0^1 f^*(t)
\log(1/t)^n \, dt < \infty \\
\end{split}
\end{equation*}

\noindent But,

\begin{equation*}
\begin{split}
\int_0^1 f^*(t) \big( \log^+ &f^*(t)\big)^n \, dt \\
&= \|f\|_1 \int_0^1 g^*(t) \big( \log^+( \|f\|_1 g^*(t)) \big)^n \, dt \\
&\lesssim \|f\|_1 \Big[ \int_0^1 g^*(t) \big( \log^+ \|f\|_1 \big)^n
\, dt + \int_0^1 g^*(t) \big( \log^+ g^*(t)\big)^n \, dt \Big] \\
&= \|f\|_1 \Big[ \big(\log^+ \|f\|_1\big)^n + \int_0^1 g^*(t) \big(
\log^+ g^*(t)\big)^n \, dt \Big] < \infty.
\end{split}
\end{equation*}

\noindent The transition from the second line to the third line
follows from the fact that $(a + b)^r \leq 2^{r-1} [a^r + b^r]$ for
any $a, b \geq 0$ and $r \in \mathbb{R}$, which is proven by
elementary calculus.
\end{proof}

\begin{thm}\label{thm:lloglnp} For any $1 < p \leq \infty$ and $n \geq 0$

\begin{equation*}
L^p(X) \subseteq L(\log L)^{n+1}(X) \subseteq \llogln(X) \subseteq
L^1(X),
\end{equation*}

\noindent with $\|f\|_1 \leq \|f\|_{\llogln} \leq \|f\|_{L(\log
L)^{n+1}} \lesssim \|f\|_p$.
\end{thm}

\begin{proof} Fix $f : \mathbb{T} \rightarrow \mathbb{C}$ and $n
\geq 0$.  Note, $\|f\|_1 = \int_0^1 f^*(t) \, dt \leq \int_0^1
f^{(*,n+1)}(t) \, dt = \|f\|_{\llogln}$.  By the same token,
$\|f\|_{\llogln} = \int_0^1 f^{(*,n+1)}(t) \, dt \leq \int_0^1
f^{(*,n+2)}(t) \, dt = \|f\|_{L(\log L)^{n+1}}$.

Now let $1 < p < \infty$.  First, as $(0,1)$ is a probability space,
we have by H\"{o}lder that $\|f\|_{\llogln} \leq (\int_0^1
f^{(*,n+1)}(t)^p \, dt)^{1/p} = \|f^{(*,n+1)}\|_p$.  Applying
Hardy's inequality (Lemma~\ref{lemma:hardy}) with $\psi(t) =
f^{(*,m)}(t)$ gives $\|f^{(*,m+1)}\|_p \lesssim \|f^{(*,m)}\|_p$.
Iterating this we have $\|f\|_{\llogln} \leq \|f^{(*,n+1)}\|_p
\lesssim \|f^{(*,n)}\|_p \lesssim \ldots \lesssim \|f^{(*,1)}\|_p =
\|f\|_p$.
\end{proof}

An interpolation result can also be proven for $\llogln$.  First, we
need to find an estimate similar to the one before.

\begin{lemma}\label{lemma:weakn} Let $T$ be a sublinear operator
which maps $L^1(X) \rightarrow L^{1,\infty}(X)$ and $L^p(X)
\rightarrow L^{q,\infty}(X)$, for some $1 < p, q < \infty$.  Then,
for $n \geq 1$,

\begin{equation*}
(Tf)^{(*,n)}(t) \lesssim \bigg[ \frac{1}{t} \int_0^{t^m}
f^{(*,n)}(s) \, ds + t^{-1/q} \int_{t^m}^1 s^{1/p - 1} f^{(*,n)}(s)
\, ds \bigg], \quad 0 < t < 1,
\end{equation*}

\noindent where $m = (\frac{1}{q} - 1) (\frac{1}{p} - 1)^{-1}$.
\end{lemma}

\begin{proof} The $n = 1$ case is precisely Lemma~\ref{lemma:weak}
(on a probability space) with $p_0 = q_0 = 1$. So, assume it is true
for $n-1$. Then,

\begin{equation*}
\begin{split}
(Tf)^{(*,n)}&(t) = \frac{1}{t} \int_0^t T^{(*,n-1)}(s) \, ds \\
&\lesssim \frac{1}{t} \int_0^t \frac{1}{s} \int_0^{s^m}
f^{(*,n-1)}(u) \, du \, ds + \frac{1}{t} \int_0^t s^{-1/q}
\int_{s^m}^1 u^{1/p - 1} f^{(*,n-1)}(u) \, du \, ds \\
&=: I + II.
\end{split}
\end{equation*}

By the change of variables $r = s^m$,

\begin{equation*}
I = \frac{1}{m} \frac{1}{t} \int_0^{t^m} \frac{1}{r} \int_0^r
f^{(*,n-1)}(u) \, du \, dr = \frac{1}{m} \frac{1}{t} \int_0^{t^m}
f^{(*,n)}(r) \, dr.
\end{equation*}

\noindent On the other hand, changing the order of integration gives

\begin{equation*}
\begin{split}
II &= \frac{1}{t} \int_0^{t^m} u^{1/p-1} f^{(*,n-1)}(u)
\int_0^{u^{1/m}} s^{-1/q} \, ds \, du \\
&\qquad + \frac{1}{t} \int_{t^m}^1 u^{1/p-1} f^{(*,n-1)}(u) \int_0^t
s^{-1/q} \, ds \, du \\
&= \frac{1}{1 - 1/q} \,\, \frac{1}{t} \int_0^{t^m} f^{(*,n-1)}(u) \,
du + \frac{1}{1 - 1/q} \,\, t^{-1/q} \int_{t^m}^1 u^{1/p-1}
f^{(*,n-1)}(u) \, du \\
&\leq \frac{1}{1 - 1/q} \bigg[ \frac{1}{t} \int_0^{t^m} f^{(*,n)}(u)
\, du + t^{-1/q} \int_{t^m}^1 u^{1/p-1} f^{(*,n)}(u) \, du \bigg].
\end{split}
\end{equation*}
\end{proof}

\begin{thm} Let $T$ be a sublinear operator
which maps $L^1(X) \rightarrow L^{1,\infty}(X)$ and $L^p(X)
\rightarrow L^{q,\infty}(X)$, for some $1 < p, q < \infty$.  Then,
for all $n \in \mathbb{N}$, we have $T : L(\log L)^{n}(X)
\rightarrow L(\log L)^{n-1}(X)$.
\end{thm}

\begin{proof} Set $m = (\frac{1}{q} - 1) (\frac{1}{p} - 1)^{-1}$,
which is positive and finite. Using Lemma~\ref{lemma:weakn} and the
same change of variables and Fubini arguments,

\begin{equation*}
\begin{split}
\|Tf\|_{L(\log L)^{n-1}} &= \int_0^1 (Tf)^{(*,n)}(t) \, dt \\
&\lesssim \int_0^1 \frac{1}{t} \int_0^{t^m} f^{(*,n)}(s) \, ds \, dt
+ \int_0^1 t^{-1/q} \int_{t^m}^1 s^{1/p-1} f^{(*,n)}(s) \, ds \, dt
\\
&= \frac{1}{m} \int_0^1 \frac{1}{u} \int_0^u f^{(*,n)}(s) \, ds \,
du + \int_0^1 s^{1/p-1} f^{(*,n)}(s) \int_0^{s^{1/m}} t^{-1/q} \, dt
\, ds \\
&= \frac{1}{m} \int_0^1 f^{(*,n+1)}(u) \, du + \frac{1}{1 - 1/q}
\int_0^1 f^{(*,n)}(s) \, ds \lesssim \|f\|_{\llogln}.
\end{split}
\end{equation*}
\end{proof}

\begin{cor}\label{cor:f-sllogln} Let $T$ be a sublinear operator.  If
for some $1 < p, r < \infty$

\begin{gather*}
\bigg\| \Big( \sum_{k=1}^\infty |Tf_k|^r \Big)^{1/r}
\bigg\|_{1,\infty} \lesssim \bigg\| \Big( \sum_{k=1}^\infty |f_k|^r
\Big)^{1/r} \bigg\|_1 \quad \text{and} \\
\bigg\| \Big( \sum_{k=1}^\infty |Tf_k|^r \Big)^{1/r} \bigg\|_{p}
\lesssim \bigg\| \Big( \sum_{k=1}^\infty |f_k|^r \Big)^{1/r}
\bigg\|_{p},
\end{gather*}

\noindent then for all $n \in \mathbb{N}$

\begin{gather*}
\bigg\| \Big( \sum_{k=1}^\infty |Tf_k|^r \Big)^{1/r} \bigg\|_{L(\log
L)^{n-1}} \lesssim \bigg\| \Big( \sum_{k=1}^\infty |f_k|^r
\Big)^{1/r} \bigg\|_{L(\log L)^{n}}.
\end{gather*}
\end{cor}

\section{$\llogl$ and Connections to Hardy-Littlewood}

Let us consider the probability space $(\mathbb{T}, m)$ and
$\llogl(\mathbb{T})$.  The maximal operator $M$ maps $L^1
\rightarrow L^{1,\infty}$ and $L^p \rightarrow L^p$ for all $1 < p <
\infty$.  Therefore, by our interpolation results, $M :
\llogl(\mathbb{T}) \rightarrow L^1$.  However, much more can be
said.

\begin{thm}\label{thm:M**point} For any $0 < t < 1$, $f^{**}(t) \sim
(Mf)^*(t)$, where the underlying constants do not depend on $f$ or
$t$. \end{thm}

\begin{proof} Fix $t$ and $f$.  We start by proving $(Mf)^*(t)
\lesssim f^{**}(t)$.  Let $\epsilon > 0$.  By Lemma~\ref{lemma:f**sum},
there are functions $g, h$ so that $f = g + h$ and $\|g\|_1 + t
\|h\|_\infty \leq tf^{**}(t) + \epsilon$.  On the other hand,

\begin{equation*}
\begin{split}
(Mf)^*(t) &\leq (Mg)^*(t/2) + (Mh)^*(t/2) = \frac{2}{t} \Big[
\frac{t}{2} (Mg)^*(t/2) \Big] + (Mh)^*(t/2) \\
&\leq \frac{2}{t} \|Mg\|_{1,\infty} + \|Mh\|_\infty \lesssim
\frac{2}{t} \|g\|_1 + \|h\|_\infty \\
&\leq \frac{2}{t} \big[ \|g\|_1 + t \|h\|_\infty \big] \leq 2
f^{**}(t) + 2\epsilon/t.
\end{split}
\end{equation*}

\noindent Letting $\epsilon \rightarrow 0$, we have the first
inequality.

For the second inequality, we may assume $(Mf)^*(t)$ is finite, or
there is nothing to prove.  Set $\Omega$ to be the closure of $\{Mf
> (Mf)^*(t)\}$. Note that $|\Omega| = \mu_{Mf}((Mf)^*(t)) \leq t$.
First, suppose $|\Omega| = 0$. Then, $|f| \leq Mf \leq (Mf)^*(t)$
a.e., which implies $f^*(s) \leq (Mf)^*(t)$ for all $s$. So,
$f^{**}(t) = t^{-1} \int_0^t f^*(s) \, ds \leq (Mf)^*(t)$.

Now, assume $|\Omega| > 0$.  As $|\Omega| \leq t < 1$, for each $x
\in \Omega$ we can choose an interval $I_x$ which contains $x$ in
its interior and $I_x \cap \Omega^c \not= \emptyset$, but also so
that most of $I_x$ is in $\Omega$.  In particular, $|I_x| \leq 2
|I_x \cap \Omega|$. Then, the interiors of $\{I_x : x \in \Omega\}$
cover $\Omega$. As $\Omega$ is compact, we can choose a finite
subcover $I_1, \ldots, I_n$.  Further, we can choose this subcover
to be minimal, in that any point is contained in at most two of the
$I_k$ (this property is inherited from $\mathbb{R}$). On the other
hand, as $I_j \cap \Omega^c \not= \emptyset$, there is a $y \in I_j
\cap \Omega^c$. This implies $|I_j|^{-1} \int_{I_j} |f| \, dm \leq
Mf(y) \leq (Mf)^*(t)$.

Define $g = f \chi_\Omega$ and $h = f \chi_{\Omega^c}$.  We have
immediately that $\|h\|_\infty = \|f \chi_{\Omega^c}\|_\infty \leq
(Mf)^*(t)$.  On the other hand,

\begin{equation*}
\begin{split}
\|g\|_1 &\leq \sum_{j=1}^n \int_{I_j} |f(x)| \, dx \leq
\sum_{j=1}^n (Mf)^*(t) |I_j| \\
&\leq 2 (Mf)^*(t) \sum_{j=1}^n |I_j \cap \Omega| \leq 4 |\Omega|
(Mf)^*(t) \leq 4t (Mf)^*(t),
\end{split}
\end{equation*}

\noindent where the next to last inequality is gained from $I_j$
being a minimal subcover.  As $f = g + h$, it follows from
Lemma~\ref{lemma:f**sum} that $t f^{**}(t) \leq \|g\|_1 + t
\|h\|_\infty \lesssim t (Mf)^*(t)$. This completes the proof.
\end{proof}

\begin{cor} For any $f : \mathbb{T} \rightarrow \mathbb{C}$, $f
\in \llogl(\mathbb{T})$ if and only if $Mf \in L^1(\mathbb{T})$. In
particular, $\|f\|_{\llogl} \sim \|Mf\|_1$.
\end{cor}

\begin{proof} Using the previous theorem, $\|f\|_{\llogl} = \int_0^1
f^{**}(t) \, dt \sim \int_0^1 (Mf)^*(t) \, dt = \|Mf\|_1$.
\end{proof}

We note that $\|M(\cdot)\|_1$ is itself a norm, with the additional
properties that $|f| \leq |g|$ a.e.~implies $\|Mf\|_1 \leq \|Mg\|_1$
and $|f_n| \uparrow |f|$ a.e.~implies $\|Mf_n\|_1 \uparrow
\|Mf\|_1$.  Therefore, on $\mathbb{T}$, we could have defined
$\|f\|_{\llogl} = \|Mf\|_1$ and $\llogl(\mathbb{T})$ the space of
functions which are mapped into $L^1$ by $M$.  There is a similar
result for $\llogln(\mathbb{T})$.

\begin{cor} $f \in L(\log L)^{n+1}(\mathbb{T})$ if and only if $Mf
\in \llogln(\mathbb{T})$, and, in particular, $\|f\|_{L(\log L)^{n+1}} \sim
\|Mf\|_{\llogln}$. \end{cor}

\begin{proof} We know $(Mf)^{(*,1)} \sim f^{(*,2)}$.  It follows from
induction that $(Mf)^{(*,n)} \sim f^{(*,n+1)}$ for all $n \geq 1$.
Thus, $\|f\|_{\llogln} = \int_0^1 f^{(*,n+1)}(t) \, dt \sim \int_0^1
(Mf)^{(*,n)}(t) \, dt = \|Mf\|_{L(\log L)^{n-1}}$.  \end{proof}

Finally, we return to the unanswered question of the end-point
estimates of the strong maximal function $M_S$.  The probability
space we focus on now is $(\mathbb{T}^d, m)$. As each of the
$j^{th}$ parameter maximal operators $M_j$ map $L^1$ to weak-$L^1$
and $L^p$ to $L^p$, we have by interpolation that $M_j : L(\log
L)^{n+1}(\mathbb{T}^d) \rightarrow \llogln(\mathbb{T}^d)$.  Thus,
for $n \geq d$,

\begin{equation*}
\begin{split}
\|M_S\|_{\llogln} &\leq \|M_1 \circ M_2 \circ \cdots \circ M_d
f\|_{\llogln} \\
&\lesssim \|M_2 \circ \cdots \circ M_d f\|_{L(\log L)^{n-1}} \\
&\lesssim \ldots \lesssim \|f\|_{L(\log L)^{n-d}}.
\end{split}
\end{equation*}

\noindent In particular, $M_S : L(\log L)^d(\mathbb{T}^d)
\rightarrow L^1(\mathbb{T}^d)$ and $M_S : L(\log
L)^{d-1}(\mathbb{T}^d) \rightarrow L^{1,\infty}(\mathbb{T}^d)$.

\chapter{Single-parameter Multipliers}

\section{Shifted Max and Square Operators}\label{sec:shifted}

For $n \in \mathbb{Z}$, define the $n$-shifted maximal operator as

\begin{equation*}
M^n f(x) = \sup_{x \in I} \frac{1}{|I|} \int_{I^n} |f(x)| \, dx,
\end{equation*}

\noindent where the supremum is taken over all intervals $I$
containing $x$, but the integral is over $I^n$.  We would like to
establish results for $M^n$ similar to those of $M$.  This is quite
simple.

Fix $f$ and $n$.  Let $x \in \mathbb{T}$ and $\epsilon > 0$. Choose
an interval $I$ containing $x$ so that $M^nf(x) \leq |I|^{-1} \int_{I^n}
|f(x)| \, dx + \epsilon$. There exists an interval $I'$ (possibly all
of $\mathbb{T}$) which contains both $I$ and $I^n$, and $|I'| \leq
(|n|+1) |I|$.  Thus,

\begin{equation*}
M^n f(x) - \epsilon \leq \frac{1}{|I|} \int_{I^n} |f(x)| \, dx \leq
(|n|+1) \frac{1}{|I'|} \int_{I'} |f(x)| \, dx \leq (|n|+1) Mf(x).
\end{equation*}

\noindent As $\epsilon$ is arbitrary, we have the pointwise estimate
$M^n f \leq (|n|+1) Mf$.  Therefore, we immediately obtain all the
$L^p$ estimates of $M$, along with the Fefferman-Stein inequalities,
for $M^n$ with an additional factor of $|n|+1$.

Now consider an adapted family $\varphi_I$.  By precisely the same
argument used in Proposition~\ref{prop:1},

\begin{equation*}
M'^n f := \sup_I \frac{1}{|I|} \langle \varphi_{I^n}, f \rangle
\chi_I \lesssim M^n f.
\end{equation*}

\noindent So, $M'^n f$ is also easily understood.

However, the shifted square function

\begin{equation*}
S^n f(x) = \bigg( \sum_I \frac{|\langle \phi_{I^n}, f
\rangle|^2}{|I|} \chi_I(x) \bigg)^{1/2}
\end{equation*}

\noindent does not permit a simple pointwise estimate.  To prove the
desired $L^p$ results, one has to go through the argument as
presented in Chapter~\ref{chap:square} with $S^n$ instead of $S$. We
refrain from doing this, as only a brief description seems
necessary.

It can be shown that $S^n : L^2 \rightarrow L^2$ exactly as before,
with no dependence on $n$. This is because in the proof of
Theorem~\ref{thm:SL2} (and the preceding lemmas), we sum over all
$I$ with the same lengths, and the shift will not be important.

Fix a dyadic interval $I$ and $a$ an $L^1$-function supported on $I$
with integral 0.  Define $I^* = (2|n| + 2)I$ if $2|n|+2 \leq 1/|I|$
and $I^* = \mathbb{T}$ otherwise.  Then, $|I^*| \leq (2|n|+2) |I|$.
If $J$ is a dyadic interval with $|J| < |I|$, we have that $J
\subset I^*$ or $J, I^*$ are disjoint.  If it is the later, then by
construction, $J^n$ and $2I$ are disjoint.  It now follows by
precisely the same argument as in the proof of
Lemma~\ref{lemma:atom} that $\|S^n a\|_{L^1(\mathbb{T} - I^*)}
\lesssim \|a\|_1$, where the underlying constant is independent of
$n$.  Applying the same decomposition as Theorem~\ref{thm:Sweak}, we
have $\|S^n f\|_{1,\infty} \lesssim (|n|+1)\|f\|_1$ for all $f \in
L^1$.

Define the shifted linearization

\begin{equation*}
T_\epsilon^n f(x) = \sum_I \epsilon_I \langle \phi_{I^n}^{1}, f
\rangle \phi_I^2(x).
\end{equation*}

\noindent By the same technique as before, $T_\epsilon^n : L^2
\rightarrow L^2$ with no dependence on $n$.  For the weak-$L^1$
result, simply replace $S^1$ with $S^{1,n}$ in the proof of
Theorem~\ref{thm:Tweak}.  The constant $C$ which is chosen at the
beginning will now depend on $n$, but as we saw, $C$ actually cancels
out by the end.  This gives $\|T_\epsilon^n f\|_{1,\infty} \lesssim
(|n|+1) \|f\|_1$.  The rest of the arguments follow as before giving
$\|S^n f\|_p \lesssim (|n|+1) \|f\|_p$ and $\|T_\epsilon^n f\|_p
\lesssim (|n|+1) \|f\|_p$.  The Fefferman-Stein inequalities also
hold for $S^n$, with the additional factor of $|n|+1$.

On a different note, let $\alpha \in [0, 1]$ and $I_\alpha = I +
\alpha|I|$.  This shifts the interval, much like $I^n$, but we use a
different notation to distinguish the roles $\alpha$ and $n$ will
play.  Define

\begin{equation*}
M_\alpha^n f(x) = \sup_{x \in I} \frac{1}{|I|} \int_{I_\alpha^n}
|f(y)| \, dy.
\end{equation*}

\noindent By the same argument as before, $M_\alpha^n f \leq (|n| +
\alpha + 1) Mf(x) \lesssim (|n|+1) Mf(x)$.  So, if we let
$M^{[n]}f(x) = \sup_\alpha M^n_\alpha f(x)$ for each $x$, then
$M^{[n]}$ satisfies all the estimates of $M$ ($L^p \rightarrow L^p$,
$L^1 \rightarrow L^{1,\infty}$, and Fefferman-Stein inequalities)
with an additional factor of $|n|+1$.

For an adapted family $\{\varphi_I\}$, let $\varphi_{I_\alpha}(x) =
\varphi_I(x - \alpha|I|)$ so that each $\varphi_{I_\alpha}$ is
uniformly adapted to $I_\alpha$. Like the argument before,
$M'^n_\alpha f(x) = \sup_I \frac{1}{|I|} \langle
\varphi_{I_\alpha^n}, f \rangle \chi_I(x) \lesssim M_\alpha^n f(x)$.
For a 0-mean family, let

\begin{equation*}
S_\alpha^n f(x) = \Big( \sum_I \frac{|\langle \phi_{I_\alpha^n}, f
\rangle|^2}{|I|} \chi_I(x) \Big)^{1/2}
\end{equation*}

\noindent and $S^{[n]} f(x) = \sup_\alpha S^n_\alpha f(x)$.  We are
interested in gaining estimates on $S^{[n]}$.  First, fix an
interval $I$.  Note, for any $x$, $\dist(x, I_\alpha) \geq
\dist(x,I) - \alpha|I|$ and

\begin{equation*}
\begin{split}
|\varphi_{I_\alpha}(x)| &\leq C_m \Big(1 + \frac{\dist(x,
I_\alpha)}{|I|} \Big)^{-m} \leq 2^m C_m \Big(2 + \frac{\dist(x,
I_\alpha)}{|I|} \Big)^{-m} \\
&\leq 2^m C_m \Big(2 - \alpha + \frac{\dist(x, I)}{|I|} \Big)^{-m}
\leq 2^m C_m \Big(1 + \frac{\dist(x, I)}{|I|} \Big)^{-m}.
\end{split}
\end{equation*}

\noindent That is, each $\varphi_{I_\alpha}$ is actually uniformly
adapted to $I$.  Fix $f$ and $n$.  For each dyadic interval $I$,
choose an $I_\#$, dependent on $f$, so that $|\langle \phi_{I_\#^n},
f \rangle| = \sup_\alpha |\langle \phi_{I_\alpha^n}, f \rangle|$.  Then,

\begin{equation*}
S^{[n]}f(x) \leq \Big( \sum_I \frac{|\langle \phi_{I^n_\#}, f
\rangle|^2}{|I|} \chi_I(x) \Big)^{1/2}.
\end{equation*}

\noindent As each $\varphi_{I_\#^n}$ is uniformly adapted to $I^n$,
we observe that $S^{[n]} f$ is bounded by a kind of $S^n f$, with a
new adapted family. Hence, $\|S^{[n]} f\|_{1,\infty} \lesssim (|n|+1) \|f\|_1$
and $\|S^{[n]} f\|_p \lesssim (|n|+1) \|f\|_p$ as before.

Finally, let

\begin{equation*}
T_\epsilon^{[n]} f(x) = \int_0^1 \sum_I \epsilon_I \langle
\phi^1_{I^n_\alpha}, f \rangle \phi^2_{I_\alpha}(x) \, d\alpha.
\end{equation*}

\noindent Let $1 < p < \infty$ and take $\|g\|_{p'} \leq 1$.  Then,
by the normal H\"{o}lder argument $|\langle T_\epsilon^{[n]} f, g
\rangle| \leq \|S^{[n]} f\|_{p} \|S^{[0]} g\|_{p'} \lesssim (|n|+1)
\|f\|_p$. As $g$ in the unit ball of $L^{p'}$ is arbitrary,
$\|T_\epsilon^{[n]} f\|_p \lesssim (|n|+1) \|f\|_p$.  To show that
$\|T_\epsilon^{[n]} f\|_{1,\infty} \lesssim (|n|+1) \|f\|_1$, one
needs to run the argument of Theorem~\ref{thm:Tweak} again, this
time with $S^1$ replaced by $S^{1,[n]}$ and $S^{2,k}$ replaced by
$S^{2,k,[0]}$.  As each of the square functions is the supremum over
$\alpha$, the integral over $\alpha$ will be irrelevant.

\section{Marcinkiewicz Multipliers}

\begin{defnn} Let $m : \mathbb{R} \rightarrow \mathbb{C}$ be smooth
away from $0$ and uniformly bounded.  We say $m$ is a Marcinkiewicz
multiplier if $|m^{(l)}(t)| \lesssim |t|^{-l}$ for $0 \leq l \leq
4$.
\end{defnn}

The restriction $l \leq 4$ is what we will need. It can often be
assumed to hold for many more derivatives.  Our definition here
differs slightly from the classical definition. Normally, $m$ is
taken only to be in $L^\infty$, not uniformly bounded. Typically,
the multiplier appears in some integral and the value of $m$ at 0 is
irrelevant. Here, however, it will applied in a sum and the value is
important.

Given a Marcinkiewicz multiplier $m$, define the Marcinkiewicz
multiplier operator for $f \in L^1(\mathbb{T})$ as

\begin{equation*}
\Lambda_m f(x) = \sum_{t \in \mathbb{Z}} m(t) \widehat{f}(t) e^{2\pi
itx}.
\end{equation*}

\noindent We will show this operator satisfies the same $L^p$
properties as its classical counterpart on $\mathbb{R}$.  First, we
show the following technical results.

\begin{lemma}\label{lemma:weaksum} Fix positive integers $k$ and
$K$.  For each $\vec{n} \in \mathbb{Z}^K$, write $\alpha(\vec{n}) =
\prod_{j=1}^K (|n_j|+1)$. Suppose we have $f_{\vec{n}} :
\mathbb{T}^d \rightarrow \mathbb{C}$ for each $\vec{n} \in
\mathbb{Z}^K$ and $\|f_{\vec{n}}\|_{p,\infty} \leq \alpha(\vec{n})$
for all $\vec{n}$ and some $p \geq 1/k$.  Set $r = k + 3$ and $F =
\sum_{\vec{n}} \alpha(\vec{n})^{-r} f_{\vec{n}}$.  Then,
$\|F\|_{p,\infty} \lesssim 1$.
\end{lemma}

\begin{proof}  Let $\lambda > 0$. Fix $C = \sum_{\vec{n}}
\alpha(\vec{n})^{-3/2}$.  It is clear that

\begin{equation*}
\{|F| > \lambda\} \subseteq \bigcup_{\vec{n}} \{|f_{\vec{n}}|
> \lambda C^{-1} \alpha(\vec{n})^{r-3/2}\}.
\end{equation*}

\noindent So, $|\{|F| > \lambda\}| \leq \sum_{\vec{n}}
|\{|f_{\vec{n}}| > \lambda C^{-1} \alpha(\vec{n})^{r-3/2}\}| \leq
\frac{C^p}{\lambda^p} \sum_{\vec{n}} \|f_{\vec{n}} \|_{p,\infty}^p
\alpha(\vec{n})^{-rp + 3p/2} \leq \frac{C^p}{\lambda^p}
\sum_{\vec{n}} \alpha(\vec{n})^{-rp + 5p/2} \lesssim \lambda^{-p}$,
because $p(-r+5/2) = p(-k-1/2) < -1$. As $\lambda$ is
arbitrary, $\|F\|_{p,\infty} \lesssim 1$.
\end{proof}

\begin{lemma}\label{lemma:mk} Let $m$ be any Marcinkiewicz multiplier
and $\psi_k^1$ the functions guaranteed by Theorem~\ref{thm:tbumps}.
For each $k \in \mathbb{N}$, there is a smooth function $m_k$ so
that $m_k \widehat{\psi^1_k} = m \widehat{\psi_k^1}$ and

\begin{equation*}
m_k(t) = \sum_{n \in \mathbb{Z}} c_{k,n} e^{-2\pi in2^{-k}t},
\end{equation*}

\noindent where $|c_{k,n}| \lesssim (|n| + 1)^{-4}$ uniformly in
$k$.
\end{lemma}

\begin{proof} Let $\varphi : \mathbb{R} \rightarrow \mathbb{C}$ be
smooth, with $\supp(\varphi) \subseteq [-1/2, -1/32] \cup [1/32,
1/2]$ and $\varphi = 1$ on $[-1/4, -1/16] \cup [1/16, 1/4]$. Define
$m_k(t) = m(t) \varphi(2^{-k} t)$. Then, $m_k = m$ on $[-2^{k-2},
-2^{k-4}] \cup [2^{k-4}, 2^{k-2}]$, or equivalently $m_k
\widehat{\psi_k^1} = m \widehat{\psi_k^1}$.  Further, $m_k$ is
supported on $E_k := [-2^{k-1}, -2^{k-5}] \cup [2^{k-5}, 2^{k-1}]
\subset [-2^{k-1}, 2^{k-1}]$, an interval of length $2^k$.

Recall that $\{e^{-2\pi inx}\}_{n \in \mathbb{Z}}$ is an orthonormal
basis for the Hilbert space $L^2([0,1])$, or any interval of length
1 in $\mathbb{R}$. Thus, $\{2^{-k/2} e^{-2\pi in2^{-k}x}\}$ is an
orthonormal basis on any interval of length $2^k$, and

\begin{equation*}
m_k(t) = \sum_{n \in \mathbb{Z}} \bigg( \int_\mathbb{R} m_k(x)
\frac{e^{2\pi in2^{-k}x}}{2^{k/2}} \, dx \bigg) \frac{e^{-2\pi
in2^{-k}t}}{2^{k/2}} = \sum_{n \in \mathbb{Z}} c_{k,n} e^{-2\pi
in2^{-k}t},
\end{equation*}

\noindent where $c_{k,n} = 2^{-k} \int_\mathbb{R} m_k(x) e^{2\pi
in2^{-k} x} \, dx$.

First, if $n = 0$, then $c_{k,n} = 2^{-k} \int_\mathbb{R} m_k \, dm
= 2^{-k} \int_{E_k} m_k \, dm$, and we see $|c_{k,n}| \leq 2^{-k} |E_k|
\|m\|_\infty \|\varphi\|_\infty \leq \|m\|_\infty \|\varphi\|_\infty
\lesssim 1$.

Now assume $n \not= 0$.  Let $C = \max\{ \|\varphi^{(l)}\|_\infty :
0 \leq l \leq 4\}$.  On $E_k$, $|m^{(l)}(x)| \lesssim |x|^{-l} \leq
|2^{k-5}|^{-l} = 2^{-kl} 2^{5l}$ for $l \leq 4$. Thus,

\begin{equation*}
|m_k^{(4)}(x)| \lesssim \sum_{l=0}^4 |m^{(l)}(x)| |2^{-k(4-l)}
\varphi^{(4-l)}(2^{-k}x)| \leq \sum_{l=0}^4 2^{-kl} 2^{5l} 2^{-4k}
2^{kl} C \lesssim 2^{-4k}.
\end{equation*}

\noindent By several iterations of integration by parts,

\begin{equation*}
\begin{split}
\bigg| \int_\mathbb{R} m_k(x) e^{2\pi in2^{-k}x} \, dx \bigg| &=
\bigg| \int_{E_k} m_k(x) e^{2\pi in2^{-k}x} \, dx \bigg| \\
&= \bigg| \int_{E_k} m_k^{(4)}(x) \frac{e^{2\pi in2^{-k}x}}{(2\pi
in2^{-k})^4} \, dx \bigg| \\
&\lesssim \frac{2^{4k}}{|n|^4} |E_k| \|m_k^{(4)} \|_\infty \lesssim
\frac{2^k}{|n|^4} \lesssim \frac{2^k}{(|n| + 1)^4}.
\end{split}
\end{equation*}

\noindent Namely, $|c_{k,n}| \lesssim (|n|+1)^{-4}$.
\end{proof}

\begin{thm}\label{thm:marcin} For any Marcinkiewicz multiplier $m$,
$\Lambda_m : L^1(\mathbb{T}) \rightarrow L^{1,\infty}(\mathbb{T})$
and $\Lambda_m : L^p(\mathbb{T}) \rightarrow L^p(\mathbb{T})$ for $1
< p < \infty$.
\end{thm}

\begin{proof} We start by noting that we can assume $m(0) = 0$.  Let
$m_0 = m$ away from 0 and $m_0(0) = 0$.  Then, $m_0$ is a
Marcinkiewicz multiplier and $\Lambda_m f(x) = m(0) \widehat{f}(0) +
\Lambda_{m_0} f(x)$. But, $|m(0) \widehat{f}(0)| = |m(0)
\int_\mathbb{T} f(x) \, dx| \lesssim \|f\|_1 \leq \|f\|_p$ for any
$p$, as $m$ is uniformly bounded. Thus, it suffices to prove the
result for $\Lambda_{m_0}$, or equivalently, assuming $m(0) = 0$.

Fix $f, g \in L^1(\mathbb{T})$.  Define the reflection of a function
by $\widetilde{f}(x) = f(-x)$.  Let $f_0 = \widetilde{\overline{f}}$
and $g_0 = \widetilde{\overline{g}}$. Then,

\begin{equation*}
\begin{split}
\langle \Lambda_m f, \widetilde{g} \rangle &= \int_\mathbb{T}
\Lambda_m f(x) g_0(x) \, dx = \int_\mathbb{T} \bigg( \sum_{t \in
\mathbb{Z}} m(t) \widehat{f}(t) e^{2\pi ixt} \bigg) g_0(x) \, dx \\
&= \sum_{t \in \mathbb{Z}} m(t) \widehat{f}(t) \int_\mathbb{T} g_0(x)
e^{2\pi ixt} \, dx = \sum_{t \in \mathbb{Z}} m(t) \widehat{f}(t)
\widehat{g_0}(-t).
\end{split}
\end{equation*}

\noindent Now apply Theorem~\ref{thm:tbumps} to write

\begin{equation*}
\begin{split}
\langle \Lambda_m f, \widetilde{g} \rangle &= \sum_{t \in
\mathbb{Z}} \sum_{k=1}^\infty m(t) \widehat{f}(t)
\widehat{\psi^1_k}(t)
\widehat{g_0}(-t) \widehat{\psi^2_k}(-t) \\
&= \sum_{k=1}^\infty \sum_{t \in \mathbb{Z}} m_{k}(t) \widehat{f}(t)
\widehat{\psi^1_k}(t) \widehat{g_0}(-t) \widehat{\psi^2_k}(-t),
\end{split}
\end{equation*}

\noindent where $m_{k}$ is as given in Lemma~\ref{lemma:mk}. Let
$\psi^1_{k,n}(x) = \psi^1_k(x - n 2^{-k})$. Then,

\begin{equation*}
\begin{split}
\langle \Lambda_m f, \widetilde{g} \rangle &= \sum_{n \in
\mathbb{Z}} \sum_{k=1}^\infty \sum_{t \in \mathbb{Z}} c_{k,n}
e^{-2\pi in2^{-k}t} \widehat{f}(t) \widehat{\psi^1_k}(t)
\widehat{g_0}(-t)
\widehat{\psi^2_k}(-t) \\
&= \sum_{n \in \mathbb{Z}} \sum_{k=1}^\infty \sum_{t \in \mathbb{Z}}
c_{k,n} \widehat{f}(t) \widehat{\psi^1_{k,n}}(t)
\widehat{g_0}(-t) \widehat{\psi^2_k}(-t) \\
&= \sum_{n \in \mathbb{Z}} \sum_{k=1}^\infty \sum_{t \in \mathbb{Z}}
c_{k,n} (f * \psi^1_{k,n}) \, \widehat{} \,\, (t) (g_0 * \psi^2_k)
\, \widehat{} \,\, (-t) \\
&= \sum_{n \in \mathbb{Z}} \sum_{k=1}^\infty c_{k,n} \int_\mathbb{T}
(f * \psi^1_{k,n})(x) (g_0 * \psi^2_k)(x) \, dx,
\end{split}
\end{equation*}

\noindent the last line being an application of Plancherel. Even
though $f, g_0$ are only assumed in $L^1$, $f * \psi^1_{k,n}$ and
$g_0 * \psi^2_k$ are smooth, thus in $L^2$.  Focusing on just the
integral portion,

\begin{equation*}
\begin{split}
\int_\mathbb{T} (f * \psi^1_{k,n})(x) &(g_0 * \psi^2_k)(x) \, dx \\
&= \int_0^1 (f * \psi^1_{k,n})(x) (g_0 * \psi^2_k)(x) \, dx \\
&= 2^{-k} \int_0^{2^k} (f * \psi^1_{k,n})(2^{-k} x) (g_0 *
\psi^2_k)(2^{-k} x) \, dx \\
&= 2^{-k} \sum_{j=0}^{2^k-1} \int_j^{j+1} (f * \psi^1_{k,n})(2^{-k}
x) (g_0 * \psi^2_k)(2^{-k} x) \, dx \\
&= 2^{-k} \sum_{j=0}^{2^k-1} \int_0^1 (f *
\psi^1_{k,n})(2^{-k}(\alpha + j)) (g_0
* \psi^2_k)(2^{-k}(\alpha + j)) \, d\alpha \\
&= 2^{-k} \sum_{j=0}^{2^k-1} \int_0^1 \langle \psi_{k, j,
n,\alpha}^1, \overline{f} \rangle \langle \psi_{k,j,\alpha}^2,
\overline{g_0} \rangle \, d\alpha,
\end{split}
\end{equation*}

\noindent where $\psi^1_{k,j,n,\alpha}(x) =
\psi^1_{k,n}(2^{-k}(\alpha+j) - x) = \psi^1_k(2^{-k}(\alpha+j+n) -
x)$ and $\psi^2_{k,j,\alpha}(x) = \psi^2_k(2^{-k}(\alpha+j) - x)$.

For a dyadic interval $I = [2^{-k} j, 2^{-k} (j+1)]$, let
$\varphi_{I_\alpha}^2 = 2^{-k} \widetilde{\psi^2_{k,j,\alpha}}$.
Similarly, let $\varphi_{I^n_\alpha}^1 = 2^{-k}
\widetilde{\psi^1_{k,j,n,\alpha}}$. It is easily checked that the
original conditions on $\psi^1, \psi^2$ guarantee that $\varphi_I^1,
\varphi_I^2$ are 0-mean adapted families.  Let $\phi_I^1 =
|I|^{-1/2} \varphi_I^1$ and $\phi_{I}^2 = |I|^{-1/2} \varphi_I^2$,
so that

\begin{equation*}
\begin{split}
\langle \Lambda_m f, \widetilde{g} \rangle &= \sum_{n \in
\mathbb{Z}} \sum_{k=1}^\infty c_{k,n} 2^{-k} \sum_{j=0}^{2^k-1}
\int_0^1 \langle \psi_{k, j, n,\alpha}^1, \overline{f} \rangle
\langle \psi_{k,
j,\alpha}^2, \overline{g_0} \rangle \, d\alpha \\
&= \sum_{n \in \mathbb{Z}} \int_0^1 \sum_{I} c_{I,n} \langle
\phi_{I^n_\alpha}^1, f_0 \rangle \langle \phi_{I_\alpha}^2, g
\rangle \, d\alpha,
\end{split}
\end{equation*}

\noindent where the inner sum is over all dyadic intervals and
$c_{I,n} = c_{k,n}$ when $|I| = 2^{-k}$.  Write $c'_{I,n} =
(|n|+1)^4 c_{I,n}$, which are uniformly bounded in $I$ and $n$ by
Lemma~\ref{lemma:mk}.  Hence,

\begin{equation*}
\begin{split}
\langle \Lambda_m f, \widetilde{g} \rangle &= \sum_{n \in
\mathbb{Z}} \frac{1}{(|n|+1)^4} \int_0^1 \sum_{I} c'_{I,n} \langle
\phi_{I^n_\alpha}^1, f_0 \rangle \langle \phi_{I_\alpha}^2, g
\rangle \, d\alpha \\
&= \sum_{n \in \mathbb{Z}} \frac{1}{(|n|+1)^4} \langle T_{c'}^{[n]}
f_0, g \rangle \\
&= \Big\langle \sum_{n \in \mathbb{Z}} \frac{1}{(|n|+1)^4}
T_{c'}^{[n]} f_0, g \Big\rangle
\end{split}
\end{equation*}

As $g \in L^1$ is arbitrary, it follows that $\widetilde{\Lambda_m
f} = \sum (|n|+1)^{-4} T_{c'}^{[n]} f_0$ a.e..  But, $\|T_{c'}^{[n]}
f_0\|_p \lesssim (|n|+1) \|f_0\|_p = (|n|+1) \|f\|_p$ and
$\|T_{c'}^{[n]} f_0\|_{1,\infty} \lesssim (|n|+1) \|f\|_1$.  So, we
have immediately that $\|\Lambda_m f\|_p \lesssim \|f\|_p$ for all
$1 < p < \infty$.  Further, by Lemma~\ref{lemma:weaksum} (with $K =
k = 1$), $\|\Lambda_m f\|_{1,\infty} \lesssim \|f\|_1$.
\end{proof}

\begin{cor} $\Lambda_m : \llogln \rightarrow L(\log L)^{n-1}$ for
any Marcinkiewicz multiplier $m$ and $n \in \mathbb{N}$.
\end{cor}

\section{Single-parameter Paraproducts}\label{sec:para1}

Return to the linearization $T_\epsilon$ defined in
Section~\ref{sec:Te}.  This linear operator can be viewed as the
simplest in a family of multilinear operators, which we call
paraproducts.  For simplicity, we will focus only on the bilinear
case, but the other operators are handled in precisely the same
manner.

For $f, g : \mathbb{T} \rightarrow \mathbb{C}$, the single-parameter
bilinear paraproducts are defined

\begin{equation*}
T_\epsilon^a(f, g)(x) = \sum_{I} \epsilon_I \frac{1}{|I|^{1/2}}
\langle \phi_I^1, f \rangle \langle \phi_I^2, g \rangle \phi_I^3(x),
\end{equation*}

\noindent for $a = 1, 2, 3$, where $\varphi_I^1$, $\varphi_I^2$, and
$\varphi_I^3$ are three adapted families with the property that
$\int_\mathbb{T} \varphi_I^i \, dm = 0$ for $i \not= a$.  As before,
the sum is over all dyadic intervals $I$, and $(\epsilon_I)$ is a
uniformly bounded sequence.  By dividing out a constant, we can
assume $|\epsilon_I| \leq 1$.  The reason for the terminology
single-parameter will be become clearer in the next chapter.

The primary goal of this section is to prove standard $L^p$
estimates of these paraproducts, which we do now.

\begin{thm}\label{thm:1para} $T_\epsilon^a : L^{p_1} \times L^{p_2}
\rightarrow L^p$ for $1 < p_1, p_2 < \infty$ and $\frac{1}{p} =
\frac{1}{p_1} + \frac{1}{p_2}$.  If $p_1$ or $p_2$ or both are equal
to 1, this still holds with $L^p$ replaced by $L^{p,\infty}$. The
underlying constants do not depend on $a$ or the sequence
$\epsilon_I$. \end{thm}

\begin{proof} We will assume that $a =
1$, so that $\int \phi_I^i \, dm = 0$ for $i = 2, 3$.  It will be
clear that the proofs for $a = 2, 3$ are essentially the same.

First, suppose $p > 1$.  Then, necessarily $p_1, p_2 > 1$ and $1 <
p' < \infty$.  Note, $1/p_1 + 1/p_2 + 1/p' = 1$.  Fix $h \in
L^{p'}(\mathbb{T})$ with $\|h\|_{p'} \leq 1$.  Then,

\begin{equation*}
\begin{split}
|\langle &T_\epsilon^1 (f,g), h \rangle| = \Big| \sum_I \epsilon_I
\frac{1}{|I|^{1/2}} \langle \phi^1_I, f \rangle \langle \phi^2_I, g
\rangle \langle \phi^3_I, h \rangle \Big| \\
&\leq \sum_I \frac{1}{|I|^{1/2}} |\langle \phi^1_I, f \rangle|
|\langle \phi^2_I, g \rangle| |\langle \phi^3_I, h \rangle| \\
&= \int_\mathbb{T} \sum_I \frac{|\langle \phi_I^1, f
\rangle|}{|I|^{1/2}} \frac{|\langle \phi_I^2, g \rangle|}{|I|^{1/2}}
\frac{|\langle \phi_I^3, h \rangle|}{|I|^{1/2}} \chi_I(x) \, dx \\
&\leq \int_\mathbb{T} \bigg( \sup_I \frac{|\langle \phi_I^1, f
\rangle|}{|I|^{1/2}} \chi_I(x) \bigg) \bigg( \sum_I \frac{|\langle
\phi_I^2, g \rangle|^2}{|I|} \chi_I(x) \bigg)^{1/2} \bigg( \sum_I
\frac{|\langle \phi_I^3, h \rangle|^2}{|I|} \chi_I(x) \bigg)^{1/2} \,
dx \\
&= \int_\mathbb{T} M'f(x) S^2g(x) S^3h(x) \, dx \\
&\leq \|M'f\|_{p_1} \|S^2g\|_{p_2} \|S^3h\|_{p'} \lesssim
\|f\|_{p_1} \|g\|_{p_2}.
\end{split}
\end{equation*}

\noindent As $h$ in the unit ball of $L^{p'}$ is arbitrary, we have
$\|T^1_\epsilon(f,g)\|_p \lesssim \|f\|_{p_1} \|g\|_{p_2}$.

Now suppose $1/2 \leq p \leq 1$.  We will show $T_\epsilon^1 :
L^{p_1} \times L^{p_2} \rightarrow L^{p,\infty}$ for all $1 \leq
p_1, p_2 < \infty$.  The fact that $L^{p,\infty}$ can be replaced by
$L^p$ where appropriate will follow immediately from interpolation
of these results. Fix $1 \leq p_1, p_2 < \infty$.

Let $\|f\|_{p_1} = \|g\|_{p_2} = 1$ and $|E| > 0$.  By
Lemma~\ref{lemma:dual}, we will be done if we can find $E' \subseteq
E$, $|E'| > |E|/2$ so that $|\langle T_\epsilon^1 (f, g), \chi_{E'}
\rangle| \lesssim 1 \leq |E|^{1-1/p}$.  Using
Theorem~\ref{thm:decomposition2}, decompose each $\phi_I^3$ into

\begin{equation*}
\phi_I^3 = \sum_{k=1}^\infty 2^{-10k} \phi_I^{3,k}
\end{equation*}

\noindent where $\phi_I^{3,k}$ is the normalization of a 0-mean
adapted family $\varphi_I^{3,k}$, which are uniformly  adapted to
$I$. Further, $\supp(\phi_I^{3,k}) \subseteq 2^kI$ for $k$ small
enough, while $\phi_I^{3,k}$ is identically 0 otherwise.  Now write

\begin{equation*}
\langle T_\epsilon^1 (f, g), \chi_{E'} \rangle = \sum_{k=1}^\infty
2^{-10k} \sum_{I} \epsilon_I \frac{1}{|I|^{1/2}} \langle \phi_I^1, f
\rangle \langle \phi_I^2, g \rangle \langle \phi_I^{3,k}, \chi_{E'}
\rangle.
\end{equation*}

\noindent Hence, it suffices to show $|\sum \epsilon_I |I|^{-1/2}
\langle \phi_I^1, f \rangle \langle \phi_I^2, g \rangle \langle
\phi_I^{3,k}, \chi_{E'} \rangle| \lesssim 2^{4k}$, so long as the
underlying constants are independent of $k$.

Let $S^2$ and $S^{3,k}$ be the square functions for $\phi_I^2$ and
$\phi_I^{3,k}$.  For each $k \in \mathbb{N}$, define

\begin{gather*}
\Omega_{-3k} = \{Mf > C2^{3k} \} \cup \{S^2g > C2^{3k} \}, \\
\widetilde{\Omega}_k = \{M(\chi_{\Omega_{-3k}}) > 1/100\}, \\
\widetilde{\widetilde{\Omega}}_k = \{M(\chi_{\widetilde{\Omega}_k})
> 2^{-k-1} \}.
\end{gather*}

\noindent and

\begin{gather*}
\Omega = \bigcup_{k \in \mathbb{N}}
\widetilde{\widetilde{\Omega}}_k.
\end{gather*}

\noindent Observe, $|\Omega|$ is less than or equal to

\begin{equation*}
100 \sum_{k=1}^\infty 2^{k+1} \|M\|^2_{L^1 \rightarrow L^{1,\infty}}
\Big[\frac{1}{C^{p_1}} 2^{-3p_1k} \|M\|_{L^{p_1} \rightarrow
L^{p_1,\infty}}^{p_1} + \frac{1}{C^{p_2}} 2^{-3p_2k}
\|S^2\|_{L^{p_2} \rightarrow L^{p_2,\infty}}^{p_2}\Big].
\end{equation*}

\noindent Therefore, we can choose $C$ independent of $f$ and $g$ so
that $|\Omega| < |E|/2$. Set $E' = E - \Omega = E \cap \Omega^c$.
Then, $E' \subseteq E$ and $|E'| > |E|/2$.

Fix $k \in \mathbb{N}$.  Set $Z_k = \{S^2 g = 0\} \cup \{S^{3, k}
(\chi_{E'}) = 0\}$. Let $\mathcal{D}$ be any finite collection of
dyadic intervals. We divide this collection into three
subcollections.  Set $\mathcal{D}_1 = \{I \in \mathcal{D} : I \cap
Z_k \not= \emptyset \}$. For the remaining intervals, let
$\mathcal{D}_2 = \{I \in \mathcal{D} - \mathcal{D}_1 : I \subseteq
\widetilde{\Omega}_k\}$ and $\mathcal{D}_3 = \{I \in \mathcal{D} -
\mathcal{D}_1 : I \cap \widetilde{\Omega}_k^c \not= \emptyset\}$.

If $I \in \mathcal{D}_1$, there is some $x \in I \cap Z_k$,
which implies $S^2g(x) = 0$ or $S^{3,k}(\chi_{E'})(x) = 0$. If it is
the first, $\langle \phi^2_J, g \rangle = 0$ for all dyadic $J$
containing $x$. In particular, $\langle \phi^2_I, g \rangle = 0$. If
it is the second, then $\langle \phi^{3,k}_I, \chi_{E'} \rangle =
0$. As this holds for all $I \in \mathcal{D}_1$, we have

\begin{equation*}
\sum_{I \in \mathcal{D}_1} \frac{1}{|I|^{1/2}} |\langle \phi_I^1, f
\rangle| |\langle \phi_I^2, g \rangle| |\langle \phi_I^{3,k},
\chi_{E'} \rangle| = 0.
\end{equation*}

Now suppose $I \in \mathcal{D}_2$, namely $I \subseteq
\widetilde{\Omega}_k$.  If $k$ is big enough so that $2^k > 1/|I|$,
then $\phi_I^{3,k}$ is identically 0 and $\langle \phi_I^{3,k},
\chi_{E'} \rangle = 0$. If $2^k \leq 1/|I|$, then $\phi_I^{3,k}$ is
supported in $2^k I$. Let $x \in 2^k I$, and observe

\begin{equation*}
M(\chi_{\widetilde{\Omega}_k})(x) \geq \frac{1}{|2^kI|} \int_{2^kI}
\chi_{\widetilde{\Omega}_k} \,\, dm \geq \frac{1}{2^k} \frac{1}{|I|}
\int_I \chi_{\widetilde{\Omega}_k} \,\, dm = 2^{-k} > 2^{-k-1}.
\end{equation*}

\noindent That is, $2^kI \subseteq \widetilde{\widetilde{\Omega}}_k
\subseteq \Omega$, a set disjoint from $E'$.  Thus, $\langle
\phi_I^{3,k}, \chi_{E'} \rangle = 0$.  As this holds for all $I
\in \mathcal{D}_2$, we have

\begin{equation*}
\sum_{I \in \mathcal{D}_2} \frac{1}{|I|^{1/2}} |\langle \phi_I^1, f
\rangle| |\langle \phi_I^2, g \rangle| |\langle \phi_I^{3,k},
\chi_{E'} \rangle| = 0.
\end{equation*}

Finally, we concentrate on $\mathcal{D}_3$.  Define $\Omega_{-3k+1}$
and $\Pi_{-3k+1}$ by

\begin{gather*}
\Omega_{-3k+1} = \{Mf > C2^{3k-1}\}, \\
\Pi_{-3k+1} = \{I \in \mathcal{D}_3 : |I \cap \Omega_{-3k+1}| >
|I|/100\}.
\end{gather*}

\noindent Inductively, define for all $n > -3k+1$,

\begin{gather*}
\Omega_{n} = \{M f > C2^{-n}\}, \\
\Pi_{n} = \{I \in \mathcal{D}_3 - \bigcup_{j=-3k+1}^{n-1} \Pi_{j} :
|I \cap \Omega_{n}|
> |I|/100\}.
\end{gather*}

\noindent As $\|f\|_{p_1} = 1$, and thus not equal to 0 a.e., $Mf > 0$
everywhere. So, it is clear that each $I \in \mathcal{D}_3$ will be
in one of these collections.

Set $\Omega_{-3k}' = \Omega_{-3k}$ for symmetry.  Define
$\Omega'_{-3k+1}$ and $\Pi'_{-3k+1}$ by

\begin{gather*}
\Omega'_{-3k+1} = \{S^2 g > C2^{3k-1}\}, \\
\Pi'_{-3k+1} = \{I \in \mathcal{D}_3 : |I \cap \Omega_{-3k+1}'| >
|I|/100\}.
\end{gather*}

\noindent Inductively, define for all $n > -3k+1$,

\begin{gather*}
\Omega'_{n} = \{S^2 g > C2^{-n}\}, \\
\Pi'_{n} = \{I \in \mathcal{D}_3 - \bigcup_{j=-3k+1}^{n-1} \Pi'_{j}
: |I \cap \Omega_{n}'|
> |I|/100\}.
\end{gather*}

\noindent As every $I \in \mathcal{D}_3$ is not in $\mathcal{D}_1$,
that is $S^2 g > 0$ on $I$, it is clear that each $I \in
\mathcal{D}_3$ will be in one of these collections.

Now, we can choose an integer $N$ big enough so that $\Omega_{-N}''
= \{S^{3,k} (\chi_{E'}) > 2^N\}$ has very small measure.  In
particular, we take $N$ big enough so that $|I \cap \Omega_{-N}''| <
|I|/100$ for all $I \in \mathcal{D}_3$, which is possible since
$\mathcal{D}_3$ is a finite collection. Define

\begin{gather*}
\Omega_{-N+1}'' = \{S^{3,k}(\chi_{E'}) > 2^{N-1}\}, \\
\Pi_{-N+1}'' = \{I \in \mathcal{D}_3 : |I \cap \Omega_{-N+1}''| >
|I|/100\},
\end{gather*}

\noindent and

\begin{gather*}
\Omega_{n}'' = \{S^{3,k}(\chi_{E'}) > 2^{-n}\}, \\
\Pi_{n}'' = \{I \in \mathcal{D}_3 - \bigcup_{j=-N+1}^{n-1} \Pi_j'' :
|I \cap \Omega_{n}''| > |I|/100\},
\end{gather*}

\noindent Again, all $I \in \mathcal{D}_3$ must be in one of these
collections.

Consider $I \in \mathcal{D}_3$, so that $I \cap
\widetilde{\Omega}_k^c \not= \emptyset$. Then, there is some $x \in
I \cap \widetilde{\Omega}_k^c$ which implies $|I \cap
\Omega_{-3k}|/|I| \leq M(\chi_{\Omega_{-3k}})(x) \leq 1/100$. Write
$\Pi_{n_1, n_2, n_3} = \Pi_{n_1} \cap \Pi'_{n_2} \cap \Pi''_{n_3}$.
So,

\begin{equation*}
\begin{split}
\sum_{I \in \mathcal{D}_3} \frac{1}{|I|^{1/2}} &|\langle \phi_I^1, f
\rangle| |\langle \phi_I^2, g \rangle | |\langle \phi_I^{3,k}, \chi_{E'} \rangle| \\
&= \sum_{n_1, n_2 > -3k, \, n_3 > -N} \bigg[ \sum_{I \in \Pi_{n_1,
n_2, n_3}} \frac{1}{|I|^{1/2}} |\langle \phi_I^1, f \rangle|
|\langle \phi_I^2, g \rangle|
|\langle \phi_I^{3,k}, \chi_{E'}\rangle| \bigg]\\
&= \sum_{n_1, n_2 > -3k, \, n_3 > -N} \bigg[ \sum_{I \in \Pi_{n_1,
n_2, n_3}} \frac{|\langle \phi_I^1, f \rangle|}{|I|^{1/2}}
\frac{|\langle \phi_I^2, g \rangle|}{|I|^{1/2}} \frac{|\langle
\phi_I^{3,k}, \chi_{E'} \rangle|}{|I|^{1/2}} |I| \bigg].
\end{split}
\end{equation*}

\noindent Suppose $I \in \Pi_{n_1, n_2, n_3}$.  If $n_1 > -3k +
1$, then $I \in \Pi_{n_1}$, which in particular says $I \notin
\Pi_{n_1 - 1}$. So, $|I \cap \Omega_{n_1 - 1}| \leq |I|/100$.  If
$n_1 = -3k + 1$, then we still have $|I \cap \Omega_{-3k}| \leq
|I|/100$, as $I \in \mathcal{D}_3$. Similarly, If $n_2 > -3k + 1$,
then $I \in \Pi_{n_2}'$, which in particular says $I \notin \Pi_{n_2
- 1}'$. So, $|I \cap \Omega'_{n_2 - 1}| \leq |I|/100$.  If $n_2 =
-3k + 1$, then $|I \cap \Omega'_{-3k}| = |I \cap \Omega_{-3k}| \leq
|I|/100$, as $I \in \mathcal{D}_3$.  Finally, if $n_3 > -N+1$, then
$I \notin \Pi_{n_3 - 1}''$ and $|I \cap \Omega''_{n_3 - 1}| \leq
|I|/100$.  If $n_3 = -N+1$, then $|I \cap \Omega''_{-N}| \leq
|I|/100$ by the choice of $N$.  So, $|I \cap \Omega_{n_1 - 1}^c \cap
\Omega_{n_2 - 1}'^c \cap \Omega_{n_3 - 1}''^c| \geq \frac{97}{100}
|I|$.  Let $\Omega_{n_1, n_2, n_3} = \bigcup\{ I : I \in \Pi_{n_1,
n_2, n_3}\}$. Then,

\begin{equation*}
|I \cap \Omega_{n_1 - 1}^c \cap \Omega'^c_{n_2-1} \cap \Omega_{n_3 -
1}''^c \cap \Omega_{n_1, n_2, n_3}| \geq \frac{97}{100} |I|
\end{equation*}

\noindent for all $I \in \Pi_{n_1, n_2, n_3}$.  Further,

\begin{equation*}
\begin{split}
&\sum_{I \in \Pi_{n_1, n_2, n_3}} \frac{|\langle \phi_I^1, f
\rangle|}{|I|^{1/2}} \frac{|\langle \phi_I^2, g \rangle}{|I|^{1/2}}
\frac{|\langle \phi_I^{3,k}, \chi_{E'} \rangle|}{|I|^{1/2}} |I| \\
&\lesssim \sum_{I \in \Pi_{n_1, n_2, n_3}} \frac{|\langle \phi_I^1,
f \rangle|}{|I|^{1/2}} \frac{|\langle \phi_I^2, g
\rangle|}{|I|^{1/2}} \frac{|\langle \phi_I^{3,k}, \chi_{E'}
\rangle|}{|I|^{1/2}} |I \cap \Omega_{n_1 - 1}^c
\cap \Omega_{n_2 - 1}'^c \cap \Omega_{n_3 - 1}''^c \cap \Omega_{n_1, n_2, n_3}| \\
&= \int_{\Omega_{n_1 - 1}^c \cap \Omega_{n_2 - 1}'^c \cap
\Omega_{n_3 - 1}''^c \cap \Omega_{n_1, n_2, n_3}} \,\, \sum_{I \in
\Pi_{n_1, n_2, n_3}} \frac{|\langle \phi_I^1, f \rangle|}{|I|^{1/2}}
\frac{|\langle \phi_I^2, g \rangle|}{|I|^{1/2}} \frac{|\langle
\phi_I^{3,k}, \chi_{E'} \rangle|}{|I|^{1/2}} \chi_I(x) \,
dx \\
&\lesssim \int_{\Omega_{n_1 - 1}^c \cap \Omega_{n_2 - 1}'^c \cap
\Omega_{n_3 - 1}''^c \cap \Omega_{n_1, n_2, n_3}} Mf(x) S^2g(x)
S^{3,k} (\chi_{E'})(x) \, dx \\
&\lesssim C^2 2^{-n_1} 2^{-n_2} 2^{-n_3} |\Omega_{n_1, n_2, n_3}|.
\end{split}
\end{equation*}

Note, $|\Omega_{n_1, n_2, n_3}| \leq |\bigcup \{I : I \in
\Pi_{n_1}\}| \leq |\{M(\chi_{\Omega_{n_1}}) > 1/100\}| \lesssim
|\Omega_{n_1}| = |\{Mf > C2^{-n_1}\}| \lesssim C^{-p_1} 2^{p_1
n_1}$. By the same argument, $|\Omega_{n_1, n_2, n_3}| \lesssim
|\Omega_{n_2}'| = |\{S^2 g > C 2^{-n_2}\}| \lesssim C^{-p_2} 2^{p_2
n_2}$, and $|\Omega_{n_1, n_2, n_3}| \lesssim |\Omega_{n_3}''| =
|\{S^{3,k}(\chi_{E'}) > 2^{-n_3}\}| \lesssim 2^{\alpha n_3}$ for any
$\alpha \geq 1$. Therefore, $|\Omega_{n_1, n_2, n_3}| \lesssim
C^{-p_1 - p_2} 2^{\theta_1 p_1 n_1} 2^{\theta_2 p_2 n_2} 2^{\theta_3
\alpha n_3}$ for any $\theta_1 + \theta_2 + \theta_3 = 1$, $0 \leq
\theta_1, \theta_2, \theta_3 \leq 1$. Hence,

\begin{equation*}
\begin{split}
\sum_{I \in \mathcal{D}_3} \frac{1}{|I|^{1/2}} &|\langle \phi_I^1, f
\rangle| |\langle \phi_I^2, g \rangle| |\langle \phi_I^{3,k},
\chi_{E'} \rangle| \\
&\lesssim \sum_{n_1, n_2 > -3k, \,\, n_3
> 0} 2^{(\theta_1 p_1 - 1)n_1} 2^{(\theta_2 p_2 - 1)n_2} 2^{(\theta_3 \alpha - 1) n_3} \quad + \\
&\,\,\, \sum_{n_1, n_2 > -3k, \,\, -N < n_3 \leq 0}
2^{(\theta_1 p_1 - 1)n_1} 2^{(\theta_2 p_2 - 1)n_2} 2^{(\theta_3 \alpha - 1) n_3} \\
&= A + B.
\end{split}
\end{equation*}

\noindent For the first term, take $\theta_1 = 1/(2p_1)$, $\theta_2
= 1/(2p_2)$, $\theta_3 = 1 - 1/(2p)$, and $\alpha = 1$. For the
second term, take $\theta_1 = 1/(3p_1)$, $\theta_2 = 1/(3p_2)$,
$\theta_3 = 1 - 1/(3p) > 0$, and $\alpha = 2/\theta_3$ to see

\begin{equation*}
\begin{split}
A &= \sum_{n_1, n_2 > -3k, \,\, n_3 > 0} 2^{-n_1/2} 2^{-n_2/2}
2^{-n_3/2p} \lesssim 2^{3k}, \\
B &= \sum_{n_1, n_2 > -3k, \,\, -N < n_3 \leq 0} 2^{-2n_1/3}
2^{-2n_2/3} 2^{n_3} \leq \sum_{n_1, n_2 > -3k, n_3 \leq 0}
2^{-2n_1/3} 2^{-2n_2/3} 2^{n_3} \lesssim 2^{4k}.
\end{split}
\end{equation*}

\noindent The estimate for $A$ is made in part because $p$ is
bounded away from 0 ($p \geq 1/2$).  Also, there is no dependence on
the number $N$, which depends on $\mathcal{D}$, or $C$, which
depends on $E$.

Combining the estimates for $\mathcal{D}_1$, $\mathcal{D}_2$, and
$\mathcal{D}_3$, we see

\begin{equation*}
\sum_{I \in \mathcal{D}} \frac{1}{|I|^{1/2}} |\langle \phi_I^1, f
\rangle| |\langle \phi_I^2, g \rangle| |\langle \phi_I^{3,k},
\chi_{E'} \rangle| \lesssim 2^{4k},
\end{equation*}

\noindent where the constant has no dependence on the collection
$\mathcal{D}$.  Hence, as $\mathcal{D}$ is arbitrary, we have

\begin{equation*}
\Big| \sum_I \epsilon_I \frac{1}{|I|^{1/2}} \langle \phi_I^1, f
\rangle \langle \phi_I^2, g \rangle \langle \phi_I^{3,k}, \chi_{E'}
\rangle \Big| \leq \sum_I \frac{1}{|I|^{1/2}} |\langle \phi_I^1, f
\rangle| |\langle \phi_I^2, g \rangle| |\langle \phi_I^{3,k},
\chi_{E'} \rangle| \lesssim 2^{4k},
\end{equation*}

\noindent which completes the proof.
\end{proof}

It should now be clear that proving the above for $a \not= 1$
follows by permuting the roles of $M$ and $S$.  In particular, $M$
will always be applied to the function in the $a^{th}$ slot and $S$
to the others.

For any $\vec{n} \in \mathbb{Z}^2$, we can define the shifted
paraproducts by

\begin{equation*}
T_\epsilon^{a, [\vec{n}]}(f, g)(x) = \int_0^1 \sum_{I} \epsilon_I
\frac{1}{|I|^{1/2}} \langle \phi_{I_\alpha^{n_1}}^{1}, f \rangle
\langle \phi_{I_\alpha^{n_2}}^2, g \rangle \phi_{I_\alpha}^3(x) \,
d\alpha,
\end{equation*}

\noindent where, as before, $\int_\mathbb{T} \varphi_I^i \, dm = 0$
for $i \not= a$.  Much like in Section~\ref{sec:shifted},
understanding these operators is just a matter of reworking the
proof.  Simply replace $M$ by $M^{[n_j]}$ and $S$ by $S^{[n_j]}$
where appropriate.  This leads to the previous estimates with an
additional factor of $(|n_1|+1)(|n_2| + 1)$.

\section{Coifmann-Meyer Operators}

We will employ the standard ``$\partial$" notation of partial derivatives.
That is, $\partial^k_j f$ is the $k^{th}$ partial derivative of $f$ in the
$j^{th}$ variable.  Further, if $\alpha = (\alpha_1, \ldots,
\alpha_n)$ is a vector of nonnegative integers and $f : \mathbb{R}^d
\rightarrow \mathbb{C}$, then

\begin{equation*}
\partial^\alpha f(\vec{x}) = \partial_1^{\alpha_1} \cdots \partial_d^{\alpha_d} f(x_1,
\ldots, x_n)
\end{equation*}

\noindent For such a vector $\alpha$, we write $|\alpha| = \alpha_1
+ \ldots + \alpha_d$.

\begin{defnn} Let $m : \mathbb{R}^d \rightarrow \mathbb{C}$ be smooth
away from $0$ and uniformly bounded.  We say $m$ is a Coifman-Meyer
multiplier if $|\partial^\alpha m(\vec{t})| \lesssim
\|\vec{t}\|^{-|\alpha|}$ for all vectors $\alpha$ with $|\alpha|
\leq d(d+3)$, where $\|\vec{t}\|$ is the standard Euclidean norm on
$\mathbb{R}^d$.
\end{defnn}

For a Coifman-Meyer multiplier $m$ on $\mathbb{R}^d$ and $L^1$ functions
$f_1, \ldots, f_d : \mathbb{T} \rightarrow \mathbb{C}$, we define
the multilinear multiplier operator $\Lambda_m(f_1, \ldots, f_d) :
\mathbb{T} \rightarrow \mathbb{C}$ as

\begin{equation*}
\Lambda_m(f_1, \ldots, f_d)(x) = \sum_{\vec{t} \in \mathbb{Z}^d}
m(\vec{t}) \widehat{f}_1(t_1) \cdots \widehat{f}_d(t_d) e^{2\pi
ix(t_1 + \ldots + t_d)}.
\end{equation*}

The principal goal we have for these operators is the
following $L^p$ result.

\begin{thmn} For any Coifman-Meyer multiplier $m$ on $\mathbb{R}^d$,
$\Lambda_m : L^{p_1} \times \ldots \times L^{p_d} \rightarrow L^p$
for $1 < p_j < \infty$ and $\frac{1}{p_1} + \ldots + \frac{1}{p_d} =
\frac{1}{p}$. If any or all of the $p_j$ are equal to 1, this still
holds with $L^p$ replaced by $L^{p,\infty}$.  \end{thmn}

For simplicity, we will focus on the $d = 2$ case, but there is no
difference in the proof.  We start with the following.

\begin{claim}\label{claim:fgh} Let $f, g, h : \mathbb{T} \rightarrow
\mathbb{C}$ be smooth.  Then,

\begin{equation*}
\sum_{s, t \in \mathbb{Z}} \widehat{f}(s) \widehat{g}(t)
\widehat{h}(-s-t) = \int_\mathbb{T} f(x) g(x) h(x) \, dx.
\end{equation*}
\end{claim}

\begin{proof} As $f$ is smooth, we have the inversion formula $f(x)
= \sum_s \widehat{f}(s) e^{2\pi ixs}$.  Similarly for $g$.  So,

\begin{equation*}
\begin{split}
\sum_{s, t \in \mathbb{Z}} \widehat{f}(s) \widehat{g}(t)
\widehat{h}(-s-t) &= \sum_{s, t \in \mathbb{Z}} \widehat{f}(s)
\widehat{g}(t) \Big( \int_\mathbb{T} h(x) e^{-2\pi ix(-s-t)} \, dx \Big) \\
&= \int_\mathbb{T} h(x) \Big( \sum_{s \in \mathbb{Z}} \widehat{f}(s)
e^{2\pi ixs} \Big) \Big( \sum_{t \in \mathbb{Z}} \widehat{g}(t)
e^{2\pi ixt} \Big) \, dx \\
&= \int_\mathbb{T} f(x) g(x) h(x) \, dx.
\end{split}
\end{equation*}
\end{proof}

\begin{lemma}\label{lemma:m2} Let $m : \mathbb{R}^2 \rightarrow
\mathbb{C}$ be any Coifmann-Meyer multiplier and $\psi_k^{a,1},
\psi_k^{a,2}$, $a = 1,2,3$, the functions guaranteed by
Theorem~\ref{thm:doubletbumps}. For each $k \in \mathbb{N}$ and $1
\leq a \leq 3$, there is a smooth function $m_{a,k}$ so that
$m_{a,k}(s,t) \widehat{\psi^{a,1}_k}(s) \widehat{\psi^{a,2}_k}(t) =
m(s,t) \widehat{\psi^{a,1}_k}(s) \widehat{\psi^{a,2}_k}(t)$ and

\begin{equation*}
m_{a,k}(s,t) = \sum_{\vec{n} \in \mathbb{Z}^2} c_{a,k,\vec{n}}
e^{-2\pi in_1 2^{-k}s} e^{-2\pi in_2 2^{-k}t},
\end{equation*}

\noindent where $|c_{a,k,\vec{n}}| \lesssim (|n_1| +
1)^{-5}(|n_2|+1)^{-5}$ uniformly in $a$ and $k$.
\end{lemma}

\begin{proof} Let $\varphi_1 : \mathbb{R}^2 \rightarrow \mathbb{C}$ be
a smooth function with

\begin{gather*}
\supp(\varphi_1) \subseteq \Big([-2^{-1}, -2^{-11}] \cup [2^{-11},
2^{-1}] \Big) \times [2^{-1}, 2^{-1}] \quad \text{ and } \\
\varphi_1 = 1 \, \text{ on } \, \Big( [-2^{-2}, -2^{-10}] \cup
[2^{-10}, 2^{-2}] \Big) \times [-2^{-2}, 2^{-2}].
\end{gather*}

\noindent Let $\varphi_3 = \varphi_1$ and $\varphi_2(x,y) =
\varphi_1(y,x)$.  Define $m_{a,k}(s,t) = m(s,t) \varphi_a(2^{-k} s,
2^{-k}t)$. Then, $m_{a,k}(s,t) \widehat{\psi_k^{a,1}}(s)
\widehat{\psi_k^{a,2}}(t) = m(s,t) \widehat{\psi_k^{a,1}}(s)
\widehat{\psi_k^{a,2}}(t)$ by construction. Further, if $E_{a,k}$ is
the support of $m_{a,k}$, then $E_{a,k} \subset [-2^{k-1},
2^{k-1}]^2$.

Recall that $\{2^{-k/2} e^{-2\pi in2^{-k}x}\}$ is an orthonormal
basis on any interval of length $2^k$, so

\begin{equation*}
\begin{split}
m_{a,k}(s,t) &= \sum_{\vec{n} \in \mathbb{Z}^2} \bigg(
\int_{\mathbb{R}^2} m_{a,k}(x,y) \frac{e^{2\pi
in_12^{-k}x}}{2^{k/2}} \frac{e^{2\pi in_22^{-k}y}}{2^{k/2}} \, dx \,
dy \bigg) \frac{e^{-2\pi in_12^{-k}s}}{2^{k/2}} \frac{e^{-2\pi in_2
2^{-k}
t}}{2^{k/2}} \\
&= \sum_{\vec{n} \in \mathbb{Z}^2} c_{a,k,\vec{n}} e^{-2\pi in_1
2^{-k}s} e^{-2\pi in_2 2^{-k}t},
\end{split}
\end{equation*}

\noindent where $c_{a,k,\vec{n}} = 2^{-2k} \int_{\mathbb{R}^2}
m_{a,k}(x,y) e^{2\pi in_12^{-k} x} e^{2\pi in_2 2^{-k} y} \, dx \,
dy$.

First, if $\vec{n} = (0,0)$, then $c_{a,k,\vec{n}} = 2^{-2k}
\int_{\mathbb{R}^2} m_{a,k} \, dm = 2^{-2k} \int_{E_k} m_{a,k} \,
dm$. So, $|c_{a,k,\vec{n}}| \leq 2^{-2k} |E_k| \|m\|_\infty
\|\varphi\|_\infty \leq \|m\|_\infty \|\varphi\|_\infty \lesssim 1$.

Assume $n_1 \not= 0$, $n_2 \not= 0$.  Let $C = \max\{ \|\partial^\alpha
\varphi_a\|_\infty : 0 \leq |\alpha|  \leq 10, a =1,2,3\}$. Note, for
$(x,y) \in E_{a,k}$, $|x| \geq 2^{k-11}$ if $a
= 1, 3$ and $|y| \geq 2^{k-11}$ if $a = 2$.  So, $\|(x,y)\| \geq
2^{k-11}$ on $E_{a,k}$ and $|\partial^{\alpha} m(x,y)| \lesssim
\|(x,y)\|^{-|\alpha|} \leq |2^{k-11}|^{-|\alpha|} = 2^{-k|\alpha|}
2^{11|\alpha|}$ for all $|\alpha| \leq 10$. Set $\beta = (5,5)$.
Write $\alpha \leq \beta$ if $\alpha_1 \leq \beta_1$ and $\alpha_2
\leq \beta_2$. Then,

\begin{equation*}
\begin{split}
|\partial^\beta m_{a,k}(x,y)| &\lesssim \sum_{\alpha \leq \beta} |\partial^\alpha
m(x,y)| |2^{-k(|\beta|-|\alpha|)} \partial^{\beta - \alpha} \varphi(2^{-k}x, 2^{-k}y)| \\
&\leq \sum_{\alpha \leq \beta} 2^{-k|\alpha|} 2^{11|\alpha|}
2^{-10k} 2^{k|\alpha|} C \lesssim 2^{-10k}.
\end{split}
\end{equation*}

\noindent By several iterations of integration by parts,

\begin{equation*}
\begin{split}
\bigg| \int_{\mathbb{R}^2} m_{a,k}(x) &e^{2\pi in_1 2^{-k}x} e^{2\pi
in_2 2^{-k}y} \, dx \, dy \bigg| \\
&= \bigg| \int_{E_{a,k}} m_{a,k}(x) e^{2\pi in_1 2^{-k}x} e^{2\pi
in_2 2^{-k}y} \, dx \, dy \bigg| \\
&= \bigg| \int_{E_{a,k}} \partial^\beta m_{a,k}(x) \frac{e^{2\pi in_1
2^{-k}x} e^{2\pi in_2 2^{-k}y}}{(2\pi in_1 2^{-k})^{5} (2\pi in_2
2^{-k})^{5}} \, dx \, dy \bigg| \\
&\lesssim \frac{2^{10k}}{|n_1|^5 |n_2|^5} |E_{a,k}| \|\partial^\beta
m_{a,k} \|_\infty \lesssim \frac{2^{2k}}{|n_1|^5 |n_2|^5} \lesssim
\frac{2^{2k}}{(|n_1| + 1)^5 (|n_2|+1)^5}.
\end{split}
\end{equation*}

\noindent Namely, $|c_{a,k,\vec{n}}| \lesssim (|n_1|+1)^{-5}
(|n_2|+1)^{-5}$.  If $n_1 = 0$, repeat the above argument with
$\beta = (0,5)$.  If $n_2 = 0$, use $\beta = (5,0)$.
\end{proof}

\begin{thm}\label{thm:c-m} For any Coifman-Meyer multiplier $m$ on
$\mathbb{R}^2$, $\Lambda_m : L^{p_1} \times L^{p_2} \rightarrow L^p$
for $1 < p_1, p_2 < \infty$ and $\frac{1}{p_1} + \frac{1}{p_2} =
\frac{1}{p}$.  If $p_1$ or $p_2$ or both are equal to 1, this still
holds with $L^p$ replaced by $L^{p,\infty}$.  \end{thm}

\begin{proof} Fix $m$ and let $f, g : \mathbb{T}
\rightarrow \mathbb{C}$.  Then,

\begin{equation*}
\Lambda_m(f,g)(x) = \sum_{s, t \in \mathbb{Z}} m(s,t) \widehat{f}(s)
\widehat{g}(t) e^{2\pi ix(s+t)}.
\end{equation*}

\noindent As in the proof of Theorem~\ref{thm:marcin}, we can assume
$m(0,0) = 0$, as $|m(0,0) \widehat{f}(0) \widehat{g}(0)| \lesssim
\|f\|_1 \|g\|_1 \leq \|f\|_{p_1} \|g\|_{p_2}$.

Let $h \in L^1(\mathbb{T})$.  Write $f_0 = \widetilde{\overline{f}}$
and similarly for $g_0, h_0$.  Then,

\begin{equation*}
\begin{split}
\langle \Lambda_m(f,g), \widetilde{h} \rangle &= \int_\mathbb{T}
\Lambda_m(f,g)(x) h_0(x) \, dx \\
&= \int_\mathbb{T} \bigg( \sum_{s, t \in \mathbb{Z}} m(s,t)
\widehat{f}(s) \widehat{g}(t) e^{2\pi ix(s+t)}
\bigg) h_0(x) \, dx \\
&= \sum_{s, t \in \mathbb{Z}} m(s,t) \widehat{f}(s) \widehat{g}(t)
\int_\mathbb{T} h_0(x) e^{2\pi ix(s+t)} \, dx \\
&= \sum_{s, t \in \mathbb{Z}} m(s,t) \widehat{f}(s) \widehat{g}(t)
\widehat{h_0}(-s-t).
\end{split}
\end{equation*}

\noindent Now apply Theorem~\ref{thm:doubletbumps} to write

\begin{equation*}
\begin{split}
\langle \Lambda_m(f,g), \widetilde{h} \rangle &= \sum_{a=1}^3
\sum_{k=1}^\infty \sum_{s, t \in \mathbb{Z}} m(s,t) \widehat{f}(s)
\widehat{\psi_k^{a,1}}(s) \widehat{g}(t) \widehat{\psi_k^{a,2}}(t)
\widehat{h_0}(-s-t) \widehat{\psi_k^{a,3}}(-s-t) \\
&= \sum_{a=1}^3 \sum_{k=1}^\infty \sum_{s, t \in \mathbb{Z}}
m_{a,k}(s,t) \widehat{f}(s) \widehat{\psi_k^{a,1}}(s) \widehat{g}(t)
\widehat{\psi_k^{a,2}}(t) \widehat{h_0}(-s-t)
\widehat{\psi_k^{a,3}}(-s-t) \\
&=: S_1 + S_2 + S_3,
\end{split}
\end{equation*}

\noindent where $m_{a,k}$ is as given in Lemma~\ref{lemma:m2}.  Let
$\psi^{a,1}_{k,n_1}(x) = \psi^{a,1}_k(x - n_1 2^{-k})$ and
$\psi_{k,n_2}^{a,2}(x) = \psi^{a,2}_k(x - n_2 2^{-k})$. Then,

\begin{equation*}
\begin{split}
S_a &= \sum_{k=1}^\infty \sum_{s,t \in \mathbb{Z}} m_{a,k}(s,t)
\widehat{f}(s) \widehat{\psi_k^{a,1}}(s) \widehat{g}(t)
\widehat{\psi_k^{a,2}}(t) \widehat{h_0}(-s-t)
\widehat{\psi_k^{a,3}}(-s-t) \\
&= \sum_{\vec{n} \in \mathbb{Z}^2} \sum_{k=1}^\infty \sum_{s,t \in
\mathbb{Z}} c_{a,k,\vec{n}} \widehat{f}(s)
\widehat{\psi_{k,n_1}^{a,1}}(s) \widehat{g}(t)
\widehat{\psi_{k,n_2}^{a,2}}(t) \widehat{h_0}(-s-t)
\widehat{\psi_k^{a,3}}(-s-t) \\
&= \sum_{\vec{n} \in \mathbb{Z}^2} \sum_{k=1}^\infty \sum_{s,t \in
\mathbb{Z}} c_{a,k,\vec{n}} (f * \psi_{k,n_1}^{a,1}) \, \widehat{}
\,\, (s) (g * \psi_{k,n_2}^{a,2}) \, \widehat{} \,\, (t) (h_0 *
\psi_k^{a,3}) \, \widehat{} \,\, (-s-t) \\
&= \sum_{\vec{n} \in \mathbb{Z}^2} \sum_{k=1}^\infty c_{a,k,\vec{n}}
\int_\mathbb{T} (f * \psi_{k,n_1}^{a,1})(x) (g *
\psi_{k,n_2}^{a,2})(x) (h_0 * \psi_{k}^{a,3})(x) \, dx,
\end{split}
\end{equation*}

\noindent where the last line is the application of
Claim~\ref{claim:fgh}.  Even though $f, g, h_0$ are not necessarily
smooth, their convolutions with smooth functions will be.  Just
as in the proof of Theorem~\ref{thm:marcin}, we can dilate and
translate to write

\begin{equation*}
\begin{split}
\int_\mathbb{T} &(f * \psi^{a,1}_{k,n_1})(x) (g *
\psi^{a,2}_{k,n_2})(x) (h_0 * \psi^{a,3}_k)(x) \, dx \\
&= 2^{-k} \int_0^{2^k} (f * \psi^{a,1}_{k,n_1})(2^{-k}x) (g *
\psi^{a,2}_{k,n_2})(2^{-k}x) (h_0 * \psi^{a,3}_k)(2^{-k}x) \, dx \\
&= 2^{-k} \sum_{j=0}^{2^k-1} \int_0^1 \langle \psi_{k, j,
n_1,\alpha}^{a,1}, \overline{f} \rangle \langle
\psi_{k,j,n_2,\alpha}^{a,2}, \overline{g} \rangle \langle
\psi_{k,j,\alpha}^{a,3}, \overline{h_0} \rangle\, d\alpha,
\end{split}
\end{equation*}

\noindent where $\psi^{a,1}_{k,j,n_1,\alpha}(x) =
\psi^{a,1}_{k,n_1}(2^{-k}(\alpha+j) - x) =
\psi^{a,1}_k(2^{-k}(\alpha+j+n_1) - x)$, and similarly for the other
two functions.

For a dyadic interval $I = [2^{-k} j, 2^{-k} (j+1)]$, let
$\varphi_{I_\alpha^{n_1}}^{a,1} = 2^{-k}
\widetilde{\psi^{a,1}_{k,j,n_1,\alpha}}$,
$\varphi_{I_\alpha^{n_2}}^{a,2} = 2^{-k}
\widetilde{\psi^{a,2}_{k,j,n_2,\alpha}}$, and
$\varphi_{I_\alpha}^{a,3} = 2^{-k}
\widetilde{\psi^{a,3}_{k,j,\alpha}}$. It is easily checked that the
original conditions on $\psi^{a,i}$ guarantee that $\varphi_I^{a,i}$
are adapted families with mean 0 when $a \not= i$. Let $\phi_I^{a,i}
= |I|^{-1/2} \varphi_I^{a,i}$, so that

\begin{equation*}
\begin{split}
S_a &= \sum_{\vec{n} \in \mathbb{Z}^2} \sum_{k=1}^\infty
c_{a,k,\vec{n}} 2^{-k} \sum_{j=0}^{2^k-1} \int_0^1 \langle \psi_{k,
j, n_1,\alpha}^{a,1}, \overline{f} \rangle \langle
\psi_{k,j,n_2,\alpha}^{a,2}, \overline{g} \rangle \langle
\psi_{k,j,\alpha}^{a,3}, \overline{h_0} \rangle\, d\alpha \\
&= \sum_{\vec{n} \in \mathbb{Z}^2} \int_0^1 \sum_{I} c_{a,I,\vec{n}}
\frac{1}{|I|^{1/2}} \langle \phi_{I^{n_1}_\alpha}^{a,1}, f_0 \rangle
\langle \phi^{a,2}_{I^{n_2}_\alpha}, g_0 \rangle \langle
\phi_{I_\alpha}^{a,3}, h \rangle \, d\alpha,
\end{split}
\end{equation*}

\noindent where the inner sum is over all dyadic intervals and
$c_{a,I,\vec{n}} = c_{a,k,\vec{n}}$ when $|I| = 2^{-k}$.  Write
$c'_{a,I,\vec{n}} = (|n_1|+1)^5 (|n_2|+1)^5 c_{a,I,\vec{n}}$,
which are uniformly bounded in $I$ and $\vec{n}$ by
Lemma~\ref{lemma:m2}. Hence,

\begin{equation*}
\begin{split}
S_a &= \sum_{\vec{n} \in \mathbb{Z}^2} \frac{1}{(|n_1|+1)^5
(|n_2|+1)^5} \int_0^1 \sum_{I} c'_{a,I,\vec{n}} \frac{1}{|I|^{1/2}}
\langle \phi_{I^{n_1}_\alpha}^{a,1} f_0 \rangle \langle
\phi^{a,2}_{I^{n_2}_\alpha}, g_0 \rangle \langle
\phi_{I_\alpha}^{a,3}, h \rangle \, d\alpha \\
&= \sum_{\vec{n} \in \mathbb{Z}^2} \frac{1}{(|n_1|+1)^5 (|n_2|+1)^5}
\langle T_{c'}^{a, [\vec{n}]}(f_0, g_0), h \rangle \\
&= \Big\langle \sum_{\vec{n} \in \mathbb{Z}^2} \frac{1}{(|n_1|+1)^5
(|n_2|+1)^5} T_{c'}^{a, [\vec{n}]}(f_0, g_0), h \Big\rangle
\end{split}
\end{equation*}

As $h \in L^1$ is arbitrary, it follows that

\begin{equation*}
\widetilde{\Lambda_m(f,g)} = \sum_{\vec{n} \in \mathbb{Z}^2}
\frac{1}{(|n_1|+1)^5 (|n_2|+1)^5} \sum_{a=1}^3 T_{c'}^{a,
[\vec{n}]}(f_0, g_0)
\end{equation*}

\noindent almost everywhere.  We know $\|T_{c'}^{a, [\vec{n}]} (f_0,
g_0)\|_p \lesssim (|n_1|+1)(|n_2|+1) \|f\|_{p_1} \|g\|_{p_2}$ when
$p_1, p_2 > 1$, and $\|T_{c'}^{a, [\vec{n}]}(f_0, g_0)\|_{p,\infty}
\lesssim (|n_1|+1)(|n_2|+1) \|f\|_{p_1} \|g\|_{p_2}$ when $p_1$ or
$p_2$ or both are equal to 1.  So, $\|\Lambda_m(f,g)\|_p \lesssim
\|f\|_{p_1} \|g\|_{p_2}$ whenever $p \geq 1$, $p_1, p_2 > 1$ follows
immediately. By Lemma~\ref{lemma:weaksum} (with $k = 2$),
$\|\Lambda_m(f,g)\|_{p,\infty} \lesssim \|f\|_{p_1} \|g\|_{p_2}$ for
all $p_1, p_2 \geq 1$; the sum over $a$ does not cause any problems.
By interpolation of these cases, $\|\Lambda_m(f,g)\|_p \lesssim
\|f\|_{p_1} \|g\|_{p_2}$ whenever $p_1, p_2 > 1$ and $p < 1$.
\end{proof}

\chapter{Bi-parameter Multipliers}

\section{Hybrid Max-Square Functions}

When considering bi-parameter multipliers, the max and square
functions of previous chapters can no longer be applied.  However,
they can be properly extended to this setting~\cite{camil1, camil2}.

We say a set $R \subset \mathbb{T}^2$ is a dyadic rectangle if there
exist dyadic intervals $I$ and $J$ so that $R = I \times J$. Given
two adapted families $\varphi_I^1$ and $\varphi_I^2$, we will write
$\varphi_R(x,y) = \varphi_I^1(x) \varphi_J^2(y)$ for $R = I \times
J$.  We will informally write $\{\varphi_R\}$ to mean the collection
over all dyadic rectangles $R$.  For $\varphi_R = \varphi^1_I \oplus
\varphi_J^2$, set $\phi_R = |R|^{-1/2} \varphi_R = \phi_I^1 \oplus
\phi_J^2$.

For functions $f : \mathbb{T}^2 \rightarrow \mathbb{C}$, define

\begin{equation*}
MMf(x,y) = \sup_{R} \frac{1}{|R|^{1/2}} |\langle \phi_R, f \rangle|
\chi_R(x,y).
\end{equation*}

\noindent If $\{\varphi_R\}$ is a family such that $\int_\mathbb{T}
\varphi_J^2 \, dm = 0$ for all $J$, then define

\begin{equation*}
MSf(x,y) = \sup_I \frac{1}{|I|^{1/2}} \bigg( \sum_J \frac{|\langle
\phi_R, f \rangle|^2}{|J|} \chi_J(y) \bigg)^{1/2} \chi_I(x),
\end{equation*}

\noindent where of course $R = I \times J$.  This $MS$
operator is similar to taking a square function $S$ of $f$ in the
its second variable, then a maximal function $M'$ in its first
variable.  Analogously, if $\int_\mathbb{T} \varphi_I^1 \, dm = 0$
for all $I$, define

\begin{equation*}
SMf(x,y) = \Bigg( \sum_I \frac{\Big( \sup_J \frac{1}{|J|^{1/2}}
|\langle \phi_R, f \rangle| \chi_J(y) \Big)^2}{|I|} \chi_I(x)
\Bigg)^{1/2}.
\end{equation*}

\noindent Finally, if $\int_\mathbb{T} \varphi_I^1 \, dm =
\int_\mathbb{T} \varphi_J^2 \, dm = 0$, set

\begin{equation*}
SSf(x,y) = \bigg( \sum_R \frac{|\langle \phi_R, f \rangle|^2}{|R|}
\chi_R(x,y) \bigg)^{1/2}.
\end{equation*}

\noindent We note that the ``$M$" in $MS$, $SM$, and $MM$ really
corresponds to an $M'$.  However, this should not cause any
confusion.

From now on, we will be less rigid about the notation.  If we write
$\phi_R$, it will be understood to be a collection over all dyadic
rectangles, where each $\phi_R = \phi_I^1 \oplus \phi_J^2$. Further,
whenever we employ $MM$, $SM$, $MS$, or $SS$, it will be understood
that there are underlying adapted families and they have integral 0
in the appropriate variable.

\begin{thm}\label{thm:MS} Each of $MM$, $MS$, $SM$, and $SS$ maps
$L^p(\mathbb{T}^2) \rightarrow L^p(\mathbb{T}^2)$ for $1 < p <
\infty$ and $\llogl(\mathbb{T}^2) \rightarrow
L^{1,\infty}(\mathbb{T}^2)$. \end{thm}

\begin{proof} Throughout this proof, we will write $\phi_R = \phi_I
\oplus \phi_J$, instead of $\phi_I^1$ and $\phi_J^2$. This
is simply for neatness.  The underlying adapted families can still
be distinct.  Recall the notation $L_j$ from
Section~\ref{sec:strong}. We apply this to $M$, $M'$, and $S$. In
particular, $M_1$, $M_2$, $M'_1$, $M_2'$, $S_1$, and $S_2$ each map
$L^p(\mathbb{T}^2) \rightarrow L^p(\mathbb{T}^2)$ for $1 < p <
\infty$, $L^1(\mathbb{T}^2) \rightarrow L^{1,\infty}(\mathbb{T}^2)$,
and $\llogl(\mathbb{T}^2) \rightarrow L^1(\mathbb{T}^2)$ by
interpolation. Further, each satisfies Fefferman-Stein inequalities
for $r = 2$.

Use Theorem~\ref{thm:decomposition1} to write

\begin{equation*}
\varphi_R = \varphi_I \oplus \varphi_J = \Big( \sum_{k_1=1}^\infty
2^{-10k_1} \varphi_I^{k_1} \Big) \oplus \Big( \sum_{k_2=1}^\infty
2^{-10k_2} \varphi_J^{k_2} \Big) =: \sum_{\vec{k} \in \mathbb{N}^2}
2^{-10|\vec{k}|} \varphi_R^{\vec{k}}
\end{equation*}

\noindent where each $\varphi_R^{\vec{k}}$ is the tensor product of
functions uniformally adapted to $I, J$ respectively.  We write
$|\vec{k}| = k_1 + k_2$.  If each $k_1, k_2$ is small enough,
$\supp(\varphi_R^{\vec{k}}) \subseteq 2^{\vec{k}} R := 2^{k_1} I
\times 2^{k_2} J$.  Otherwise, $\varphi_R^{\vec{k}}$ is identically
0.  Let $K(R)$ be the subset of $\mathbb{N}^2$ for which the first
case occurs.  As they are uniformally adapted,
$\|\varphi_R^{\vec{k}}\|_\infty \lesssim 1$ uniformly in $\vec{k}$ and
$R$.  Fix $R$ and suppose $(x,y) \in R$.  Then,

\begin{equation*}
\begin{split}
\frac{1}{|R|} |\langle \varphi_R, f \rangle| \chi_R(x,y) &\leq
\frac{1}{|R|} \sum_{\vec{k} \in \mathbb{N}^2} 2^{-10|\vec{k}|}
\int_{\mathbb{T}^2} |\varphi_R^{\vec{k}}| |f| \, dm \\
&= \frac{1}{|R|} \sum_{\vec{k} \in K(R)} 2^{-10|\vec{k}|}
\int_{2^{\vec{k}}R} |\varphi_R^{\vec{k}}| |f| \, dm \\
&\lesssim \sum_{\vec{k} \in K(R)} 2^{-9|\vec{k}|}
\frac{1}{|2^{\vec{k}}R|}
\int_{2^{\vec{k}} R} |f| \, dm \\
&\leq \sum_{\vec{k} \in \mathbb{N}^2} 2^{-9|\vec{k}|} M_Sf(x,y)
\lesssim M_Sf(x,y).
\end{split}
\end{equation*}

\noindent If $(x,y)$ is not in $R$, then this inequality holds
trivially.  As $R$ is arbitrary, $MMf \lesssim M_Sf \leq M_1 \circ
M_2f$.  Hence,

\begin{gather*}
\|MMf\|_p \lesssim \|M_1 \circ M_2 f\|_p \lesssim \|M_2 f\|_p
\lesssim \|f\|_p, \\
\|MMf\|_{1,\infty} \lesssim \|M_1 \circ M_2 f\|_{1,\infty} \lesssim
\|M_2 f\|_1 \lesssim \|f\|_{\llogl}.
\end{gather*}

We abuse notation slightly and write $\langle f, \phi_I \rangle$ to
mean $\int_\mathbb{T} \overline{\phi}_I(x) f(x,y) \, dx$, a function
of the variable $y$. Thus, $\langle \phi_R, f \rangle = \langle
\phi_J, \langle f, \phi_I \rangle \rangle$ makes sense. Also, we can
consider the two variable function $\langle f, \phi_I \rangle
\chi_I$.  In this manner,

\begin{equation*}
\begin{split}
SMf(x,y) &= \Bigg( \sum_I \frac{\Big( \sup_J \frac{1}{|J|^{1/2}}
|\langle \phi_R, f \rangle| \chi_J(y) \Big)^2}{|I|} \chi_I(x)
\Bigg)^{1/2} \\
&= \bigg( \sum_I \Big( \sup_J \frac{1}{|J|^{1/2}} \big|\big\langle
\phi_J, \frac{\langle f, \phi_I \rangle}{|I|^{1/2}} \chi_I(x)
\big\rangle\big| \chi_J(y) \Big)^2 \bigg)^{1/2} \\
&= \bigg( \sum_I M'_2\Big( \frac{\langle f, \phi_I
\rangle}{|I|^{1/2}} \chi_I \Big)(x,y)^2 \bigg)^{1/2}.
\end{split}
\end{equation*}

\noindent By the Fefferman-Stein inequalities on $M'$ (or $M'_2$),

\begin{equation*}
\begin{split}
\|SMf\|_p &=  \bigg\| \bigg( \sum_I M_2'\Big( \frac{\langle f,
\phi_I \rangle}{|I|^{1/2}} \chi_I \Big)^2 \bigg)^{1/2} \bigg\|_p
\\
&\lesssim \bigg\| \bigg( \sum_I \frac{|\langle f, \phi_I
\rangle|^2}{|I|} \chi_I \bigg)^{1/2} \bigg\|_p = \|S_1 f\|_p
\lesssim \|f\|_p,
\end{split}
\end{equation*}

\noindent and

\begin{equation*}
\begin{split}
\|SMf\|_{1,\infty} &=  \bigg\| \bigg( \sum_I M_2'\Big( \frac{\langle
f, \phi_I \rangle}{|I|^{1/2}} \chi_I \Big)^2 \bigg)^{1/2}
\bigg\|_{1,\infty}
\\
&\lesssim \bigg\| \bigg( \sum_I \frac{|\langle f, \phi_I
\rangle|^2}{|I|} \chi_I \bigg)^{1/2} \bigg\|_1 = \|S_1 f\|_1
\lesssim \|f\|_{\llogl}.
\end{split}
\end{equation*}

On the other hand,

\begin{equation*}
\begin{split}
MSf(x,y) &= \sup_I \frac{1}{|I|^{1/2}} \bigg( \sum_J \frac{|\langle
\phi_R, f \rangle|^2}{|J|} \chi_J(y) \bigg)^{1/2} \chi_I(x) \\
&\leq \Bigg( \sum_J \frac{\Big( \sup_I \frac{1}{|I|^{1/2}} |\langle
\phi_R, f \rangle| \chi_I(x) \Big)^2}{|J|} \chi_J(y) \Bigg)^{1/2}.
\end{split}
\end{equation*}

\noindent This is essentially $SM$ with the roles of $I$ and $J$
reversed.  The same arguments as above can now be applied.

Finally,

\begin{equation*}
\begin{split}
SSf(x,y) &= \bigg( \sum_R \frac{|\langle \phi_R, f \rangle|^2}{|R|}
\chi_R(x,y) \bigg)^{1/2} \\
&= \bigg[ \sum_I \sum_J \frac{1}{|J|} \big|\big\langle \phi_J,
\frac{\langle f, \phi_I \rangle}{|I|^{1/2}} \chi_I(x)
\big\rangle\big|^2 \chi_J(y) \bigg]^{1/2} \\
&= \bigg[ \sum_I S_2\Big( \frac{\langle f, \phi_I
\rangle}{|I|^{1/2}} \chi_I\Big)(x,y)^2 \bigg]^{1/2},
\end{split}
\end{equation*}

\noindent so that by the Fefferman-Stein inequalities on $S_2$,

\begin{equation*}
\begin{split}
\|SSf\|_p &= \bigg\| \bigg( \sum_I S_2\Big( \frac{\langle f, \phi_I
\rangle}{|I|^{1/2}} \chi_I \Big)^2 \bigg)^{1/2} \bigg\|_p
\\
&\lesssim \bigg\| \bigg( \sum_I \frac{|\langle f, \phi_I
\rangle|^2}{|I|} \chi_I \bigg)^{1/2} \bigg\|_p = \|S_1 f\|_p
\lesssim \|f\|_p,
\end{split}
\end{equation*}

\noindent and

\begin{equation*}
\begin{split}
\|SSf\|_{1,\infty} &= \bigg\| \bigg( \sum_I S_2\Big( \frac{\langle
f, \phi_I \rangle}{|I|^{1/2}} \chi_I \Big)^2 \bigg)^{1/2}
\bigg\|_{1,\infty}
\\
&\lesssim \bigg\| \bigg( \sum_I \frac{|\langle f, \phi_I
\rangle|^2}{|I|} \chi_I \bigg)^{1/2} \bigg\|_1 = \|S_1 f\|_1
\lesssim \|f\|_{\llogl}.
\end{split}
\end{equation*}
\end{proof}

Let $R = I \times J$ be a dyadic rectangle.  For $\vec{n} \in
\mathbb{Z}^2$ and $\vec{\alpha} \in [0,1]^2$, let
$R^{\vec{n}}_{\vec{\alpha}} = I^{n_1}_{\alpha_1} \times
J^{n_2}_{\alpha_2}$ and $\varphi_{R^{\vec{n}}_{\vec{\alpha}}} =
\varphi_{I^{n_1}_{\alpha_1}} \oplus \varphi_{J^{n_2}_{\alpha_2}}$.
In this way, we can define shifted versions of each of $MM$, $SM$,
$MS$, and $SS$.  For example,

\begin{equation*}
SS^{\vec{n}}_{\vec{\alpha}}f(x,y) = \bigg( \sum_R \frac{|\langle
\phi_{R_{\vec{\alpha}}^{\vec{n}}}, f \rangle|^2}{|R|} \chi_R(x,y)
\bigg)^{1/2},
\end{equation*}

\noindent and $SS^{[\vec{n}]}f(x,y) = \sup_{\vec{\alpha}}
SS^{\vec{n}}_{\vec{\alpha}}f(x,y)$.  We first note that
$SS^{\vec{n}}$ satisfies all the above properties with an additional
factor of $(|n_1|+1)(|n_2|+1)$.  This follows easily by replacing in
the previous proof $S_1, S_2$ by $S_1^{n_1}, S_2^{n_2}$.  Then, as
before, we observe that $SS^{[\vec{n}]}f$ is bounded by an
$SS^{\vec{n}}f$, with a particular adapted tensor product which
depends on $f$.  So, $SS^{[\vec{n}]}$ satisfies the above with the
additional factor of $(|n_1|+1)(|n_2|+1)$.  The same holds for
$SM^{[\vec{n}]}$, $MS^{[\vec{n}]}$, and $MM^{[\vec{n}]}$.

Although we will not explicitly need the following result, it is
interesting enough to mention here.

\begin{thm} Each of $MM$, $MS$, $SM$, and
$SS$ maps $L(\log L)^{n+2} \rightarrow L(\log L)^n$.
\end{thm}

\begin{proof} This is simply a matter of repeating the arguments of
the previous proof and applying the interpolation results of
Corollaries~\ref{cor:f-sllogl} and~\ref{cor:f-sllogln}.  We have
immediately that $\|MMf\|_{L(\log L)^n} \lesssim \|M_1 \circ M_2
f\|_{L(\log L)^n} \lesssim \|M_2 f\|_{L(\log L)^{n+1}} \lesssim
\|f\|_{L(\log L)^{n+2}}$.  Further,

\begin{equation*}
\begin{split}
\|SMf\|_{\llogln} &=  \bigg\| \bigg( \sum_I M_2'\Big( \frac{\langle
f, \phi_I \rangle}{|I|^{1/2}} \chi_I \Big)^2
\bigg)^{1/2} \bigg\|_{\llogln} \\
&\lesssim \bigg\| \bigg( \sum_I \frac{|\langle f, \phi_I
\rangle|^2}{|I|} \chi_I \bigg)^{1/2} \bigg\|_{L(\log L)^{n+1}} \\
&= \|S_1 f\|_{L(\log L)^{n+1}} \lesssim \|f\|_{L(\log L)^{n+2}},
\end{split}
\end{equation*}

\noindent and

\begin{equation*}
\begin{split}
\|SSf\|_{\llogln} &=  \bigg\| \bigg( \sum_I S_2 \Big( \frac{\langle
f, \phi_I \rangle}{|I|^{1/2}} \chi_I \Big)^2
\bigg)^{1/2} \bigg\|_{\llogln} \\
&\lesssim \bigg\| \bigg( \sum_I \frac{|\langle f, \phi_I
\rangle|^2}{|I|} \chi_I \bigg)^{1/2} \bigg\|_{L(\log L)^{n+1}} \\
&= \|S_1 f\|_{L(\log L)^{n+1}} \lesssim \|f\|_{L(\log L)^{n+2}}.
\end{split}
\end{equation*}

\noindent Finally, $MS$ is pointwise smaller than an $SM$ type
operator, and therefore satisfies the same bounds.
\end{proof}

\section{Bi-parameter Paraproducts}

In Section~\ref{sec:para1}, we defined single-parameter
paraproducts. In order to study bi-parameter multiplier operators,
we will need to define and investigate the appropriate bi-parameter
paraproducts. For simplicity, as before, we will focus only on the
bilinear case.

For $f, g : \mathbb{T}^2 \rightarrow \mathbb{C}$, the bi-parameter
bilinear paraproducts are defined

\begin{equation*}
T_\epsilon^{a,b}(f, g)(x,y) = \sum_{R} \epsilon_R
\frac{1}{|R|^{1/2}} \langle \phi_R^1, f \rangle \langle \phi_R^2, g
\rangle \phi_R^3(x,y),
\end{equation*}

\noindent for $a, b = 1, 2, 3$, where $\varphi_R^1$, $\varphi_R^2$,
and $\varphi_R^3$ are each the tensor product of two adapted
families, as in the previous section. The sum is over all dyadic
rectangles $R$, and $(\epsilon_R)$ is a uniformly bounded sequence.
By dividing out a constant, we can assume $|\epsilon_R| \leq 1$.
Further, if $\phi^i_R = \phi^i_I \oplus \phi^i_J$, then
$\int_\mathbb{T} \phi^i_I \, dx = 0$ for $i \not= a$ and
$\int_\mathbb{T} \phi^i_J \, dx = 0$ for $i \not= b$.

\begin{thm}\label{thm:2para} $T_\epsilon^{a,b} : L^{p_1} \times L^{p_2}
\rightarrow L^p$ for $1 < p_1, p_2 < \infty$ and $\frac{1}{p} =
\frac{1}{p_1} + \frac{1}{p_2}$.  If $p_1$ or $p_2$ or both are equal
to 1, this still holds with $L^p$ replaced by $L^{p,\infty}$ and
$L^{p_j}$ replaced by $\llogl$. The underlying constants do not
depend on $a$, $b$, or the sequence $\epsilon_R$. \end{thm}

\begin{proof} We will assume $a = 1$ and $b = 2$, as the other cases
will follow similarly.

First, suppose $p > 1$.  Then, necessarily $p_1, p_2 > 1$ and $1 <
p' < \infty$.  Note, $1/p_1 + 1/p_2 + 1/p' = 1$.  Fix $h \in
L^{p'}(\mathbb{T})$ with $\|h\|_{p'} \leq 1$.  Then,

\begin{equation*}
\begin{split}
|\langle T_\epsilon^{1,2} (f,g), h \rangle| &= \Big| \sum_R
\epsilon_R \frac{1}{|R|^{1/2}} \langle \phi^1_R, f \rangle \langle
\phi^2_R, g \rangle \langle \phi^3_R, h \rangle \Big| \\
&\leq \sum_R \frac{1}{|R|^{1/2}} |\langle \phi^1_R, f \rangle|
|\langle \phi^2_R, g \rangle| |\langle \phi^3_R, h \rangle| \\
&= \int_{\mathbb{T}^2} \sum_R \frac{|\langle \phi_R^1, f
\rangle|}{|R|^{1/2}} \frac{|\langle \phi_R^2, g \rangle|}{|R|^{1/2}}
\frac{|\langle \phi_R^3, h \rangle|}{|R|^{1/2}} \chi_R(x,y) \, dx \,
dy.
\end{split}
\end{equation*}

\noindent Concentrating on the integrand,

\begin{gather*}
\sum_{R} \frac{|\langle \phi^1_{R}, f \rangle|}{|R|^{1/2}}
\frac{|\langle \phi^2_R, g \rangle|}{|R|^{1/2}} \frac{|\langle
\phi^3_{R}, h \rangle|}{|R|^{1/2}} \chi_R(x,y) = \\
\sum_I \sum_{J} \frac{|\langle \phi^1_R, f \rangle|}{|R|^{1/2}}
\frac{|\langle \phi^2_R, g \rangle|}{|R|^{1/2}} \frac{|\langle
\phi^3_{R}, h \rangle|}{|R|^{1/2}} \chi_R(x,y) \leq \\
\sum_I \Bigg[ \bigg( \frac{1}{|I|^{1/2}} \chi_I(x) \sup_{J}
\frac{|\langle \phi^2_R, g\rangle|}{|J|^{1/2}} \chi_J(y) \bigg)
\times \bigg( \sum_{J} \frac{|\langle \phi^1_{R}, f
\rangle|}{|R|^{1/2}} \frac{|\langle \phi^3_{R}, h
\rangle|}{|R|^{1/2}} \chi_R(x,y) \bigg) \Bigg].
\end{gather*}

\noindent Applying H\"{o}lder's inequality, the last term is bounded by

\begin{gather*}
SM(g)(x,y) \bigg(\sum_I \Big( \sum_{J} \frac{|\langle \phi^1_{R}, f
\rangle|}{|R|^{1/2}} \frac{|\langle \phi^3_{R}, h
\rangle|}{|R|^{1/2}} \chi_R(x,y) \Big)^2 \bigg)^{1/2}.
\end{gather*}

\noindent Applying H\"{o}lder to the inner sum,

\begin{gather*}
\bigg(\sum_I \Big( \sum_{J} \frac{|\langle \phi^1_{R}, f
\rangle|}{|R|^{1/2}} \frac{|\langle \phi^3_{R}, h
\rangle|}{|R|^{1/2}} \chi_R(x,y) \Big)^2
\bigg)^{1/2} \leq \\
\bigg(\sum_I \Big(\sum_{J} \frac{|\langle \phi^1_{R}, f
\rangle|^2}{|R|} \chi_R(x,y) \Big) \Big( \sum_{J} \frac{|\langle
\phi^3_{R}, h \rangle|^2}{|R|} \chi_R(x,y) \Big) \bigg)^{1/2}
\leq \\
\bigg(\sup_I \frac{1}{|I|} \chi_I(x) \sum_{J} \frac{|\langle
\phi^1_{R}, f \rangle|^2}{|J|} \chi_J(y) \bigg)^{1/2} \bigg(\sum_I
\sum_{J} \frac{|\langle \phi^3_R, h \rangle|^2}{|R|} \chi_R(x,y)
\bigg)^{1/2} = \\
MS(f)(x,y) SS(h)(x,y).
\end{gather*}

\noindent Hence,

\begin{equation*}
\begin{split}
|\langle T_\epsilon^{1,2} (f,g), h \rangle| &\leq
\int_{\mathbb{T}^2} MSf(x,y) SMg(x,y) SSh(x,y) \, dx \, dy \\
&\leq \|MSf\|_{p_1} \|SMg\|_{p_2} \|SSh\|_{p'} \lesssim \|f\|_{p_1}
\|g\|_{p_2}.
\end{split}
\end{equation*}

\noindent As $h$ in the unit ball of $L^{p'}$ is arbitrary, we have
$\|T^{1,2}_\epsilon(f,g)\|_p \lesssim \|f\|_{p_1} \|g\|_{p_2}$.

Now suppose $1/2 \leq p \leq 1$.  As in the proof of
Theorem~\ref{thm:1para}, by interpolation it suffices to show
$T_\epsilon^{1,2} : L^{p_1} \times L^{p_2} \rightarrow L^{p,\infty}$
for all $1 < p_1, p_2 < \infty$.  We concentrate on the special case
$T_\epsilon^{1,2} : \llogl \times \llogl \rightarrow L^{1/2,\infty}$, but all
others follow in the same way.

Let $\|f\|_{\llogl} = \|g\|_{\llogl} = 1$ and $E \subseteq
\mathbb{T}^2$ with $|E| > 0$. Lemma~\ref{lemma:dual} is valid on
$\mathbb{T}^d$ for any dimension $d$.  So, we will be done if we can
find $E' \subseteq E$, $|E'| > |E|/2$ so that $|\langle
T_\epsilon^{1,2} (f, g), \chi_{E'} \rangle| \lesssim 1 \leq
|E|^{-1}$.

For $\vec{k} \in \mathbb{N}^2$ and $R = I \times J$ a dyadic
interval, denote $2^{\vec{k}} R = 2^{k_1} I \times 2^{k_2}J$, and
$|\vec{k}| = k_1 + k_2$.  Use Theorem~\ref{thm:decomposition2} to
write

\begin{equation*}
\phi_R^3 = \sum_{\vec{k} \in \mathbb{N}^2} 2^{-10|\vec{k}|}
\phi_R^{3,\vec{k}}
\end{equation*}

\noindent where each $\phi_R^{3,\vec{k}}$ is the normalization of
the tensor product of two 0-mean adapted families which are
uniformally adapted to $I$, $J$ respectively. Further,
$\supp(\phi_R^{3,\vec{k}}) \subseteq 2^{\vec{k}}R$ for $\vec{k}$ small
enough, while $\phi_I^{3,\vec{k}}$ is identically 0 otherwise. Now

\begin{equation*}
\langle T_\epsilon^{1,2} (f, g), \chi_{E'} \rangle = \sum_{\vec{k}
\in \mathbb{N}^2} 2^{-10|\vec{k}|} \sum_{R} \epsilon_R
\frac{1}{|R|^{1/2}} \langle \phi_R^1, f \rangle \langle \phi_R^2, g
\rangle \langle \phi_R^{3,\vec{k}}, \chi_{E'} \rangle.
\end{equation*}

\noindent Hence, it suffices to show $|\sum \epsilon_R |R|^{-1/2}
\langle \phi_R^1, f \rangle \langle \phi_R^2, g \rangle \langle
\phi_R^{3,\vec{k}}, \chi_{E'} \rangle| \lesssim 2^{4|\vec{k}|}$, so
long as the underlying constants are independent of $\vec{k}$.

Let $SS^{\vec{k}}$ be the double square operator with
$\phi_R^{\vec{k}}$.  For each $\vec{k} \in \mathbb{N}^2$, define

\begin{gather*}
\Omega_{-3|\vec{k}|} = \{MSf > C2^{3|\vec{k}|} \} \cup \{SMg > C2^{3|\vec{k}|} \}, \\
\widetilde{\Omega}_{\vec{k}} = \{M_S(\chi_{\Omega_{-3|\vec{k}|}}) > 1/100\}, \\
\widetilde{\widetilde{\Omega}}_{\vec{k}} =
\{M_S(\chi_{\widetilde{\Omega}_{\vec{k}}})
> 2^{-|\vec{k}|-1} \}.
\end{gather*}

\noindent and

\begin{gather*}
\Omega = \bigcup_{\vec{k} \in \mathbb{N}^2}
\widetilde{\widetilde{\Omega}}_{\vec{k}}.
\end{gather*}

\noindent Observe,

\begin{equation*}
|\Omega| \leq \sum_{\vec{k} \in \mathbb{N}^2} 2^{-3|\vec{k}|}
2^{2|\vec{k}|+1} \frac{100^2}{C} \|M_S\|^4_{L^2 \rightarrow
L^2} \Big[\|MS\|_{\llogl \rightarrow L^{1,\infty}} +
\|SM\|_{\llogl \rightarrow L^{1,\infty}}\Big].
\end{equation*}

\noindent Therefore, we can choose $C$ independent of $f$ and $g$ so
that $|\Omega| < |E|/2$. Set $E' = E - \Omega = E \cap \Omega^c$.
Then, $E' \subseteq E$ and $|E'| > |E|/2$.

Fix $\vec{k} \in \mathbb{N}^2$.  Set $Z_{\vec{k}} = \{MSf = 0\} \cup
\{SMg = 0\} \cup \{SS^{\vec{k}}\chi_{E'} = 0\}$. Let $\mathcal{D}$
be any finite collection of dyadic rectangles. We divide this
collection into three subcollections.  Set $\mathcal{D}_1 = \{R \in
\mathcal{D} : R \cap Z_{\vec{k}} \not= \emptyset \}$. For the
remaining rectangles, let $\mathcal{D}_2 = \{R \in \mathcal{D} -
\mathcal{D}_1 : R\subseteq \widetilde{\Omega}_{\vec{k}}\}$ and
$\mathcal{D}_3 = \{R \in \mathcal{D} - \mathcal{D}_1 : R \cap
\widetilde{\Omega}_{\vec{k}}^c \not= \emptyset\}$.

If $R \in \mathcal{D}_1$, then there is some $(x,y) \in R \cap
Z_{\vec{k}}$. Namely, $MSf(x,y) = 0$, $SMg(x,y) = 0$, or
$SS^{\vec{k}}(\chi_{E'})(x,y) = 0$. If it is the first, $\langle
\phi^1_R, f \rangle = 0$. If it is the second, then $\langle
\phi_R^2, g \rangle = 0$, and if it is the third, $\langle
\phi^{3,\vec{k}}_R, \chi_{E'} \rangle = 0$. As this holds for all $R
\in \mathcal{D}_1$, we have

\begin{equation*}
\sum_{R \in \mathcal{D}_1} \frac{1}{|R|^{1/2}} |\langle \phi_R^1, f
\rangle| |\langle \phi_R^2, g \rangle| |\langle \phi_R^{3,\vec{k}},
\chi_{E'} \rangle| = 0.
\end{equation*}

Now suppose $R \in \mathcal{D}_2$, namely $R \subseteq
\widetilde{\Omega}_{\vec{k}}$.  For some $\vec{k}$,
$\phi_R^{3,\vec{k}}$ is identically 0 and $\langle
\phi_R^{3,\vec{k}}, \chi_{E'} \rangle = 0$. For all others,
$\phi_I^{3,\vec{k}}$ is supported in $2^{\vec{k}} R$. Let $(x,y) \in
2^{\vec{k}} R$, and observe

\begin{equation*}
M_S(\chi_{\widetilde{\Omega}_{\vec{k}}})(x,y) \geq
\frac{1}{|2^{\vec{k}}R|} \int_{2^{\vec{k}}R}
\chi_{\widetilde{\Omega}_{\vec{k}}} \,\, dm \geq
\frac{1}{2^{|\vec{k}|}} \frac{1}{|R|} \int_R
\chi_{\widetilde{\Omega}_{\vec{k}}} \,\, dm = 2^{-|\vec{k}|} >
2^{-|\vec{k}|-1}.
\end{equation*}

\noindent That is, $2^{\vec{k}}R \subseteq
\widetilde{\widetilde{\Omega}}_{\vec{k}} \subseteq \Omega$, a set
disjoint from $E'$.  Thus, $\langle \phi_R^{3,\vec{k}}, \chi_{E'}
\rangle = 0$. As this holds for all $R \in \mathcal{D}_2$, we have

\begin{equation*}
\sum_{R \in \mathcal{D}_2} \frac{1}{|R|^{1/2}} |\langle \phi_R^1, f
\rangle| |\langle \phi_R^2, g \rangle| |\langle \phi_R^{3,\vec{k}},
\chi_{E'} \rangle| = 0.
\end{equation*}

Finally, we concentrate on $\mathcal{D}_3$.  Define
$\Omega_{-3|\vec{k}|+1}$ and $\Pi_{-3|\vec{k}|+1}$ by

\begin{gather*}
\Omega_{-3|\vec{k}|+1} = \{MSf > C2^{3|\vec{k}|-1}\}, \\
\Pi_{-3|\vec{k}|+1} = \{I \in \mathcal{D}_3 : |I \cap
\Omega_{-3|\vec{k}|+1}|
> |R|/100\}.
\end{gather*}

\noindent Inductively, define for all $n > -3|\vec{k}|+1$,

\begin{gather*}
\Omega_{n} = \{MSf > C2^{-n}\}, \\
\Pi_{n} = \{R \in \mathcal{D}_3 - \bigcup_{j=-3|\vec{k}|+1}^{n-1}
\Pi_{j} : |R \cap \Omega_{n}|
> |R|/100\}.
\end{gather*}

\noindent As every $R \in \mathcal{D}_3$ is not in $\mathcal{D}_1$,
that is $MS f > 0$ on $R$, it is clear that each $R \in
\mathcal{D}_3$ will be in one of these collections.

Set $\Omega_{-3|\vec{k}|}' = \Omega_{-3|\vec{k}|}$ for symmetry.
Define $\Omega'_{-3|\vec{k}|+1}$ and $\Pi'_{-3|\vec{k}|+1}$ by

\begin{gather*}
\Omega'_{-3|\vec{k}|+1} = \{SM g > C2^{3|\vec{k}|-1}\}, \\
\Pi'_{-3|\vec{k}|+1} = \{R \in \mathcal{D}_3 : |R \cap
\Omega_{-3|\vec{k}|+1}'|
> |R|/100\}.
\end{gather*}

\noindent Inductively, define for all $n > -3|\vec{k}|+1$,

\begin{gather*}
\Omega'_{n} = \{SM g > C2^{-n}\}, \\
\Pi'_{n} = \{R \in \mathcal{D}_3 - \bigcup_{j=-3|\vec{k}|+1}^{n-1}
\Pi'_{j} : |R \cap \Omega'_{n}| > |R|/100\}.
\end{gather*}

\noindent As every $R \in \mathcal{D}_3$ is not in $\mathcal{D}_1$,
that is $SM g > 0$ on $R$, it is clear that each $R \in
\mathcal{D}_3$ will be in one of these collections.

Now, we can choose an integer $N$ big enough so that $\Omega_{-N}''
= \{SS^{\vec{k}} (\chi_{E'}) > 2^N\}$ has very small measure.  In
particular, we take $N$ big enough so that $|R \cap \Omega_{-N}''| <
|R|/100$ for all $R \in \mathcal{D}_3$, which is possible since
$\mathcal{D}_3$ is a finite collection. Define

\begin{gather*}
\Omega_{-N+1}'' = \{SS^{\vec{k}}(\chi_{E'}) > 2^{N-1}\}, \\
\Pi_{-N+1}'' = \{R \in \mathcal{D}_3 : |R \cap \Omega_{-N+1}''| >
|R|/100\},
\end{gather*}

\noindent and

\begin{gather*}
\Omega_{n}'' = \{SS^{\vec{k}}(\chi_{E'}) > 2^{-n}\}, \\
\Pi_{n}'' = \{R \in \mathcal{D}_3 - \bigcup_{j=-N+1}^{n-1} \Pi_j'' :
|R \cap \Omega_{n}''| > |R|/100\},
\end{gather*}

\noindent Again, all $R \in \mathcal{D}_3$ must be in one of these
collections.

Consider $R \in \mathcal{D}_3$, so that $R \cap
\widetilde{\Omega}_{\vec{k}}^c \not= \emptyset$. Then, there is some
$(x,y) \in R \cap \widetilde{\Omega}_{\vec{k}}^c$ which implies $|R
\cap \Omega_{-3|\vec{k}|}|/|R| \leq
M_S(\chi_{\Omega_{-3|\vec{k}|}})(x,y) \leq 1/100$. Write $\Pi_{n_1,
n_2, n_3} = \Pi_{n_1} \cap \Pi'_{n_2} \cap \Pi''_{n_3}$. So,

\begin{equation*}
\begin{split}
\sum_{R \in \mathcal{D}_3} \frac{1}{|R|^{1/2}} &|\langle \phi_R^1, f
\rangle| |\langle \phi_R^2, g \rangle | |\langle \phi_R^{3,\vec{k}}, \chi_{E'} \rangle| \\
&= \sum_{n_1, n_2 > -3|\vec{k}|, \, n_3 > -N} \bigg[ \sum_{R \in
\Pi_{n_1, n_2, n_3}} \frac{1}{|R|^{1/2}} |\langle \phi_R^1, f
\rangle| |\langle \phi_R^2, g \rangle|
|\langle \phi_R^{3,\vec{k}}, \chi_{E'}\rangle| \bigg]\\
&= \sum_{n_1, n_2 > -3|\vec{k}|, \, n_3 > -N} \bigg[ \sum_{R \in
\Pi_{n_1, n_2, n_3}} \frac{|\langle \phi_R^1, f \rangle|}{|R|^{1/2}}
\frac{|\langle \phi_R^2, g \rangle|}{|R|^{1/2}} \frac{|\langle
\phi_R^{3,\vec{k}}, \chi_{E'} \rangle|}{|R|^{1/2}} |R| \bigg].
\end{split}
\end{equation*}

\noindent Suppose $R \in \Pi_{n_1, n_2, n_3}$.  If $n_1 >
-3|\vec{k}| + 1$, then $R \in \Pi_{n_1}$, which in particular says
$R \notin \Pi_{n_1 - 1}$. So, $|R \cap \Omega_{n_1 - 1}| \leq
|R|/100$. If $n_1 = -3|\vec{k}| + 1$, then we still have $|R \cap
\Omega_{-3|\vec{k}|}| \leq |R|/100$, as $R \in \mathcal{D}_3$.
Similarly, If $n_2 > -3k + 1$, then $R \in \Pi_{n_2}'$, which in
particular says $R \notin \Pi_{n_2 - 1}'$. So, $|R \cap \Omega'_{n_2
- 1}| \leq |R|/100$.  If $n_2 = -3|\vec{k}| + 1$, then we still have
$|R \cap \Omega'_{-3|\vec{k}|}| = |R \cap \Omega_{-3|\vec{k}|}| \leq
|R|/100$, as $R \in \mathcal{D}_3$.  Finally, if $n_3 > -N+1$, then
$R \notin \Pi_{n_3 - 1}''$ and $|R \cap \Omega''_{n_3 - 1}| \leq
|R|/100$.  If $n_3 = -N+1$, then $|R \cap \Omega''_{-N}| \leq
|R|/100$ by the choice of $N$.  So, $|R \cap \Omega_{n_1 - 1}^c \cap
\Omega_{n_2 - 1}'^c \cap \Omega_{n_3 - 1}''^c| \geq \frac{97}{100}
|R|$.  Let $\Omega_{n_1, n_2, n_3} = \bigcup\{ R : R \in \Pi_{n_1,
n_2, n_3}\}$. Then,

\begin{equation*}
|R \cap \Omega_{n_1 - 1}^c \cap \Omega'^c_{n_2-1} \cap \Omega_{n_3 -
1}''^c \cap \Omega_{n_1, n_2, n_3}| \geq \frac{97}{100} |R|
\end{equation*}

\noindent for all $R \in \Pi_{n_1, n_2, n_3}$.  Further,

\begin{equation*}
\begin{split}
&\sum_{R \in \Pi_{n_1, n_2, n_3}} \frac{|\langle \phi_R^1, f
\rangle|}{|R|^{1/2}} \frac{|\langle \phi_R^2, g \rangle}{|R|^{1/2}}
\frac{|\langle \phi_R^{3,\vec{k}}, \chi_{E'} \rangle|}{|R|^{1/2}} |R| \\
&\lesssim \sum_{R \in \Pi_{n_1, n_2, n_3}} \frac{|\langle \phi_R^1,
f \rangle|}{|R|^{1/2}} \frac{|\langle \phi_R^2, g
\rangle|}{|R|^{1/2}} \frac{|\langle \phi_R^{3,\vec{k}}, \chi_{E'}
\rangle|}{|R|^{1/2}} |R \cap \Omega_{n_1 - 1}^c
\cap \Omega_{n_2 - 1}'^c \cap \Omega_{n_3 - 1}''^c \cap \Omega_{n_1, n_2, n_3}| \\
&= \int_{\Omega_{n_1 - 1}^c \cap \Omega_{n_2 - 1}'^c \cap
\Omega_{n_3 - 1}''^c \cap \Omega_{n_1, n_2, n_3}} \,\, \sum_{I \in
\Pi_{n_1, n_2, n_3}} \frac{|\langle \phi_R^1, f \rangle|}{|R|^{1/2}}
\frac{|\langle \phi_R^2, g \rangle|}{|R|^{1/2}} \frac{|\langle
\phi_R^{3,\vec{k}}, \chi_{E'} \rangle|}{|R|^{1/2}} \chi_R \,
dm \\
&\leq \int_{\Omega_{n_1 - 1}^c \cap \Omega_{n_2 - 1}'^c \cap
\Omega_{n_3 - 1}''^c \cap \Omega_{n_1, n_2, n_3}} MSf(x,y) SMg(x,y)
SS^{\vec{k}}(\chi_{E'})(x,y) \, dx \, dy \\
&\lesssim C^2 2^{-n_1} 2^{-n_2} 2^{-n_3} |\Omega_{n_1, n_2, n_3}|.
\end{split}
\end{equation*}

Note, $|\Omega_{n_1, n_2, n_3}| \leq |\bigcup \{R : R \in
\Pi_{n_1}\}| \leq |\{M_S(\chi_{\Omega_{n_1}}) > 1/100\}| \lesssim
|\Omega_{n_1}| = |\{MSf > C2^{-n_1}\}| \lesssim C^{-1} 2^{n_1}$.  By the
same argument, $|\Omega_{n_1, n_2, n_3}| \lesssim |\Omega_{n_2}'| =
|\{SM g > C 2^{-n_2}\}| \lesssim C^{-1} 2^{n_2}$, and $|\Omega_{n_1, n_2,
n_3}| \lesssim |\Omega_{n_3}''| = |\{SS^{\vec{k}}(\chi_{E'}) >
2^{-n_3}\}| \lesssim 2^{\alpha n_3}$ for any $\alpha \geq 1$. Thus,
$|\Omega_{n_1, n_2, n_3}| \lesssim C^{-2} 2^{\theta_1 n_1}
2^{\theta_2 n_2} 2^{\theta_3 \alpha n_3}$ for any $\theta_1 +
\theta_2 + \theta_3 = 1$, $0 \leq \theta_1, \theta_2, \theta_3 \leq
1$. Hence,

\begin{equation*}
\begin{split}
\sum_{R \in \mathcal{D}_3} \frac{1}{|R|^{1/2}} &|\langle \phi_R^1, f
\rangle| |\langle \phi_R^2, g \rangle| |\langle \phi_R^{3,\vec{k}},
\chi_{E'} \rangle| \\
&\lesssim \sum_{n_1, n_2 > -3|\vec{k}|, \,\, n_3
> 0} 2^{(\theta_1 - 1)n_1} 2^{(\theta_2 - 1)n_2} 2^{(\theta_3 \alpha - 1) n_3} \quad + \\
&\,\,\, \sum_{n_1, n_2 > -3|\vec{k}|, \,\, -N < n_3 \leq 0}
2^{(\theta_1 - 1)n_1} 2^{(\theta_2 - 1)n_2} 2^{(\theta_3 \alpha - 1) n_3} \\
&= A + B.
\end{split}
\end{equation*}

\noindent For the first term, take $\theta_1 = 1/2$, $\theta_2 =
1/2$, $\theta_3 = 0$, and $\alpha = 1$. For the second term, take
$\theta_1 = 1/3$, $\theta_2 = 1/3$, $\theta_3 = 1/3$, and $\alpha =
6$ to see

\begin{equation*}
\begin{split}
A &= \sum_{n_1, n_2 > -3|\vec{k}|, \,\, n_3 > 0} 2^{-n_1/2} 2^{-n_2/2}
2^{-n_3} \lesssim 2^{3|\vec{k}|}, \\
B &= \sum_{n_1, n_2 > -3|\vec{k}|, \,\, -N < n_3 \leq 0} 2^{-2n_1/3}
2^{-2n_2/3} 2^{n_3} \\
&\leq \sum_{n_1, n_2 > -3|\vec{k}|, n_3 \leq 0} 2^{-2n_1/3}
2^{-2n_2/3} 2^{n_3} \lesssim 2^{4|\vec{k}|}.
\end{split}
\end{equation*}

\noindent Note, there is no dependence on the number $N$, which
depends on $\mathcal{D}$, or $C$, which depends on $E$.

Combining the estimates for $\mathcal{D}_1$, $\mathcal{D}_2$, and
$\mathcal{D}_3$, we see

\begin{equation*}
\sum_{R \in \mathcal{D}} \frac{1}{|R|^{1/2}} |\langle \phi_R^1, f
\rangle| |\langle \phi_R^2, g \rangle| |\langle \phi_R^{3,\vec{k}},
\chi_{E'} \rangle| \lesssim 2^{4|\vec{k}|},
\end{equation*}

\noindent where the constant has no dependence on the collection
$\mathcal{D}$.  Hence, as $\mathcal{D}$ is arbitrary, we have

\begin{equation*}
\begin{split}
\Big| \sum_R \epsilon_R \frac{1}{|R|^{1/2}} \langle \phi_R^1, f
\rangle &\langle \phi_R^2, g \rangle \langle \phi_R^{3,\vec{k}},
\chi_{E'} \rangle \Big| \\
&\leq \sum_R \frac{1}{|R|^{1/2}} |\langle \phi_R^1, f \rangle|
|\langle \phi_R^2, g \rangle| |\langle \phi_R^{3,\vec{k}}, \chi_{E'}
\rangle| \lesssim 2^{4|\vec{k}|},
\end{split}
\end{equation*}

\noindent which completes the proof.
\end{proof}

It should now be clear that proving the above for $(a, b) \not= (1,
2)$ follows by permuting the roles of $MM$, $MS$, $SM$, and $SS$.
For instance, if $(a,b) = (1,1)$, then we consider $MMf$, $SSg$, and
$SS^{\vec{k}}\chi_{E'}$.

For any $\vec{n} \in \mathbb{Z}^4$, where $\vec{n}_1 = (n_1, n_2)$
and $\vec{n}_2 = (n_3, n_4)$, we can define the shifted paraproducts
by

\begin{equation*}
T_\epsilon^{a, b, [\vec{n}]}(f, g)(\vec{x}) = \int_{[0,1]^2}
\sum_{R} \epsilon_R \frac{1}{|R|^{1/2}} \langle
\phi_{R_{\vec{\alpha}}^{\vec{n}_1}}^1, f \rangle \langle
\phi_{R_{\vec{\alpha}}^{\vec{n}_2}}^2, g \rangle
\phi_{R_{\vec{\alpha}}}^3(\vec{x}) \, d\vec{\alpha},
\end{equation*}

\noindent Like the previous cases, simply rework the proof. For
instance, if $(a,b) = (1,2)$, replace $MSf$ by $MS^{[\vec{n}_1]}f$,
$SMg$ by $SM^{[\vec{n}_2]}g$, and $SS^{\vec{k}}(\chi_{E'})$ by
$SS^{\vec{k}, [0]}(\chi_{E'})$. This leads to the previous estimates
with an additional factor of $\prod_{j=1}^4 (|n_j|+1)$.

\section{Multiplier Operators}

We now wish to extend Coifman-Meyer operators to a broader
bi-parameter setting.  In particular, we investigate a new, wider
class of multipliers $m$, which act as if they are the product of
two Coifman-Meyer multipliers.

Given a vector $\vec{t} = (t_1, \ldots, t_{2d}) \in
\mathbb{R}^{2d}$, denote $\rho_1(\vec{t}) = (t_1, t_3, \ldots,
t_{2d-1})$ and $\rho_2(\vec{t}) = (t_2, t_4, \ldots, t_{2d})$, which
are both vectors in $\mathbb{R}^d$.  For multi-indices of
nonnegative integers $\alpha$, we can also employ this notation.  In
particular, $|\rho_1(\alpha)| = \alpha_1 + \alpha_3 + \ldots +
\alpha_{2d-1}$, and similarly for $\rho_2(\alpha)$.  Conversely, for
$1 \leq j \leq d$, let $\vec{t}_j = (t_{2j-1}, t_{2j}) \in
\mathbb{R}^2$, so that $\vec{t} = (\vec{t}_1, \ldots, \vec{t}_d)$.

\begin{defnn} Let $m : \mathbb{R}^{2d} \rightarrow \mathbb{C}$ be smooth
away the origin and uniformly bounded. We say $m$ is a bi-parameter
multiplier if $|\partial^\alpha m(\vec{t})| \lesssim
\|\rho_1(\vec{t})\|^{-|\rho_1(\alpha)|}
\|\rho_2(\vec{t})\|^{-|\rho_2(\alpha)|}$ for all vectors $\alpha$
with $|\alpha| \leq 2d(d+3)$, where $\|\cdot\|$ is the Euclidean
norm on $\mathbb{R}^d$.
\end{defnn}

Given such a multiplier $m$ on $\mathbb{R}^{2d}$ and $L^1$ functions $f_1,
\ldots, f_d : \mathbb{T}^2 \rightarrow \mathbb{C}$, we define the
associated multiplier operator $\Lambda_m^{(2)}(f_1, \ldots, f_d) :
\mathbb{T}^2 \rightarrow \mathbb{C}$ as

\begin{equation*}
\Lambda_m^{(2)}(f_1, \ldots, f_d)(\vec{x}) = \sum_{\vec{t} \in
\mathbb{Z}^{2d}} m(\vec{t}) \widehat{f}_1(\vec{t}_1) \cdots
\widehat{f}_d(\vec{t}_d) e^{2\pi i \vec{x} \cdot (\vec{t}_1 + \ldots
+ \vec{t}_d)}.
\end{equation*}

\noindent Consider the following theorem.

\begin{thmn} For any bi-parameter multiplier $m$ on $\mathbb{R}^{2d}$,
$\Lambda^{(2)}_m : L^{p_1} \times \ldots \times L^{p_d} \rightarrow
L^p$ for $1 < p_j < \infty$ and $\frac{1}{p_1} + \ldots +
\frac{1}{p_d} = \frac{1}{p}$. If any or all of the $p_j$ are equal
to 1, this still holds with $L^p$ replaced by $L^{p,\infty}$ and
$L^{p_j}$ replaced by $\llogl$.  In particular, $\Lambda_m^{(2)} :
\llogl \times \ldots \times \llogl \rightarrow L^{1/d, \infty}$.
\end{thmn}

As before, we will focus on the $d = 2$ case for simplicity, but
this makes no substantiative difference in the proof.  We note that
in this case, the bi-parameter multiplier condition can be stated

\begin{equation*}
|\partial^{(\alpha, \beta)} m(\vec{s}, \vec{t})| \lesssim \|(s_1,
t_1)\|^{-\alpha_1-\beta_1} \|(s_2, t_2)\|^{-\alpha_2-\beta_2}
\end{equation*}

\noindent for all two-dimensional indices $\alpha, \beta$ with
$|\alpha|, |\beta| \leq 10$.

\begin{remark}\label{remark:tripletbumps} Let $\psi_k^{a,i}$ be the
functions in Theorem~\ref{thm:doubletbumps}.  For $1 \leq a, b \leq
3$ and $k, k' \in \mathbb{N}$, define $\psi^{a,b,i}_{k,k'}(\vec{s})
= \psi^{a,i}_k(s_1) \psi^{b,i}_{k'}(s_2)$.  Let $E_j = \{ \vec{x}
\in \mathbb{Z}^4 : \rho_j(\vec{x}) \not= (0,0)\}$ and $E = E_1 \cap
E_2$. Then,

\begin{equation*}
\begin{split}
\chi_E(\vec{s}, \vec{t}) &= \chi_{\mathbb{N}^2 - (0,0)}(s_1, t_1)
\chi_{\mathbb{N}^2 - (0,0)}(s_2, t_2) \\
&= \bigg( \sum_{a=1}^3 \sum_{k=1}^\infty \widehat{\psi^{a,1}_k}(s_1)
\widehat{\psi^{a,2}_k}(t_1) \widehat{\psi^{a,3}_k}(-s_1-t_1) \bigg)
\times \\
&\qquad \bigg( \sum_{b=1}^3 \sum_{k'=1}^\infty
\widehat{\psi^{b,1}_{k'}}(s_2) \widehat{\psi^{b,2}_{k'}}(t_2)
\widehat{\psi^{b,3}_{k'}}(-s_2-t_2) \bigg) \\
&= \sum_{a,b = 1}^3 \sum_{k, k' = 1}^\infty
\widehat{\psi^{a,b,1}_{k,k'}}(\vec{s})
\widehat{\psi^{a,b,2}_{k,k'}}(\vec{t})
\widehat{\psi^{a,b,3}_{k,k'}}(-\vec{s}-\vec{t}).
\end{split}
\end{equation*}
\end{remark}

\begin{lemma}\label{lemma:m3} Let $m : \mathbb{R}^4 \rightarrow
\mathbb{C}$ be a bi-parameter multiplier and $\psi_{k,k'}^{a,b,1},
\psi_{k,k'}^{a,b,2}$ the functions in
Remark~\ref{remark:tripletbumps}. For every $k,k' \in \mathbb{N}$
and $1 \leq a, b \leq 3$, there is a smooth function $m_{a,b,k,k'}$
satisfying $m_{a,b,k,k'}(\vec{s},\vec{t})
\widehat{\psi^{a,b,1}_{k,k'}}(\vec{s})
\widehat{\psi^{a,b,2}_{k,k'}}(\vec{t}) = m(\vec{s},\vec{t})
\widehat{\psi^{a,b,1}_{k,k'}}(\vec{s})
\widehat{\psi^{a,b,2}_{k,k'}}(\vec{t})$ and

\begin{equation*}
m_{a,b,k,k'}(\vec{s},\vec{t}) = \sum_{\vec{n} \in \mathbb{Z}^4}
c_{a,b,k,k',\vec{n}} e^{-2\pi i 2^{-k} \rho_1(\vec{n}) \cdot
(s_1,t_1)} e^{-2\pi i2^{-k'} \rho_2(\vec{n}) \cdot (s_2, t_2)},
\end{equation*}

\noindent where $|c_{a,b,k,k',\vec{n}}| \lesssim \prod_{j=1}^4
(|n_j| + 1)^{-5}$ uniformly in $a, b, k, k'$.
\end{lemma}

\begin{proof} For simplicity, assume $a = b = 1$.  Let $\varphi :
\mathbb{R}^4 \rightarrow \mathbb{C}$ be a smooth function with

\begin{gather*}
\supp(\varphi) \subseteq [-2^{-1}, 2^{-1}]^2 \times \Big([-2^{-1},
-2^{-11}] \cup [2^{-11}, 2^{-1}] \Big)^2 \quad \text{ and } \\
\varphi = 1 \, \text{ on } \, [-2^{-2}, 2^{-2}]^2 \times \Big(
[-2^{-2}, -2^{-10}] \cup [2^{-10}, 2^{-2}] \Big)^2.
\end{gather*}

\noindent Define $m_{a,b,k,k'}(\vec{s},\vec{t}) = m(\vec{s},\vec{t})
\varphi(2^{-k} s_1, 2^{-k'} s_2, 2^{-k}t_1, 2^{-k'}t_2)$. Then by
construction, $m_{a,b,k,k'}(\vec{s},\vec{t})
\widehat{\psi_{k,k'}^{a,b,1}}(\vec{s})
\widehat{\psi_{k,k'}^{a,b,2}}(\vec{t}) = m(\vec{s},\vec{t})
\widehat{\psi_{k,k'}^{a,b,1}}(\vec{s})
\widehat{\psi_{k,k'}^{a,b,2}}(\vec{t})$. Further, if $E_{a,b,k,k'}$
is the support of $m_{a,b,k,k'}$, then $|E_{a,b,k,k'}| \leq 2^{2k}
2^{2k'}$.

Recall that $\{2^{-k/2} e^{-2\pi in2^{-k}x}\}$ is an orthonormal
basis on any interval of length $2^k$, so that $\{2^{-k} e^{-2\pi
i2^{-k} \vec{n} \cdot \vec{x}}\}$ is an orthonormal basis on any
square of side length $2^{k}$.  Thus,

\begin{equation*}
\begin{split}
m_{a,b,k,k'}(\vec{s},\vec{t}) = \sum_{\vec{n} \in \mathbb{Z}^4}
c_{a,b,k,k',\vec{n}} \, e^{-2\pi i 2^{-k} \rho_1(\vec{n}) \cdot
(s_1,t_1)} \, e^{-2\pi i 2^{-k'} \rho_2(\vec{n}) \cdot (s_2, t_2)},
\end{split}
\end{equation*}

\noindent where $c_{a,b,k,k',\vec{n}}$ is

\begin{equation*}
2^{-2k} 2^{-2k'} \bigg( \int_{\mathbb{R}^4}
m_{a,b,k,k'}(\vec{x},\vec{y}) \, e^{2\pi i 2^{-k} \rho_1(\vec{n})
\cdot (x_1,y_1)} \, e^{2\pi i 2^{-k'} \rho_2(\vec{n}) \cdot (x_2,
y_2)} \, d\vec{x} \, d\vec{y} \bigg).
\end{equation*}

We may assume that if $\vec{n} = (n_1, n_2, n_3, n_4)$, each of
$n_j$ is nonzero, as these cases are handled similarly.  Let $C =
\max\{ \|\partial^\alpha \varphi\|_\infty : 0 \leq |\alpha| \leq 20\}$.
Note, if $(\vec{x}, \vec{y}) \in E_{a,b,k,k'}$, then $\|(x_1, y_1)\|
\geq 2^{k-11}$ and $\|(x_2, y_2)\| \geq 2^{k'-11}$. So,
$|\partial^{(\alpha, \beta)} m(\vec{x},\vec{y})| \lesssim
\|(x_1,y_1)\|^{-\alpha_1-\beta_1} \|(x_2,y_2)\|^{-\alpha_2-\beta_2}
\lesssim  2^{-k(\alpha_1 + \beta_1)} 2^{-k'(\alpha_2 + \beta_2)}$
for all $|\alpha|, |\beta| \leq 10$.  Set $\gamma = (5,5,5,5)$, and
observe

\begin{equation*}
\begin{split}
|\partial^\gamma m_{a,b,k,k'}&(\vec{x},\vec{y})| \\
&\lesssim \sum_{(\alpha,
\beta) \leq \gamma} |\partial^{(\alpha, \beta)} m_{a,b,k,k'}(\vec{x},
\vec{y})| |\partial^{\gamma - (\alpha, \beta)} \varphi(2^{-k}x_1,
2^{-k'}x_2, 2^{-k}y_1, 2^{-k'}y_2)| \\
&\leq \sum_{(\alpha, \beta) \leq \gamma} C |\partial^{\alpha, \beta}
m(\vec{x}, \vec{y})|
2^{-k(10 - \alpha_1 - \beta_1)} 2^{-k'(10 -\alpha_2-\beta_2)} \\
&\lesssim  2^{-10k} 2^{-10k'}.
\end{split}
\end{equation*}

\noindent By several iterations of integration by parts,

\begin{equation*}
\begin{split}
&\bigg| \int_{\mathbb{R}^4} m_{a,b,k,k'}(\vec{x},\vec{y}) \, e^{2\pi
i 2^{-k} \rho_1(\vec{n}) \cdot (x_1,y_1)} \, e^{2\pi i 2^{-k'}
\rho_2(\vec{n}) \cdot (x_2, y_2)} \, d\vec{x} \, d\vec{y} \bigg| \\
&= \bigg| \int_{E_{a,b,k,k'}} \partial^{\gamma}
m_{a,b,k,k'}(\vec{x},\vec{y}) \, \frac{e^{2\pi i 2^{-k}
\rho_1(\vec{n}) \cdot (x_1,y_1)} \, e^{2\pi i 2^{-k'}
\rho_2(\vec{n}) \cdot (x_2, y_2)}
 (2\pi i)^{-20}}{(n_1 2^{-k})^{5} (n_2 2^{-k'})^{5} (n_3 2^{-k})^{5} (n_4
2^{-k'})^{5}}  \, d\vec{x} \, d\vec{y} \bigg| \\
&\lesssim \frac{2^{10k} 2^{10k'}}{|n_1|^5 |n_2|^5 |n_3|^5 |n_4|^5}
|E_{a,b,k,k'}| \|\partial^\gamma m_{a,b,k,k'}\|_\infty \lesssim 2^{2k}
2^{2k'} \prod_{j=1}^4 (|n_j| + 1)^{-5}.
\end{split}
\end{equation*}

\noindent Namely, $|c_{a,b,k,k',\vec{n}}| \lesssim \prod
(|n_j|+1)^{-5}$.  To handle the cases when $n_j = 0$ for some $j$,
adjust the above argument with $\gamma = (0,5,5,5)$ or $\gamma =
(5,0,5,5)$, and so on.  For $(a,b) \not= (1,1)$, we simply need to
choose $\varphi$ differently.
\end{proof}

\begin{thm}\label{thm:bimult} For any bi-parameter multiplier $m$ on
$\mathbb{R}^4$, $\Lambda_m^{(2)} : L^{p_1} \times L^{p_2}
\rightarrow L^p$ for $1 < p_1, p_2 < \infty$ and $\frac{1}{p_1} +
\frac{1}{p_2} = \frac{1}{p}$.  If $p_1$ or $p_2$ or both are equal
to 1, this still holds with $L^p$ replaced by $L^{p,\infty}$ and
$L^{p_j}$ replaced by $\llogl$.
\end{thm}

\begin{proof} Fix $m$ and let $f, g : \mathbb{T}^2
\rightarrow \mathbb{C}$.  Then,

\begin{equation*}
\Lambda_m^{(2)}(f,g)(\vec{x}) = \sum_{\vec{s}, \vec{t} \in
\mathbb{Z}^2} m(\vec{s},\vec{t}) \widehat{f}(\vec{s})
\widehat{g}(\vec{t}) e^{2\pi i\vec{x} \cdot (\vec{s}+\vec{t})}.
\end{equation*}

\noindent As in the proofs of Theorem~\ref{thm:marcin} and
Theorem~\ref{thm:c-m}, we can assume $m(\vec{0}, \vec{0}) = 0$.

Let $m_1(s_1, t_1) = m(s_1,0,t_1,0)$.  Then, $m_1 : \mathbb{R}^2
\rightarrow \mathbb{C}$ is a Coifman-Meyer multiplier.  Let
$F_1(x_1) = \int_\mathbb{T} f(x_1, x_2) \, dx_2$ and $G_1(x_1) =
\int_\mathbb{T} g(x_1, x_2) \, dx_2$, so that $\widehat{f}(s_1,0) =
\widehat{F}_1(s_1)$ and $\widehat{g}(t_1, 0) = \widehat{G}_1(t_1)$.
Let

\begin{equation*}
\Lambda_{m_1}(F_1, G_1)(x) = \sum_{s_1, t_1 \in \mathbb{Z}} m_1(s_1,
t_1) \widehat{F}_1(s_1) \widehat{G_1}(t_1) e^{2\pi ix(s_1 + t_1)},
\end{equation*}

\noindent a standard Coifman-Meyer operator.  Now define
$m_2(s_2, t_2) = m(0,s_2,0,t_2)$ and $F_2, G_2$ as expected.  Let
$\Lambda_{m_2}(F_2, G_2)(y)$ be the appropriate Coifman-Meyer
operator.  Finally, let $m_0$ be a bi-parameter multiplier which
agrees with $m$ on integers away from the planes $\{(s_1,t_1) = 0\}$
and $\{(s_2,t_2) = 0\}$, but is 0 on these planes.  Then,

\begin{equation*}
\Lambda_m^{(2)}(f,g)(x,y) = \Lambda_{m_0}^{(2)}(f,g)(x,y) +
\Lambda_{m_1}(F_1,G_1)(x) + \Lambda_{m_2}(F_2,G_2)(y).
\end{equation*}

By Theorem~\ref{thm:c-m}, if $p_1, p_2 > 1$, then
$\|\Lambda_{m_1}(F_1, G_1)\|_{L^p(\mathbb{T})} \lesssim
\|F_1\|_{L^{p_1}(\mathbb{T})} \|G_1\|_{L^{p_2}(\mathbb{T})}$.  By
generalized Minkowski, $\|F_1\|_{L^{p_1}(\mathbb{T})} \leq
\|f\|_{p_1}$ and $\|G_1\|_{L^{p_2}(\mathbb{T})} \leq \|g\|_{p_2}$.
Therefore, $\|\Lambda_{m_1}(F_1, G_1)\|_{L^p(\mathbb{T})} \lesssim
\|f\|_{p_1} \|g\|_{p_2}$. If $p_1 = 1$, then $\|F_1\|_1 \leq \|f\|_1
\leq \|f\|_{\llogl}$. Similarly for $p_2 = 1$.  Thus, the term
$\Lambda_{m_1}(F_1, G_1)$, and by symmetry $\Lambda_{m_2}(F_2,
G_2)$, satisfies all the estimates we want.  Hence, it suffices to
consider the operator $\Lambda_{m_0}$.  Equivalently, we can assume
$m$ is 0 on the planes $\{(s_1, t_1) = 0\}$ and $\{(s_2, t_2) =
0\}$.

Let $h \in L^1(\mathbb{T}^2)$.  Let $f_0 = \widetilde{\overline{f}}$
and similarly for $g_0, h_0$.  Then,

\begin{equation*}
\begin{split}
\langle \Lambda_m^{(2)}(f,g), \widetilde{h} \rangle &=
\int_{\mathbb{T}^2} \Lambda^{(2)}_m(f,g)(x) h_0(x) \, dx \\
&= \int_{\mathbb{T}^2} \bigg( \sum_{\vec{s}, \vec{t} \in
\mathbb{Z}^2}^\infty m(\vec{s},\vec{t}) \widehat{f}(\vec{s})
\widehat{g}(\vec{t}) e^{2\pi i\vec{x}(\vec{s}+\vec{t})}
\bigg) h_0(\vec{x}) \, d\vec{x} \\
&= \sum_{\vec{s}, \vec{t} \in \mathbb{Z}^2} m(\vec{s},\vec{t})
\widehat{f}(\vec{s}) \widehat{g}(\vec{t})
\int_{\mathbb{T}^2} h_0(\vec{x}) e^{2\pi i\vec{x}(\vec{s}+\vec{t})}
\, d\vec{x} \\
&= \sum_{\vec{s}, \vec{t} \in \mathbb{Z}^2} m(\vec{s},\vec{t})
\widehat{f}(\vec{s}) \widehat{g}(\vec{t})
\widehat{h_0}(-\vec{s}-\vec{t}).
\end{split}
\end{equation*}

\noindent Now employ Remark~\ref{remark:tripletbumps} to write

\begin{equation*}
\begin{split}
\langle &\Lambda_m^{(2)}(f,g), \widetilde{h} \rangle \\
&= \sum_{a,b=1}^3 \sum_{k,k'=1}^\infty \sum_{\vec{s}, \vec{t} \in
\mathbb{Z}^2} m(\vec{s},\vec{t}) \widehat{f}(\vec{s})
\widehat{\psi_{k,k'}^{a,b,1}}(\vec{s}) \widehat{g}(\vec{t})
\widehat{\psi_{k,k'}^{a,b,2}}(\vec{t})
\widehat{h_0}(-\vec{s}-\vec{t}) \widehat{\psi_{k,k'}^{a,b,3}}(-\vec{s}-\vec{t}) \\
&= \sum_{a,b=1}^3 \sum_{k,k'=1}^\infty \sum_{\vec{s}, \vec{t} \in
\mathbb{Z}^2} m_{a,b,k,k'}(\vec{s},\vec{t}) \widehat{f}(\vec{s})
\widehat{\psi_{k, k'}^{a,b,1}}(\vec{s}) \widehat{g}(\vec{t})
\widehat{\psi_{k, k'}^{a,b,2}}(\vec{t})
\widehat{h_0}(-\vec{s}-\vec{t})
\widehat{\psi_{k, k'}^{a,b,3}}(-\vec{s}-\vec{t}) \\
&=: \sum_{a,b=1}^3 S_{a,b},
\end{split}
\end{equation*}

\noindent where $m_{a,b,k,k'}$ is as given in Lemma~\ref{lemma:m3}.
Let $\psi^{a,b,1}_{k,k',\vec{n}_1}(\vec{x}) =
\psi^{a,b,1}_{k,k'}(\vec{x} - (2^{-k} n_1, 2^{-k'} n_2))$ and
$\psi_{k,k',\vec{n}_2}^{a,b,2}(\vec{x}) =
\psi^{a,b,2}_{k,k'}(\vec{x} - (2^{-k} n_3, 2^{-k'} n_4)))$. Then,

\begin{equation*}
\begin{split}
S_{a,b} &= \sum_{\vec{n} \in \mathbb{Z}^4} \sum_{k,k'=1}^\infty
\sum_{\vec{s},\vec{t} \in \mathbb{Z}^2} c_{a,b,k,k',\vec{n}} (f *
\psi_{k,k',\vec{n}_1}^{a,b,1}) \, \widehat{} \, (\vec{s}) (g *
\psi_{k,k',\vec{n}_2}^{a,b,2}) \, \widehat{} \, (\vec{t}) (h_0 *
\psi_{k,k'}^{a,b,3}) \, \widehat{}
\, (-\vec{s}-\vec{t}) \\
&= \sum_{\vec{n} \in \mathbb{Z}^4} \sum_{k, k' =1}^\infty
c_{a,b,k,k',\vec{n}} \int_{\mathbb{T}^2} (f *
\psi_{k,k',\vec{n}_1}^{a,b,1})(\vec{x}) (g *
\psi_{k,k,\vec{n}_2}^{a,b,2})(\vec{x}) (h_0 *
\psi_{k,k'}^{a,b,3})(\vec{x}) \, d\vec{x}.
\end{split}
\end{equation*}

\noindent The last line is gained from showing Claim~\ref{claim:fgh}
is valid on $\mathbb{T}^d$ for any $d$.  Just as in the previous
proofs, we can dilate and translate to write

\begin{equation*}
\begin{split}
\int_{\mathbb{T}^2} &(f * \psi_{k,k',\vec{n}_1}^{a,b,1})(\vec{x}) (g
* \psi_{k,k',\vec{n}_2}^{a,b,2})(\vec{x}) (h_0 *
\psi_{k,k'}^{a,b,3})(\vec{x}) \, d\vec{x} \\
&= 2^{-k} 2^{-k'} \sum_{j=0}^{2^k-1} \sum_{j'=0}^{2^{k'}-1}
\int_{[0,1]^2} \langle \psi_{k, k', j, j',
\vec{n}_1,\vec{\alpha}}^{a,b,1}, \overline{f} \rangle \langle
\psi_{k,k',j,j',\vec{n}_2,\vec{\alpha}}^{a,b,2}, \overline{g}
\rangle \langle \psi_{k,k',j,j',\vec{\alpha}}^{a,b,3},
\overline{h_0} \rangle\, d\vec{\alpha},
\end{split}
\end{equation*}

\noindent where
$\psi^{a,b,1}_{k,k',j,j',\vec{n}_1,\vec{\alpha}}(\vec{x}) =
\psi^{a,b,1}_{k,k'}(2^{-k}(\alpha_1+j+n_1) - x_1,
2^{-k'}(\alpha_2+j'+n_2) - x_2)$, and similarly for the other two
functions.

For a dyadic rectangle $R = [2^{-k} j, 2^{-k} (j+1)] \times [2^{-k'}
j', 2^{-k'}(j'+1)]$, let
$\varphi_{R_{\vec{\alpha}}^{\vec{n}_1}}^{a,b,1}$ be the reflection
of $2^{-k} 2^{-k'} \psi^{a,b,1}_{k,k',j,j',\vec{n}_1,\vec{\alpha}}$,
and similarly for $\varphi_{R_{\vec{\alpha}}^{\vec{n}_2}}^{a,b,2}$
and $\varphi_{R_{\vec{\alpha}}}^{a,b,3}$.  It is easily checked that
the construction of $\psi_{k,k'}^{a,b,i}$ guarantees that
$\varphi_R^{a,b,i}$ are the tensor products of adapted families with
$\int_\mathbb{T} \varphi_I^{a,b,i} \, dx = 0$ when $a \not= i$ and
$\int_\mathbb{T} \varphi_J^{a,b,i} \, dx = 0$ when $b \not= i$. Let
$\phi_R^{a,b,i} = |R|^{-1/2} \varphi_R^{a,b,i}$, so that

\begin{equation*}
\begin{split}
S_{a,b} = \sum_{\vec{n} \in \mathbb{Z}^4} \int_{[0,1]^2} \sum_{R}
c_{a,b,R,\vec{n}} \frac{1}{|R|^{1/2}} \langle
\phi_{R^{\vec{n}_1}_{\vec{\alpha}}}^{a,b,1}, f_0 \rangle \langle
\phi_{R^{\vec{n}_2}_{\vec{\alpha}}}^{a,b,2}, g_0 \rangle \langle
\phi_{R_{\vec{\alpha}}}^{a,b,3}, h \rangle \, d\alpha,
\end{split}
\end{equation*}

\noindent where the inner sum is over all dyadic rectangles and
$c_{a,b,R,\vec{n}} = c_{a,b,k,k',\vec{n}}$ when $R = I \times J$ with
$|I| = 2^{-k}, |J| = 2^{-k'}$. Write $c'_{a,b,R,\vec{n}} =
\prod_{j=1}^4 (|n_j|+1)^{5} c_{a,b,R,\vec{n}}$, which are uniformly
bounded in $R$ and $\vec{n}$ by Lemma~\ref{lemma:m3}. Hence,

\begin{equation*}
\begin{split}
S_{a,b} &= \sum_{\vec{n} \in \mathbb{Z}^4} \prod_{j=1}^4
\frac{1}{(|n_j|+1)^5} \int_{[0,1]} \sum_{R} c'_{a,b,R,\vec{n}}
\frac{1}{|R|^{1/2}} \langle
\phi_{R^{\vec{n_1}}_{\vec{\alpha}}}^{a,b,1} f_0 \rangle \langle
\phi_{R^{\vec{n_2}}_{\vec{\alpha}}}^{a,b,2}, g_0 \rangle \langle
\phi_{R_{\vec{\alpha}}}^{a,b,2}, h \rangle \, d\alpha \\
&= \Big\langle \sum_{\vec{n} \in \mathbb{Z}^4} \prod_{j=1}^4
\frac{1}{(|n_j|+1)^5} T_{c'}^{a, b, [\vec{n}]}(f_0, g_0), h
\Big\rangle
\end{split}
\end{equation*}

As $h \in L^1$ is arbitrary, it follows that

\begin{equation*}
\widetilde{\Lambda_m^{(2)}(f,g)} = \sum_{\vec{n} \in \mathbb{Z}^4}
\prod_{j=1}^4 \frac{1}{(|n_j|+1)^5} \sum_{a,b=1}^3 T_{c'}^{a, b,
[\vec{n}]}(f_0, g_0)
\end{equation*}

\noindent almost everywhere.  We know $\|T_{c'}^{a, b, [\vec{n}]}
(f_0, g_0)\|_p \lesssim \prod_j (|n_j|+1) \|f\|_{p_1} \|g\|_{p_2}$
whenever $p_1, p_2 > 1$.  So, $\|\Lambda^{(2)}_m(f,g)\|_p \lesssim
\|f\|_{p_1} \|g\|_{p_2}$ whenever $p \geq 1$ follows immediately. By
Lemma~\ref{lemma:weaksum} (with $k = 2$),
$\|\Lambda^{(2)}_m(f,g)\|_{p,\infty} \lesssim \|f\|_{p_1}
\|g\|_{p_2}$ for all $p_1, p_2 > 1$.  By interpolation of these
cases, $\|\Lambda^{(2)}_m(f,g)\|_p \lesssim \|f\|_{p_1} \|g\|_{p_2}$
whenever $p_1, p_2 > 1$ and $p < 1$.

On the other hand, $\|T_{c'}^{a, b, [\vec{n}]}(f_0,
g_0)\|_{p,\infty} \lesssim \prod_j (|n_j|+1) \|f\|_{\llogl}
\|g\|_{p_2}$ whenever $p_1 = 1$.  By applying
Lemma~\ref{lemma:weaksum} again,
$\|\Lambda_m^{(2)}(f,g)\|_{p,\infty} \lesssim \|f\|_{\llogl}
\|g\|_{p_2}$.  The cases $p_2 = 1$ and $p_1 = p_2 = 1$ follow in the
same way.
\end{proof}

\chapter{Multi-parameter Multipliers}

Finally, we would like to consider multipliers, and their
corresponding operators, which are multi-parameter.  That is, $m$
acts as if the product of $s$ Coifman-Meyer multipliers.

For a vector $\vec{t} \in \mathbb{R}^{sd}$, let $\rho_j(\vec{t}) =
(t_j, t_{j+s}, \ldots, t_{j+s(d-1)}) \in \mathbb{R}^d$ for $1 \leq j
\leq s$. Conversely, for $1 \leq j \leq d$, let $\vec{t}_j =
(t_{s(j-1) +1}, \ldots, t_{js}) \in \mathbb{R}^s$ so that $\vec{t} =
(\vec{t}_1, \ldots, \vec{t}_d)$.

Let $m : \mathbb{R}^{sd} \rightarrow \mathbb{C}$ be smooth away from
the origin and uniformly bounded.  We say $m$ is an $s$-parameter
multiplier if

\begin{equation*}
|\partial^\alpha m(\vec{t})| \lesssim \prod_{j=1}^s
\|\rho_j(\vec{t})\|^{-|\rho_j(\alpha)|}
\end{equation*}

\noindent for all indices $|\alpha| \leq sd(d+3)$, where $\|\cdot\|$
is the Euclidean norm on $\mathbb{R}^d$.

Given such a multiplier $m$ on $\mathbb{R}^{sd}$ and $L^1$ functions $f_1,
\ldots, f_d : \mathbb{T}^s \rightarrow \mathbb{C}$, we define the
associated multiplier operator $\Lambda_m^{(s)}(f_1, \ldots, f_d) :
\mathbb{T}^s \rightarrow \mathbb{C}$ as

\begin{equation*}
\Lambda_m^{(s)}(f_1, \ldots, f_d)(\vec{x}) = \sum_{\vec{t} \in
\mathbb{Z}^{sd}} m(\vec{t}) \widehat{f}_1(\vec{t}_1) \cdots
\widehat{f}_d(\vec{t}_d) e^{2\pi i \vec{x} \cdot (\vec{t}_1 + \ldots
+ \vec{t}_d)}.
\end{equation*}

\noindent The $L^p$ estimates of previous chapters still hold with
minor modifications.

\begin{thm}\label{thm:last} For any $s$-parameter multiplier $m$
on $\mathbb{R}^{sd}$, $\Lambda^{(s)}_m : L^{p_1} \times \ldots
\times L^{p_d} \rightarrow L^p$ for $1 < p_j < \infty$ and
$\frac{1}{p_1} + \ldots + \frac{1}{p_d} = \frac{1}{p}$. If any or
all of the $p_j$ are equal to 1, this still holds with $L^p$
replaced by $L^{p,\infty}$ and $L^{p_j}$ replaced by $L(\log
L)^{s-1}$.  In particular, $\Lambda_m^{(s)} : L(\log L)^{s-1} \times
\ldots \times L(\log L)^{s-1}\rightarrow L^{1/d, \infty}$.
\end{thm}

In view of all the results, we now have a good view of the
heuristics.  Away from $p_j = 1$, each of these operators act the
same.  However, it is these endpoint cases which are the most
interesting.  Each time we go up a parameter, we ``gain a $\log$" at
the endpoint.

It will not be our goal in this chapter to explicitly prove this
result.  Indeed, it should be clear that the method of proof
employed on increasing complex multiplier operators throughout this
text can be used.  Instead, we give a brief survey of how the
argument would go.

By induction, we can assume this theorem is known for
$(s-1)$-parameter multipliers.  Like in the proof of Theorem~\ref{thm:bimult},
this allows us to assume $m$ is 0 on the planes $\{\rho_j(\vec{t}) =
0\}$.  Then, we can introduce bump functions which are the $s$-fold
tensor products of the functions in Theorem~\ref{thm:doubletbumps}
(as the functions in Remark~\ref{remark:tripletbumps} are the 2-fold
tensor products).  By the same dilation and translations, our
problem boils to understanding the appropriate paraproducts.

We say $Q \subset \mathbb{T}^{s}$ is a dyadic rectangle if $Q =
I_1 \times \ldots \times I_s$ for dyadic intervals $I_j$.  Define
$\varphi_Q : \mathbb{T}^s \rightarrow \mathbb{C}$ to be the
$s$-fold tensor product of adapted families.  The appropriate
(bilinear) paraproducts in this setting are

\begin{equation*}
T_\epsilon^{a_1, \ldots, a_s}(f, g)(\vec{x}) = \sum_Q \epsilon_Q
\frac{1}{|Q|^{1/2}} \langle \phi^1_Q, f \rangle \langle \phi^2_Q, g
\rangle \phi^3_Q(\vec{x})
\end{equation*}

\noindent where the sum is over all dyadic rectangles $Q$ and
$(\epsilon_Q)$ is a uniformly bounded sequence.  Each $a_j$ ranges
over $1, 2, 3$.  If $\phi^i_Q = \phi^i_{I_1} \oplus \ldots \oplus
\phi^i_{I_s}$, then $\int_\mathbb{T} \phi^i_{I_j} \, dx = 0$ whenever
$i \not= a_j$.

To complete the proof on $s$-parameter multiplier operators, it
suffices to show the associated paraproducts satisfy the same
bounds.  The stopping time argument presented in
Theorems~\ref{thm:Tweak}, \ref{thm:1para}, and~\ref{thm:2para} works
equally well in all dimensions, given the correct $s$-fold hybrid
operators.  For example, when $s=3$, we consider $SSS$, $SSM$,
$MSM$, etc.  Therefore, we will understand the paraproducts if we
can show each $s$-fold hybrid operator maps $L^p \rightarrow L^p$
for $1 < p < \infty$ and $L(\log L)^{s-1} \rightarrow L^{1,\infty}$.

For illustrative purposes, we show this for three specific operators
when $s = 3$. For $f : \mathbb{T}^3 \rightarrow \mathbb{C}$ define

\begin{gather*}
SSSf(x,y,z) = \bigg( \sum_Q \frac{|\langle \phi_Q, f
\rangle|^2}{|Q|} \chi_Q(x,y,z) \bigg)^{1/2}, \\
SSMf(x,y,z) = \bigg( \sum_{I_1} \sum_{I_2} \frac{\Big( \sup_{I_3}
\frac{1}{|I_3|^{1/2}} |\langle \phi_Q, f \rangle| \chi_{I_3}(z)
\Big)^2}{|I_1| |I_2|} \chi_{I_1}(x) \chi_{I_2}(y) \bigg)^{1/2},
\end{gather*}

\noindent and

\begin{equation*}
SMMf(x,y,z) = \bigg( \sum_{I_1} \frac{\Big( \sup_{I_2} \sup_{I_3}
\frac{1}{|I_2|^{1/2}} \frac{1}{|I_3|^{1/2}} |\langle \phi_Q, f
\rangle| \chi_{I_2}(y) \chi_{I_3}(z) \Big)^2}{|I_1|} \chi_{I_1}(x)
\bigg)^{1/2}.
\end{equation*}

Start with $SSSf$.  Using the same notational conveniences as
before,

\begin{equation*}
\begin{split}
SSSf &= \bigg( \sum_{I_1} \sum_{I_2} \sum_{I_3} \frac{1}{|I_3|}
\Big| \Big\langle \phi_{I_3}, \frac{\langle f, \phi_{I_1} \oplus
\phi_{I_2} \rangle}{|I_1|^{1/2} |I_2|^{1/2}} \chi_{I_1} \chi_{I_2}
\Big\rangle \Big|^2 \chi_{I_3}
\bigg)^{1/2} \\
&= \bigg( \sum_{I_1} \sum_{I_2} S_3 \Big(\frac{\langle f, \phi_{I_1}
\oplus \phi_{I_2} \rangle}{|I_1|^{1/2} |I_2|^{1/2}} \chi_{I_1}
\chi_{I_2} \Big)^2 \bigg)^{1/2}.
\end{split}
\end{equation*}

\noindent So,

\begin{equation*}
\begin{split}
\|SSSf\|_p &= \bigg\| \Big( \sum_{I_1} \sum_{I_2} S_3
\Big(\frac{\langle f, \phi_{I_1} \oplus \phi_{I_2}
\rangle}{|I_1|^{1/2} |I_2|^{1/2}}
\chi_{I_1} \chi_{I_2} \Big)^2 \Big)^{1/2} \bigg\|_p \\
&\lesssim \bigg\| \Big( \sum_{I_1} \sum_{I_2} \frac{|\langle f,
\phi_{I_1} \oplus \phi_{I_2} \rangle|^2}{|I_1|
|I_2|} \chi_{I_1} \chi_{I_2} \Big)^{1/2} \bigg\|_p \\
&= \bigg\| \Big( \sum_{I_1} S_2\Big( \frac{\langle f, \phi_{I_1}
\rangle}{|I_1|^{1/2}} \chi_{I_1} \Big)^2 \Big)^{1/2} \bigg\|_p
\lesssim \bigg\| \Big( \sum_{I_1} \frac{|\langle f, \phi_{I_1}
\rangle|^2}{|I_1|}
\chi_{I_1} \Big)^{1/2} \bigg\|_p \\
&= \|S_1f\|_p \lesssim \|f\|_p,
\end{split}
\end{equation*}

\noindent and

\begin{equation*}
\begin{split}
\|SSSf\|_{1,\infty} &= \bigg\| \Big( \sum_{I_1} \sum_{I_2} S_3
\Big(\frac{\langle f, \phi_{I_1} \oplus \phi_{I_2}
\rangle}{|I_1|^{1/2} |I_2|^{1/2}}
\chi_{I_1} \chi_{I_2} \Big)^2 \Big)^{1/2} \bigg\|_{1,\infty} \\
&\lesssim \bigg\| \Big( \sum_{I_1} S_2\Big( \frac{\langle f,
\phi_{I_1} \rangle}{|I_1|^{1/2}} \chi_{I_1} \Big)^2 \Big)^{1/2}
\bigg\|_1 \lesssim \|S_1f\|_{\llogl} \lesssim \|f\|_{L(\log L)^2}.
\end{split}
\end{equation*}

\noindent Using the same kind of argument

\begin{equation*}
\begin{split}
\|SSMf\|_p &= \bigg\| \Big( \sum_{I_1} \sum_{I_2} M_3'
\Big(\frac{\langle f, \phi_{I_1} \oplus \phi_{I_2}
\rangle}{|I_1|^{1/2} |I_2|^{1/2}}
\chi_{I_1} \chi_{I_2} \Big)^2 \Big)^{1/2} \bigg\|_p \\
&\lesssim \bigg\| \Big( \sum_{I_1} \sum_{I_2} \frac{|\langle f,
\phi_{I_1} \oplus \phi_{I_2} \rangle|^2}{|I_1|
|I_2|} \chi_{I_1} \chi_{I_2} \Big)^{1/2} \bigg\|_p \\
&= \bigg\| \Big( \sum_{I_1} S_2\Big( \frac{\langle f, \phi_{I_1}
\rangle}{|I_1|^{1/2}} \chi_{I_1} \Big)^2 \Big)^{1/2} \bigg\|_p
\lesssim \bigg\| \Big( \sum_{I_1} \frac{|\langle f, \phi_{I_1}
\rangle|^2}{|I_1|}
\chi_{I_1} \Big)^{1/2} \bigg\|_p \\
&= \|S_1f\|_p \lesssim \|f\|_p,
\end{split}
\end{equation*}

\noindent and

\begin{equation*}
\begin{split}
\|SSSf\|_{1,\infty} &= \bigg\| \Big( \sum_{I_1} \sum_{I_2} M_3'
\Big(\frac{\langle f, \phi_{I_1} \oplus \phi_{I_2}
\rangle}{|I_1|^{1/2} |I_2|^{1/2}}
\chi_{I_1} \chi_{I_2} \Big)^2 \Big)^{1/2} \bigg\|_{1,\infty} \\
&\lesssim \bigg\| \Big( \sum_{I_1} S_2\Big( \frac{\langle f,
\phi_{I_1} \rangle}{|I_1|^{1/2}} \chi_{I_1} \Big)^2 \Big)^{1/2}
\bigg\|_1 \lesssim \|S_1f\|_{\llogl} \lesssim \|f\|_{L(\log L)^2}.
\end{split}
\end{equation*}

\noindent By the same method used in Theorem~\ref{thm:MS} for $MM$,

\begin{equation*}
\sup_{I_2} \sup_{I_3} \frac{1}{|I_2|^{1/2}} \frac{1}{|I_3|^{1/2}}
|\langle \phi_Q, f \rangle| \chi_{I_1} \chi_{I_2} \chi_{I_3}
\lesssim M_2 \circ M_3(\langle \phi_{I_1}, f \rangle \chi_{I_1}).
\end{equation*}

\noindent Thus,

\begin{equation*}
\begin{split}
\|SMMf\|_p &\lesssim \bigg\| \Big( \sum_{I_1} M_2 \circ M_3 \Big(
\frac{\langle f, \phi_{I_1} \rangle}{|I_1|^{1/2}}
\chi_{I_1} \Big)^2 \Big)^{1/2} \bigg\|_p \\
&\lesssim \bigg\| \Big( \sum_{I_1} M_3 \Big( \frac{\langle f,
\phi_{I_1} \rangle}{|I_1|^{1/2}}
\chi_{I_1} \Big)^2 \Big)^{1/2} \bigg\|_p \\
&\lesssim \bigg\| \Big( \sum_{I_1} \frac{|\langle f, \phi_{I_1}
\rangle|^2}{|I_1|} \chi_{I_1} \Big)^{1/2} \bigg\|_p = \|S_1 f\|_p
\lesssim \|f\|_p,
\end{split}
\end{equation*}

\noindent and

\begin{equation*}
\begin{split}
\|SMMf\|_{1,\infty} &\lesssim \bigg\| \Big( \sum_{I_1} M_2 \circ M_3
\Big( \frac{\langle f, \phi_{I_1} \rangle}{|I_1|^{1/2}}
\chi_{I_1} \Big)^2 \Big)^{1/2} \bigg\|_{1,\infty} \\
&\lesssim \bigg\| \Big( \sum_{I_1} M_3 \Big( \frac{\langle f,
\phi_{I_1} \rangle}{|I_1|^{1/2}}
\chi_{I_1} \Big)^2 \Big)^{1/2} \bigg\|_1 \\
&\lesssim \bigg\| \Big( \sum_{I_1} \frac{|\langle f, \phi_{I_1}
\rangle|^2}{|I_1|} \chi_{I_1} \Big)^{1/2} \bigg\|_{\llogl} = \|S_1
f\|_{\llogl} \lesssim \|f\|_{L(\log L)^2}.
\end{split}
\end{equation*}

We also have $MMMf \lesssim M_Sf \leq M_1 \circ M_2 \circ M_3f$, for
which the desired estimates clearly hold.  Finally, we note as
before that $SMS$ and $MSS$ are pointwise smaller than a kind of
$SSM$, while $MMS$ and $MSM$ are smaller than a kind of $SMM$.

The recipe for arbitrary $s$-fold hybrid operators should now be
clear.  It suffices to consider only the ones of the form
$SS...SMM...M$.  In this case, the $M...MM$ part is pointwise
smaller than $M_j \circ M_{j+1} \circ \cdots \circ M_s$.  Repeated
iterations of Fefferman-Stein eliminate these $M_j$, while the
remaining $SS...S$ part can be dealt with as usual.

In conclusion, Theorem~\ref{thm:last} can be proven by the same
methods presented in earlier chapters, with only minor adjustments
here and there.

\begin{singlespacing}

\end{singlespacing}


\begin{thebibliography}{10}
\addcontentsline{toc}{chapter}{Bibliography}

\bibitem{inter}
C.~Bennett and R.~Sharpley.
\newblock {\em Interpolation of Operators}.
\newblock Academic Press, New York, 1988.

\bibitem{billingsley}
P.~Billingsley.
\newblock {\em Probabilty and Measure}.
\newblock Wiley-Interscience, New York, 1995.

\bibitem{cz}
A.~P. Calder\'{o}n and A.~Zygmund.
\newblock On the existence of certain singular integrals.
\newblock {\em Acta Mathematica}, 88:85--139, 1952.

\bibitem{christ}
M.~Christ and J.-L. Journ\'{e}.
\newblock Polynomial growth estimates for multilinear singular integral
  operators.
\newblock {\em Acta Mathematica}, 159:51--80, 1987.

\bibitem{meyer}
R.~Coifman and Y.~Meyer.
\newblock {\em Onelettes et Op\'{e}rateurs, III: Op\'{e}rateurs
  multilin\'{e}aires}.
\newblock Hermann, Paris, 1991.

\bibitem{fefferman}
C.~Fefferman.
\newblock Estimates for the double {Hilbert} transforms.
\newblock {\em Studia Mathematica}, 44:1--15, 1972.

\bibitem{feffermanstein}
C.~Fefferman and E.~M. Stein.
\newblock Some maximal inequalities.
\newblock {\em American Journal of Mathematics}, 93:107--15, 1971.

\bibitem{spanish}
J.~Garcia-Cuerva and J.~L.~Rubio de~Francia.
\newblock {\em Weighted Norm Inequalities and Related Topics}.
\newblock North-Holland, New York, 1985.

\bibitem{grafakos}
L.~Grafakos and R.~Torres.
\newblock Multilinear {Calder\'{o}n-Zygmund} theory.
\newblock {\em Advances in Mathematics}, 165:124--64, 2002.

\bibitem{hardylittlewood1}
G.~H. Hardy and J.~E. Littlewood.
\newblock A maximal theorem with function-theoretic appliations.
\newblock {\em Acta Mathematica}, 54:81--116, 1930.

\bibitem{hardyinequality}
G.~H. Hardy, J.~E. Littlewood, and G.~Polya.
\newblock {\em Inequalities}.
\newblock Cambridge University Press, Cambridge, 1932.

\bibitem{jessen}
B.~Jessen, J.~Marcinkiewicz, and A.~Zygmund.
\newblock Note on the differentiability of multiple integrals.
\newblock {\em Fundamenta Mathematicae}, 25:217--34, 1935.

\bibitem{journe}
J.-L. Journ\'{e}.
\newblock {Calder\'{o}n-Zygmund} operators on product spaces.
\newblock {\em Revista Matematica Iberoamericana}, 1:55--91, 1985.

\bibitem{kato}
T.~Kato and G.~Ponce.
\newblock Commutator estimates and the {Euler} and {Navier-Stokes} equations.
\newblock {\em Communications on Pure and Applied Mathematics}, 41:891--907,
  1988.

\bibitem{katznelson}
Y.~Katznelson.
\newblock {\em An Introduction to Harmonic Analysis}.
\newblock Dover Publications, New York, 1968.

\bibitem{kenig}
C.~Kenig and E.~M. Stein.
\newblock Multilinear estimates and fractional integration.
\newblock {\em Mathematical Research Letters}, 6:1--15, 1999.

\bibitem{me1}
S.~Lenhart, V.~Protopopescu, and J.~Workman.
\newblock Minimizing transient times in a couple solid-state laser model.
\newblock {\em Mathematical Methods in the Applied Sciences}, 29:373--86, 2006.

\bibitem{me2}
S.~Lenhart and J.~Workman.
\newblock An introduction to optimal control applied to immunology problems.
\newblock {\em Proceedings of Symposia in Applied Mathematics}, 64:85--99,
  2006.

\bibitem{me3}
S.~Lenhart and J.~Workman.
\newblock {\em Optimal Control Applied to Biological Models}.
\newblock CRC Press, New York, 2007.

\bibitem{lieb}
E.~H. Lieb and M.~Loss.
\newblock {\em Analysis}.
\newblock American Mathematical Society, Providence, RI, 2001.

\bibitem{littlewoodpaley1}
J.~E. Littlewood and R.~E. Paley.
\newblock Theorems on {Fourier} series and power series {(I)}.
\newblock {\em Journal of the London Mathematical Society}, 6:230--3, 1931.

\bibitem{littlewoodpaley2}
J.~E. Littlewood and R.~E. Paley.
\newblock Theorems on {Fourier} series and power series {(II)}.
\newblock {\em Journal of the London Mathematical Society}, 42:52--89, 1936.

\bibitem{littlewoodpaley3}
J.~E. Littlewood and R.~E. Paley.
\newblock Theorems on {Fourier} series and power series {(III)}.
\newblock {\em Journal of the London Mathematical Society}, 43:105--26, 1937.

\bibitem{marcin}
J.~Marcinkiewicz.
\newblock Sur les multiplicateurs des s\'{e}ries de {Fourier}.
\newblock {\em Studia Mathematica}, 8:78--91, 1939.

\bibitem{marcin2}
J.~Marcinkiewicz.
\newblock Sur l'interpolation d'op\'{e}rations.
\newblock {\em Comptes Rendus Mathematique Acad\'{e}mie des Sciences},
  208:172--3, 1939.

\bibitem{camil1}
C.~Muscalu, J.~Pipher, T.~Tao, and C.~Thiele.
\newblock Bi-parameter paraproducts.
\newblock {\em Acta Mathematica}, 193:269--96, 2004.

\bibitem{camil2}
C.~Muscalu, J.~Pipher, T.~Tao, and C.~Thiele.
\newblock Multi-parameter paraproducts.
\newblock {\em Revista Matematica Iberoamericana}, 22:963--76, 2006.

\bibitem{rudin}
W.~Rudin.
\newblock {\em Real and Complex Analysis}.
\newblock McGraw-Hill, New York, 1987.

\bibitem{steininter}
E.~M. Stein.
\newblock Interpolation of linear operators.
\newblock {\em Transactions of the American Mathematical Society}, 83:482--92,
  1956.

\bibitem{littlewoodpaley4}
E.~M. Stein.
\newblock On the functions of {Littlewood-Paley}, {Lusin}, and {Marcinkiewicz}.
\newblock {\em Transactions of the American Mathematical Society}, 88:430--66,
  1958.

\bibitem{steinllogl}
E.~M. Stein.
\newblock Note on the class $\llogl$.
\newblock {\em Studia Mathematica}, 32:305--10, 1969.

\bibitem{stein}
E.~M. Stein.
\newblock {\em Harmonic Analysis}.
\newblock Princeton University Press, Princeton, NJ, 1993.

\bibitem{stein3}
E.~M. Stein and R.~Shakarchi.
\newblock {\em Fourier Analysis: An Introduction}.
\newblock Princeton University Press, Princeton, NJ, 2003.

\bibitem{stein2}
E.~M. Stein and G.~Weiss.
\newblock {\em Introduction to Fourier Analysis on Euclidean Spaces}.
\newblock Princeton University Press, Princeton, NJ, 1971.

\bibitem{wiener}
N.~Wiener.
\newblock The ergodic theorem.
\newblock {\em Duke Mathematical Journal}, 5:1--18, 1939.
\end{thebibliography}
\end{document}